\documentclass[10pt,reqno]{amsart}
\usepackage{amsmath,amstext,amssymb,amsopn,amsthm,mathrsfs}
\usepackage{epsfig,graphics,latexsym,graphicx,color}
\allowdisplaybreaks

\newtheorem{theorem}{Theorem}[subsection]
\newtheorem{lemma}[theorem]{Lemma}
\newtheorem{corollary}[theorem]{Corollary}
\newtheorem{proposition}[theorem]{Proposition}

\newtheorem{remark}[theorem]{Remark}       

\numberwithin{equation}{subsection}

\begin{document}
\title[]{On the Capacities Associated with Local Muckenhoupt Weights}
\date{}
\author[Keng Hao Ooi]
{Keng Hao Ooi}
\address{Department of Mathematics,
National Central University,
No.300, Jhongda Rd., Jhongli City, Taoyuan County 32001, Taiwan (R.O.C.).
Department of Mathematics, National Taiwan Normal University, No. 88, Section 4, Tingzhou Rd, Wenshan District, Taipei City, 116, Taiwan (R.O.C.). 
Supported by Ministry of Science and Technology, NSTC 112-2811-M-003-023}
\email{kooi1@math.ncu.edu.tw}

\begin{abstract}
We develop a theory of capacities associated with local Muckenhoupt weights. Fundamental properties of local Muckenhoupt weights will be revisited. Weak type boundedness of nonlinear potential and capacitary strong type inequalities associated with such weights will be addressed. The boundedness of the local maximal function on the spaces of Choquet integrals associated with such weighted capacities will be justified as an application of the theory. We also address on the thinness of sets with the Kellogg property as another application. 
\end{abstract}

\maketitle

\tableofcontents

\newpage
\section{Introduction}
To measure the discontinuities of Sobolev functions, one uses the capacities. To begin with, let $n\in\mathbb{N}$, $0<\alpha<\infty$, and $1<p<\infty$. We define the Bessel capacities ${\rm Cap}_{\alpha,p}(\cdot)$ by 
\begin{align}\label{pre bessel}
\text{Cap}_{\alpha,p}(E)&=\inf\left\{\|f\|_{L^{p}(\mathbb{R}^{n})}^{p}:f\geq 0,~G_{\alpha}\ast f\geq 1~\text{on}~E\right\},
\end{align}
where $E\subseteq\mathbb{R}^{n}$ is an arbitrary set,
\begin{align*}
G_{\alpha}(x)=\mathcal{F}^{-1}\left[\left(1+4\pi^{2}\left|\cdot\right|^{2}\right)^{-\frac{\alpha}{2}}\right](x),\quad x\in\mathbb{R}^{n},
\end{align*}
is the Bessel kernel, and $\mathcal{F}^{-1}$ is the inverse distributional Fourier transform on $\mathbb{R}^{n}$. A Sobolev function $f$ is of the form that $f=G_{\alpha}\ast g$ for some $g\in L^{p}(\mathbb{R}^{n})$. It can be proved that for any $\varepsilon>0$, there is an open set $G\subseteq\mathbb{R}^{n}$ such that ${\rm Cap}_{\alpha,p}(G)<\varepsilon$ and $f|_{G^{c}}$ is a continuous function on $G^{c}$ (see \cite[Proposition 6.1.2]{AH2}). Further, the standard Sobolev embedding theorem states that 
\begin{align*}
|E|^{\frac{n-\alpha p}{n}}\leq C(n,\alpha,p){\rm Cap}_{\alpha,p}(E),\quad E\subseteq\mathbb{R}^{n},\quad 1<p<\frac{n}{\alpha},
\end{align*}
where we use the convention that $C(\alpha,\beta,\gamma,\ldots)$ to designate the positive constant depending only on the parameters $\alpha,\beta,\gamma,\ldots$. The previous embedding shows that the Bessel capacities ${\rm Cap}_{\alpha,p}(\cdot)$ can be viewed as a refined Lebesgue measure as ${\rm Cap}_{\alpha,p}(E)=0$ entails $|E|=0$. Hence ${\rm Cap}_{\alpha,p}(\cdot)$ is more sensitive in measuring the discontinuities of Sobolev functions than the Lebesgue measure.

A very thorough study of the theory of capacities is given by Maz'ya and Havin \cite{MH}. Landkof \cite{LN} also gave a modern perspective of such a theory. While Adams and Hedberg \cite{AH2} also provided an excellent standard graduate text in that area. Later, Adams \cite{AD} developed the weighted version of the theory of capacities. A modern book on weighted capacities is due to Turesson \cite{TB}. The weighted theories therein all assume the Muckenhoupt $A_{p}$ condition on such weights. For instance, \cite[Theorem 3.1.2]{TB}, which known to be the Muckenhoupt-Wheeden type inequality, states that 
\begin{align}\label{TB 312}
\int_{\mathbb{R}^{n}}(I_{\alpha,\rho}\ast\mu)(x)^{p}\omega(x)dx\leq C(n,\alpha,p,[\omega]_{A_{\infty}})\int_{\mathbb{R}^{n}}M_{\alpha,\rho}\mu(x)^{p}\omega(x)^{p}dx,
\end{align}
where $\mu$ is a positive measure on $\mathbb{R}^{n}$, $I_{\alpha,\rho}(\cdot)=\left|\cdot\right|^{-(n-\alpha)}\chi_{\left|\cdot\right|\leq\rho}$ is the local Riesz kernel, $M_{\alpha,\rho}\mu(\cdot)=\sup\limits_{0<r\leq\rho}r^{-(n-\alpha)}\mu(B_{r}(\cdot))$, $0<\rho<\infty$, and $\omega$ is an $A_{\infty}$ weight. While the estimate (\ref{TB 312}) is used to show that the weighted local Riesz $R_{\alpha,p;\rho}^{\omega}(\cdot)$ and weighted Bessel $B_{\alpha,p}^{\omega}(\cdot)$ capacities are equivalent that 
\begin{align}\label{TB equi}
C(n,\alpha,p,[\omega]_{A_{p}})^{-1}R_{\alpha,p;\rho}^{\omega}(E)\leq B_{\alpha,p}^{\omega}(E)\leq C(n,\alpha,p,[\omega]_{A_{p}})R_{\alpha,p;\rho}^{\omega}(E),
\end{align}
where $E\subseteq\mathbb{R}^{n}$ is arbitrary, $\omega$ is an $A_{p}$ weight, 
\begin{align*}
R_{\alpha,p;\rho}^{\omega}(E)&=\inf\left\{\|f\|_{L^{p}(\omega)}^{p}:f\geq 0,~I_{\alpha,\rho}\ast f\geq 1~\text{on}~E\right\},\\
B_{\alpha,p}^{\omega}(E)&=\inf\left\{\|f\|_{L^{p}(\omega)}^{p}:f\geq 0,~G_{\alpha}\ast f\geq 1~\text{on}~E\right\},
\end{align*}
and $\|f\|_{L^{p}(\omega)}=\left(\int_{\mathbb{R}^{n}}|f(x)|^{p}\omega(x)dx\right)^{\frac{1}{p}}$ (see \cite[Theorem 3.3.7 and Lemma 3.3.8]{TB}). However, there are some results which suggest that the weight condition on (\ref{TB 312}) and (\ref{TB equi}) is not optimal. Let ${\bf M}^{\rm loc}$ be the (centered) local Hardy-Littlewood maximal function defined by 
\begin{align*}
{\bf M}^{\rm loc}f(x)=\sup_{0<r\leq 1}\frac{1}{|B_{r}(x)|}\int_{B_{r}(x)}|f(y)|dy,
\end{align*}
where $f$ is locally integrable on $\mathbb{R}^{n}$. A weight $\omega$ belongs to local Muckenhoupt $A_{1}^{\rm loc}$ class if there is a constant $C>0$ such that
\begin{align*}
{\bf M}^{\rm loc}\omega(x)\leq C\omega(x)\quad\text{a.e.}.
\end{align*}
Let $V^{\mu}=G_{\alpha}\ast(G_{\alpha}\ast\mu)^{p'-1}$ be the nonlinear potential of ${\rm Cap}_{\alpha,p}(\cdot)$. Roughly speaking, $V^{\mu}$ is the minimizer of (\ref{pre bessel}). It is shown in \cite[Theorem 3.1]{OP} that $(V^{\mu})^{\delta}$ is an $A_{1}^{\rm loc}$ weight for admissible exponents $\delta$. On the other hand, if we define the Choquet integrals by
\begin{align*}
\int_{\mathbb{R}^{n}}|f|d{\rm Cap}_{\alpha,p}=\int_{0}^{\infty}{\rm Cap}_{\alpha,p}(\{x\in\mathbb{R}^{n}:|f(x)|>t\})dt
\end{align*}
for admissible functions $f$, then it is shown in \cite{OP2} that 
\begin{align*}
\int_{\mathbb{R}^{n}}\left({\bf M}^{\rm loc}f\right)^{q}d{\rm Cap}_{\alpha,p}\leq C(n,\alpha,p,q)\int_{\mathbb{R}^{n}}|f|^{q}d{\rm Cap}_{\alpha,p},\quad q>\frac{n-\alpha p}{n},
\end{align*}
where ${\bf M}^{\rm loc}$ is replaced by the usual Hardy-Littlewood maximal function ${\bf M}$ if the Riesz capacities ${\rm cap}_{\alpha,p}(\cdot)$ (defined by replacing $G_{\alpha}(\cdot)$ with $I_{\alpha}(\cdot)=\left|\cdot\right|^{-(n-\alpha)}$ in (\ref{pre bessel})) are taken into account. In view of the aforementioned results, one may reasonably suspect that the weighted theory corresponding to (\ref{TB 312}) and (\ref{TB equi}) should work with local Muckenhoupt $A_{p}^{\rm loc}$ class. By $A_{p}^{\rm loc}$ class we mean that 
\begin{align*}
[\omega]_{A_{p}^{\rm loc}}&=\sup_{|Q|\leq 1}\left(\frac{1}{|Q|}\int_{Q}\omega(x)dx\right)\left(\frac{1}{|Q|}\int_{Q}\omega(x)^{-\frac{1}{p-1}}dx\right)^{p-1},\quad 1<p<\infty,\\
[\omega]_{A_{p}^{\rm loc}}&=\sup_{|Q|\leq 1}\left(\frac{1}{|Q|}\int_{Q}\omega(x)dx\right)\left\|\omega^{-1}\right\|_{L^{\infty}(Q)},\quad p=1,\quad Q~\text{is any cube in}~\mathbb{R}^{n},
\end{align*}
it can shown that the maximal characterization of $A_{1}^{\rm loc}$ coincides with the above definition (see Proposition \ref{repeating}). The definition of $A_{p}^{\rm loc}$ for $1<p<\infty$ goes back to Rychkov \cite{RV}. We will show that (\ref{TB 312}) holds for $\omega\in A_{\infty}^{\rm loc}=\bigcup\limits_{1\leq p<\infty}A_{p}^{\rm loc}$ (see Proposition \ref{312}). Moreover, it will be shown that 
\begin{align}\label{pre local}
C^{-1}R_{\alpha,p;\rho_{2}}^{\omega}(E)\leq R_{\alpha,p;\rho_{1}}^{\omega}(E)\leq CR_{\alpha,p;\rho_{2}}^{\omega}(E),\quad E\subseteq\mathbb{R}^{n}
\end{align}
for some suitable constant $C>0$, where $\omega^{-\frac{1}{p-1}}\in A_{\infty}^{\rm loc}$ (see Proposition \ref{338}). Meanwhile, if $\omega$ is an $A_{p}^{\rm loc}$ weight that does not decay too fast, then (\ref{TB equi}) holds (see Theorem \ref{337}). On the other hand, $B_{\alpha,p}^{\omega}(\cdot)$ becomes trivial, i.e., $B_{\alpha,p}^{\omega}(\mathbb{R}^{n})=0$ for $\omega=e^{\delta\left|\cdot\right|}\in A_{p}^{\rm loc}$, where $\delta>0$ is large enough (see Theorem \ref{becoming zero}). This also reminisces a result from Aikawa \cite{AH} that the weighted Riesz capacities $R_{\alpha,p}^{\omega}(\cdot)$ are trivial for rapidly decay $\omega$.

The proof of (\ref{pre local}), which mainly relies on (\ref{TB 312}), is not verbatim to that of \cite[Theorem 3.1.2]{TB}. We need to introduce the local Muckenhoupt $A_{p;\rho}^{\rm loc}$ class with scale $0<\rho<\infty$, which is by definition, the condition $|Q|\leq 1$ being replaced by the length of cubes $\ell(Q)\leq\rho$ in the previous definition of $A_{p}^{\rm loc}$. We will prove that
\begin{align*}
\int_{\mathbb{R}^{n}}(I_{\alpha,\rho}\ast\mu)(x)^{p}\omega(x)dx\leq C(n,\alpha,p,[\omega]_{A_{\infty;\rho}^{\rm loc}})\int_{\mathbb{R}^{n}}M_{\alpha,\rho}\mu(x)^{p}\omega(x)^{p}dx
\end{align*}
(see Proposition \ref{312}). The constant $C(n,\alpha,p,[\omega]_{A_{\infty;\rho}^{\rm loc}})$ plays a crucial role in our theory. For fixed $n$, $\alpha$, and $p$, we will require the function $\mathcal{N}:[0,\infty)\rightarrow[0,\infty)$ defined by $\mathcal{N}(\cdot)=C(n,\alpha,p,\cdot)$ is increasing, and this will be needed in proving that 
\begin{align*}
R_{\alpha,p;\rho}^{\omega}(E)\leq C(n,\alpha,p,[\omega]_{A_{\infty;20\rho}^{\rm loc}})R_{\alpha,p;2\rho}^{\omega}(E),\quad E\subseteq\mathbb{R}^{n},
\end{align*}
and (\ref{pre local}) will then follow (see Remarks \ref{crucial decreasing} and \ref{very crucial remark}). To obtain the fact that $\mathcal{N}(\cdot)$ is monotonic increasing, we first prove $A_{p;\rho_{1}}^{\rm loc}=A_{p;\rho_{2}}^{\rm loc}$ for any $0<\rho_{1},\rho_{2}<\infty$, and there exists a constant $C>0$ such that 
\begin{align*}
C^{-1}[\omega]_{A_{p;\rho_{2}}^{\rm loc}}\leq[\omega]_{A_{p;\rho_{1}}^{\rm loc}}\leq C[\omega]_{A_{p;\rho_{2}}^{\rm loc}},
\end{align*}
while a very careful analysis on the parameters in $C$ needs to be taken throughout the proof (see, for example, Theorems \ref{important}, \ref{further properties Ainfty}, and Remark \ref{reverse Ap}). The analysis of such a constant $C$ is not covered in Rychkov \cite{RV}.

Now we introduce the capacitary strong type inequality (CSI), which reads as 
\begin{align*}
\int_{\mathbb{R}^{n}}(G_{\alpha}\ast f)(x)^{p}d{\rm Cap}_{\alpha,p}\leq C(n,\alpha,p)\int_{\mathbb{R}^{n}}f(x)^{p}dx,
\end{align*}
where $f\in L^{p}(\mathbb{R}^{n})$ is any nonnegative function (see \cite[Theorem 7.1.1]{AH2}). It is used in studying the trace class inequalities and the compactness of embedding of potentials (see \cite{MV2}). The weighted analogue of such an inequality is given by \cite[Theorem 5.1]{AD}, where the weighted Riesz capacities and $A_{\infty}$ weights are assumed therein. Our corresponding result tackles with $R_{\alpha,p;\rho}^{\omega}(\cdot)$ and $A_{\infty}^{\rm loc}$ weights (see Theorem \ref{CSI}). The proof of Theorem \ref{CSI} uses bounded maximum principle for Wolff potential, to wit:
\begin{align*}
R_{\alpha,p;\rho}^{\omega}\left(\left\{x\in\mathbb{R}^{n}:W_{\omega;\rho}^{\mu}(x)>t\right\}\right)\leq C(n)\frac{\mu(\mathbb{R}^{n})}{t^{p-1}},\quad 0<t<\infty
\end{align*}
(see Theorem \ref{bounded}). Such a principle also appears in \cite[Theorem 2.6.3]{AH2} for convolution kernels. However, \cite[Theorem 2.6.3]{AH2} is not applicable to the Wolff potential $W_{\omega;\rho}^{\mu}$ since it is not of convolution type. As an application of the CSI, we obtain the boundedness of ${\bf M}_{\rho}^{\rm loc}$ (defined by replacing $0<r\leq 1$ with $\ell(Q)\leq\rho$ in ${\bf M}^{\rm loc}$) on the spaces of Choquet integrals associated with $R_{\alpha,p;\rho}^{\omega}(\cdot)$, to wit:
\begin{align}\label{pre strong}
\int_{\mathbb{R}^{n}}\left({\bf M}_{\rho}^{\rm loc}f\right)^{q}d\mathcal{R}_{\alpha,p;\rho}^{\omega}\leq C(n,\alpha,p,[\omega]_{A_{1;160\rho}^{\rm loc}})\int_{\mathbb{R}^{n}}|f|^{q}d\mathcal{R}_{\alpha,p;\rho}^{\omega},~ q>\frac{n-\alpha p}{n}
\end{align}
(see Theorem \ref{main2}), where the capacities $\mathcal{R}_{\alpha,p;\rho}^{\omega}(\cdot)$ is an equivalent variant of $R_{\alpha,p;\rho}^{\omega}(\cdot)$ (see Theorem \ref{R cal}). When $\omega=1$, the estimate (\ref{pre strong}) reduces to \cite[Theorem 1.2]{OP2}, and it has been shown in \cite{OP2} that the exponent $\frac{n-\alpha p}{n}$ is sharp. A related result to the estimate (\ref{pre strong}) is given in \cite{OV}, where Hausdorff content is taken into account instead of capacities. However, the argument given in \cite{OV} uses certain type of covering lemma of dyadic cubes, and this argument seems not to work for capacities since capacities are not defined in terms of covering. When $q=\frac{n-\alpha p}{n}$, the estimate (\ref{pre strong}) reduces to a weak type estimate (see Theorem \ref{main1}).

The last application in this paper is about the thinness of sets. It is used in studying the regularity of the Dirichlet problem for the $p$-Laplace equation $\Delta_{p}u=0$ (see Maz'ya \cite{MV} and Kilpel\"{a}inen and Mal\'{y} \cite{KM}). The Kellogg property will also be addressed (see Theorem \ref{kellogg}). 

We end this section by listing out all the main notations used in this paper. The sets $\mathbb{N}$, $\mathbb{Z}$, $\mathbb{Q}$, and $\mathbb{R}$ have their usual meanings. The symbol $C(\alpha,\beta,\gamma,\ldots)$ designates the positive constant depending only on the parameters $\alpha,\beta,\gamma,\ldots$. Given two quantities $A$ and $B$, we write $A\approx B$ to refer $C^{-1}B\leq A\leq CB$, where the constant $C>0$ is independent of the main parameters in $A$ and $B$. The characteristic function of a set $E\subseteq\mathbb{R}^{n}$ is denoted by $\chi_{E}$. Denote by $\left|\cdot\right|_{\infty}$ the Euclidean maximum norm that
\begin{align*}
|x|_{\infty}=\max\{|x_{1}|,\ldots,|x_{n}|\},\quad x=(x_{1},\ldots,x_{n})\in\mathbb{R}^{n}.
\end{align*}
For the word ``cube", we always refer to open cube with symbol $Q$. We denote the closed cubes by $\overline{Q}$. We write $Q_{r}(x)$ the cube with length $\ell(Q)=2r$ and center $x$, and $\overline{Q}_{r}(x)$ is the closed version of $Q_{r}(x)$. For any $\lambda>0$ and cube $Q$, we write $\lambda Q$ to refer the cube concentric to $Q$ with side length $\lambda$ times that of $Q$. Similar notation will do for $\lambda\overline{Q}$ for closed cube $\overline{Q}$. Suppose that $\omega$ is a weight. Then we write $\omega(E)=\int_{E}\omega(x)dx$ for Lebesgue measurable sets $E\subseteq\mathbb{R}^{n}$. While ${\rm Avg}_{E}f=\frac{1}{|E|}\int_{E}f(x)dx$ designates the average of a Lebesgue measurable function $f$ on $E$ with positive Lebesgue measure $|E|$. For any weight $\omega$ and $1<p<\infty$, we write $\omega'=\omega^{-\frac{1}{p-1}}=\omega^{1-p'}$, where $p'$ is the H\"{o}lder's conjugate of $p$ that $\frac{1}{p}+\frac{1}{p'}=1$. For any $0<p<\infty$, the space $L^{p}(\omega)$ consists of Lebesgue measurable functions $f:\mathbb{R}^{n}\rightarrow[-\infty,\infty]$ with 
\begin{align*}
\|f\|_{L^{p}(\omega)}=\left(\int_{\mathbb{R}^{n}}|f(x)|^{p}\omega(x)dx\right)^{\frac{1}{p}}<\infty.
\end{align*}
Given a positive measure $\mu$ on $\mathbb{R}^{n}$, the space $L_{\mu}^{p}(\mathbb{R}^{n})$ for $0<p<\infty$ consists of $\mu$-measurable functions $f:\mathbb{R}^{n}\rightarrow[-\infty,\infty]$ with
\begin{align*}
\|f\|_{L_{\mu}^{p}(\mathbb{R}^{n})}=\left(\int_{\mathbb{R}^{n}}|f(x)|^{p}d\mu(x)\right)^{\frac{1}{p}}<\infty.
\end{align*}
The symbols $L^{p}(\omega)^{+}$ and $L_{\mu}^{p}(\mathbb{R}^{n})^{+}$ are the positive cone (nonnegative elements) of $L^{p}(\omega)$ and $L_{\mu}^{p}(\mathbb{R}^{n})$ respectively. The space of positive measures on $\mathbb{R}^{n}$ with support in $E$ is denoted by $\mathcal{M}^{+}(E)$. Finally, the space $C_{0}(\Omega)$ consists of continuous functions $f$ on $\Omega$ with compact support, where $\Omega$ is an open subset of $\mathbb{R}^{n}$.

\newpage
\section{On the Local Muckenhoupt Class $A_{p;\rho}^{\rm loc}$}
Let $0<\rho<\infty$, $1\leq p<\infty$, and $n\in\mathbb{N}$. A nonnegative measurable function $\omega$ is said to be a weight provided that $\omega$ is locally integrable on $\mathbb{R}^{n}$ and $\omega(x)>0$ almost everywhere. We define the local Muckenhoupt class $A_{p;\rho}^{\rm loc}$ to consist of all weights $\omega$ such that 
\begin{align*}
[\omega]_{A_{p;\rho}^{\rm loc}}&=\sup_{\ell(Q)\leq\rho}\left(\frac{1}{|Q|}\int_{Q}\omega(x)dx\right)\left(\frac{1}{|Q|}\int_{Q}\omega(x)^{-\frac{1}{p-1}}dx\right)^{p-1},\quad 1<p<\infty,\\
[\omega]_{A_{p;\rho}^{\rm loc}}&=\sup_{\ell(Q)\leq\rho}\left(\frac{1}{|Q|}\int_{Q}\omega(x)dx\right)\left\|\omega^{-1}\right\|_{L^{\infty}(Q)},\quad p=1,
\end{align*}
where the suprema are taken over all cubes $Q$ with volumes $\ell(Q)\leq\rho$. When the condition $\ell(Q)\leq\rho$ is replaced to all cubes $Q$ with arbitrary volumes, we obtain the usual Muckenhoupt $A_{p}$ class. As an immediate result, we have $A_{p}\subseteq A_{p;\rho}^{\rm loc}$.

\subsection{Basic Properties of $A_{p;\rho}^{\rm loc}$}
\enskip

We begin with the following basic properties of $A_{p;\rho}^{\rm loc}$.
\begin{proposition}\label{many properties}
Let $0<\rho<\infty$, $1\leq p<\infty$, and $\omega\in A_{p;\rho}^{\rm loc}$. Then
\begin{enumerate}
\item $[\delta^{\lambda}(\omega)]_{A_{p;\rho}^{\rm loc}}=[\omega]_{A_{p;\rho/\lambda}^{\rm loc}}$, where $\delta^{\lambda}(\omega)(x)=\omega(\lambda x_{1},\ldots,\lambda x_{n})$, $x\in\mathbb{R}^{n}$, and $\lambda>0$.

\medskip
\item $[\tau^{z}\omega]_{A_{p;\rho}^{\rm loc}}=[\omega]_{A_{p;\rho}^{\rm loc}}$, where $\tau^{z}(\omega)(x)=\omega(x-z)$, $x,z\in\mathbb{R}^{n}$.

\medskip
\item $[\lambda\omega]_{A_{p;\rho}^{\rm loc}}=[\omega]_{A_{p;\rho}^{\rm loc}}$, where $\lambda>0$.

\medskip 
\item If $1<p<\infty$, then the weight $\omega^{1-p'}\in A_{p';\rho}^{\rm loc}$ with
\begin{align*}
[\omega^{1-p'}]_{A_{p';\rho}^{\rm loc}}=[\omega]_{A_{p;\rho}^{\rm loc}}^{p'-1}.
\end{align*}
Therefore, $\omega\in A_{2;\rho}^{\rm loc}$ if and only if $\omega^{-1}\in A_{2;\rho}^{\rm loc}$ with $[\omega^{-1}]_{A_{2;\rho}^{\rm loc}}=[\omega]_{A_{2;\rho}^{\rm loc}}$.

\medskip
\item $[\omega]_{A_{p;\rho}^{\rm loc}}\geq 1$ for all $\omega\in A_{p;\rho}^{\rm loc}$. The equality holds if and only if $\omega$ is almost everywhere a constant.

\medskip
\item For $1\leq p<q<\infty$, it holds that $A_{p;\rho}^{\rm loc}\subseteq A_{q;\rho}^{\rm loc}$ with
\begin{align*}
[\omega]_{A_{q;\rho}^{\rm loc}}\leq[\omega]_{A_{p;\rho}^{\rm loc}}.
\end{align*}

\medskip
\item If $\omega\in A_{1}^{\rm loc}$, then
\begin{align*}
\lim_{q\rightarrow 1^{+}}[\omega]_{A_{q;\rho}^{\rm loc}}=[\omega]_{A_{1;\rho}^{\rm loc}}.
\end{align*}

\medskip 
\item The following formula
\begin{align*}
[\omega]_{A_{p;\rho}^{\rm loc}}=\sup_{\substack{\ell(Q)\leq\rho\\ f\in L^{p}(Q,\omega dx)\\\int_{Q}|f|^{p}\omega dx>0}}\frac{\left(\dfrac{1}{|Q|}\displaystyle\int_{Q}|f(x)|dx\right)^{p}}{\dfrac{1}{\omega(Q)}\displaystyle\int_{Q}|f(x)|^{p}\omega(x)dx}
\end{align*}
holds.
\end{enumerate}
\end{proposition}

\begin{proof}
\noindent{\bf Property (1):} For $Q=Q_{r}(c)$, it holds that $\{\lambda^{-1}x:x\in Q\}=Q_{\frac{r}{\lambda}}\left(\frac{c}{\lambda}\right)$ and $|\{\lambda^{-1}x:x\in Q\}|=\lambda^{-n}|Q|$. The result follows by the standard dilation formula for the Lebesgue integrals.

\medskip
\noindent{\bf Property (2):} We have $|z+Q|=|Q|$, $z\in\mathbb{R}^{n}$. The result follows by the translation formula for the Lebesgue integrals.

\medskip
\noindent{\bf Property (3):} Trivial.

\medskip
\noindent{\bf Property (4):} The result follows by switching the duality role of $\omega$ in the definition of $A_{p;\rho}^{\rm loc}$. 

\medskip
\noindent{\bf Property (5):} As in the proof of Proposition \ref{local strong}, we have
\begin{align*}
1=\frac{1}{|Q|}\int_{Q}\omega(x)^{\frac{1}{p}}\omega(x)^{-\frac{1}{p}}dx\leq[\omega]_{A_{p;\rho}^{\rm loc}}.
\end{align*}
The equality holds if and only if $\omega(x)^{\frac{1}{p}}=c\cdot\omega(x)^{-\frac{1}{p}}$ almost everywhere for some $c>0$.

\medskip
\noindent{\bf Property (6):} Observe that $0<q'-1<p'-1\leq\infty$ and the statement 
\begin{align*}
[\omega]_{A_{q;\rho}^{\rm loc}}\leq[\omega]_{A_{p;\rho}^{\rm loc}}
\end{align*}
is equivalent to
\begin{align*}
\|\omega^{-1}\|_{L^{q'-1}(Q,\frac{dx}{|Q|})}\leq\|\omega^{-1}\|_{L^{p'-1}(Q,\frac{dx}{|Q|})},
\end{align*}
which holds by H\"{o}lder's inequality.

\medskip
\noindent{\bf Property (7):} Note that {\bf Property (6)} implies that
\begin{align}\label{first entailment}
[\omega]_{A_{q;\rho}^{\rm loc}}\leq[\omega]_{A_{1;\rho}^{\rm loc}},\quad q>1.
\end{align}
Let $\eta<[\omega]_{A_{1;\rho}^{\rm loc}}$ be given. There is a cube $Q\subseteq\mathbb{R}^{n}$ with $\ell(Q)\leq\rho$ and  
\begin{align*}
\eta<\left(\frac{1}{|Q|}\int_{Q}\omega(x)dx\right)\|\omega^{-1}\|_{L^{\infty}(Q)}.
\end{align*}
It is a standard fact
\begin{align*}
\lim_{p\rightarrow\infty}\|f\|_{L^{p}(X,\mu)}=\|f\|_{L^{\infty}(X,\mu)}
\end{align*}
provided that $f\in L^{p_{0}}(X,\mu)$ for some $0<p_{0}<\infty$ and $(X,\mu)$ is an arbitrary measure space with positive measure $\mu$ on $X$. As a consequence,
\begin{align*}
\eta&<\lim_{p\rightarrow\infty}\left(\frac{1}{|Q|}\int_{Q}\omega(x)dx\right)\left(\frac{1}{|Q|}\int_{Q}\omega(x)^{-p}dx\right)^{\frac{1}{p}}\\
&=\lim_{q\rightarrow 1^{+}}\left(\frac{1}{|Q|}\int_{Q}\omega(x)dx\right)\left(\frac{1}{|Q|}\int_{Q}\omega(x)^{-\frac{1}{q-1}}dx\right)^{q-1}\\
&\leq\liminf_{q\rightarrow 1^{+}}[\omega]_{A_{q;\rho}^{\rm loc}}.
\end{align*}
Since $\eta<[\omega]_{A_{1;\rho}^{\rm loc}}$ is arbitrary, one has $[\omega]_{A_{1;\rho}^{\rm loc}}\leq\liminf\limits_{q\rightarrow 1^{+}}[\omega]_{A_{q;\rho}^{\rm loc}}$, which yields 
\begin{align*}
\lim_{q\rightarrow 1^{+}}[\omega]_{A_{q;\rho}^{\rm loc}}=[\omega]_{A_{1;\rho}^{\rm loc}}
\end{align*}
by combining (\ref{first entailment}).

\medskip
\noindent{\bf Property (8):} For $1<p<\infty$, one has by H\"{o}lder's inequality that 
\begin{align*}
&\left(\frac{1}{|Q|}\int_{Q}|f(x)|dx\right)^{p}\\
&=\left(\frac{1}{|Q|}\int_{Q}|f(x)|\omega(x)^{\frac{1}{p}}\omega(x)^{-\frac{1}{p}}dx\right)^{p}\\
&\leq\frac{1}{|Q|^{p}}\left(\int_{Q}|f(x)|^{p}\omega(x)dx\right)\left(\int_{Q}\omega(x)^{-\frac{p'}{p}}dx\right)^{\frac{p}{p'}}\\
&=\left(\frac{1}{\omega(Q)}\int_{Q}|f(x)|^{p}\omega(x)dx\right)\left(\frac{1}{|Q|}\int_{Q}\omega(x)dx\right)\left(\frac{1}{|Q|}\int_{Q}\omega(x)^{-\frac{1}{p-1}}dx\right)^{p-1}\\
&\leq[\omega]_{A_{p;\rho}^{\rm loc}}\left(\frac{1}{\omega(Q)}\int_{Q}|f(x)|^{p}\omega(x)dx\right).
\end{align*}
The above estimates also hold for $p=1$ by obvious modification, which proves the $\geq$ direction of the inequality in {\bf Property 8}. For the other direction, we simply let $f=\omega^{\frac{p'}{p}}$.
\end{proof}

If $k\geq 0$ is a bounded function with strictly positive lower bound, then $k\omega$ is still an $A_{p;\rho}^{\rm loc}$ weight for $\omega\in A_{p;\rho}^{\rm loc}$, as the following shown.
\begin{proposition}
Let $0<\rho<\infty$ and $k$ be a nonnegative measurable function such that $k,k^{-1}\in L^{\infty}(\mathbb{R}^{n})$. If $\omega\in A_{p;\rho}^{\rm loc}$ for some $1\leq p<\infty$, then $k\omega\in A_{p;\rho}^{\rm loc}$ with
\begin{align*}
[k\omega]_{A_{p;\rho}^{\rm loc}}\leq\|k\|_{L^{\infty}(\mathbb{R}^{n})}\|k^{-1}\|_{L^{\infty}(\mathbb{R}^{n})}\|\omega\|_{A_{p;\rho}^{\rm loc}}.
\end{align*}
\end{proposition}

\begin{proof}
Let $Q$ be a cube with $\ell(Q)\leq\rho$. Assume that $p=1$. Then
\begin{align*}
&\left(\frac{1}{|Q|}\int_{Q}k(x)\omega(x)dx\right)\left\|(k\omega)^{-1}\right\|_{L^{\infty}(Q)}\\
&\leq\|k\|_{L^{\infty}(Q)}\|k^{-1}\|_{L^{\infty}(Q)}\left(\frac{1}{|Q|}\int_{Q}\omega(x) dx\right)\|\omega^{-1}\|_{L^{\infty}(Q)}\\
&\leq\|k\|_{L^{\infty}(\mathbb{R}^{n})}\|k^{-1}\|_{L^{\infty}(\mathbb{R}^{n})}\|\omega\|_{A_{1;\rho}^{\rm loc}},
\end{align*}
which yields $[k\omega]_{A_{1;\rho}^{\rm loc}}\leq\|k\|_{L^{\infty}(\mathbb{R}^{n})}\|k^{-1}\|_{L^{\infty}(\mathbb{R}^{n})}\|\omega\|_{A_{1;\rho}^{\rm loc}}$.

Consider now that $1<p<\infty$. We have 
\begin{align*}
&\left(\frac{1}{|Q|}\int_{Q}k(x)\omega(x)dx\right)\left(\frac{1}{|Q|}\int_{Q}[k(x)\omega(x)]^{-\frac{1}{p-1}}dx\right)^{p-1}\\
&\leq\|k\|_{L^{\infty}(\mathbb{R}^{n})}\left(\frac{1}{|Q|}\int_{Q}\omega(x)dx\right)\left(\|\omega^{-1}\|_{L^{\infty}(\mathbb{R}^{n})}^{\frac{1}{p-1}}\frac{1}{|Q|}\int_{Q}\omega(x)^{-\frac{1}{p-1}}dx\right)^{p-1}\\
&\leq\|k\|_{L^{\infty}(\mathbb{R}^{n})}\|k^{-1}\|_{L^{\infty}(\mathbb{R}^{n})}\|\omega\|_{A_{p;\rho}^{\rm loc}},
\end{align*}
which yields $[k\omega]_{A_{p;\rho}^{\rm loc}}\leq\|k\|_{L^{\infty}(\mathbb{R}^{n})}\|k^{-1}\|_{L^{\infty}(\mathbb{R}^{n})}\|\omega\|_{A_{p;\rho}^{\rm loc}}$.
\end{proof}

\begin{proposition}
Let $0<\rho<\infty$ and $\omega_{1},\omega_{2}\in A_{1;\rho}^{\rm loc}$. Then for any $1<p<\infty$, $\omega_{1}\omega_{2}^{1-p}\in A_{p;\rho}^{\rm loc}$ with 
\begin{align*}
\left[\omega_{1}\omega_{2}^{1-p}\right]_{A_{p;\rho}^{\rm loc}}\leq[\omega_{1}]_{A_{1;\rho}^{\rm loc}}[\omega_{2}]_{A_{1;\rho}^{\rm loc}}^{p-1}.
\end{align*}
\end{proposition}

\begin{proof}
Let $Q$ be a cube with $\ell(Q)\leq\rho$. Then
\begin{align*}
&\left(\frac{1}{|Q|}\int_{Q}\omega_{1}(x)\omega_{2}(x)^{1-p}dx\right)\left(\frac{1}{|Q|}\int_{Q}[\omega_{1}(x)\omega_{2}(x)^{1-p}]^{-\frac{1}{p-1}}dx\right)^{p-1}\\
&\leq\left(\frac{1}{|Q|}\int_{Q}\omega_{1}(x)dx\right)\|\omega_{2}^{-1}\|_{L^{\infty}(Q)}^{p-1}\left(\frac{1}{|Q|}\int_{Q}\omega_{2}(x)^{-\frac{1}{p-1}}dx\right)^{p-1}\|\omega_{1}^{-1}\|_{L^{\infty}(Q)}\\
&\leq[\omega_{1}]_{A_{1;\rho}^{\rm loc}}\left(\left(\frac{1}{|Q|}\int_{Q}\omega_{2}(x)^{-\frac{1}{p-1}}dx\right)\|\omega_{2}^{-1}\|_{L^{\infty}(Q)}\right)^{p-1}\\
&\leq[\omega_{1}]_{A_{1;\rho}^{\rm loc}}[\omega_{2}]_{A_{1;\rho}^{\rm loc}}^{p-1},
\end{align*}
as expected.
\end{proof}

If we raise an $A_{p;\rho}^{\rm loc}$ weight $\omega$ up to exponent $0<\delta<1$, then $\omega^{\delta}$ belongs to a smaller class $A_{q;\rho}^{\rm loc}$, where $q$ is the linear ratio between $1$ and $p$ with respect to $\delta$.
\begin{lemma}\label{small}
Let $0<\rho<\infty$, $1\leq p<\infty$, $0<\delta<1$, and $\omega\in A_{p;\rho}^{\rm loc}$. Then $\omega^{\delta}\in A_{q;\rho}^{\rm loc}$ for $q=\delta p+1-\delta$ with 
\begin{align*}
\left[\omega^{\delta}\right]_{A_{q;\rho}^{\rm loc}}\leq[\omega]_{A_{p;\rho}^{\rm loc}}^{\delta}.
\end{align*}
\end{lemma}

\begin{proof}
Let $Q$ be a cube with $\ell(Q)\leq\rho$. Consider first that $p=1$. Then $q=1$. We have
\begin{align*}
\left(\frac{1}{|Q|}\int_{Q}\omega(x)^{\delta}dx\right)\left\|\omega^{-\delta}\right\|_{L^{\infty}(Q)}&=\left(\frac{1}{|Q|}\int_{Q}\omega(x)^{\delta}dx\right)\left\|\omega^{-1}\right\|_{L^{\infty}(Q)}^{\delta}\\
&\leq\left(\frac{1}{|Q|}\int_{Q}\omega(x)dx\right)^{\delta}\left\|\omega^{-1}\right\|_{L^{\infty}(Q)}^{\delta}\\
&\leq[\omega]_{A_{1;\rho}^{\rm loc}}^{\delta},
\end{align*}
which yields $\left[\omega^{\delta}\right]_{A_{1;\rho}^{\rm loc}}\leq[\omega]_{A_{1;\rho}^{\rm loc}}^{\delta}$.

Consider now that $1<p<\infty$. We have
\begin{align*}
&\left(\frac{1}{|Q|}\int_{Q}\omega(x)^{\delta}dx\right)\left(\frac{1}{|Q|}\int_{Q}\left[\omega(x)^{\delta}\right]^{-\frac{1}{q-1}}dx\right)^{q-1}\\
&\leq\left(\frac{1}{|Q|}\int_{Q}\omega(x)dx\right)^{\delta}\left(\frac{1}{|Q|}\int_{Q}\omega(x)^{-\frac{\delta}{q-1}}dx\right)^{q-1}\\
&=\left(\frac{1}{|Q|}\int_{Q}\omega(x)dx\right)^{\delta}\left(\frac{1}{|Q|}\int_{Q}\omega(x)^{-\frac{1}{p-1}}dx\right)^{\delta(p-1)}\\
&\leq[\omega]_{A_{p;\rho}^{\rm loc}}^{\delta},
\end{align*}
which yields $\left[\omega^{\delta}\right]_{A_{q;\rho}^{\rm loc}}\leq[\omega]_{A_{p;\rho}^{\rm loc}}^{\delta}$.
\end{proof}

The only weights with uniformly bounded $\left[\cdot\right]_{A_{p;\rho}^{\rm loc}}$ constants are those of $A_{1;\rho}^{\rm loc}$ class.
\begin{proposition}
Let $0<\rho<\infty$. Assume that there is a constant $C>0$ such that $[\omega]_{A_{p;\rho}^{\rm loc}}\leq C$ for all $1<p<\infty$. Then $\omega\in A_{1;\rho}^{\rm loc}$ with $[\omega]_{A_{1;\rho}^{\rm loc}}\leq C$.
\end{proposition}

\begin{proof}
Let $Q$ be a cube with $\ell(Q)\leq\rho$. Then 
\begin{align*}
\left(\frac{1}{|Q|}\int_{Q}\omega(x)dx\right)\left(\frac{1}{|Q|}\int_{Q}\omega(x)^{-\frac{1}{p-1}}dx\right)^{p-1}\leq C.
\end{align*}
Note that 
\begin{align*}
&\lim_{p\rightarrow 1^{+}}\left(\frac{1}{|Q|}\int_{Q}\omega(x)dx\right)\left(\frac{1}{|Q|}\int_{Q}\omega(x)^{-\frac{1}{p-1}}dx\right)^{p-1}\\
&=\left(\frac{1}{|Q|}\int_{Q}\omega(x)dx\right)\lim_{p\rightarrow 1^{+}}\frac{1}{|Q|^{p-1}}\lim_{p\rightarrow 1^{+}}\left(\int_{Q}\omega(x)^{-\frac{1}{p-1}}dx\right)^{p-1}\\
&=\left(\frac{1}{|Q|}\int_{Q}\omega(x)dx\right)\lim_{q\rightarrow\infty}\left(\int_{Q}\omega(x)^{-q}dx\right)^{\frac{1}{q}}\\
&=\left(\frac{1}{|Q|}\int_{Q}\omega(x)dx\right)\|\omega^{-1}\|_{L^{\infty}(Q)}.
\end{align*}
Therefore,
\begin{align*}
[\omega]_{A_{1;\rho}^{\rm loc}}=\sup_{\ell(Q)\leq\rho}\left(\frac{1}{|Q|}\int_{Q}\omega(x)dx\right)\|\omega^{-1}\|_{L^{\infty}(Q)}\leq C<\infty,
\end{align*}
as expected.
\end{proof}

The truncation $\min(\omega,k)$ of an $A_{p;\rho}^{\rm loc}$ weight $\omega$ up to height $k$ still belongs to the same class.
\begin{proposition}
Let $0<\rho<\infty$, $k>0$, $1\leq p<\infty$, and $\omega\in A_{p;\rho}^{\rm loc}$. Then 
\begin{align*}
[\min(\omega,k)]_{A_{p;\rho}^{\rm loc}}\leq C(p)[\omega]_{A_{p;\rho}^{\rm loc}},
\end{align*} 
where $C(p)=1$ for $p=1$, $C(p)=2$ for $1<p\leq 2$, and $C(p)=2^{p-1}$ for $2<p<\infty$.
\end{proposition}

\begin{proof}
Let $Q$ be a cube with $\ell(Q)\leq\rho$. Consider first that $p=1$. Then
\begin{align*}
\frac{1}{|Q|}\int_{Q}\min(\omega(x),k)dx&\leq\min\left(k,\frac{1}{|Q|}\int_{Q}\omega(x)dx\right)\\
&\leq\min\left(k,[\omega]_{A_{1;\rho}^{\rm loc}}\cdot{\rm ess.inf}_{Q}\omega\right)\\
&\leq\max\left([\omega]_{A_{1;\rho}^{\rm loc}},1\right)\min\left(k,{\rm ess.inf}_{Q}\omega\right)\\
&\leq[\omega]_{A_{1;\rho}^{\rm loc}}\cdot{\rm ess.inf}_{Q}\omega,
\end{align*}
which yields
\begin{align*}
[\min(\omega,k)]_{A_{1;\rho}^{\rm loc}}\leq[\omega]_{A_{1;\rho}^{\rm loc}}.
\end{align*}

Now consider that $1<p\leq 2$. We have
\begin{align*}
\frac{1}{|Q|}\int_{Q}\min(\omega(x),k)^{-\frac{1}{p-1}}dx&=\frac{1}{|Q|}\int_{Q}\max\left(\omega(x)^{-\frac{1}{p-1}},k^{-\frac{1}{p-1}}\right)dx\\
&\leq\frac{1}{|Q|}\int_{Q}\omega(x)^{-\frac{1}{p-1}}dx+\frac{1}{|Q|}\int_{Q}k^{-\frac{1}{p-1}}dx\\
&=\frac{1}{|Q|}\int_{Q}\omega(x)^{-\frac{1}{p-1}}dx+k^{-\frac{1}{p-1}}.
\end{align*}
As $0<p-1\leq 1$, it follows that 
\begin{align*}
\left(\frac{1}{|Q|}\int_{Q}\min(\omega(x),k)^{-\frac{1}{p-1}}dx\right)^{p-1}&\leq\left(\frac{1}{|Q|}\int_{Q}\omega(x)^{-\frac{1}{p-1}}dx+k^{-\frac{1}{p-1}}\right)^{p-1}\\
&\leq\left(\frac{1}{|Q|}\int_{Q}\omega(x)^{-\frac{1}{p-1}}dx\right)^{p-1}+k^{-1}.
\end{align*}
As a result, we have
\begin{align*}
&\left(\frac{1}{|Q|}\int_{Q}\min(\omega(x),k)dx\right)\left(\frac{1}{|Q|}\int_{Q}\min(\omega(x),k)^{-\frac{1}{p-1}}dx\right)^{p-1}\\
&\leq\min\left(k,\frac{1}{|Q|}\int_{Q}\omega(x)dx\right)\left(\left(\frac{1}{|Q|}\int_{Q}\omega(x)^{-\frac{1}{p-1}}dx\right)^{p-1}+k^{-1}\right)\\
&\leq\left(\frac{1}{|Q|}\int_{Q}\omega(x)dx\right)\left(\frac{1}{|Q|}\int_{Q}\omega(x)^{-\frac{1}{p-1}}dx\right)^{p-1}+k\cdot k^{-1}\\
&\leq[\omega]_{A_{p;\rho}^{\rm loc}}+1\\
&\leq 2[\omega]_{A_{p;\rho}^{\rm loc}},
\end{align*}
which shows $[\min(\omega,k)]_{A_{p;\rho}^{\rm loc}}\leq 2[\omega]_{A_{p;\rho}^{\rm loc}}$ for $1<p\leq 2$.

Consider the final case where $p>2$. Then $p-1>1$ and hence
\begin{align*}
\left(\frac{1}{|Q|}\int_{Q}\min(\omega(x),k)^{-\frac{1}{p-1}}dx\right)^{p-1}&\leq\left(\frac{1}{|Q|}\int_{Q}\omega(x)^{-\frac{1}{p-1}}dx+k^{-\frac{1}{p-1}}\right)^{p-1}\\
&\leq 2^{p-2}\left[\left(\frac{1}{|Q|}\int_{Q}\omega(x)^{-\frac{1}{p-1}}dx\right)^{p-1}+k^{-1}\right].
\end{align*}
Subsequently,
\begin{align*}
&\left(\frac{1}{|Q|}\int_{Q}\min(\omega(x),k)dx\right)\left(\frac{1}{|Q|}\int_{Q}\min(\omega(x),k)^{-\frac{1}{p-1}}dx\right)^{p-1}\\
&\leq 2^{p-2}\min\left(k,\frac{1}{|Q|}\int_{Q}\omega(x)dx\right)\left(\left(\frac{1}{|Q|}\int_{Q}\omega(x)^{-\frac{1}{p-1}}dx\right)^{p-1}+k^{-1}\right)\\
&\leq 2^{p-2}\left([\omega]_{A_{p;\rho}^{\rm loc}}+1\right)\\
&\leq 2^{p-1}[\omega]_{A_{p;\rho}^{\rm loc}},
\end{align*}
which implies that $[\min(\omega,k)]_{A_{p;\rho}^{\rm loc}}\leq 2^{p-1}[\omega]_{A_{p;\rho}^{\rm loc}}$ for $1<p\leq 2$ and the proof is now complete.
\end{proof}

Recall the standard interpolation theorem that 
\begin{align*}
\|f\|_{L^{p}}\leq\|f\|_{L^{p_{0}}}^{1-\theta}\|f\|_{L^{p_{1}}}^{\theta},
\end{align*}
where $0<p_{0}<p<p_{1}<\infty$ and 
\begin{align*}
\frac{1}{p}=\frac{1-\theta}{p_{0}}+\frac{\theta}{p_{1}}.
\end{align*}
We have a similar statement for $A_{p;\rho}^{\rm loc}$ weights.
\begin{proposition}
Let $0<\rho<\infty$, $\omega_{0}\in A_{p_{0};\rho}^{\rm loc}$, $\omega_{1}\in A_{p_{1};\rho}^{\rm loc}$, and $1\leq p_{0},p_{1}<\infty$. Suppose that $0\leq\theta\leq 1$ satisfies
\begin{align*}
\frac{1}{p}=\frac{1-\theta}{p_{0}}+\frac{\theta}{p_{1}},\qquad\omega^{\frac{1}{p}}=\omega_{0}^{\frac{1-\theta}{p_{0}}}\omega_{1}^{\frac{\theta}{p_{1}}}.
\end{align*}
Then $\omega\in A_{p;\rho}^{\rm loc}$ with 
\begin{align*}
[\omega]_{A_{p;\rho}^{\rm loc}}\leq[\omega]_{A_{p_{0};\rho}^{\rm loc}}^{(1-\theta)\frac{p}{p_{0}}}[\omega_{1}]_{A_{p_{1};\rho}^{\rm loc}}^{\theta\frac{p}{p_{1}}}.
\end{align*}
\end{proposition}

\begin{proof}
Let $Q$ be a cube with $\ell(Q)\leq\rho$. By applying H\"{o}lder's inequality to the exponents $\frac{p_{0}}{(1-\theta)p}$ and $\frac{p_{1}}{\theta p}$, we have
\begin{align*}
\frac{1}{|Q|}\int_{Q}\omega(x)dx&=\frac{1}{|Q|}\int_{Q}\omega_{0}(x)^{(1-\theta)\frac{p}{p_{0}}}\omega_{1}(x)^{\theta\frac{p}{p_{1}}}dx\\
&\leq\left(\frac{1}{|Q|}\int_{Q}\omega_{0}(x)dx\right)^{(1-\theta)\frac{p}{p_{0}}}\left(\frac{1}{|Q|}\int_{Q}\omega_{1}(x)dx\right)^{\theta\frac{p}{p_{1}}}.
\end{align*}
On the other hand, since 
\begin{align*}
\frac{1}{p'}=\frac{1-\theta}{p_{0}'}+\frac{\theta}{p_{1}'},
\end{align*}
applying H\"{o}lder's inequality again, it holds that
\begin{align*}
&\frac{1}{|Q|}\int_{Q}\omega(x)^{-\frac{1}{p-1}}dx\\
&=\frac{1}{|Q|}\int_{Q}\omega_{0}(x)^{-(1-\theta)\frac{p}{p_{0}(p-1)}}\omega_{1}(x)^{-\theta\frac{p}{p_{1}(p-1)}}dx\\
&=\frac{1}{|Q|}\int_{Q}\omega_{0}(x)^{-(1-\theta)\frac{p'}{p_{0}}}\omega_{1}(x)^{-\theta\frac{p'}{p_{1}}}dx\\
&\leq\left(\frac{1}{|Q|}\int_{Q}\omega_{0}(x)^{-\frac{p_{0}'}{p_{0}}}dx\right)^{(1-\theta)\frac{p'}{p_{0}'}}\left(\frac{1}{|Q|}\int_{Q}\omega_{1}(x)^{-\frac{p_{1}'}{p_{1}}}dx\right)^{\theta\frac{p'}{p_{1}'}}\\
&=\left[\left(\frac{1}{|Q|}\int_{Q}\omega_{0}(x)^{-\frac{1}{p_{0}-1}}dx\right)^{p_{0}-1}\right]^{(1-\theta)\frac{p'}{p_{0}}}\left[\left(\frac{1}{|Q|}\int_{Q}\omega_{1}(x)^{-\frac{1}{p_{1}-1}}dx\right)^{p_{1}-1}\right]^{\theta\frac{p'}{p_{1}}}.
\end{align*}
Therefore, we have
\begin{align*}
[\omega]_{A_{p;\rho}^{\rm loc}}&=\sup_{|Q|\leq\rho}\left(\frac{1}{|Q|}\int_{Q}\omega(x)dx\right)\left(\frac{1}{|Q|}\int_{Q}\omega(x)^{-\frac{1}{p-1}}dx\right)^{p-1}\\
&\leq[\omega]_{A_{p_{0};\rho}^{\rm loc}}^{(1-\theta)\frac{p}{p_{0}}}[\omega_{1}]_{A_{p_{1};\rho}^{\rm loc}}^{\theta\frac{p}{p_{1}}},
\end{align*}
as expected.
\end{proof}

The following shows that the triangle inequality holds for $\left[\cdot\right]_{A_{p;\rho}^{\rm loc}}$.
\begin{proposition}
Let $0<\rho<\infty$, $1\leq p_{1},p_{2}<\infty$, $\omega_{1}\in A_{p_{1};\rho}^{\rm loc}$, and $\omega_{2}\in A_{p_{2};\rho}^{\rm loc}$. Then $\omega_{1}+\omega_{2}\in A_{p;\rho}^{\rm loc}$ with 
\begin{align*}
[\omega_{1}+\omega_{2}]_{A_{p;\rho}^{\rm loc}}\leq[\omega_{1}]_{A_{p_{1};\rho}^{\rm loc}}+[\omega_{2}]_{A_{p_{1};\rho}^{\rm loc}},
\end{align*}
where $p=\max(p_{1},p_{2})$. 
\end{proposition}

\begin{proof}
We will use the convention that $\frac{1}{0}=\infty$ in the sequel. Let $Q$ be a cube with $\ell(Q)\leq\rho$. Note that
\begin{align*}
&\frac{1}{\omega_{1}+\omega_{2}}\leq\min\left(\frac{1}{\omega_{1}},\frac{1}{\omega_{2}}\right),\\
&\frac{1}{p-1}=\min\left(\frac{1}{p_{1}-1},\frac{1}{p_{2}-1}\right).
\end{align*}
Denote by
\begin{align*}
m=\min\left(\left\|\omega_{1}^{-1}\right\|_{L^{\frac{1}{p_{1}-1}}(Q,\frac{dx}{|Q|})},\left\|\omega_{2}^{-1}\right\|_{L^{\frac{1}{p_{2}-1}}(Q,\frac{dx}{|Q|})}\right).
\end{align*}
Then
\begin{align*}
&\|\omega_{1}+\omega_{2}\|_{L^{1}(Q,\frac{dx}{|Q|})}\left\|(\omega_{1}+\omega_{2})^{-1}\right\|_{L^{\frac{1}{p-1}}(Q,\frac{dx}{|Q|})}\\
&\leq\left(\|\omega_{1}\|_{L^{1}(Q,\frac{dx}{|Q|})}+\|\omega_{2}\|_{L^{1}(Q,\frac{dx}{|Q|})}\right)m\\
&\leq\|\omega_{1}\|_{L^{1}(Q,\frac{dx}{|Q|})}\left\|\omega_{1}^{-1}\right\|_{L^{\frac{1}{p-1}}(Q,\frac{dx}{|Q|})}+\|\omega_{2}\|_{L^{1}(Q,\frac{dx}{|Q|})}\left\|\omega_{2}^{-1}\right\|_{L^{\frac{1}{p-1}}(Q,\frac{dx}{|Q|})}\\
&\leq[\omega_{1}]_{A_{p_{1};\rho}^{\rm loc}}+[\omega_{2}]_{A_{p_{1};\rho}^{\rm loc}},
\end{align*}
and the result follows by taking supremum over all such cubes $Q$.
\end{proof}

\subsection{Relations between $A_{p;\rho}^{\rm loc}$ and $A_{p}$}
\enskip

The following states that an $A_{p;\rho}^{\rm loc}$ weight can be extended to an element of $A_{p}$ based on a given cube $Q$ with $\ell(Q)=\rho$ (see \cite[Lemma 1.1]{RV}). 
\begin{theorem}
Let $0<\rho<\infty$, $1\leq p<\infty$, $\omega\in A_{p;\rho}^{\rm loc}$, and $Q$ be a cube with $\ell(Q)=\rho$. Then there exists a $\overline{\omega}\in A_{p}$ such that $\overline{\omega}=\omega$ on $Q$ and
\begin{align*}
[\overline{\omega}]_{A_{p}}\leq\max\left(2^{np},(2\rho+1)^{np}\right)[\omega]_{A_{p;\rho}^{\rm loc}}.
\end{align*} 
\end{theorem}

\begin{proof}
We may assume that the ``left lower" corner of $Q$ is the origin. Let $\widetilde{Q}$ be the cube with $\ell(\widetilde{Q})=2\rho$ centered at the origin. Define $\overline{\omega}=\omega$ on $Q$. We extend $\overline{\omega}$ from $Q$ to $\widetilde{Q}$ symmetrically with respect to all the coordinate axes. Subsequently, we extend $\overline{\omega}$ from $\widetilde{Q}$ to $\mathbb{R}^{n}$ periodically with the period $2\rho$ to all coordinates. Note that $\overline{\omega}$ is defined everywhere except on the edges of such cubes.

Let $R$ be a cube with $\ell(R)\geq\rho$. Choose a $k\in\mathbb{N}$ such that $k\rho\leq\ell(R)<(k+1)\rho$. Then $R$ is contained at most $N=(\lfloor(k+1)\rho\rfloor+1)^{n}$ cubes $Q_{j}$ being integer translates of $Q$. For $p=1$, we have
\begin{align*}
\left(\frac{1}{|R|}\int_{R}\overline{\omega}(x)dx\right)\left\|\overline{\omega}^{-1}\right\|_{L^{\infty}(R)}&\leq\frac{1}{|R|}\sum_{j=1}^{N}\left(\int_{Q_{j}}\overline{\omega}(x)dx\right)\sup_{1\leq l\leq N}\left\|\overline{\omega}^{-1}\right\|_{L^{\infty}(Q_{l})}\\
&\leq\frac{N}{|R|}\left(\int_{Q}\omega(x)dx\right)\|\omega^{-1}\|_{L^{\infty}(Q)}\\
&\leq\left(2+\frac{1}{\rho}\right)^{n}\rho^{n}\frac{1}{|Q|}\left(\int_{Q}\omega(x)dx\right)\|\omega^{-1}\|_{L^{\infty}(Q)}\\
&\leq(2\rho+1)^{n}[\omega]_{A_{1;\rho}^{\rm loc}}.
\end{align*}
Similarly, for $1<p<\infty$, one can show that 
\begin{align*}
&\left(\frac{1}{|R|}\int_{R}\overline{\omega}(x)dx\right)\left(\frac{1}{|R|}\int_{R}\overline{\omega}(x)^{-\frac{1}{p-1}}dx\right)^{p-1}\\
&\leq\left(\frac{N}{|R|}\right)^{p}\left(\frac{1}{|Q|}\int_{Q}\overline{\omega}(x)dx\right)\left(\frac{1}{|Q|}\int_{Q}\overline{\omega}(x)^{-\frac{1}{p-1}}dx\right)^{p-1}\\
&\leq(2\rho+1)^{np}[\omega]_{A_{p;\rho}^{\rm loc}}.
\end{align*}
Now we consider the case where $\ell(R)<\rho$. We may assume that $R$ contains the origin. Denote by $\Pi_{j}$ the octants of $\mathbb{R}^{n}$, $j=1,\ldots,2^{n}$. Then $R\cap\Pi_{j}$ is contained in a cube being integer translate of $Q$, which yields
\begin{align*}
\int_{R}\overline{\omega}(x)dx=\sum_{j=1}^{2^{n}}\int_{R\cap\Pi_{j}}\overline{\omega}(x)dx=\sum_{j=1}^{2^{n}}\int_{R\cap\Pi_{j}}\omega(x)dx\leq 2^{n}\int_{R}\omega(x)dx.
\end{align*}
For $p=1$, we have 
\begin{align*}
\left(\frac{1}{|R|}\int_{R}\overline{\omega}(x)dx\right)\left\|\overline{\omega}^{-1}\right\|_{L^{\infty}(R)}&\leq\left(\frac{2^{n}}{|R|}\int_{R}\omega(x)dx\right)\sup_{1\leq j\leq 2^{n}}\left\|\overline{\omega}^{-1}\right\|_{L^{\infty}(R\cap\Pi_{j})}\\
&\leq\left(\frac{2^{n}}{|R|}\int_{R}\omega(x)dx\right)\left\|\omega^{-1}\right\|_{L^{\infty}(R)}\\
&\leq 2^{n}\|\omega\|_{A_{1;\rho}^{\rm loc}}.
\end{align*}
Similarly, for $1<p<\infty$, we have 
\begin{align*}
&\left(\frac{1}{|R|}\int_{R}\overline{\omega}(x)dx\right)\left(\frac{1}{|R|}\int_{R}\overline{\omega}(x)^{-\frac{1}{p-1}}dx\right)^{p-1}\\
&\leq\frac{2^{np}}{|R|^{p}}\left(\int_{R}\overline{\omega}(x)dx\right)\left(\int_{R}\overline{\omega}(x)^{-\frac{1}{p-1}}dx\right)^{p-1}\\
&\leq 2^{np}[\omega]_{A_{p;\rho}^{\rm loc}},
\end{align*}
the result follows by combining the above estimates.
\end{proof}

It is a standard fact that an $\omega\in A_{p}$ weight is of doubling: there exists a constant $C>0$ such that  $\omega(tQ)\leq Ct^{n}\omega(Q)$ for any $t\geq 1$ and cubes $Q$. For weights in $A_{p;\rho}^{\rm loc}$, we have the exponential growths (see \cite[Lemma 1.4]{RV}).
\begin{proposition}\label{growth}
Let $0<\rho<\infty$, $1\leq p<\infty$, and $\omega\in A_{p;\rho}^{\rm loc}$. Then
\begin{align*}
\omega(tQ)\leq\left(C(n,p)\max\left(\frac{1}{\rho^{np}},1\right)[\omega]_{A_{p;\rho}^{\rm loc}}\right)^{t}\omega(Q)
\end{align*}
for all $t\geq 1$ and cubes $Q$ with $\ell(Q)=\rho$.
\end{proposition}

\begin{proof}
Consider first that $t\in\mathbb{N}$. Let $S=(t+1)Q\setminus(t-1)Q$. There is a sequence of disjoint cubes $I_{k}\subseteq S$ with $\ell(I_{k})=\rho$, $|I_{k}\cap tQ|\geq\frac{\rho^{n}}{2^{n}}$, and each point of $S$, except those on the edges of $I_{k}$, belongs to at most $C_{n}$ of the $I_{k}$, here we use the convention that $(t-1)Q=\emptyset$ for $t=1$. For $p=1$, we have 
\begin{align*}
\omega(S)\leq\sum_{k}\omega(I_{k})\leq[\omega]_{A_{1;\rho}^{\rm loc}}\sum_{k}\left\|\omega^{-1}\right\|_{L^{\infty}(I_{k})}^{-1}\leq[\omega]_{A_{1;\rho}^{\rm loc}}\sum_{k}\left\|\omega^{-1}\right\|_{L^{\infty}(I_{k}\cap tQ)}^{-1}.
\end{align*}
Subsequently,
\begin{align*}
\frac{\rho^{n}}{2^{n}}=\int_{I_{k}\cap tQ}\omega(x)\omega(x)^{-1}dx\leq\left\|\omega^{-1}\right\|_{L^{\infty}(I_{k}\cap tQ)}\left(\int_{I_{k}\cap tQ}\omega(x)dx\right),
\end{align*}
and hence
\begin{align*}
\omega(S)\leq\frac{2^{n}}{\rho^{n}}[\omega]_{A_{1;\rho}^{\rm loc}}\sum_{k}\left(\int_{I_{k}\cap tQ}\omega(x)dx\right)\leq\frac{2^{n}}{\rho^{n}}[\omega]_{A_{1;\rho}^{\rm loc}}C(n)\omega(tQ)
\end{align*}
For $1<p<\infty$, using H\"{o}lder's inequality, one obtains 
\begin{align*}
\frac{\rho^{n}}{2^{n}}=\int_{I_{k}\cap tQ}\omega(x)^{\frac{1}{p}}\omega(x)^{-\frac{1}{p}}dx\leq\left(\int_{I_{k}\cap tQ}\omega(x)dx\right)^{\frac{1}{p}}\left(\int_{I_{k}\cap tQ}\omega(x)^{-\frac{1}{p-1}}dx\right)^{\frac{p-1}{p}}.
\end{align*}
Then we have similarly that
\begin{align*}
\omega(S)\leq\left(\frac{2^{n}}{\rho^{n}}\right)^{p}[\omega]_{A_{p;\rho}^{\rm loc}}\sum_{k}\left(\int_{I_{k}\cap tQ}\omega(x)dx\right)\leq\left(\frac{2^{n}}{\rho^{n}}\right)^{p}[\omega]_{A_{p;\rho}^{\rm loc}}C(n)\omega(tQ).
\end{align*}
Hence 
\begin{align*}
\omega((t+1)Q)\leq C(n,p)\max\left(\frac{1}{\rho^{np}},1\right)[\omega]_{A_{p;\rho}^{\rm loc}}\omega(tQ),\quad 1\leq p<\infty.
\end{align*}
Iteration yields
\begin{align*}
\omega((t+1)Q)\leq\left(C(n,p)\max\left(\frac{1}{\rho^{np}},1\right)[\omega]_{A_{p;\rho}^{\rm loc}}\right)^{t}\omega(Q),\quad t\in\mathbb{N},\quad 1\leq p<\infty.
\end{align*}
Since $[\omega]_{A_{p;\rho}^{\rm loc}}\geq 1$, the above estimate also holds for $t=0$. As a consequence, for any $t\geq 1$ and $1\leq p<\infty$, we have
\begin{align*}
\omega(tQ)&\leq\omega\left((\lfloor t\rfloor+1)Q\right)\\
&\leq\left(C(n,p)\max\left(\frac{1}{\rho^{np}},1\right)[\omega]_{A_{p;\rho}^{\rm loc}}\right)^{\lfloor t\rfloor}\omega(Q)\\
&\leq\left(C(n,p)\max\left(\frac{1}{\rho^{np}},1\right)[\omega]_{A_{p;\rho}^{\rm loc}}\right)^{t}\omega(Q),
\end{align*}
and the result follows.
\end{proof}

\begin{remark}
\rm It is routine to check that $e^{c\left|\cdot\right|}\in A_{p;\rho}^{\rm loc}$ for $1\leq p<\infty$ and $c\in\mathbb{R}$, which shows that the exponential bound in Proposition \ref{growth} cannot be improved to polynomial growths.
\end{remark}

Suppose that $\mu$ is a positive measure on $\mathbb{R}^{n}$ and $0<\rho<\infty$. We say that $\mu$ satisfies the local doubling property with respect to $\rho$ provided that 
\begin{align}\label{local doubling property}
\mu(Q)\leq D\cdot\mu\left(\frac{1}{3}Q\right),\quad\ell(Q)\leq\rho
\end{align}
for some constant $D>0$. In which case, we have either $\mu(\mathbb{R}^{n})=0$ or $\mu(Q)>0$ for every cube $Q$.

In contrast to Proposition \ref{growth}, $A_{p;\rho}^{\rm loc}$ weights possess the local doubling property.
\begin{proposition}\label{local strong}
Let $0<\rho<\infty$, $1\leq p<\infty$, and $\omega\in A_{p;\rho}^{\rm loc}$. For any measurable subset $E$ of a cube $Q$ with $\ell(Q)\leq\rho$, it holds that
\begin{align*}
\omega(Q)\leq[\omega]_{A_{p;\rho}^{\rm loc}}\left(\frac{|Q|}{|E|}\right)^{p}\omega(E).
\end{align*}
In particular, $\omega$ satisfies the local doubling property with respect to $\rho$.
\end{proposition}

\begin{proof}
If $p=1$, then
\begin{align*}
|E|=\int_{E}\omega(x)\omega(x)^{-1}dx\leq\|\omega^{-1}\|_{L^{\infty}(Q)}\int_{E}\omega(x)dx=\|\omega^{-1}\|_{L^{\infty}(Q)}\omega(E),
\end{align*}
which yields
\begin{align*}
\omega(Q)\leq\frac{|Q|}{|E|}\left(\frac{1}{|Q|}\omega(Q)\|\omega^{-1}\|_{L^{\infty}(Q)}\right)\omega(E)\leq[\omega]_{A_{1;\rho}^{\rm loc}}\frac{|Q|}{|E|}\omega(E).
\end{align*}
If $1<p<\infty$, then
\begin{align*}
|E|&=\int_{E}\omega(x)^{\frac{1}{p}}\omega(x)^{-\frac{1}{p}}dx\\
&\leq\left(\int_{E}\omega(x)dx\right)^{\frac{1}{p}}\left(\int_{E}\omega(x)^{-\frac{1}{p-1}}dx\right)^{\frac{1}{p'}}\\
&\leq\omega(E)^{\frac{1}{p}}|Q|^{\frac{1}{p'}}\left(\frac{1}{|Q|}\int_{Q}\omega(x)^{-\frac{1}{p-1}}dx\right)^{\frac{1}{p'}}\\
&\leq[\omega]_{A_{p;\rho}^{\rm loc}}^{\frac{1}{p}}\omega(E)^{\frac{1}{p}}|Q|^{\frac{1}{p'}}\left(\frac{1}{|Q|}\int_{Q}\omega(x)dx\right)^{-\frac{1}{p}}\\
&=[\omega]_{A_{p;\rho}^{\rm loc}}^{\frac{1}{p}}\left(\frac{\omega(E)}{\omega(Q)}\right)^{\frac{1}{p}}|Q|,
\end{align*}
which gives the result by routine simplification.
\end{proof}

\subsection{Characterization of $A_{p;\rho}^{\rm loc}$ Weights via Local Maximal Function}
\enskip

The following covering lemma is taken from \cite[Lemma 7.1.10]{GL} and it can be viewed as a weak version of Besicovitch covering theorem (see \cite[Theorem 18.1]{DE}).
\begin{lemma}\label{weak besicovitch}
Let $K$ be a bounded set in $\mathbb{R}^{n}$ and for every $x\in X$, let $Q_{x}$ be an open cube with center $x$ and sides parallel to the axes. Then there are an $m\in\mathbb{N}\cup\{\infty\}$ and a sequence of points $\{x_{j}\}_{j=1}^{m}$ in $K$ such that
\begin{align*}
K\subseteq\bigcup_{j=1}^{m}Q_{x_{j}}
\end{align*}
and for almost all $y\in\mathbb{R}^{n}$, one has
\begin{align*}
\sum_{j=1}^{m}\chi_{Q_{x_{j}}}(y)\leq 72^{n}.
\end{align*}
\end{lemma}

Let $0<\rho<\infty$. We define the (uncentered) local Hardy-Littlewood maximal function ${\bf M}_{\rho}^{\rm loc}$ by
\begin{align}\label{uncenter}
{\bf M}_{\rho}^{\rm loc}f(x)=\sup_{\substack{x\in Q\\\ell(Q)\leq\rho}}\frac{1}{|Q|}\int_{Q}|f(y)|dy,\quad x\in\mathbb{R}^{n},
\end{align}
where $f$ is a locally integrable function on $\mathbb{R}^{n}$. We can now characterize the $A_{1;\rho}^{\rm loc}$ weights in terms of maximal function.
\begin{proposition}\label{repeating}
Let $\omega$ be a weight. Then $\omega\in A_{1;\rho}^{\rm loc}$ if and only if there exists a constant $C>0$ such that 
\begin{align}\label{first A1}
{\bf M}_{\rho}^{\rm loc}\omega(x)\leq C\cdot\omega(x)\quad{\rm a.e.}.
\end{align}
The infimum of all such $C>0$ in $(\ref{first A1})$ equals to $[\omega]_{A_{1;\rho}^{\rm loc}}$.
\end{proposition}

\begin{proof}
Assume that $\omega\in A_{1}^{\rm loc}$. Suppose that $Q$ is an arbitrary cube with $\ell(Q)\leq\rho$. Then there is a set $N_{Q}$ with $|N_{Q}|=0$ and 
\begin{align*}
\frac{1}{|Q|}\int_{Q}\omega(t)dt\cdot\omega(x)^{-1}\leq[\omega]_{A_{1;\rho}^{\rm loc}},\quad x\in Q\setminus N_{Q},
\end{align*}
and hence
\begin{align*}
\frac{1}{|Q|}\int_{Q}\omega(t)dt\leq[\omega]_{A_{1;\rho}^{\rm loc}}\omega(x),\quad x\in Q\setminus N_{Q}.
\end{align*}
Let $\{Q\}$ be the countable collection of all cubes with rational vertices and $|Q|\leq\rho$. Denote by $N=\bigcup_{Q}N_{Q}$. Then we have $|N|=0$. Let $0<\varepsilon<1$ be given. For all $x\in\mathbb{R}^{n}\setminus N$ and cubes $R$ that containing $x$ with $\ell(R)<\rho$, there is a $Q\in\{Q\}$ with $R\subseteq Q$ and $|Q|\leq(1+\varepsilon)|R|$. As a consequence, we have
\begin{align*}
\frac{1}{|R|}\int_{R}\omega(t)dt\leq\frac{1+\varepsilon}{|Q|}\int_{Q}\omega(t)dt\leq(1+\varepsilon)[\omega]_{A_{1;\rho}^{\rm loc}}\omega(x).
\end{align*}
Taking $\varepsilon\rightarrow 0$, then
\begin{align*}
\frac{1}{|R|}\int_{R}\omega(t)dt\leq[\omega]_{A_{1;\rho}^{\rm loc}}\omega(x),\quad x\in\mathbb{R}^{n}\setminus N.
\end{align*}
One may easily see that the above estimate also holds for cubes $R$ with $x\in R$ and $\ell(R)=\rho$ by approximating $R$ with $x\in R_{N}\subsetneq R$, $\ell(R_{N})<\rho$, and $\chi_{R_{N}}\uparrow\chi_{R}$ except on the boundary of $R$, hence (\ref{first A1}) follows.

On the other hand, if (\ref{first A1}) holds, say, 
\begin{align*}
{\bf M}_{\rho}^{\rm loc}\omega(x)\leq C\cdot\omega(x)\quad x\in\mathbb{R}^{n}\setminus N
\end{align*}
for some $N$ with $|N|=0$, then for any cube $Q$ that containing $x\in\mathbb{R}^{n}\setminus N$ with $\ell(Q)\leq\rho$, one has
\begin{align*}
\frac{1}{|Q|}\int_{Q}\omega(t)dt\leq{\bf M}_{\rho}^{\rm loc}\omega(x)\leq C\cdot\omega(x).
\end{align*}
Necessary, $\omega(x)>0$, it follows that 
\begin{align*}
\frac{1}{|Q|}\int_{Q}\omega(t)dt\cdot\omega(x)^{-1}\leq C.
\end{align*}
Since the above estimate holds for all $x\in\mathbb{R}^{n}\setminus N$, it holds that
\begin{align*}
[\omega]_{A_{1;\rho}^{\rm loc}}=\sup_{\ell(Q)\leq\rho}\left(\frac{1}{|Q|}\int_{Q}\omega(t)dt\right)\|\omega^{-1}\|_{L^{\infty}(Q)}\leq C,
\end{align*}
and the proof is now complete.
\end{proof}

The following shows that $A_{p;\rho}^{\rm loc}$ are all the same class regardless of the scaling $\rho$.
\begin{theorem}\label{important}
Let $0<\rho_{1}<\rho_{2}<\infty$ and $1\leq p<\infty$. Then
\begin{align*}
A_{p;\rho_{1}}^{\rm loc}=A_{p:\rho_{2}}^{\rm loc}.
\end{align*}
In fact, it holds that
\begin{align}\label{important formula}
[\omega]_{A_{p;\rho}^{\rm loc}}\leq\left(C(n,p)\max\left(\frac{\delta^{np}}{\rho^{np}},1\right)[\omega]_{A_{p;\rho/\delta}^{\rm loc}}\right)^{p\delta+1}[\omega]_{A_{p;\rho/\delta}^{\rm loc}},\quad\delta>1.
\end{align}
\end{theorem}

\begin{proof}
For $0<\rho_{1}<\rho_{2}<\infty$, it is clear that $[\omega]_{A_{p;\rho_{1}}^{\rm loc}}\leq[\omega]_{A_{p;\rho_{2}}^{\rm loc}}$, which yields $A_{p;\rho_{2}}^{\rm loc}\subseteq A_{p:\rho_{1}}^{\rm loc}$. To prove the converse, it suffices to justify (\ref{important formula}). Let $Q$ be a cube with $\ell(Q)\leq\rho$. Consider first that $p=1$. We use the maximal characterization of $A_{1;\rho}^{\rm loc}$ condition as in Proposition \ref{repeating}. Choose an $N\subseteq\mathbb{R}^{n}$ with $|N|=0$ and 
\begin{align*}
{\bf M}_{\rho/\delta}^{\rm loc}\omega(x)\leq[\omega]_{A_{1;\rho/\delta}^{\rm loc}}\omega(x),\quad x\in\mathbb{R}^{n}\setminus N.
\end{align*}
Fix an $x\in\mathbb{R}^{n}\setminus N$, we see that 
\begin{align*}
{\bf M}_{\rho}^{\rm loc}\omega(x)&=\max\left(\sup_{\substack{x\in Q\\\ell(Q)\leq\rho/\delta}}\frac{\omega(Q)}{|Q|},\sup_{\substack{x\in Q\\\rho/\delta<\ell(Q)\leq\rho}}\frac{\omega(Q)}{|Q|}\right)\\
&=\max\left({\bf M}_{\rho/\delta}^{\rm loc}(x),\sup_{\substack{x\in Q\\\rho/\delta<\ell(Q)\leq\rho}}\frac{\omega(Q)}{|Q|}\right)
\end{align*}
For any cube $Q$ with $\frac{\rho}{\delta}<\ell(Q)\leq\rho$ and $x\in Q$, there is a cube $R\subseteq Q$ with $x\in R$ and $\ell(R)=\frac{1}{\beta}\ell(Q)$, $\beta=\frac{\delta\ell(Q)}{\rho}>1$. Note that $\ell(R)=\frac{\rho}{\delta}$. We claim that $Q\subseteq(\beta+1)R$. Indeed, by letting $c_{R}$ the center of $R$, for all $z\in Q$, one has 
\begin{align*}
|z-c_{R}|_{\infty}\leq|z-x|_{\infty}+|x-c_{R}|_{\infty}\leq\frac{\ell(Q)}{2}+\frac{\ell(R)}{2}=(\beta+1)\frac{\ell(R)}{2},
\end{align*}
as claimed. Using Proposition \ref{growth}, one obtains
\begin{align*}
\frac{\omega(Q)}{|Q|}&\leq\frac{\omega(Q)}{|R|}\\
&\leq\frac{\omega((\beta+1)R)}{|R|}\\
&\leq\left(C(n)\max\left(\frac{1}{\ell(R)^{n}},1\right)[\omega]_{A_{1;\ell(R)}^{\rm loc}}\right)^{\delta+1}\frac{\omega(R)}{|R|}\\
&\leq\left(C(n)\max\left(\frac{\delta^{n}}{\rho^{n}},1\right)[\omega]_{A_{1;\rho/\delta}^{\rm loc}}\right)^{\delta+1}\frac{\omega(R)}{|R|}\\
&\leq\left(C(n)\max\left(\frac{\delta^{n}}{\rho^{n}},1\right)[\omega]_{A_{1;\rho/\delta}^{\rm loc}}\right)^{\delta+1}{\bf M}_{\rho/\delta}^{\rm loc}(x),
\end{align*}
and hence
\begin{align*}
{\bf M}_{\rho}^{\rm loc}\omega(x)&\leq\max\left({\bf M}_{\rho/\delta}^{\rm loc}(x),\left(C(n)\max\left(\frac{\delta^{n}}{\rho^{n}},1\right)[\omega]_{A_{1;\rho/\delta}^{\rm loc}}\right)^{\delta+1}{\bf M}_{\rho/\delta}^{\rm loc}(x)\right)\\
&=\left(C(n)\max\left(\frac{\delta^{n}}{\rho^{n}},1\right)[\omega]_{A_{1;\rho/\delta}^{\rm loc}}\right)^{\delta+1}{\bf M}_{\rho/\delta}^{\rm loc}(x),
\end{align*}
which yields
\begin{align*}
[\omega]_{A_{1;\rho}^{\rm loc}}\leq\left(C(n)\max\left(\frac{\delta^{n}}{\rho^{n}},1\right)[\omega]_{A_{1;\rho/\delta}^{\rm loc}}\right)^{\delta+1}[\omega]_{A_{1;\rho/\delta}^{\rm loc}}.
\end{align*}
For $1<p<\infty$, it holds that 
\begin{align*}
[\omega]_{A_{p;\rho}^{\rm loc}}&=\max\left(\sup_{\ell(Q)\leq\rho/\delta}\frac{1}{|Q|^{p}}\omega(Q)[\omega'(Q)]^{p-1},\sup_{\rho/\delta<\ell(Q)\leq\rho}\frac{1}{|Q|^{p}}\omega(Q)[\omega'(Q)]^{p-1}\right)\\
&=\max\left([\omega]_{A_{p;\rho/\delta}^{\rm loc}},\sup_{\rho/\delta<\ell(Q)\leq\rho}\frac{1}{|Q|^{p}}\omega(Q)[\omega'(Q)]^{p-1}\right).
\end{align*}
For any $\frac{\rho}{\delta}<\eta\leq\rho$ and cube $Q$ with $\ell(Q)=\eta$, Proposition \ref{growth} yields
\begin{align*}
&\omega(Q)[\omega'(Q)]^{p-1}\\
&=\omega\left(\frac{\eta\delta}{\rho}\cdot\frac{\rho}{\eta\delta}Q\right)\left[\omega'\left(\frac{\eta\delta}{\rho}\cdot\frac{\rho}{\eta\delta}Q\right)\right]^{p-1}\\
&\leq\left(C(n,p)\max\left(\frac{\delta^{np}}{\rho^{np}},1\right)[\omega]_{A_{p;\rho/\delta}^{\rm loc}}\right)^{p\frac{\eta\delta}{\rho}}\omega\left(\frac{\rho}{\eta\delta}Q\right)\left[\omega'\left(\frac{\rho}{\eta\delta}Q\right)\right]^{p-1}\\
&\leq\left(C(n,p)\max\left(\frac{\delta^{np}}{\rho^{np}},1\right)[\omega]_{A_{p;\rho/\delta}^{\rm loc}}\right)^{p\delta^{1/n}}\omega\left(\frac{\rho}{\eta\delta}Q\right)\left[\omega'\left(\frac{\rho}{\eta\delta}Q\right)\right]^{p-1}.
\end{align*}
Therefore, we have 
\begin{align*}
[\omega]_{A_{p;\rho}^{\rm loc}}&\leq\max\left([\omega]_{A_{p;\rho/\delta}^{\rm loc}},\left(C(n,p)\max\left(\frac{\delta^{np}}{\rho^{np}},1\right)[\omega]_{A_{p;\rho/\delta}^{\rm loc}}\right)^{p\delta}[\omega]_{A_{p;\rho/\delta}^{\rm loc}}\right)\\
&=\left(C(n,p)\max\left(\frac{\delta^{np}}{\rho^{np}},1\right)[\omega]_{A_{p;\rho/\delta}^{\rm loc}}\right)^{p\delta}[\omega]_{A_{p;\rho/\delta}^{\rm loc}},
\end{align*}
which yields the result for $1<p<\infty$.
\end{proof}
In view of Theorem \ref{important}, we write $A_{p}^{\rm loc}$ instead of $A_{p;1}^{\rm loc}$ and the symbol $A_{p;\rho}^{\rm loc}$ is used only when we wish to measure $\omega\in A_{p}^{\rm loc}$ in terms of $\left[\cdot\right]_{A_{p;\rho}^{\rm loc}}$.

On the other hand, the $A_{p}^{\rm loc}$ weights can also be characterized in terms of weak type boundedness of maximal function.
\begin{proposition}\label{weak type equivalent}
For $0<\rho<\infty$ and $1\leq p<\infty$, the weak type $(p,p)$ estimate 
\begin{align}\label{71}
\omega(\{x\in\mathbb{R}^{n}:{\bf M}_{\rho}^{\rm loc}f(x)>\lambda\})\leq\frac{C}{\lambda^{p}}\int_{\mathbb{R}^{n}}|f(x)|^{p}\omega(x)dx
\end{align}
entails $\omega\in A_{p;\rho}^{\rm loc}$ with $[\omega]_{A_{p;\rho}^{\rm loc}}\leq C$.
\end{proposition}

\begin{proof}
Let $Q$ be an arbitrary cube with $\ell(Q)\leq\rho$. Consider first that $p=1$. Then ${\bf M}_{\rho}^{\rm loc}f(x)\geq{\rm Avg}_{Q}|f|$ for all $x\in Q$, it follows by (\ref{71}) that for all $0<\lambda<{\rm Avg}_{Q}|f|$, we have 
\begin{align}\label{719}
\omega(Q)\leq\omega(\{x\in\mathbb{R}^{n}:{\bf M}_{\rho}^{\rm loc}f(x)>\lambda\})\leq\frac{C}{\lambda^{p}}\int_{\mathbb{R}^{n}}|f(x)|\omega(x)dx.
\end{align}
Taking $f\chi_{Q}$ instead of $f$ in (\ref{719}), we deduce that
\begin{align}\label{72}
\left({\rm Avg}_{Q}|f|\right)^{p}=\frac{1}{|Q|^{p}}\left(\int_{Q}|f(t)|dt\right)^{p}\leq\frac{C}{\omega(Q)}\int_{Q}|f(x)|\omega(x)dx.
\end{align}
Taking $S\subseteq Q$ be a measurable set and $f=\chi_{S}$, we obtain
\begin{align}\label{7111}
\left(\frac{|S|}{|Q|}\right)^{p}\leq C\frac{\omega(S)}{\omega(Q)}.
\end{align}
Now we consider first the case where $p=1$. Denote by 
\begin{align*}
m_{E}=\inf\{b>0:|\{x\in E:\omega(x)<b\}|>0\}
\end{align*}
for any measurable set $E\subseteq\mathbb{R}^{n}$. Then for every $a>m_{Q}$, there exists a measurable set $S_{a}\subseteq Q$ with $|S_{a}|>0$ and $\omega(x)<a$ for all $x\in S_{a}$. Applying (\ref{7111}) to the set $S_{a}$, we obtain
\begin{align*}
\frac{1}{|Q|}\int_{Q}\omega(t)dt\leq\frac{C}{|S_{a}|}\int_{S_{a}}\omega(t)dt\leq C\cdot a,
\end{align*}
which yields
\begin{align*}
\frac{1}{|Q|}\int_{Q}\omega(t)dt\leq C\cdot\omega(x)
\end{align*}
for all cubes $Q$ with $\ell(Q)\leq\rho$ and almost everywhere $x\in Q$. By repeating the proof of Proposition \ref{repeating}, one deduces that $\omega\in A_{1;\rho}^{\rm loc}$ with $[\omega]_{A_{1;\rho}}^{\rm loc}\leq C$.

It remains to consider the case where $1<p<\infty$. Let $f_{N}=\min(\omega^{-\frac{1}{p-1}},N)\chi_{Q}$ in (\ref{72}), $N\in\mathbb{N}$. Then
\begin{align*}
\omega(Q)\left(\frac{1}{|Q|}\int_{Q}\min\left(\omega(x)^{-\frac{1}{p-1}},N\right)dx\right)^{p}\leq C\int_{Q}\min(\omega(x)^{-\frac{1}{p-1}},N)dx,
\end{align*}
which yields
\begin{align*}
\left(\frac{1}{|Q|}\int_{Q}\omega(x)dx\right)\left(\frac{1}{|Q|}\int_{Q}\min\left(\omega(x)^{-\frac{1}{p-1}},N\right)dx\right)^{p-1}\leq C.
\end{align*}
Taking $N\rightarrow\infty$, monotone convergence theorem implies that $[\omega]_{A_{p;\rho}^{\rm loc}}\leq C$, which gives the result.
\end{proof}
It will be seen later that the converse of Proposition \ref{weak type equivalent} also holds. 

The following is a local form of Fefferman-Stein type inequality \cite{FS}. Note that the estimate (\ref{FS1}) below gives the converse statement of Proposition \ref{weak type equivalent} for the case where $p=1$.
\begin{theorem}\label{FS}
Let $0<\rho<\infty$ and $\omega$ be a weight. Then 
\begin{align}\label{FS1}
\omega\left(\left\{x\in\mathbb{R}^{n}:{\bf M}_{\rho}^{\rm loc}f(x)>\alpha\right\}\right)\leq\frac{72^{n}}{\alpha}\int_{\mathbb{R}^{n}}|f(x)|{\bf M}_{\rho}^{\rm loc}\omega(x)dx,\quad\alpha>0,
\end{align}
and
\begin{align}\label{FS2}
\int_{\mathbb{R}^{n}}{\bf M}_{\rho}^{\rm loc}f(x)^{p}\omega(x)dx\leq C_{n,q}\int_{\mathbb{R}^{n}}|f(x)|^{p}{\bf M}_{\rho}^{\rm loc}\omega(x)dx,\quad 1<p<\infty.
\end{align}
\end{theorem}

\begin{proof}
The estimate (\ref{FS2}) will follow by (\ref{FS1}) and Marcinkiewicz interpolation theorem as we have
\begin{align*}
\left\|{\bf M}_{\rho}^{\rm loc}f\right\|_{L^{\infty}(\omega)}=\left\|{\bf M}_{\rho}^{\rm loc}f\right\|_{L^{\infty}(\mathbb{R}^{n})}\leq\|f\|_{L^{\infty}(\mathbb{R}^{n})}=\|f\|_{L^{\infty}({\bf M}_{\rho}^{\rm loc}\omega)}
\end{align*}
by noting that the measures $\omega dx, {\bf M}_{\rho}^{\rm loc}\omega dx$ are absolutely continuous with respect to the Lebesgue measure and vice versa.

It remains to show (\ref{FS1}). Denote by $E_{\alpha}=\left\{x\in\mathbb{R}^{n}:{\bf M}_{\rho}^{\rm loc}f(x)>\alpha\right\}$. Let $K\subseteq E_{\alpha}$ be an arbitrary compact set. For each $x\in K$, there is a cube $Q_{x}$ that containing $x$ such that $\ell(Q_{x})\leq\rho$ and 
\begin{align*}
\frac{1}{|Q_{x}|}\int_{Q_{x}}|f(y)|dy>\alpha.
\end{align*}
We could assume all the cubes $Q_{x}$ are open. Using Lemma \ref{weak besicovitch}, there are an $m\in\mathbb{N}\cup\{\infty\}$ and a sequence $\{x_{j}\}_{j=1}^{m}\subseteq K$ such that 
\begin{align*}
K\subseteq\bigcup_{j=1}^{m}Q_{x_{j}}
\end{align*}
and for almost everywhere $y\in\mathbb{R}^{n}$, one has
\begin{align*}
\sum_{j=1}^{m}\chi_{Q_{x_{j}}}(y)\leq 72^{n}.
\end{align*}
On the other hand, for each $z\in Q_{x_{j}}$, we have
\begin{align*}
\frac{1}{\left|Q_{x_{j}}\right|}\int_{Q_{x_{j}}}\omega(y)dy\leq{\bf M}_{\rho}^{\rm loc}\omega(z).
\end{align*}
Multiplying the above inequalities with $|f(z)|$ and integrating with respect to $\left(Q_{x_{j}},dz\right)$, it follows that
\begin{align*}
\omega\left(Q_{x_{j}}\right)\frac{1}{\left|Q_{x_{j}}\right|}\int_{Q_{x_{j}}}|f(z)|dz\leq\int_{Q_{x_{j}}}|f(z)|{\bf M}_{\rho}^{\rm loc}\omega(z)dz.
\end{align*}
Then
\begin{align*}
\omega\left(Q_{x_{j}}\right)\leq\frac{1}{\alpha}\int_{Q_{x_{j}}}|f(z)|{\bf M}_{\rho}^{\rm loc}\omega(z)dz
\end{align*}
and hence
\begin{align*}
\omega(K)&=\omega\left(\bigcup_{j=1}^{m}Q_{x_{j}}\right)\\
&\leq\sum_{j=1}^{m}\omega\left(Q_{x_{j}}\right)\\
&\leq\frac{1}{\alpha}\int_{\mathbb{R}^{n}}\left(\sum_{j=1}^{m}\chi_{Q_{x_{j}}}(z)\right)|f(z)|{\bf M}_{\rho}^{\rm loc}\omega(z)dz\\
&\leq\frac{72^{n}}{\alpha}\int_{\mathbb{R}^{n}}|f(z)|{\bf M}_{\rho}^{\rm loc}\omega(z)dz.
\end{align*}
The estimate (\ref{FS1}) follows by the inner regularity of the Lebesgue measure applied to $E_{\alpha}$. 
\end{proof}

In the sequel, for $0<\rho<\infty$, we define the center local maximal function $\mathcal{M}_{\rho}^{\rm loc}$ by
\begin{align*}
\mathcal{M}_{\rho}^{\rm loc}f(x)=\sup_{0<r\leq\rho/2}\frac{1}{|Q_{r}(x)|}\int_{Q_{r}(x)}|f(y)|dy,
\end{align*}
where $f$ is a locally integrable function on $\mathbb{R}^{n}$. Note that the supremum is taken over all cubes $Q_{r}(x)$ with lengths $\ell(Q_{r}(x))=2r\leq\rho$.

The converse of Proposition \ref{weak type equivalent} for the case where $p=1$ is also shown in the following if we replace the maximal function ${\bf M}_{\rho}^{\rm loc}$ by $\mathcal{M}_{\rho}^{\rm loc}$.
\begin{theorem}\label{weak type theorem}
Let $0<\rho<\infty$ and $\omega\in A_{1;\rho}^{\rm loc}$. Then
\begin{align*}
\left\|\mathcal{M}_{\rho}^{\rm loc}\right\|_{L^{1}(\omega)\rightarrow L^{1,\infty}(\omega)}\leq 72^{n}[\omega]_{A_{1;\rho}^{\rm loc}}.
\end{align*}
Conversely, if $\omega$ is a weight such that 
\begin{align}\label{partial converse}
\left\|\mathcal{M}_{\rho}^{\rm loc}\right\|_{L^{1}(\omega)\rightarrow L^{1,\infty}(\omega)}\leq C
\end{align}
for some constant $C>0$, then $\omega\in A_{1}^{\rm loc}$.
\end{theorem}

\begin{proof}
One can show this theorem by noting that $\mathcal{M}_{\rho}^{\rm loc}f\leq{\bf M}_{\rho}^{\rm loc}f$ and applying Proposition \ref{repeating} and (\ref{FS1}) of Theorem \ref{FS}. However, we give a somewhat different proof as follows.

Denote by $d\mu=\omega dx$ and 
\begin{align*}
\mathcal{M}_{\mu}f(x)=\sup_{r>0}\frac{1}{\mu(Q_{r}(x))}\int_{Q_{r}(x)}|f(y)|d\mu(y)
\end{align*}
for admissible functions $f$. We first show that
\begin{align}\label{only dimension}
\left\|\mathcal{M}_{\mu}\right\|_{L^{1}(\omega)\rightarrow L^{1,\infty}(\omega)}\leq 72^{n}
\end{align}
Let $\alpha>0$ be given. Denote by $E_{\alpha}=\{x\in\mathbb{R}^{n}:\mathcal{M}_{\mu}f(x)>\alpha\}$. Let $K$ be an arbitrary compact subset of $E_{\alpha}$. For each $x\in K$, there exists a cube $Q_{r_{x}}(x)$ such that 
\begin{align*}
\frac{1}{\mu(Q_{r_{x}}(x))}\int_{Q_{r_{x}}(x)}|f(y)|d\mu(y)>\alpha.
\end{align*}
Using Lemma \ref{weak besicovitch}, there are an $m\in\mathbb{N}\cup\{\infty\}$ and a sequence $\{x_{j}\}_{j=1}^{m}\subseteq K$ such that 
\begin{align*}
K\subseteq\bigcup_{j=1}^{m}Q_{r_{x_{j}}}(x_{j})
\end{align*}
and for almost everywhere $y\in\mathbb{R}^{n}$, one has
\begin{align*}
\sum_{j=1}^{m}\chi_{Q_{r_{x_{j}}}(x_{j})}(y)\leq 72^{n}.
\end{align*}
Then
\begin{align*}
\mu(K)\leq\mu\left(\bigcup_{j=1}^{m}Q_{r_{x_{j}}}(x_{j})\right)\leq\sum_{j=1}^{m}\mu\left(Q_{r_{x_{j}}}(x_{j})\right)\leq\frac{1}{\alpha}\sum_{j=1}^{m}\int_{Q_{r_{x_{j}}}(x_{j})}|f(y)|d\mu(y).
\end{align*}
The absolute continuity of $\mu$ with respect to Lebesgue measure yields
\begin{align*}
\mu(K)\leq\frac{1}{\alpha}\int_{\mathbb{R}^{n}}\sum_{j=1}^{m}\chi_{Q_{r_{x_{j}}}(x_{j})}(y)|f(y)|d\mu(y)\leq \frac{72^{n}}{\alpha}\int_{\mathbb{R}^{n}}|f(y)|d\mu(y).
\end{align*}
The inner regularity of $\mu$ yields
\begin{align*}
\omega(E_{\alpha})=\mu(E_{\alpha})\leq\frac{72^{n}}{\alpha}\|f\|_{L^{1}(\mu)}=\frac{72^{n}}{\alpha}\|f\|_{L^{1}(\omega)}.
\end{align*}
Now we observe that
\begin{align*}
\mathcal{M}_{\rho}^{\rm loc}f(x)&=\sup_{0<r\leq\rho}\frac{1}{|Q_{r}(x)|}\int_{Q_{r}(x)}|f(y)|dy\\
&\leq\sup_{0<r\leq\rho}\frac{\omega(Q_{r}(x))}{|Q_{r}(x)|}\left\|\omega^{-1}\right\|_{L^{\infty}(Q_{r}(x))}\frac{1}{\omega(Q_{r}(x))}\int_{Q_{r}(x)}|f(y)|\omega(y)dy\\
&\leq[\omega]_{A_{1;\rho}^{\rm loc}}\mathcal{M_{\mu}}f(x).
\end{align*}
As a result,
\begin{align*}
\sup_{\alpha>0}\alpha\cdot\omega(\{x\in\mathbb{R}^{n}:\mathcal{M}_{\rho}^{\rm loc}f(x)>\alpha\})\leq 72^{n}[\omega]_{A_{1;\rho}^{\rm loc}}\int_{\mathbb{R}^{n}}|f(y)|\omega(y)dy,
\end{align*}
as expected. Finally, (\ref{partial converse}) follows by Proposition \ref{weak type equivalent} as $\mathcal{M}_{\rho}^{\rm loc}f\leq{\bf M}_{\rho}^{\rm loc}f$.
\end{proof}

On the other hand, the converse statement of Proposition \ref{weak type equivalent} for the case where $1<p<\infty$ is shown in the following as $2^{-n}{\bf M}_{\rho/2}^{\rm loc}f\leq\mathcal{M}_{\rho}^{\rm loc}f\leq{\bf M}_{\rho}^{\rm loc}f$ holds pointwise.

\begin{theorem}\label{center strong type}
Let $0<\rho<\infty$, $1<p<\infty$, and $\omega\in A_{p;\rho}^{\rm loc}$. Then
\begin{align}\label{7125}
\left\|\mathcal{M}_{\rho/3}^{\rm loc}\right\|_{L^{p}(\omega)\rightarrow L^{p}(\omega)}\leq C(n,p)[\omega]_{A_{p;\rho}^{\rm loc}}^{\frac{1}{p-1}}.
\end{align}
Conversely, if $\omega$ is a weight such that 
\begin{align*}
\left\|\mathcal{M}_{\rho}^{\rm loc}\right\|_{L^{p}(\omega)\rightarrow L^{p}(\omega)}\leq C
\end{align*}
for some constant $C>0$, then $\omega\in A_{p}^{\rm loc}$.
\end{theorem}

\begin{proof}
Fix an open cube $Q=Q_{r}(z)$ with $\ell(Q)=2r\leq\frac{\rho}{3}$. We have 
\begin{align*}
\frac{1}{|Q|}\int_{Q}|f(x)|dx=\frac{\omega(Q)^{\frac{1}{p-1}}\omega'(3Q)}{|Q|^{\frac{p}{p-1}}}\left(\frac{|Q|}{\omega(Q)}\left(\frac{1}{\omega'(3Q)}\int_{Q}|f(x)|dx\right)^{p-1}\right)^{\frac{1}{p-1}}.
\end{align*}
For any $y\in Q$, it holds that $Q\subseteq Q_{2r}(y)\subseteq Q_{3r}(z)=3Q$ and hence
\begin{align*}
\frac{1}{\omega'(3Q)}\int_{Q}|f(x)|dx\leq\frac{1}{\omega'(Q_{2r}(y))}\int_{Q_{2r}(y)}|f(x)|dx\leq\mathcal{M}_{\omega'}\left(|f|\omega'^{-1}\right)(y).
\end{align*}
Raising both sides with exponent $p-1$ and taking integral with respect to $(dy,Q)$, we have
\begin{align*}
|Q|\left(\frac{1}{\omega'(3Q)}\int_{Q}|f(x)|dx\right)^{p-1}\leq\int_{Q}\left[\mathcal{M}_{\omega'}\left(|f|\omega'^{-1}\right)\right]^{p-1}dy,
\end{align*}
and hence
\begin{align*}
\frac{1}{|Q|}\int_{Q}|f(x)|dx\leq\frac{\omega(Q)^{\frac{1}{p-1}}\omega'(3Q)}{|Q|^{\frac{p}{p-1}}}\left(\frac{1}{\omega(Q)}\int_{Q}\left[\mathcal{M}_{\omega'}\left(|f|\omega'^{-1}\right)\right]^{p-1}dy\right)^{\frac{1}{p-1}}.
\end{align*}
Note that $\ell(3Q)=3\ell(Q)\leq\rho$. Hence
\begin{align*}
\frac{\omega(Q)\omega'(3Q)^{p-1}}{|Q|^{p}}\leq 3^{np}[\omega]_{A_{p;\rho}^{\rm loc}},
\end{align*}
it follows that
\begin{align*}
\frac{1}{|Q|}\int_{Q}|f(x)|dx\leq 3^{\frac{np}{p-1}}[\omega]_{A_{p;\rho}^{\rm loc}}^{\frac{1}{p-1}}\left(\mathcal{M}_{\omega}\left[\left[\mathcal{M}_{\omega'}\left(|f|\omega'^{-1}\right)\right]^{p-1}\omega^{-1}\right](z)\right)^{\frac{1}{p-1}}.
\end{align*}
Taking supremum over all $Q=Q_{r}(z)$ with $\ell(Q)=2r\leq\frac{\rho}{3}$, one has 
\begin{align*} 
\mathcal{M}_{\rho/3}^{\rm loc}f\leq 3^{\frac{np}{p-1}}[\omega]_{A_{p;\rho}^{\rm loc}}^{\frac{1}{p-1}}\left(\mathcal{M}_{\omega}\left[\left[\mathcal{M}_{\omega'}\left(|f|\omega'^{-1}\right)\right]^{p-1}\omega^{-1}\right]\right)^{\frac{1}{p-1}}.
\end{align*}
Taking $L^{p}(\omega)$ norms both sides, we have 
\begin{align*}
\left\|\mathcal{M}_{\rho/3}^{\rm loc}f\right\|_{L^{p}(\omega)}&\leq 3^{\frac{np}{p-1}}[\omega]_{A_{p;\rho}^{\rm loc}}^{\frac{1}{p-1}}\left\|\mathcal{M}_{\omega}\left[\left[\mathcal{M}_{\omega'}\left(|f|\omega'^{-1}\right)\right]^{p-1}\omega^{-1}\right]\right\|_{L^{p'}(\omega)}^{\frac{1}{p-1}}\\
&\leq 3^{\frac{np}{p-1}}[\omega]_{A_{p;\rho}^{\rm loc}}^{\frac{1}{p-1}}\left\|\mathcal{M}_{\omega}\right\|_{L^{p'}(\omega)\rightarrow L^{p'}(\omega)}\left\|\left[\mathcal{M}_{\omega'}\left(|f|\omega'^{-1}\right)\right]^{p-1}\omega^{-1}\right\|_{L^{p'}(\omega)}^{\frac{1}{p-1}}\\
&=3^{\frac{np}{p-1}}[\omega]_{A_{p;\rho}^{\rm loc}}^{\frac{1}{p-1}}\left\|\mathcal{M}_{\omega}\right\|_{L^{p'}(\omega)\rightarrow L^{p'}(\omega)}\left\|\mathcal{M}_{\omega'}\left(|f|\omega'^{-1}\right)\right\|_{L^{p}(\omega')}\\
&\leq 3^{\frac{np}{p-1}}[\omega]_{A_{p;\rho}^{\rm loc}}^{\frac{1}{p-1}}\left\|\mathcal{M}_{\omega}\right\|_{L^{p'}(\omega)\rightarrow L^{p'}(\omega)}\left\|\mathcal{M}_{\omega'}\right\|_{L^{p}(\omega')\rightarrow L^{p}(\omega')}\|f\|_{L^{p}(\omega)},
\end{align*}
and the estimate (\ref{7125}) follows provided we show that 
\begin{align}\label{7128}
\left\|\mathcal{M}_{\omega}\right\|_{L^{q}(\omega)\rightarrow L^{q}(\omega)}\leq C(n,q)<\infty
\end{align}
for any $1<q<\infty$ and weight $\omega$. Indeed, by combining (\ref{only dimension}) and the simple fact that 
\begin{align*}
\left\|\mathcal{M}_{\omega}\right\|_{L^{\infty}(\omega)\rightarrow L^{\infty}(\omega)}\leq 1,
\end{align*}
then (\ref{7128}) follows by Marcinkiewicz interpolation theorem. 

As in the last part of the proof of Theorem \ref{weak type theorem}, the converse statement follows by Proposition \ref{weak type equivalent}.
\end{proof}

Denote the general uncentered local maximal function by
\begin{align*}
{\bf M}_{\mu,\rho}^{\rm loc}f(x)=\sup_{\substack{x\in Q\\\ell(Q)\leq\rho}}\frac{1}{\mu(Q)}\int_{Q}|f(y)|d\mu(y)
\end{align*}
for admissible functions $f$, where $\mu$ is a positive measure on $\mathbb{R}^{n}$. A generalization of Theorems \ref{weak type theorem} and \ref{center strong type} in terms of uncentered local maximal function is given as follows.
\begin{proposition}\label{use weak type}
Let $0<\rho<\infty$ and $\mu$ be a positive measure on $\mathbb{R}^{n}$. Assume that $\mu$ satisfies the local doubling property $(\ref{local doubling property})$ with respect to $\rho$. Suppose that $\mu$ is absolutely continuous with respect to the Lebesgue measure. Then 
\begin{align}\label{measure weak}
\left\|{\bf M}_{\mu,\rho/3}^{\rm loc}\right\|_{L^{1}(\mu)\rightarrow L^{1,\infty}(\mu)}\leq D
\end{align}
and 
\begin{align}\label{measure strong}
\left\|{\bf M}_{\mu,\rho/3}^{\rm loc}\right\|_{L^{p}(\mu)\rightarrow L^{p}(\mu)}\leq C_{n,p}\cdot D
\end{align}
for every $1<p<\infty$.
\end{proposition}

\begin{proof}
We may assume that $\mu(\mathbb{R}^{n})>0$. Then $\mu(Q)>0$ for every cube $Q$. It suffices to prove (\ref{measure weak}) since (\ref{measure strong}) follows by Marcinkiewicz interpolation theorem since
\begin{align*}
\left\|{\bf M}_{\mu,\rho/3}^{\rm loc}\right\|_{L^{\infty}(\mu)\rightarrow L^{\infty}(\mu)}\leq 1.
\end{align*}
Let $f\in L^{1}(\mu)$ and $\alpha>0$ be given. Denote by $E_{\alpha}=\left\{x\in\mathbb{R}^{n}:{\bf M}_{\mu,\rho/3}^{\rm loc}f(x)>\alpha\right\}$. Assume that $K$ is a compact subset of $E_{\alpha}$. For every $x\in K$, there is a cube $Q_{x}$ such that $x\in Q_{x}\subseteq E_{\alpha}$ and 
\begin{align*}
\frac{1}{\mu(Q_{x})}\int_{Q_{x}}|f(y)|d\mu(y)>\alpha.
\end{align*}
Since $K$ is compact, $K$ possesses a finite covering $\left\{Q_{x_{i}}\right\}_{i=1}^{k}$. By \cite[Lemma 2.1.5]{GL}, there is a finite disjoint subcollection $\left\{Q_{x_{j_{i}}}\right\}_{i=1}^{l}$ of $\left\{Q_{x_{i}}\right\}_{i=1}^{k}$ such that
\begin{align*}
\bigcup_{i=1}^{k}Q_{x_{i}}\subseteq\bigcup_{i=1}^{l}3Q_{x_{j_{i}}}.
\end{align*}
As a consequence, we have
\begin{align*}
\mu(K)\leq\left(\bigcup_{i=1}^{k}Q_{x_{i}}\right)\leq\mu\left(\bigcup_{i=1}^{l}3Q_{x_{j_{i}}}\right)\leq\sum_{i=1}^{l}\mu\left(3Q_{x_{j_{i}}}\right)\leq D\sum_{i=1}^{l}\mu\left(Q_{x_{j_{i}}}\right),
\end{align*}
where the latter is due to the fact that $\ell\left(3Q_{x_{j_{i}}}\right)=3\ell\left(Q_{x_{j_{i}}}\right)\leq\rho$, hence the local doubling property applies. 

Subsequently, it holds that
\begin{align*}
D\sum_{i=1}^{l}\mu\left(Q_{x_{j_{i}}}\right)&\leq\frac{D}{\alpha}\sum_{i=1}^{l}\int_{Q_{x_{j_{i}}}}|f(y)|d\mu(y)\\
&\leq\frac{D}{\alpha}\int_{\bigcup_{i=1}^{l}Q_{x_{j_{i}}}}|f(y)|dy\\
&\leq\frac{D}{\alpha}\int_{\mathbb{R}^{n}}|f(y)|d\mu(y).
\end{align*}
The inner regularity of $\mu$ entails 
\begin{align*}
\mu(E_{\alpha})\leq\frac{D}{\alpha}\int_{\mathbb{R}^{n}}|f(y)|d\mu(y),
\end{align*}
hence (\ref{measure weak}) follows.
\end{proof}

Recall that the Lebesgue differentiation theorem says that
\begin{align*}
\lim_{r\rightarrow 0}\frac{1}{B_{r}(x)}\int_{B_{r}(x)}|f(x)-f(y)|dy=0
\end{align*}
holds almost everywhere $x\in\mathbb{R}^{n}$, where $f$ is a locally integrable function on $\mathbb{R}^{n}$. A standard proof to the above differentiation theorem is to use the weak type $(1,1)$ boundedness of ${\bf M}$. Having established the weak type boundedness in Proposition \ref{use weak type}, one expects the following differentiation theorem in terms of local doubling measure $\mu$.
\begin{corollary}\label{differentiation}
Let $0<\rho<\infty$ and $\mu$ be a positive measure on $\mathbb{R}^{n}$. Assume that $\mu$ satisfies the local doubling property $(\ref{local doubling property})$ with respect to $\rho$. Suppose that $\mu$ is absolutely continuous with respect to the Lebesgue measure and $Q$ is a fixed cube. For each $x\in\overline{Q}$, let $\left\{\overline{Q}_{x,N}\right\}_{N=1}^{\infty}$ be a sequence of closed cubes that containing $x$ with $\overline{Q}=\overline{Q}_{x,1}\supseteq \overline{Q}_{x,2}\supseteq\cdots\supseteq \overline{Q}_{x,N}\supseteq\cdots$ and each $\overline{Q}_{x,N+1}$ is a dyadic subcube of $\overline{Q}_{x,N}$. Then for any $f\in L_{\mu}^{1}\left(\overline{Q}\right)$,
\begin{align*}
\lim_{N\rightarrow\infty}\frac{1}{\mu(Q_{x,N})}\int_{Q_{x,N}}|f(y)-f(x)|d\mu(y)=0
\end{align*}
holds almost everywhere $x\in\mathbb{R}^{n}$ with respect to $\mu$.
\end{corollary}

\begin{proof}
Since the lengths of $Q_{x,N}$ decreases to zero uniformly in $x$, we may assume that each $Q_{x,N}$ satisfies $\ell(Q_{x,N})\leq\frac{\rho}{3}$. Let $\varepsilon,\eta>0$ be given. Extend $f$ canonically to $\mathbb{R}^{n}$ by letting $f(y)=0$ for $y\in\overline{Q}^{c}$. Then  
\begin{align*}
\int_{\mathbb{R}^{n}}|f(y)|d\mu(y)<\infty
\end{align*}
and there is a compactly supported continuous function $g$ on $\mathbb{R}^{n}$ such that
\begin{align*}
\int_{\mathbb{R}^{n}}|f(y)-g(y)|d\mu(y)<\varepsilon\cdot\frac{\eta}{2(D+1)}.
\end{align*}
Fix an $x\in Q$. We have
\begin{align*}
&\int_{Q_{x,N}}|f(y)-f(x)|d\mu(y)\\
&=\int_{Q_{x,N}}|f(y)-g(y)+g(y)-g(x)+g(x)-f(x)|d\mu(y)\\
&\leq\int_{Q_{x,N}}|f(y)-g(y)|d\mu(y)+\int_{Q_{x,N}}|g(y)-g(x)|d\mu(y)+\int_{Q_{x,N}}|f(x)-g(x)|d\mu(y)\\
&\leq\mu(Q_{x,N}){\bf M}_{\mu,\rho/3}^{\rm loc}(f-g)(x)+\int_{Q_{x,N}}|g(y)-g(x)|d\mu(y)+|f(x)-g(x)|\mu(Q_{x,N}).
\end{align*}
The uniform continuity of $g$ implies
\begin{align*}
\lim_{N\rightarrow\infty}\frac{1}{\mu(Q_{x,N})}\int_{Q_{x,N}}|g(y)-g(x)|d\mu(y)=0.
\end{align*}
Using (\ref{measure weak}) of Proposition \ref{use weak type}, we have 
\begin{align*}
&\mu\left(\left\{x\in\mathbb{R}^{n}:\limsup_{N\rightarrow\infty}\frac{1}{\mu(Q_{x,N})}\int_{Q_{x,N}}|f(y)-f(x)|d\mu(y)>\varepsilon\right\}\right)\\
&\leq\mu\left(\left\{x\in\mathbb{R}^{n}:{\bf M}_{\mu,\rho/3}^{\rm loc}(f-g)(x)>\frac{\varepsilon}{2}\right\}\right)
+\mu\left(\left\{x\in\mathbb{R}^{n}:|f(x)-g(x)|>\frac{\varepsilon}{2}\right\}\right)\\
&\leq\frac{2(D+1)}{\varepsilon}\int_{\mathbb{R}^{n}}|f(y)-g(y)|d\mu(y)\\
&<\eta.
\end{align*}
The arbitrariness of $\eta>0$ gives
\begin{align*}
\mu\left(\left\{x\in\mathbb{R}^{n}:\limsup_{N\rightarrow\infty}\frac{1}{\mu(Q_{x,N})}\int_{Q_{x,N}}|f(y)-f(x)|d\mu(y)>\varepsilon\right\}\right)=0.
\end{align*}
By letting $\varepsilon=\frac{1}{k}$, $k=1,2,\ldots$, we obtain
\begin{align*}
&\mu\left(\left\{x\in\mathbb{R}^{n}:\limsup_{N\rightarrow\infty}\frac{1}{\mu(Q_{x,N})}\int_{Q_{x,N}}|f(y)-f(x)|d\mu(y)>0\right\}\right)\\
&=\mu\left(\bigcup_{k\in\mathbb{N}}\left\{x\in\mathbb{R}^{n}:\limsup_{N\rightarrow \infty}\frac{1}{\mu(Q_{x,N})}\int_{Q_{x,N}}|f(y)-f(x)|d\mu(y)>\frac{1}{k}\right\}\right)\\
&\leq\sum_{k\in\mathbb{N}}\mu\left(\left\{x\in\mathbb{R}^{n}:\limsup_{N\rightarrow\infty}\frac{1}{\mu(Q_{x,N})}\int_{Q_{x,N}}|f(y)-f(x)|d\mu(y)>\frac{1}{k}\right\}\right)\\
&=0,
\end{align*}
which completes the proof.
\end{proof}

\subsection{The Local Reverse H\"{o}lder's Inequality for $A_{p;\rho}^{\rm loc}$}  
\enskip

Given a probability space $(X,\mu)$, we have by H\"{o}lder's inequality that 
\begin{align*}
\int_{X}|f(x)|d\mu(x)\leq\left(\int_{X}|f(x)|^{p}d\mu(x)\right)^{\frac{1}{p}},\quad 1<p<\infty
\end{align*}
for all measurable functions $f$. The local reverse form of H\"{o}lder's inequality is valid for $A_{p}^{\rm loc}$ weights. First we start with the following lemma.
\begin{lemma}\label{reverse lemma}
Let $0<\rho<\infty$, $1\leq p<\infty$, $0<\alpha<1$, and $\omega\in A_{p}^{\rm loc}$. Then there exists a $0<\beta<1$ such that whenever $S$ is a measurable subset of a cube $Q$ with $\ell(Q)\leq\rho$ and $|S|\leq\alpha|Q|$, it holds that $\omega(S)\leq\beta\omega(Q)$.
\end{lemma}

\begin{proof}
Taking $f=\chi_{A}$ in the {\bf Property 8} of Proposition \ref{many properties}, we have
\begin{align}\label{721}
\left(\frac{|A|}{|Q|}\right)^{p}\leq[\omega]_{A_{p;\rho}^{\rm loc}}\frac{\omega(A)}{\omega(Q)}
\end{align}
whenever $A\subseteq Q$ is a measurable set. Let $S=Q\setminus A$. Then
\begin{align*}
\left(1-\frac{|S|}{|Q|}\right)^{p}\leq[\omega]_{A_{p;\rho}^{\rm loc}}\left(1-\frac{\omega(A)}{\omega(Q)}\right).
\end{align*}
Given $0<\alpha<1$, we let 
\begin{align*}
\beta=1-\frac{(1-\alpha)^{p}}{[\omega]_{A_{p;\rho}^{\rm loc}}}.
\end{align*}
A routine simplification yields $\omega(S)\leq\beta\omega(Q)$.
\end{proof}

The local reverse H\"{o}lder's inequality is given as follows.
\begin{theorem}\label{reverse}
Let $0<\rho<\infty$ and $\omega$ be a weight. Assume that $\mu$ is a positive measure on $\mathbb{R}^{n}$ which satisfies the local doubling property $(\ref{local doubling property})$ with respect to $\rho$. Suppose that $\mu$ is absolutely continuous with respect to the Lebesgue measure and there exist $0<\alpha,\beta<1$ such that 
\begin{align*}
\mu(S)\leq\alpha\mu(Q)\quad\text{entails}\quad\int_{S}\omega(x)d\mu(x)\leq\beta\int_{Q}\omega(x)d\mu(x)
\end{align*}
whenever $S$ is a measurable subset of a cube $Q$ with $\ell(Q)\leq\frac{\rho}{3}$. Then there exist $0<C,\gamma<\infty$ which depend only on $n$, $D$, $\alpha$, and $\beta$ such that for every cube $Q$ with $\ell(Q)\leq\frac{\rho}{3}$, it holds that
\begin{align*}
\left(\frac{1}{\mu(Q)}\int_{Q}\omega(x)^{1+\gamma}d\mu(x)\right)^{\frac{1}{1+\gamma}}\leq\frac{C}{\mu(Q)}\int_{Q}\omega(x)d\mu(x).
\end{align*}
\end{theorem}

\begin{proof}
Let $Q$ be a cube with $\ell(Q)\leq\frac{\rho}{3}$. Denote by
\begin{align*}
\alpha_{0}=\frac{1}{|Q|}\int_{Q}\omega(x)d\mu(x).
\end{align*}
Define $\alpha_{k}=(D\alpha^{-1})^{k}\alpha_{0}$ for each $k\geq 1$. For each dyadic subcube $R$ of $Q$, let
\begin{align}\label{725}
\frac{1}{\mu(R)}\int_{R}\omega(y)d\mu(y)>\alpha_{k}
\end{align}
be the selection criterion. The cube $Q$ is not selected since it does not satisfy the selection criterion. We divide $Q$ into $2^{n}$ dyadic subcubes and select those which satisfy (\ref{725}). Then we subdivide each unselected cubes into $2^{n}$ dyadic subcubes and continue in this way indefinitely. Denote by $\{Q_{k,j}\}_{j}$ the collection of all selected subcubes of $Q$. We have the following properties.
\begin{enumerate}
\item $\alpha_{k}<\dfrac{1}{\mu(Q_{k,j})}\displaystyle\int_{Q_{k,j}}\omega(x)d\mu(x)\leq D\cdot 2^{k}$.

\medskip
\item For $\mu$-a.e. $x\in Q\setminus U_{k}$, we have $\omega(x)\leq\alpha_{k}$, where $U_{k}=\bigcup_{j}Q_{k,j}$.

\medskip
\item Each $Q_{k+1,j}$ is contained in some $Q_{k,l}$.
\end{enumerate}
Property (1) is satisfied since the unique dyadic parent of $Q_{k,j}$ is not chosen in the selection procedure and is contained in $3Q_{k,j}$ with $\ell(3Q_{k,j})=3\ell(Q_{k,j})\leq\rho$ hence the local doubling property applies. Property (2) follows by Corollary \ref{differentiation} since each $x\in Q$ is contained in a sequence of unselected closed dyadic cubes of decreasing lengths. On the other hand, each $Q_{k,j}$ is the maximal dyadic subcube of $Q$ which satisfies (\ref{725}). Since
\begin{align*}
\frac{1}{\mu(Q_{k+1,j})}\int_{Q_{k+1,k}}\omega(x)d\mu(x)>\alpha_{k+1}>\alpha_{k},
\end{align*} 
it follows that $Q_{k+1,j}$ is contained in some maximal cube that satisfying (\ref{725}), and hence Property (3) follows.

For fixed $k,l$, we have 
\begin{align*}
2^{n}\alpha_{k}&\geq\frac{1}{\mu(Q_{k,l})}\int_{Q_{k,l}}\omega(x)d\mu(x)\\
&\geq\frac{1}{\mu(Q_{k,l})}\int_{Q_{k,l}\cap U_{k+1}}\omega(x)d\mu(x)\\
&=\frac{1}{\mu(Q_{k,l})}\sum_{j}\int_{Q_{k+1,j}\cap Q_{k,l}}\omega(x)d\mu(x)\\
&=\frac{1}{\mu(Q_{k,l})}\sum_{j:Q_{k+1,j}\subseteq Q_{k,l}}\mu(Q_{k+1,j})\frac{1}{\mu(Q_{k+1,j})}\int_{Q_{k+1,j}}\omega(x)d\mu(x)\\
&\geq\frac{1}{\mu(Q_{k,l})}\sum_{j:Q_{k+1,j}\subseteq Q_{k,l}}\mu(Q_{k+1,j})\alpha_{k+1}\\
&=\frac{\mu(Q_{k,l}\cap U_{k+1})}{\mu(Q_{k,l})}\alpha_{k+1}\\
&=\frac{\mu(Q_{k,l}\cap U_{k+1})}{\mu(Q_{k,l})}2^{n}\alpha^{-1}\alpha_{k},
\end{align*}
it follows that $\mu(Q_{k,l}\cap U_{k+1})\leq\alpha\mu(Q_{k,l})$. We have by assumption that
\begin{align*}
\int_{Q_{k,l}\cap U_{k+1}}\omega(x)d\mu(x)\leq\beta\int_{Q_{k,l}}\omega(x)d\mu(x).
\end{align*}
Summing the above expression with respect to $l$, we obtain
\begin{align*}
\int_{U_{k+1}}\omega(x)d\mu(x)&=\int_{U_{k}\cap U_{k+1}}\omega(x)d\mu(x)\\
&=\sum_{l}\int_{Q_{k,l}\cap U_{k+1}}\omega(x)d\mu(x)\\
&\leq\beta\sum_{l}\int_{Q_{k,l}}\omega(x)d\mu(x)\\
&=\beta\int_{U_{k}}\omega(x)d\mu(x),
\end{align*}
which yields
\begin{align*}
\int_{U_{k}}\omega(x)d\mu(x)\leq\beta^{k}\int_{U_{0}}\omega(x)d\mu(x).
\end{align*}
We also have $\mu(\alpha_{k+1})\leq\alpha\mu(U_{k})$ and hence $\mu(U_{k})\rightarrow 0$ as $k\rightarrow\infty$. Then
\begin{align*}
Q=(Q\setminus U_{0})\bigcup\left(\bigcup_{k=0}^{\infty}(U_{k}\setminus U_{k+1})\right)
\end{align*}
modulo a set of $\mu$ measure zero. Let $\gamma>0$ be determined later. Since $\omega(x)\leq\alpha_{k}$ for $\mu$-a.e. $x\in Q\setminus U_{k}$, it follows that
\begin{align*}
\int_{Q}\omega(x)^{1+\gamma}d\mu(x)&=\int_{Q\setminus U_{0}}\omega(x)^{\gamma}\omega(x)d\mu(x)+\sum_{k=0}^{\infty}\int_{U_{k}\setminus U_{k+1}}\omega(x)^{\gamma}\omega(x)d\mu(x)\\
&\leq\alpha_{0}^{\gamma}\omega(Q\setminus U_{0})+\sum_{k=0}^{\infty}\alpha_{k+1}^{\gamma}\omega(U_{k})\\
&\leq\alpha_{0}^{\gamma}\omega(Q\setminus U_{0})+\sum_{k=0}^{\infty}\left((D\cdot\alpha^{-1})^{k+1}\alpha_{0}\right)^{\gamma}\beta^{k}\omega(U_{0})\\
&\leq\alpha_{0}^{\gamma}\left(1+(D\cdot\alpha^{-1})^{\gamma}\sum_{k=0}^{\infty}(D\cdot\alpha^{-1})^{\gamma k}\beta^{k}\right)\omega(Q)\\
&=\left(\frac{1}{\mu(Q)}\int_{Q}\omega(x)d\mu(x)\right)^{\gamma}\left(1+\frac{(D\alpha^{-1})^{\gamma}}{1-(D\cdot\alpha^{-1})^{\gamma}\beta}\right)\int_{Q}\omega(x)d\mu(x),
\end{align*}
where $\gamma$ is chosen such that 
\begin{align*}
0<\gamma<\frac{-\log\beta}{\log D-\log\alpha}.
\end{align*}
By letting 
\begin{align*}
C=\left(1+\frac{(D\cdot\alpha^{-1})^{\gamma}}{1-(D\cdot\alpha^{-1})^{\gamma}\beta}\right)^{\frac{1}{\gamma+1}},
\end{align*}
the result follows.
\end{proof}

\begin{remark}\label{reverse Ap}
\rm If $\omega\in A_{p;\rho}^{\rm loc}$, $0<\rho<\infty$, $1\leq p<\infty$, then by letting $\mu$ to be the Lebesgue measure in Theorem \ref{reverse}, one can use Lemma \ref{reverse lemma} to conclude that 
\begin{align}\label{technical}
\left(\frac{1}{|Q|}\int_{Q}\omega(x)^{1+\gamma}dx\right)^{\frac{1}{1+\gamma}}\leq\frac{C}{|Q|}\int_{Q}\omega(x)dx,\quad\ell(Q)\leq\frac{\rho}{3}
\end{align}
with the constants $\gamma,C>0$ depending only on $n$, $p$, and $[\omega]_{A_{p;\rho}^{\rm loc}}$. In fact, the constant $C$ can be taken as 
\begin{align}\label{larger}
C=\left[1+\frac{1}{\left(1-\frac{3}{4}[\omega]_{A_{p;\rho}^{\rm loc}}^{-1}\right)^{\frac{1}{2}}-\left(1-\frac{3}{4}[\omega]_{A_{p;\rho}^{\rm loc}}^{-1}\right)}\right]^{\frac{2\log\frac{3^{n}}{\alpha}}{2\log\frac{3^{n}}{\alpha}-\log\left(1-\frac{3}{4}[\omega]_{A_{p;\rho}^{\rm loc}}^{-1}\right)}}
\end{align}
(see \cite[Remark 7.2.3]{GL}). Since $1-\frac{3}{4}[\omega]_{A_{p;\rho}^{-1}}\geq\frac{1}{4}$ and the function $t\rightarrow\sqrt{t}-t$ is decreasing on $(\frac{1}{4},1)$, the constant $C$ increases as $[\omega]_{A_{p;\rho}^{\rm loc}}$ increases.

On the other hand, the scaling $\frac{\rho}{3}$ in (\ref{technical}) is purely technical. Fix a $0<\rho'<\infty$. For any weight $\omega\in A_{p;\rho'}^{\rm loc}$, since $A_{p;\rho'}^{\rm loc}=A_{p;3\rho'}^{\rm loc}$, if we let $\rho=3\rho'$, then (\ref{technical}) reduces to 
\begin{align*}
\left(\frac{1}{|Q|}\int_{Q}\omega(x)^{1+\gamma}dx\right)^{\frac{1}{1+\gamma}}\leq\frac{C}{|Q|}\int_{Q}\omega(x)dx,\quad\ell(Q)\leq\rho',
\end{align*}
where the constant $C$ now depends on a larger quantity $[\omega]_{A_{p;3\rho'}^{\rm loc}}$, and $C$ is indeed larger as well by the argument following (\ref{larger}).
\end{remark}

Given a nonnegative measurable function $\omega$, we know that ${\bf M}\omega$ is not integrable on $\mathbb{R}^{n}$ unless $\omega=0$ almost everywhere. Moreover, ${\bf M}\omega$ is generally not locally integrable on $\mathbb{R}^{n}$ unless $\omega$ satisfies the $L\log L$ condition (see \cite[8.14, page 43]{SE2}). Nevertheless, with the aid of the local reverse H\"{o}lder's inequality, we can show that ${\bf M}_{\rho}^{\rm loc}\omega$ is locally integrable on $\mathbb{R}^{n}$ if $\omega\in A_{p}^{\rm loc}$.
\begin{theorem}
Let $0<\rho<\infty$, $1\leq p<\infty$, and $\omega\in A_{p}^{\rm loc}$. Then ${\bf M}_{\rho}^{\rm loc}\omega$ is locally integrable on $\mathbb{R}^{n}$.
\end{theorem}

\begin{proof}
Let $M>0$ and $R=[-M,M]^{n}$. Let $0<\rho'<\infty$ be determined later. Note that $\omega\in A_{p;3\rho'}^{\rm loc}$. Using Theorem \ref{reverse Ap}, there are $0<\gamma,C<\infty$ such that
\begin{align}\label{use reverse}
\left(\frac{1}{|Q|}\int_{Q}\omega(x)^{1+\gamma}dx\right)^{\frac{1}{1+\gamma}}\leq\frac{C}{|Q|}\int_{Q}\omega(x)dx,\quad\ell(Q)\leq\rho'.
\end{align}
Let $\overline{R}=[-M-2\rho,M+2\rho]^{n}$. Then it is easy to see that ${\bf M}_{\rho}^{\rm loc}\omega(x)={\bf M}_{\rho}^{\rm loc}\left(\omega\chi_{\overline{R}}\right)(x)$ for all $x\in R$. As a consequence,
\begin{align*}
\int_{R}{\bf M}_{\rho}^{\rm loc}\omega(x)dx&=\int_{R}{\bf M}_{\rho}^{\rm loc}\left(\omega\chi_{\overline{R}}\right)(x)dx\\
&\leq\left(\int_{R}{\bf M}_{\rho}^{\rm loc}\left(\omega\chi_{\overline{R}}\right)(x)^{1+\gamma}dx\right)^{\frac{1}{1+\gamma}}|R|^{\frac{\gamma}{1+\gamma}}\\
&\leq|R|^{\frac{\gamma}{1+\gamma}}\left(\int_{\mathbb{R}^{n}}{\bf M}\left(\omega\chi_{\overline{R}}\right)(x)^{1+\gamma}dx\right)^{\frac{1}{1+\gamma}}\\
&\leq C(n,\gamma)|R|^{\frac{\gamma}{1+\gamma}}\left(\int_{\mathbb{R}^{n}}\left(\omega\chi_{\overline{R}}\right)(x)^{1+\gamma}dx\right)^{\frac{1}{1+\gamma}}\\
&=C(n,\gamma)|R|^{\frac{\gamma}{1+\gamma}}\left(\int_{\overline{R}}\omega(x)^{1+\gamma}dx\right)^{\frac{1}{1+\gamma}},
\end{align*}
where we have used the standard fact that ${\bf M}:L^{q}(\mathbb{R}^{n})\rightarrow L^{q}(\mathbb{R}^{n})$ is bounded for any $1<q<\infty$ in the third inequality. Now we choose $\rho'=2M+4\rho$ in (\ref{use reverse}). Then ${\bf M}_{\rho}^{\rm loc}\omega\in L^{1}(R)$. Since the cube $R$ is of arbitrary large, the local integrability of ${\bf M}_{\rho}^{\rm loc}\omega$ follows.
\end{proof}

An $A_{p}^{\rm loc}$ weight $\omega$ is indeed locally integrable on $\mathbb{R}^{n}$, but we can show that $\omega^{1+\sigma}$ is also locally integrable on $\mathbb{R}^{n}$ as well for some $\sigma>0$.
\begin{theorem}\label{second reverse}
Let $1\leq p<\infty$ and $\omega\in A_{p}^{\rm loc}$. Then there exists a constant $\gamma>0$ depending only on $n$, $p$, and $[\omega]_{A_{p}^{\rm loc}}$ such that $\omega^{1+\gamma}\in A_{p}^{\rm loc}$.
\end{theorem}

\begin{proof}
Let $0<\rho<\infty$ and $C$ be the constant in Remark \ref{reverse Ap}. When $p=1$, we apply Remark \ref{reverse Ap} to obtain
\begin{align*}
\frac{1}{|Q|}\int_{Q}\omega(x)^{1+\gamma}\leq\left(\frac{C}{|Q|}\int_{Q}\omega(x)dx\right)^{1+\gamma}\leq C^{1+\gamma}[\omega]_{A_{1;\rho}^{\rm loc}}^{1+\gamma}\omega(x)^{1+\gamma}
\end{align*}
for almost everywhere $x\in Q$ with $\ell(Q)\leq\frac{\rho}{3}$. Therefore, we have $\omega^{1+\gamma}\in A_{1;\rho/3}^{\rm loc}$ with
\begin{align*}
\left[\omega^{1+\gamma}\right]_{A_{1;\rho/3}^{\rm loc}}\leq C^{1+\gamma}[\omega]_{A_{1;\rho}^{\rm loc}}^{1+\gamma}.
\end{align*}
When $1<p<\infty$, Remark \ref{reverse Ap} entails some $\gamma_{1},\gamma_{2}>0$ and $C_{1},C_{2}>0$ such that 
\begin{align*}
\left(\frac{1}{|Q|}\int_{Q}\omega(x)^{1+\gamma_{1}}dx\right)^{\frac{1}{1+\gamma}}&\leq\frac{C_{1}}{|Q|}\int_{Q}\omega(x)dx,\\
\left(\frac{1}{|Q|}\int_{Q}\omega(x)^{-\frac{1}{p-1}(1+\gamma_{2})}dx\right)^{\frac{1}{1+\gamma}}&\leq\frac{C_{2}}{|Q|}\int_{Q}\omega(x)^{-\frac{1}{p-1}}dx
\end{align*}
where $\ell(Q)\leq\frac{\rho}{3}$. Taking $\gamma=\min(\gamma_{1},\gamma_{2})$, the above estimates are satisfied with $\gamma$ in place of $\gamma_{1},\gamma_{2}$. Hence $\omega^{1+\gamma}\in A_{p;\rho/3}^{\rm loc}$ with 
\begin{align*}
\left[\omega^{1+\gamma}\right]_{A_{p;\rho/3}^{\rm loc}}\leq\left(C_{1}C_{2}^{p-1}\right)^{1+\gamma}[\omega]_{A_{p;\rho}^{\rm loc}}^{1+\gamma}.
\end{align*}
which completes the proof.
\end{proof}

\begin{corollary}
Let $1\leq p<\infty$ and $\omega\in A_{p}^{\rm loc}$. Then there exists a $q$ depending only on $n$, $p$, and $[\omega]_{A_{p}^{\rm loc}}$ such that $\omega\in A_{q}^{\rm loc}$. In other words,
\begin{align*}
A_{p}^{\rm loc}=\bigcup_{q\in(1,p)}A_{q}^{\rm loc}.
\end{align*}
\end{corollary}

\begin{proof}
Let $\gamma$, $C_{1}$, and $C_{2}$ be as in the proof of Theorem \ref{second reverse}. For $0<\rho<\infty$, since $\omega\in A_{p;\rho}^{\rm loc}$, it holds that $\omega^{1+\gamma}\in A_{p;\rho/3}^{\rm loc}$. By letting $\delta=\frac{1}{1+\gamma}$ and
\begin{align*}
q=p\frac{1}{1+\gamma}+1-\frac{1}{1+\gamma}=\frac{p+\gamma}{1+\gamma}
\end{align*}
in Lemma \ref{small}, we have $\omega\in A_{q;\rho/3}^{\rm loc}$, $1<q<p$, and 
\begin{align*}
[\omega]_{A_{q;\rho/3}^{\rm loc}}=\left[\left(\omega^{1+\gamma}\right)^{\frac{1}{1+\gamma}}\right]_{A_{q;\rho/3}^{\rm loc}}\leq\left[\omega^{1+\gamma}\right]_{A_{p;\rho/3}^{\rm loc}}^{\frac{1}{1+\gamma}}\leq C_{1}C_{2}^{p-1}[\omega]_{A_{p;\rho}^{\rm loc}},
\end{align*}
then the result follows.
\end{proof}

\begin{lemma}\label{kolmogorov}
Let $S$ be an operator that maps $L^{1}(\mathbb{R}^{n})$ to $L^{1,\infty}(\mathbb{R}^{n})$ with norm $B$. Suppose that $f\in L^{1}(\mathbb{R}^{n})$. For any finite measure set $A$ and $0<q<1$, it holds that
\begin{align*}
\int_{A}|S(f)(x)|^{q}dx\leq\frac{1}{1-q}B^{q}|A|^{1-q}\|f\|_{L^{1}(\mathbb{R}^{n})}^{q}.
\end{align*}
In particular, 
\begin{align*}
\int_{A}{\bf M}f(x)^{q}dx\leq\frac{1}{1-q}3^{nq}|A|^{1-q}\|f\|_{L^{1}(\mathbb{R}^{n})}^{q}.
\end{align*}
\end{lemma}

\begin{proof}
We have
\begin{align*}
&\int_{A}|S(f)(x)|^{q}dx\\
&=\int_{0}^{\infty}q\alpha^{q-1}|\{x\in A:|S(f)(x)|>\alpha\}|d\alpha\\
&=\left(\int_{0}^{\frac{B}{|A|}\|f\|_{L^{1}(\mathbb{R}^{n})}}+\int_{\frac{B}{|A|}\|f\|_{L^{1}(\mathbb{R}^{n})}}^{\infty}\right)q\alpha^{q-1}|\{x\in A:|S(f)(x)|>\alpha\}|d\alpha\\
&\leq\frac{B^{q}}{|A|^{q}}\|f\|_{L^{1}(\mathbb{R}^{n})}^{q}|A|+\int_{\frac{B}{|A|}\|f\|_{L^{1}(\mathbb{R}^{n})}}^{\infty}q\alpha^{q-1}\alpha^{-1}B\|f\|_{L^{1}(\mathbb{R}^{n})}d\alpha\\
&=B^{q}|A|^{1-q}\|f\|_{L^{1}(\mathbb{R}^{n})}^{q}+\frac{q}{1-q}B\|f\|_{L^{1}(\mathbb{R}^{n})}\left(\frac{B}{|A|}\|f\|_{L^{1}(\mathbb{R}^{n})}\right)^{q-1}\\
&=\frac{1}{1-q}B^{q}|A|^{1-q}\|f\|_{L^{1}(\mathbb{R}^{n})}^{q},
\end{align*}
which completes the proof.
\end{proof}

As an application of the local reverse H\"{o}lder's inequality, we can decompose an $A_{1}^{\rm loc}$ weight as a bounded function multiplied with a maximal function with small exponent.
\begin{theorem}\label{decomposition}
Let $0<\rho<\infty$ and $\omega\in A_{1;\rho}^{\rm loc}$. Then there exist $0<\varepsilon<1$, a nonnegative function $k$ such that $k,k^{-1}\in L^{\infty}(\mathbb{R}^{n})$, and a nonnegative locally integrable function satisfying ${\bf M}_{\rho}^{\rm loc}f(x)<\infty$ almost everywhere with
\begin{align}\label{7214}
\omega(x)=k(x){\bf M}_{\rho}^{\rm loc}f(x)^{\varepsilon}\quad{\rm a.e.}.
\end{align}
\end{theorem}

\begin{proof}
Note that $\omega\in A_{1;3\rho}^{\rm loc}$. By Remark \ref{reverse Ap}, there exist $0<\gamma,C<\infty$ such that 
\begin{align*}
\left(\frac{1}{|Q|}\int_{Q}\omega(x)^{1+\gamma}dx\right)^{\frac{1}{1+\gamma}}\leq\frac{C}{|Q|}\int_{Q}\omega(x)dx\leq C[\omega]_{A_{1;3\rho}^{\rm loc}}\omega(x),\quad\ell(Q)\leq\rho
\end{align*}
for all $x\in Q\setminus E_{Q}$, where $E_{Q}\subseteq Q$ satisfies $|E_{Q}|=0$. Let
\begin{align*}
\varepsilon=\frac{1}{1+\gamma},\quad f=\omega^{1+\gamma}=\omega^{\frac{1}{\varepsilon}},
\end{align*}
and $N$ be the union of $E_{Q}$ over all $Q$ with rational vertices and $\ell(Q)\leq\rho$. Then we have
\begin{align*}
{\bf M}_{\rho}^{\rm loc}f(x)\leq C^{1+\gamma}[\omega]_{A_{1;3\rho}^{\rm loc}}\omega(x),\quad x\in\mathbb{R}^{n}\setminus N.
\end{align*}
By letting
\begin{align*}
k(x)=\frac{f(x)^{\varepsilon}}{{\bf M}_{\rho}^{\rm loc}f(x)^{\varepsilon}},\quad x\in\mathbb{R}^{n},
\end{align*}
(\ref{7214}) follows as $C^{-1}[\omega]_{A_{1;\rho}^{\rm loc}}^{-1}\leq k(x)\leq 1$ almost everywhere.
\end{proof}

\begin{remark}
\rm The converse statement of Theorem \ref{decomposition} is somewhat tricky. Let $0<\rho<\infty$, $0<\varepsilon<1$, and $f$ be a locally integrable function with ${\bf M}_{\rho}^{\rm loc}f(x)<\infty$ almost everywhere. Define $\omega={\bf M}_{\rho}^{\rm loc}f$. We claim that
\begin{align}\label{7217}
\frac{1}{|Q|}\int_{Q}{\bf M}_{\rho}^{\rm loc}f(y)^{\varepsilon}dy\leq\frac{C_{n}}{1-\varepsilon}{\bf M}_{2\rho}^{\rm loc}f(x)^{\varepsilon}
\end{align}
for almost everywhere $x\in Q$ with $\ell(Q)\leq\frac{\rho}{3}$. We write
\begin{align*}
f=f\chi_{3Q}+f\chi_{(3Q)^{c}}.
\end{align*}
Then Lemma \ref{kolmogorov} entails
\begin{align*}
\frac{1}{|Q|}\int_{Q}{\bf M}_{\rho}^{\rm loc}(f\chi_{3Q})(y)^{\varepsilon}dy&\leq\frac{C_{n}'}{1-\varepsilon}\left(\frac{1}{|Q|}\int_{\mathbb{R}^{n}}(f\chi_{3Q})(y)dy\right)^{\varepsilon}\\
&\leq\frac{C_{n}''}{1-\varepsilon}{\bf M}_{\rho}^{\rm loc}f(x)^{\varepsilon}
\end{align*}
for all $x\in Q$ since $\ell(3Q)=3\ell(Q)\leq\rho$. On the other hand, for all $x,y\in Q$ and cube $R$ with $y\in R$, $R\cap(3Q)^{c}\ne\emptyset$, $\ell(R)\leq\rho$, we have $x\in 2R$. Indeed, choose a $z_{0}\in R\cap(3Q)^{c}$. Denote by $c_{Q}$ and $c_{R}$ the centers of $Q$ and $R$ respectively. Then 
\begin{align*}
\frac{3\ell(Q)}{2}\leq|z_{0}-c_{Q}|_{\infty}\leq|z_{0}-y|_{\infty}+|y-c_{R}|_{\infty}\leq\frac{\ell(R)}{2}+\frac{\ell(Q)}{2},
\end{align*}
which yields $\ell(Q)\leq\frac{1}{2}\ell(R)$. On the other hand, we have
\begin{align*}
|x-c_{R}|_{\infty}\leq|x-y|_{\infty}+|y-c_{R}|_{\infty}\leq\frac{\ell(Q)}{2}+\frac{\ell(R)}{2}\leq\frac{1}{2}\left(\frac{1}{2}\ell(R)+\ell(R)\right)<\ell(R),
\end{align*}
which shows that $x\in 2R$. Therefore, for any $y\in Q$, it holds that
\begin{align*}
{\bf M}_{\rho}^{\rm loc}\left(f\chi_{(3Q)^{c}}\right)(y)&=\sup_{\substack{y\in R\\\ell(R)\leq\rho}}\frac{1}{|R|}\int_{R\cap(3Q)^{c}}|f(z)|dz\\
&\leq\sup_{\substack{x\in 2R\\\ell(R)\leq\rho}}\frac{1}{|R|}\int_{R\cap(3Q)^{c}}|f(z)|dz\\
&\leq 2^{n}{\bf M}_{2\rho}^{\rm loc}f(x).
\end{align*}
We obtain 
\begin{align*}
\frac{1}{|Q|}\int_{Q}{\bf M}_{\rho}^{\rm loc}f(y)^{\varepsilon}dy&\leq\frac{1}{|Q|}\int_{Q}\left({\bf M}_{\rho}^{\rm loc}(f\chi_{3Q})(y)+{\bf M}_{\rho}^{\rm loc}\left(f\chi_{(3Q)^{c}}\right)(y)\right)^{\varepsilon}dy\\
&\leq\frac{1}{|Q|}\int_{Q}{\bf M}_{\rho}^{\rm loc}(f\chi_{3Q})(y)^{\varepsilon}dy+\frac{1}{|Q|}\int_{Q}{\bf M}_{\rho}^{\rm loc}\left(f\chi_{(3Q)^{c}}\right)(y)^{\varepsilon}dy\\
&\leq\frac{C_{n}''}{1-\varepsilon}{\bf M}_{\rho}^{\rm loc}f(x)^{\varepsilon}+2^{n}{\bf M}_{2\rho}^{\rm loc}f(x)^{\varepsilon}\\
&\leq\frac{C_{n}}{1-\varepsilon}{\bf M}_{2\rho}^{\rm loc}f(x)^{\varepsilon},
\end{align*}
which yields (\ref{7217}). However, we can only conclude that 
\begin{align*}
{\bf M}_{\rho/3}^{\rm loc}\left[\left({\bf M}_{\rho}^{\rm loc}f\right)^{\varepsilon}\right]\leq\frac{C_{n}}{1-\varepsilon}\left({\bf M}_{2\rho}^{\rm loc}f\right)^{\varepsilon},
\end{align*}
which is very closed to the $A_{1}^{\rm loc}$ condition but not exactly.
\end{remark}

Another application of the local reverse H\"{o}lder's inequality is about the following reverse property to (\ref{721}).
\begin{proposition}
Let $0<\rho<\infty$, $1\leq p<\infty$, and $\omega\in A_{p}^{\rm loc}$. Then there exist $\delta>0$ and $C>0$ depending only on $n$, $p$, and $[\omega]_{A_{p}^{\rm loc}}$ such that for any measurable subset $S$ of a cube $Q$ with $\ell(Q)\leq\rho$, it holds that
\begin{align*}
\frac{\omega(S)}{\omega(Q)}\leq C\left(\frac{|S|}{|Q|}\right)^{\delta}.
\end{align*}
\end{proposition}

\begin{proof}
Let $C$ and $\gamma$ be as in Theorem \ref{reverse}. We have by H\"{o}lder's inequality that
\begin{align*}
\frac{\omega(S)}{\omega(Q)}&=\frac{1}{\omega(Q)}\int_{Q}\omega(x)\chi_{S}(x)dx\\
&\leq\frac{1}{\omega(Q)}\left(\int_{Q}\omega(x)^{1+\gamma}dx\right)^{\frac{1}{1+\gamma}}|S|^{\frac{\gamma}{1+\gamma}}\\
&=\frac{1}{\omega(Q)}\left(\frac{1}{|Q|}\int_{Q}\omega(x)^{1+\gamma}dx\right)^{\frac{1}{1+\gamma}}|Q|^{\frac{1}{1+\gamma}}|S|^{\frac{\gamma}{1+\gamma}}\\
&\leq\frac{C}{\omega(Q)}\left(\frac{1}{|Q|}\int_{Q}\omega(x)dx\right)|Q|^{\frac{1}{1+\gamma}}|S|^{\frac{\gamma}{1+\gamma}}\\
&=C\left(\frac{|S|}{|Q|}\right)^{\delta},
\end{align*}
where $\delta=\frac{\gamma}{1+\gamma}$ and the second inequality is an application of Theorem \ref{reverse}.
\end{proof}

\subsection{The $A_{\infty;\rho}^{\rm loc}$ Class}
\enskip

Recall the Jensen's inequality that 
\begin{align}\label{731}
\left(\int_{X}|h|^{q}d\mu\right)^{\frac{1}{q}}\geq\exp\left(\int_{X}\log|h|d\mu\right)
\end{align}
for all $0<q<\infty$ and measurable functions $h$ on a probability space $(X,\mu)$. Putting $f=\omega^{-1}$ with $\omega\in A_{p;\rho}^{\rm loc}$, $0<\rho<\infty$, and $q=\frac{1}{p-1}$, one has
\begin{align*}
\frac{\omega(Q)}{|Q|}\left(\frac{1}{|Q|}\int_{Q}\omega(x)^{-\frac{1}{p-1}}dx\right)^{p-1}\geq\frac{\omega(Q)}{|Q|}\exp\left(\frac{1}{|Q|}\int_{Q}\log\omega(x)^{-1}dx\right),
\end{align*}
where $Q$ is any cube with $\ell(Q)\leq\rho$. Thus we define that 
\begin{align*}
[\omega]_{A_{\infty;\rho}^{\rm loc}}=\sup_{\ell(Q)\leq\rho}\left(\frac{1}{|Q|}\int_{Q}\omega(x)dx\right)\exp\left(\frac{1}{|Q|}\int_{Q}\log\omega(x)^{-1}dx\right),
\end{align*}
and $A_{\infty;\rho}^{\rm loc}$ consists of those weights $\omega$ with finite quantities $[\omega]_{A_{\infty;\rho}^{\rm loc}}$. As an immediate consequence of the above Jensen's inequality, we have 
\begin{align*}
[\omega]_{A_{\infty;\rho}^{\rm loc}}\leq[\omega]_{A_{p;\rho}^{\rm loc}},\quad 1\leq p<\infty,\quad 0<\rho<\infty,
\end{align*}
and hence
\begin{align*}
\bigcup_{1\leq p<\infty}A_{p}^{\rm loc}\subseteq A_{\infty}^{\rm loc}.
\end{align*}
Now we introduce some elementary properties of $A_{\infty;\rho}^{\rm loc}$.
\begin{proposition}\label{properties of Ainfty}
Let $0<\rho<\infty$ and $\omega\in A_{\infty;\rho}^{\rm loc}$. Then
\begin{enumerate}
\item $[\delta^{\lambda}(\omega)]_{A_{\infty;\rho}^{\rm loc}}=[\omega]_{A_{\infty;\frac{\rho}{\lambda}}^{\rm loc}}$, where $\delta^{\lambda}(\omega)(x)=\omega(\lambda x_{1},\ldots,\lambda x_{n})$, $x\in\mathbb{R}^{n}$, and $\lambda>0$.

\medskip
\item $[\tau^{z}\omega]_{A_{\infty;\rho}^{\rm loc}}=[\omega]_{A_{\infty;\rho}^{\rm loc}}$, where $\tau^{z}(\omega)(x)=\omega(x-z)$, $x,z\in\mathbb{R}^{n}$.

\medskip
\item $[\lambda\omega]_{A_{\infty;\rho}^{\rm loc}}=[\omega]_{A_{\infty;\rho}^{\rm loc}}$, where $\lambda>0$.

\medskip
\item $[\omega]_{A_{\infty;\rho}^{\rm loc}}\geq 1$.

\medskip 
\item The following formula
\begin{align*}
[\omega]_{A_{\infty;\rho}^{\rm loc}}=\sup_{\substack{\ell(Q)\leq\rho\\\log|f|\in L^{1}(Q)\\
0<\int_{Q}|f|\omega dx<\infty}}\frac{\omega(Q)}{\displaystyle\int_{Q}|f(x)|\omega(x)dx}\exp\left(\frac{1}{|Q|}\int_{Q}\log|f(x)|dx\right)
\end{align*}
holds.

\medskip
\item The measure $\omega dx$ satisfies the local doubling with respect to $\rho$. To be more specific, for all $\lambda>1$ and cubes $Q$ with $\ell(Q)\leq\rho$, one has 
\begin{align*}
\omega(Q)\leq 2^{\lambda^{n}}[\omega]_{A_{\infty;\rho}^{\rm loc}}^{\lambda^{n}}\omega\left(\frac{1}{\lambda}Q\right).
\end{align*}
\end{enumerate}
\end{proposition}

\begin{proof}
\noindent{\bf Properties (1), (2), and (3):} These can be argued similarly as in the proof of Proposition \ref{many properties} for the first three properties. 

\medskip
\noindent{\bf Property (4):} This follows by the Jensen's inequality (\ref{731}) with the probability space $(Q,\frac{dx}{|Q|})$.

\medskip
\noindent{\bf Property (5):} By taking $f=\omega^{-1}$, we obtain the $\leq$ direction of the inequality. For the converse, (\ref{731}) gives
\begin{align*}
\exp\left(\frac{1}{|Q|}\int_{Q}\log(|f(x)|\omega(x))dx\right)\leq\frac{1}{|Q|}\int_{Q}|f(x)|\omega(x)dx,
\end{align*}
and it can be written as 
\begin{align*}
\frac{\omega(Q)}{\displaystyle\int_{Q}|f(x)|\omega(x)dx}\exp\left(\frac{1}{|Q|}\int_{Q}\log|f(x)|dx\right)\leq\frac{\omega(Q)}{|Q|}\exp\left(-\frac{1}{|Q|}\int_{Q}\log|\omega(x)|dx\right).
\end{align*}
Taking supremum over all cubes $Q$ with $\ell(Q)\leq\rho$, we obtain the other direction of the inequality.

\medskip
\noindent{\bf Property (6):} We apply {\bf Property (5)} to the cube $Q$ and the function 
\begin{align*}
f=c\chi_{\frac{1}{\lambda}Q}+\chi_{\mathbb{R}^{n}\setminus\frac{1}{\lambda}Q},
\end{align*}
where $c>1$ is chosen such that $c=(2[\omega]_{A_{\infty;\rho}^{\rm loc}})^{\lambda^{n}}$. Then we obtain
\begin{align*}
\frac{\omega(Q)}{c\cdot\omega\left(\frac{1}{\lambda}Q\right)+\omega\left(Q\setminus\frac{1}{\lambda}Q\right)}\exp\left(\frac{\log c}{\lambda^{n}}\right)\leq[\omega]_{A_{\infty;\rho}^{\rm loc}}.
\end{align*}
The estimate then follows by routine simplification.
\end{proof}

\begin{lemma}\label{create doubling}
Let $0<\rho<\infty$ and $\mu$ be a positive measure on $\mathbb{R}^{n}$. Suppose that there are $0<\alpha,\beta<1$ such that for all cubes $Q$ with $\ell(Q)\leq\rho$ and measurable sets $A\subseteq Q$,
\begin{align*}
|A|\leq\alpha|Q|\quad\text{entails}\quad\mu(A)\leq\beta\mu(Q).
\end{align*}
Then $\mu$ is a local doubling measure with respect to $\rho$. 
\end{lemma}

\begin{proof}
Given a cube $Q$ with $\ell(Q)\leq\rho$ and measurable set $S\subseteq Q$, let $A=Q\setminus S$. We have by assumption that
\begin{align*}
|Q|-|S|\leq\alpha|Q|\quad\text{entails}\quad\mu(Q)-\mu(S)\leq\beta\mu(Q),
\end{align*}
which yields
\begin{align*}
|S|\geq(1-\alpha)|Q|\quad\text{entails}\quad\mu(S)\geq(1-\beta)\mu(Q).
\end{align*}
Choose a $0<\lambda<1$ such that $\lambda^{n}>1-\alpha$. Since $0<\lambda<1$, we have $S=\lambda Q\subseteq Q$. Then $|S|=\lambda^{n}|Q|>(1-\alpha)|Q|$ entails
\begin{align*}
\mu(Q)\leq\frac{1}{1-\beta}\mu(\lambda Q).
\end{align*}
Pick a $q\in\mathbb{N}$ such that $\lambda^{q}<1/3$. Since $\ell(\lambda Q)=\lambda\ell(Q)<\ell(Q)\leq\rho$, iteration yields
\begin{align*}
\mu(Q)\leq\frac{1}{(1-\beta)^{q}}\mu(\lambda^{q}Q)\leq\frac{1}{(1-\beta)^{q}}\mu\left(\frac{1}{3}Q\right),
\end{align*}
and the result follows.
\end{proof}

Further characterizations of $A_{\infty;\rho}^{\rm loc}$ are given as follows.
\begin{theorem}\label{further properties Ainfty}
Let $0<\rho<\infty$ and $\omega$ be a weight. Then $\omega\in A_{\infty;\rho}^{\rm loc}$ if and only if any one of the following conditions holds.
\begin{enumerate}
\item There exist $0<\alpha,\beta<1$ such that for all cubes $Q$ with $\ell(Q)\leq\rho$, it holds that
\begin{align*}
\left|\left\{x\in Q:\omega(x)\leq\gamma\cdot{\rm Avg}_{Q}\omega\right\}\right|\leq\delta|Q|.
\end{align*}

\medskip
\item There exist $0<\alpha,\beta<1$ such that for all cubes $Q$ with $\ell(Q)\leq\rho$ and measurable subsets $A$ of $Q$, it holds that
\begin{align*}
|A|\leq\alpha|Q|\quad\text{entails}\quad\omega(A)\leq\beta\omega(Q).
\end{align*}

\medskip
\item There exist $0<C_{1},\varepsilon<\infty$ such that for all cubes with $\ell(Q)\leq\rho$, it holds that
\begin{align*}
\left(\frac{1}{|Q|}\int_{Q}\omega(x)^{1+\varepsilon}dx\right)\leq\frac{C_{1}}{|Q|}\int_{Q}\omega(x)dx.
\end{align*}

\medskip
\item There exist $0<C_{2},\varepsilon_{0}<\infty$ such that for all cubes $Q$ with $\ell(Q)\leq\rho$ and measurable subsets $A$ of $Q$, it holds that
\begin{align*}
\frac{\omega(A)}{\omega(Q)}\leq C_{2}\left(\frac{|A|}{|Q|}\right)^{\varepsilon_{0}}.
\end{align*}

\medskip
\item There exist $0<\alpha',\beta'<1$ such that for all cubes $Q$ with $\ell(Q)\leq\rho$ and measurable subsets $A$ of $Q$, it holds that
\begin{align*}
\omega(A)<\alpha'\omega(Q)\quad\text{entails}\quad|A|<\beta'|Q|.
\end{align*}

\medskip
\item There exist $p,C_{3}<\infty$ such that $[\omega]_{A_{p;\rho}^{\rm loc}}\leq C_{3}$. In other words, $\omega\in A_{p}^{\rm loc}$ for some $1\leq p<\infty$.
\end{enumerate}
\end{theorem}

\begin{proof}
\noindent{\bf $\omega\in A_{\infty;\rho}^{\rm loc}$ entails (1):} By {\bf Property (3)} of Proposition \ref{properties of Ainfty}, we have $[\lambda\omega]_{A_{\infty;\rho}^{\rm loc}}=[\omega]_{A_{\infty;\rho}^{\rm loc}}$ for any $\lambda>0$. Fix a cube $Q$ with $\ell(Q)\leq\rho$. By choosing 
\begin{align*}
\lambda=\exp\left(-\frac{1}{|Q|}\int_{Q}\omega(x)dx\right),
\end{align*}
one may assume that 
\begin{align*}
\int_{Q}\log\omega(x)dx=0.
\end{align*}
Then ${\rm Avg}_{Q}\leq[\omega]_{A_{\infty;\rho}^{\rm loc}}$. Hence
\begin{align*}
&\left|\left\{x\in Q:\omega(x)\leq\gamma\cdot{\rm Avg}_{Q}\omega\right\}\right|\\
&\leq\left|\left\{x\in Q:\omega(x)\leq\gamma[\omega]_{A_{\infty;\rho}^{\rm loc}}\right\}\right|\\
&=\left|\left\{x\in Q:\log\left(1+\omega(x)^{-1}\right)\geq\log\left(1+\left(\gamma[\omega]_{A_{\infty;\rho}^{\rm loc}}\right)^{-1}\right)\right\}\right|\\
&\leq\frac{1}{\log\left(1+\left(\gamma[\omega]_{A_{\infty;\rho}^{\rm loc}}\right)^{-1}\right)}\int_{Q}\log\frac{1+\omega(x)}{\omega(x)}dx\\
&=\frac{1}{\log\left(1+\left(\gamma[\omega]_{A_{\infty;\rho}^{\rm loc}}\right)^{-1}\right)}\int_{Q}\log(1+\omega(x))dx\\
&\leq\frac{1}{\log\left(1+\left(\gamma[\omega]_{A_{\infty;\rho}^{\rm loc}}\right)^{-1}\right)}\int_{Q}\omega(x)dx\\
&\leq\frac{[\omega]_{A_{\infty;\rho}^{\rm loc}}|Q|}{\log\left(1+\left(\gamma[\omega]_{A_{\infty;\rho}^{\rm loc}}\right)^{-1}\right)}\\
&=\frac{1}{2}|Q|,
\end{align*}
where $\gamma=[\omega]_{A_{\infty;\rho}^{\rm loc}}^{-1}\left(e^{2[\omega]_{A_{\infty;\rho}^{\rm loc}}}-1\right)^{-1}$ and $\delta=\frac{1}{2}$.

\medskip
\noindent{\bf (1) entails (2):} Let $Q$ be a fixed cube with $\ell(Q)\leq\rho$. Assume that $A\subseteq Q$ is measurable with $\omega(A)>\beta\omega(Q)$ for some $\beta>0$ to be determined later. By letting $S=Q\setminus A$, we have $\omega(S)<(1-\beta)\omega(Q)$. We write $S=S_{1}\cup S_{2}$, where
\begin{align*}
S_{1}&=\left\{x\in S:\omega(x)>\gamma\cdot{\rm Avg}_{Q}\omega\right\},\\
S_{2}&=\left\{x\in S:\omega(x)\leq\gamma\cdot{\rm Avg}_{Q}\omega\right\}.
\end{align*}
The assumption in (1) entails $|S_{2}|\leq\gamma|Q|$. On the other hand, we have 
\begin{align*}
|S_{1}|\leq\frac{1}{\gamma\cdot{\rm Avg}_{Q}\omega}\int_{S}\omega(x)dx=\frac{|Q|}{\gamma}\frac{\omega(S)}{\omega(Q)}<\frac{1-\beta}{\gamma}|Q|.
\end{align*}
Then
\begin{align*}
|S|\leq|S_{1}|+|S_{2}|<\frac{1-\beta}{\gamma}|Q|+\delta|Q|=\left(\delta+\frac{1-\beta}{\gamma}\right)|Q|.
\end{align*}
Let $\alpha=\frac{1-\delta}{2}$ and $\beta=1-\frac{(1-\delta)\gamma}{2}$. We obtain $|S|<(1-\alpha)|Q|$, then $|A|>\alpha|Q|$.

\medskip
\noindent{\bf (2) entails (3):} Note that the Lebesgue measure is local doubling with respect to any $0<\rho'<\infty$. By letting $\rho'=3\rho$, then Theorem \ref{reverse} finishes the proof with $D=3^{n}$ and 
\begin{align*}
\varepsilon=-\frac{1}{2}\frac{\log\beta}{\log 3^{n}-\log\alpha},\quad C_{1}=1+\frac{(3^{n}\alpha^{-1})^{\varepsilon}}{1-(3^{n}\alpha^{-1})^{\varepsilon}\beta}.
\end{align*}

\medskip
\noindent{\bf (3) entails (4):} We have by H\"{o}lder's inequality that
\begin{align*}
\int_{A}\omega(x)dx&\leq\left(\int_{A}\omega(x)^{1+\varepsilon}dx\right)^{\frac{1}{1+\varepsilon}}|A|^{\frac{\varepsilon}{1+\varepsilon}}\\
&\leq\left(\frac{1}{|Q|}\int_{Q}\omega(x)^{1+\varepsilon}dx\right)^{\frac{1}{1+\varepsilon}}|Q|^{\frac{1}{1+\varepsilon}}|A|^{\frac{\varepsilon}{1+\varepsilon}}\\
&\leq\frac{C_{1}}{|Q|}\int_{Q}\omega(x)dx|Q|^{\frac{1}{1+\varepsilon}}|A|^{\frac{\varepsilon}{1+\varepsilon}},
\end{align*} 
which yields
\begin{align*}
\frac{\omega(A)}{\omega(Q)}\leq C_{1}\left(\frac{|A|}{|Q|}\right)^{\frac{\varepsilon}{1+\varepsilon}}.
\end{align*}
We may now choose $\varepsilon_{0}=\frac{\varepsilon}{1+\varepsilon}$ and $C_{2}=C_{1}$.

\medskip
\noindent{\bf (4) entails (5):} Pick an $0<\alpha''<1$ such that $\beta''=C_{2}(\alpha'')^{\varepsilon_{0}}<1$. Then
\begin{align}\label{735}
|A|\leq\alpha''|Q|\quad\text{entails}\quad\omega(A)\leq\beta''\omega(Q)
\end{align}
for all cubes $Q$ with $\ell(Q)\leq\rho$ and measurable subsets $A$ of $Q$. Replacing $A$ by $Q\setminus A$, the entailment in (\ref{735}) is equivalent to
\begin{align}\label{735 prime}
|A|\geq(1-\alpha'')|Q|\quad\text{entails}\quad\omega(A)\geq(1-\beta'')\omega(Q).
\end{align}
In other words, for measurable sets $A\subseteq Q$, one has
\begin{align}\label{736}
\omega(A)<(1-\beta'')\omega(Q)\quad\text{entails}\quad|A|<(1-\alpha'')|Q|.
\end{align}
Now we let $\alpha'=1-\beta''$ and $\beta'=1-\alpha''$.

\medskip
\noindent{\bf (5) entails (6):} By letting $d\mu(x)=\omega(x)dx$, the assumption says that
\begin{align*}
\mu(A)<\alpha'\mu(Q)\quad\text{entails}\quad|A|<\beta'|B|.
\end{align*}
A similar form of the equivalence between (\ref{735}) and (\ref{736}) yields
\begin{align*}
|A|\leq(1-\beta')|B|\quad\text{entails}\quad\mu(A)\leq(1-\alpha')\omega(Q),
\end{align*}
which implies by Lemma \ref{create doubling} that $\mu$ is local doubling with respect to $\rho$. Note that the assumption is also equivalent to 
\begin{align*}
\mu(A)<\alpha'\mu(Q)\quad\text{entails}\quad\int_{A}\omega(x)^{-1}d\mu(x)<\beta'\int_{B}\omega(x)^{-1}d\mu(x),
\end{align*}
which implies by Theorem \ref{reverse} that
\begin{align*}
\left(\frac{1}{\mu(Q)}\int_{Q}\omega(x)^{-1-\gamma}d\mu(x)\right)^{\frac{1}{1+\gamma}}\leq\frac{C}{\mu(Q)}\int_{Q}\omega(x)^{-1}d\mu(x)
\end{align*}
for all cubes $Q$ with $|Q|\leq\frac{\rho}{3}$. By letting $p=1+\frac{1}{\gamma}$, the above can be written as 
\begin{align*}
\left(\frac{1}{\omega(Q)}\int_{Q}\omega(x)^{-\frac{1}{p-1}}dx\right)^{\frac{p-1}{p}}\leq\frac{C}{\omega(Q)}|Q|,
\end{align*}
which yields the $A_{p;\rho/3}^{\rm loc}$ condition of $\omega$ with $[\omega]_{A_{p;\rho/3}^{\rm loc}}\leq C^{p}$. Using Theorem \ref{important}, we may take
\begin{align*}
C_{3}=\left(C(n,p)\max\left(\frac{3^{np}}{\rho^{np}},1\right)C^{p}\right)^{3p+1}C^{p}.
\end{align*}

\medskip
\noindent{\bf (6) entails $\omega\in A_{\infty;\rho}^{\rm loc}$:} Simply note that $[\omega]_{A_{\infty;\rho}^{\rm loc}}\leq[\omega]_{A_{p;\rho}^{\rm loc}}$.
\end{proof}
Since $A_{p;\rho_{1}}^{\rm loc}=A_{p;\rho_{2}}^{\rm loc}$ for any $0<\rho_{1},\rho_{2}<\infty$ and $1\leq p<\infty$, the characterizations above show that $A_{\infty;\rho_{1}}^{\rm loc}=A_{\infty;\rho_{2}}^{\rm loc}$. Denote $A_{\infty;1}^{\rm loc}$ by $A_{\infty}^{\rm loc}$. Then we have 
\begin{align*}
A_{\infty}^{\rm loc}=\bigcup_{1\leq p<\infty}A_{p}^{\rm loc}.
\end{align*}
As before, the symbol $A_{\infty;\rho}^{\rm loc}$ is used when we want to measure $\omega\in A_{\infty}^{\rm loc}$ with $\left[\cdot\right]_{A_{\infty;\rho}^{\rm loc}}$. It is worth noting that Rychkov \cite{RV} defined $A_{\infty}^{\rm loc}(\alpha)$ to be the class that consisting of weights $\omega$ such that
\begin{align*}
[\omega]_{A_{\infty}^{\rm loc}(\alpha)}=\sup_{|Q|\leq 1}\left(\sup_{F\subseteq Q,|F|\geq\alpha|Q|}\frac{\omega(Q)}{\omega(F)}\right)<\infty,\quad 0<\alpha<1.
\end{align*}
As we have seen that (\ref{735}) and (\ref{735 prime}) are equivalent, and they are also equivalent to say that $\omega\in A_{p}^{\rm loc}$, we see that $A_{\infty}^{\rm loc}(\alpha)=A_{\infty}^{\rm loc}$. Moreover, by Lemma \ref{reverse lemma}, (\ref{735 prime}), and Theorem \ref{further properties Ainfty}, the finiteness of $\left[\cdot\right]_{A_{\infty}^{\rm loc}(\alpha_{0})}$ for a fixed $0<\alpha_{0}<1$ implies $\left[\cdot\right]_{A_{\infty}^{\rm loc}(\alpha)}<\infty$ for all $0<\alpha<1$. 
\begin{remark}\label{increasing remark}
\rm Throughout the proof of Theorem \ref{further properties Ainfty} (1), one may choose that $\gamma=[\omega]_{A_{\infty;\rho}^{\rm loc}}^{-1}\left(e^{c[\omega]_{A_{\infty;\rho}^{\rm loc}}}-1\right)^{-1}$ and $\delta=\frac{1}{c}$ for arbitrary $c>1$. Since $[\omega]_{A_{\infty;\rho}^{\rm loc}}\geq 1$, the term $\gamma$ decreases exponentially as $[\omega]_{A_{\infty;\rho}^{\rm loc}}$ increases. As a direct consequence, the term $\beta=1-\frac{(1-\delta)\gamma}{2}$ in (2) satisfies $\beta\geq\frac{1}{4}$ for large $c>1$. After a routine simplification, the term $C_{1}$ in (3) becomes 
\begin{align*}
C_{1}=1+\frac{(3^{n}\alpha^{-1})^{\varepsilon}}{1-(3^{n}\alpha^{-1})^{\varepsilon}\beta}=1+\frac{1}{\sqrt{\beta}-\beta},
\end{align*}
and $C_{1}$ increases as $[\omega]_{A_{\infty;\rho}^{\rm loc}}$ increases since the function $t\rightarrow\sqrt{t}-t$ is decreasing on $(\frac{1}{4},1)$. Finally, we conclude that when $[\omega]_{A_{\infty;\rho}^{\rm loc}}$ increases, the term $\varepsilon_{0}$ decreases while $C_{2}=C_{1}$ increases in (4). 
\end{remark}

\newpage
\section{On the Capacities Associated with $A_{p;\rho}^{\rm loc}$}
Let $0<\alpha<\infty$, $1<p<\infty$, and $n\in\mathbb{N}$. We define the Bessel capacities by 
\begin{align*}
\text{Cap}_{\alpha,p}(E)&=\inf\left\{\|f\|_{L^{p}(\mathbb{R}^{n})}^{p}:f\geq 0,~G_{\alpha}\ast f\geq 1~\text{on}~E\right\},
\end{align*}
where 
\begin{align*}
G_{\alpha}(x)=\mathcal{F}^{-1}\left[\left(1+4\pi^{2}\left|\cdot\right|^{2}\right)^{-\frac{\alpha}{2}}\right](x),\quad x\in\mathbb{R}^{n},
\end{align*}
is the Bessel kernel, $\mathcal{F}^{-1}$ is the inverse distributional Fourier transform on $\mathbb{R}^{n}$. When $0<\alpha<n$, the Riesz capacities are defined by
\begin{align*}
\text{cap}_{\alpha,p}(E)&=\inf\left\{\|f\|_{L^{p}(\mathbb{R}^{n})}^{p}:f\geq 0,~I_{\alpha}\ast f\geq 1~\text{on}~E\right\},
\end{align*}
where $I_{\alpha}(x)=\mathcal{F}^{-1}(\left|\cdot\right|^{-\alpha})(x)$, $x\in\mathbb{R}^{n}$ is the Riesz kernel, and we have
\begin{align*}
I_{\alpha}(x)=C(n,\alpha)|x|^{\alpha-n},\quad x\in\mathbb{R}^{n}
\end{align*}
for some constant $C(n,\alpha)>0$ depending only on $n$ and $\alpha$. Meanwhile, the Bessel and Riesz kernels coincide locally at $x=0$, to wit:
\begin{align}\label{Bessel zero}
G_{\alpha}(x)=C(n,\alpha)|x|^{\alpha-n}+o(|x|^{\alpha-n}),\quad|x|\rightarrow 0.
\end{align}
On the other hand, the behavior of $G_{\alpha}(\cdot)$ is given as
\begin{align*}
G_{\alpha}(x)=C(\alpha)|x|^{-\frac{n+1-\alpha}{2}}e^{-|x|}+o\left(|x|^{-\frac{n+1-\alpha}{2}}e^{-|x|}\right),\quad|x|\rightarrow\infty,
\end{align*}
which yields
\begin{align}\label{Bessel infinity}
G_{\alpha}(x)=O\left(e^{-c|x|}\right),\quad|x|\rightarrow\infty
\end{align}
for any $0<c<1$ (see \cite[Section 1.2.5]{AH2}).

Assume further that $1<p<\frac{n}{\alpha}$. Then the Bessel capacities are of local Riesz type in the sense that
\begin{align*}
{\rm cap}_{\alpha,p}(E)\leq C(n,\alpha,p){\rm Cap}_{\alpha,p}(E)
\end{align*}
holds for arbitrary set $E\subseteq\mathbb{R}^{n}$, and 
\begin{align*}
{\rm Cap}_{\alpha,p}(E)\leq C(n,\alpha,p,R){\rm cap}_{\alpha,p}(E)
\end{align*}
holds for arbitrary set $E\subseteq\mathbb{R}^{n}$ with diameter at most $R>0$ (see \cite[Proposition 5.1.4]{AH2}).

Now we investigate the general construction of capacities, from which we can construct the capacities associated with local Muckenhoupt weights. 

\subsection{General Construction of Capacities}
\enskip

Here we present the construction of capacities in accord with Meyers' theory \cite{MN}. The assertions in this subsection are mainly taken from \cite[Chapter 3]{TB}, which their proofs will be omitted here, the readers may consult the chapter therein for details.

Let $n,m\in\mathbb{N}$ and $\nu$ be a positive measure on $\mathbb{R}^{m}$. Suppose that $k$ is a nonnegative measurable function on $\mathbb{R}^{n}\times\mathbb{R}^{m}$. We say that $k$ is a kernel provided that the function $k(\cdot,y)$ is lower semicontinuous on $\mathbb{R}^{n}$ $\nu$-a.e. $y\in\mathbb{R}^{m}$ and the function $k(x,\cdot)$ is $\nu$-measurable on $\mathbb{R}^{m}$ for every $x\in\mathbb{R}^{n}$. If $\mu$ is a positive measure on $\mathbb{R}^{n}$, the potential $k(\mu,\cdot)$ of $\mu$ is defined by 
\begin{align*}
k(\mu,y)=\int_{\mathbb{R}^{n}}k(x,y)d\mu(x),\quad\nu\text{-a.e.}~y\in\mathbb{R}^{m}.
\end{align*}
For general measure $\mu$ on $\mathbb{R}^{n}$, we write the canonical Jordan decomposition of $\mu$ as $\mu=\mu^{+}-\mu^{-1}$. For $\nu$-a.e. $y\in\mathbb{R}^{m}$, we then define 
\begin{align*}
k(\mu,y)=k(\mu^{+},y)-k(\mu^{-},y)
\end{align*}
provided that either $k(\mu^{+},y)$ or $k(\mu^{-},y)$ is finite, and $k(\mu,y)=\infty$ for otherwise. Let $f$ be a nonnegative $\nu$-measurable function on $\mathbb{R}^{m}$. The potential $k(\cdot,f\nu)$ is then defined by
\begin{align*}
k(x,f\nu)=\int_{\mathbb{R}^{m}}k(x,y)f(y)d\nu(y),\quad x\in\mathbb{R}^{n}.
\end{align*}
If $f=f^{+}-f^{-}$ is an arbitrary $\nu$-measurable function on $\mathbb{R}^{m}$, we define 
\begin{align*}
k(x,f\nu)=k(x,f^{+}\nu)-k(x,f^{-}\nu),\quad x\in\mathbb{R}^{n}
\end{align*}
provided that at least one of $k(x,f^{+}\nu)$, $k(x,f^{-}\nu)$ is finite.

Let $1<p<\infty$. The $L^{p}$-capacity $C_{k,\nu,p}(\cdot)$ of a set $E\subseteq\mathbb{R}^{n}$ with respect to $k$ and $\nu$ is defined by
\begin{align*}
C_{k,\nu,p}(E)=\inf\left\{\|f\|_{L_{\nu}^{p}(\mathbb{R}^{m})}^{p}:f\in L_{\nu}^{p}(\mathbb{R}^{m})^{+},~k(\cdot,f\nu)\geq 1~\text{on}~E\right\},
\end{align*}
where $L_{\nu}^{p}(\mathbb{R}^{m})^{+}$ is the cone of nonnegative functions in $L_{\nu}^{p}(\mathbb{R}^{m})$. Given a property $P(\cdot)$ defined on $\mathbb{R}^{n}$, we say that $P(\cdot)$ holds $C_{k,\nu,p}(\cdot)$-quasi-everywhere $x\in\mathbb{R}^{n}$ provided that the set
\begin{align*}
N=\{x\in\mathbb{R}^{n}:P(x)~\text{is false}\}
\end{align*} 
satisfies that $C_{k,\nu,p}(N)=0$. Suppose that $f\in L_{\nu}^{p}(\mathbb{R}^{m})$. Then $k(x,f\nu)$ is defined $C_{k,\nu,p}(\cdot)$-quasi-everywhere $x\in\mathbb{R}^{n}$. 
\begin{proposition}
If $f\in L_{\nu}^{p}(\mathbb{R}^{m})$ and 
\begin{align*}
E=\{x\in\mathbb{R}^{n}:k(x,|f|\nu)=\infty\},
\end{align*}
then $C_{k,\nu,p}(E)=0$.
\end{proposition}

The dual definition of the capacity $C_{k,\nu,p}(\cdot)$ is a consequence of the minimax theorem.
\begin{proposition}\label{dual definition}
If $E\subseteq\mathbb{R}^{n}$ is a Borel set, then 
\begin{align*}
C_{k,\nu,p}(E)^{\frac{1}{p}}=\sup\left\{\mu(E):\mu\in\mathcal{M}^{+}(E),~\|k(\mu,\cdot)\|_{L_{\nu}^{p'}(\mathbb{R}^{m})}\leq 1\right\}.
\end{align*}
Assume further that $0<C_{k,\nu,p}(E)<\infty$. Then 
\begin{align}\label{323}
C_{k,nu,p}(E)^{-p'}=\inf\{\|k(\mu,\cdot)\|_{L_{\nu}^{p'}(\mathbb{R}^{m})}^{p'}:\mu\in\mathcal{M}^{+}(E),~\mu(E)=1\}.
\end{align}
\end{proposition}

The capacity $C_{k,\nu,p}(\cdot)$ satisfies the outer regularity.
\begin{proposition}\label{outer}
For every $E\subseteq\mathbb{R}^{n}$, it holds that
\begin{align*}
C_{k,\nu,p}(E)=\inf\{C_{k,\nu,p}(G):G\supseteq E,~G~\text{open}\}.
\end{align*}
\end{proposition}

The capacity $C_{k,\nu,p}(\cdot)$ is subadditive.
\begin{proposition}
If $E_{N}$ is a subset of $\mathbb{R}^{n}$ for $N=1,2,\ldots$, then
\begin{align*}
C_{k,\nu,p}\left(\bigcup_{N=1}^{\infty}E_{N}\right)\leq\sum_{N=1}^{\infty}C_{k,\nu,p}(E_{N}).
\end{align*}
\end{proposition}

The next two propositions show that $C_{k,\nu,p}(\cdot)$ is continuous from the right on compact subsets of $\mathbb{R}^{n}$ and continuous from the left on arbitrary subsets of $\mathbb{R}^{n}$.
\begin{proposition}\label{decreasing compact}
If $K_{1}\supseteq K_{2}\supseteq\cdots$ is a decreasing sequence of compact subsets of $\mathbb{R}^{n}$, then
\begin{align*}
C_{k,\nu,p}\left(\bigcap_{N=1}^{\infty}K_{N}\right)=\lim_{N\rightarrow\infty}C_{k,\nu,p}(E_{N}).
\end{align*}
\end{proposition}

\begin{proposition}\label{monotone}
If $E_{1}\subseteq E_{2}\subseteq\cdots$ is an increasing sequence of subsets of $\mathbb{R}^{n}$, then
\begin{align*}
C_{k,\nu,p}\left(\bigcup_{N=1}^{\infty}E_{N}\right)=\lim_{N\rightarrow\infty}C_{k,\nu,p}(E_{N}).
\end{align*}
\end{proposition}

By a theorem of Choquet (see \cite{CG}), one can show that the capacity $C_{k,\nu,p}(\cdot)$ is inner regular with respect to Borel subsets of $\mathbb{R}^{n}$.
\begin{proposition}\label{borel}
If $E\subseteq\mathbb{R}^{n}$ is a Borel set, then 
\begin{align*}
C_{k,\nu,p}(E)=\sup\{C_{k,\nu,p}(K):K\subseteq E,~K~\text{compact}\}.
\end{align*} 
\end{proposition}

Assume that $\mu$ is a positive measure on $\mathbb{R}^{n}$. The nonlinear potential $V_{k,\nu,p}^{\mu}$ of $\mu$ is defined to be 
\begin{align*}
V_{k,\nu,p}^{\mu}(x)=k\left(x,k(\mu,\cdot)^{p'-1}\nu\right),\quad x\in\mathbb{R}^{n}.
\end{align*}
Note that $V_{k,\nu,p}^{\mu}$ is lower semicontinuous by Fatou's lemma. Moreover, we have
\begin{align*}
\int_{\mathbb{R}^{n}}V_{k,\nu,p}^{\mu}(x)d\mu(x)&=\int_{\mathbb{R}^{n}}\int_{\mathbb{R}^{m}}k(x,y)k(\mu,y)^{p'-1}d
\nu(y)d\mu(x)\\
&=\int_{\mathbb{R}^{n}}\int_{\mathbb{R}^{m}}k(x,y)\left(\int_{\mathbb{R}^{n}}k(z,y)d\mu(z)\right)^{p'-1}d
\nu(y)d\mu(x)\\
&=\int_{\mathbb{R}^{m}}\int_{\mathbb{R}^{n}}k(x,y)\left(\int_{\mathbb{R}^{n}}k(z,y)d\mu(z)\right)^{p'-1}d
\mu(x)d\nu(y)\\
&=\int_{\mathbb{R}^{m}}\left(\int_{\mathbb{R}^{n}}k(x,y)d\mu(x)\right)^{p'}d
\mu(x)d\nu(y)\\
&=\int_{\mathbb{R}^{m}}k(\mu,y)^{p'}d\nu(y)\\
&=\|k(\mu,\cdot)\|_{L_{\nu}^{p'}(\mathbb{R}^{m})}^{p'}.
\end{align*}

\begin{proposition}\label{cap measure}
If $K\subseteq\mathbb{R}^{n}$ is compact with $0<C_{k,\nu,p}(K)<\infty$, then there is a measure $\mu^{K}\in\mathcal{M}^{+}(K)$ such that
\begin{align*}
V_{k,\nu,p}^{\mu^{K}}(x)&\geq 1\quad C_{k,\nu,p}(\cdot)\text{-quasi-everywhere}~x\in K,\\
V_{k,\nu,p}^{\mu^{K}}(x)&\leq 1\quad\text{for every}~x\in{\rm supp}\left(\mu^{K}\right),
\end{align*} 
and 
\begin{align*}
\mu^{K}(K)=\int_{\mathbb{R}^{m}}k(\mu,y)^{p'}d\nu(y)=\int_{\mathbb{R}^{n}}V_{k,\nu,p}^{\mu^{K}}(x)d\mu(x)=C_{k,\nu,p}(K).
\end{align*}
\end{proposition}

A consequence of the above proposition is the following.
\begin{proposition}\label{3214}
If $K\subseteq\mathbb{R}^{n}$ is compact, then
\begin{align*}
C_{k,\nu,p}(K)=\sup\left\{\mu(K):\mu\in\mathcal{M}^{+}(K),~V_{k,\nu,p}^{\mu}(x)\leq 1~\text{for every}~x\in{\rm supp}(\mu)\right\}.
\end{align*}
\end{proposition}

Proposition \ref{cap measure} can be extended to arbitrary sets $E\subseteq\mathbb{R}^{n}$ if the kernel $k$ satisfies certain extra conditions.
\begin{proposition}\label{trick}
Suppose that the function $k(\cdot,\varphi\nu)$ is continuous on $\mathbb{R}^{n}$ and $\lim\limits_{|x|\rightarrow\infty}k(x,\varphi\nu)=0$ for every function $\varphi$ in a dense subset of $L_{\nu}^{p}(\mathbb{R}^{m})$. Let $E\subseteq\mathbb{R}^{n}$ be an arbitrary subset of $\mathbb{R}^{n}$ with $0<C_{k,\nu,p}(E)<\infty$. Then there is a measure $\mu^{E}\in\mathcal{M}^{+}\left(\overline{E}\right)$ such that
\begin{align*}
V_{k,\nu,p}^{\mu^{E}}(x)&\geq 1\quad C_{k,\nu,p}(\cdot)\text{-quasi-everywhere}~x\in E,\\
V_{k,\nu,p}^{\mu^{E}}(x)&\leq 1\quad\text{for every}~x\in{\rm supp}\left(\mu^{E}\right),
\end{align*} 
and 
\begin{align*}
\mu^{E}\left(\overline{E}\right)=\int_{\mathbb{R}^{m}}k(\mu,y)^{p'}d\nu(y)=\int_{\mathbb{R}^{n}}V_{k,\nu,p}^{\mu^{E}}(x)d\mu(x)=C_{k,\nu,p}(E).
\end{align*}
\end{proposition}

\subsection{Definitions of Weighted Capacities}
\enskip

We begin by introducing the capacities associated with general weights. Let $0<\alpha<n$, $1<p<\infty$, $0<\rho<\infty$, and $n\in\mathbb{N}$. Suppose that $\omega$ is a weight. We define the weighted Bessel capacities $B_{\alpha,p}^{\omega}(\cdot)$ by 
\begin{align*}
B_{\alpha,p}^{\omega}(E)=\inf\left\{\|f\|_{L^{p}(\omega)}^{p}:f\in L^{p}(\omega)^{+},~G_{\alpha}\ast f\geq 1~\text{on}~E\right\},\quad E\subseteq\mathbb{R}^{n}.
\end{align*}
While the weighted local Riesz capacities are defined to be 
\begin{align*}
R_{\alpha,p;\rho}^{\omega}(E)=\inf\left\{\|f\|_{L^{p}(\omega)}^{p}:f\in L^{p}(\omega)^{+},~I_{\alpha,\rho}\ast f\geq 1~\text{on}~E\right\},\quad E\subseteq\mathbb{R}^{n},
\end{align*}
where $I_{\alpha,\rho}(\cdot)$ is the local Riesz kernel defined by
\begin{align*}
I_{\alpha,\rho}(x)=|x|^{\alpha-n}\chi_{\{|x|<\rho\}},\quad x\in\mathbb{R}^{n}.
\end{align*} 
Since the definition of $A_{p}^{\rm loc}$ weights are framed in terms of cubes along with the norm $\left|\cdot\right|_{\infty}$, as it will be seen later that it is more convenient to work with the definition of $I_{\alpha,\rho}(\cdot)$ that
\begin{align*}
I_{\alpha,\rho}(x)=|x|_{\infty}^{\alpha-n}\chi_{Q_{\rho}(x)},\quad x\in\mathbb{R}^{n}.
\end{align*}
To show that the above weighted capacities can be instantiated by the general theory given in the previous subsection, we proceed in a similar way as in Adams \cite{AD}. Let 
\begin{align*}
k(x,y)=I_{\alpha,\rho}(x-y)\omega(y)^{-1},\quad x,y\in\mathbb{R}^{n}
\end{align*}
and $d\nu(y)=\omega(y)dy$. Then 
\begin{align*}
k(x,f\nu)=\int_{\mathbb{R}^{n}}k(x,y)f(y)d\nu(y)=\int_{\mathbb{R}^{n}}I_{\alpha,\rho}(x-y)f(y)dy=(I_{\alpha,\rho}\ast f)(x).
\end{align*}
As a result, we have 
\begin{align*}
C_{k,\nu,p}(E)=R_{\alpha,p;\rho}^{\omega}(E),\quad E\subseteq\mathbb{R}^{n}.
\end{align*}
The corresponding nonlinear potential of a positive measure $\mu$ on $\mathbb{R}^{n}$ is given by
\begin{align*}
V_{\omega;\rho}^{\mu}(x)&=V_{k,\nu,p}^{\mu}(x),\quad x\in\mathbb{R}^{n}\\
&=k\left(x,k(\mu,\cdot)^{p'-1}\nu\right)\\
&=\int_{\mathbb{R}^{n}}k(x,y)k(\mu,y)^{p'-1}d\nu(y)\\
&=\int_{\mathbb{R}^{n}}I_{\alpha,\rho}(x-y)\omega(y)^{-1}\left(\int_{\mathbb{R}^{n}}I_{\alpha,\rho}(x-y)\omega(y)^{-1}d\mu(x)\right)^{p'-1}\omega(y)dy\\
&=\int_{\mathbb{R}^{n}}I_{\alpha,\rho}(x-y)\left(\int_{\mathbb{R}^{n}}I_{\alpha,\rho}(x-y)d\mu(x)\right)^{p'-1}\omega(y)^{-\frac{1}{p-1}}dy\\
&=\left(I_{\alpha,\rho}\ast\left((I_{\alpha,\rho}\ast\mu)^{p'-1}\omega'\right)\right)(x).
\end{align*}
Moreover, we have
\begin{align*}
\int_{\mathbb{R}^{n}}V_{\omega;\rho}^{\mu}(x)d\mu(x)&=\int_{\mathbb{R}^{n}}k(\mu,y)^{p'}d\nu(y)\\
&=\int_{\mathbb{R}^{n}}\left(\int_{\mathbb{R}^{n}}I_{\alpha,\rho}(x-y)\omega(y)^{-1}d\mu(x)\right)^{p'}\omega(y)dy\\
&=\int_{\mathbb{R}^{n}}\left(\int_{\mathbb{R}^{n}}I_{\alpha,\rho}(x-y)d\mu(x)\right)^{p'}\omega(y)^{-\frac{1}{p-1}}dy\\
&=\int_{\mathbb{R}^{n}}(I_{\alpha,\rho}\ast\mu)(y)^{p'}\omega'(y)dy.
\end{align*}
Exactly the same way, one can show that $B_{\alpha,p}^{\omega}(\cdot)=C_{k,\nu,p}(\cdot)$ by letting 
\begin{align*}
k(x,y)=G_{\alpha}(x-y)\omega(y)^{-1},\quad d\nu(y)=\omega(y)dy.
\end{align*}
Now we instantiate Proposition \ref{trick} to the capacities $R_{\alpha,p;\rho}^{\omega}(\cdot)$. Note that $C_{0}(\mathbb{R}^{n})$ is dense in $L_{\nu}^{p}(\mathbb{R}^{n})=L^{p}(\omega^{-1})$. The term 
\begin{align*}
\lim_{|x|\rightarrow\infty}k(x,\varphi\nu)=\lim_{|x|\rightarrow\infty}\int_{\mathbb{R}^{n}}I_{\alpha,\rho}(x-y)\varphi(y)dy=0
\end{align*}
clearly holds for every function $\varphi\in C_{0}(\mathbb{R}^{n})$. The premise of Proposition \ref{trick} is then fulfilled hence the following holds.
\begin{proposition}\label{nonlinear use in CSI}
If $E\subseteq\mathbb{R}^{n}$ is an arbitrary set with $0<R_{\alpha,p;\rho}^{\omega}(E)<\infty$, then there is a measure $\mu^{E}\in\mathcal{M}^{+}\left(\overline{E}\right)$ such that
\begin{align*}
V_{\omega;\rho}^{\mu^{E}}(x)&\geq 1\quad R_{\alpha,p;\rho}^{\omega}(\cdot)\text{-quasi-everywhere}~x\in E,\\
V_{\omega;\rho}^{\mu^{E}}(x)&\leq 1\quad\text{for every}~x\in{\rm supp}\left(\mu^{E}\right),
\end{align*} 
and 
\begin{align*}
\mu^{E}\left(\overline{E}\right)=\int_{\mathbb{R}^{n}}(I_{\alpha,\rho}\ast\mu)(y)^{p'}\omega'(y)dy=\int_{\mathbb{R}^{n}}V_{\omega;\rho}^{\mu^{E}}(x)d\mu(x)=R_{\alpha,p;\rho}^{\omega}(E).
\end{align*}
\end{proposition}
A similar statement holds for $B_{\alpha,p}^{\omega}(\cdot)$ if we make use of (\ref{Bessel infinity}). Indeed, we have 
\begin{align*}
|k(x,\varphi\nu)|&=\left|\int_{\mathbb{R}^{n}}G_{\alpha}(x-y)\varphi(y)dy\right|\\
&\leq C\int_{\mathbb{R}^{n}}e^{-\frac{|x-y|}{2}}|\varphi(y)|dy\\
&\leq Ce^{-\frac{|x|-M}{2}}\int_{|y|\leq M}|\varphi(y)|dy\\
&\rightarrow 0
\end{align*}
as $|x|\rightarrow\infty$, where ${\rm supp}(\varphi)\subseteq\{|y|\leq M\}$ for some $M>0$ and $\varphi\in C_{0}(\mathbb{R}^{n})$.

A variant of weighted local Riesz capacities $\mathcal{R}_{\alpha,p;\rho}^{\omega}(\cdot)$ is given as follows. Let $\chi$ be the characteristic function for the open cube $Q_{1}(0)$. For $x\in\mathbb{R}^{n}$ and $(y,t)\in\mathbb{R}^{n+1}$, define the kernel $k$ on $\mathbb{R}^{n}\times\mathbb{R}^{n+1}$ by
\begin{align*}
k(x,(y,t))=\begin{cases}
t^{-(n-\alpha)}\chi\left(\frac{x-y}{t}\right),\quad\text{if}~0<t<\rho,\\
0,\quad\text{otherwise}.
\end{cases}
\end{align*}
We also define a positive measure $\nu$ on $\mathbb{R}^{n+1}$ by 
\begin{align*}
d\nu(y,t)=\chi_{0<t<\rho}\cdot\omega'(y)dy\frac{dt}{t}.
\end{align*}
If $f$ is a function on $\mathbb{R}^{n+1}$, then
\begin{align*}
k(x,f\nu)=\int_{0}^{\rho}\left(\int_{|x-y|_{\infty}<t}\frac{f(y,t)}{t^{n-\alpha}}\omega'(y)dy\right)\frac{dt}{t},\quad x\in\mathbb{R}^{n}.
\end{align*}
If $\mu$ is a positive measure on $\mathbb{R}^{n}$, then
\begin{align*}
k(\mu,(y,t))=\chi_{0<t<\rho}\cdot t^{-(n-\alpha)}\mu(Q_{t}(y)),\quad(y,t)\in\mathbb{R}^{n+1}.
\end{align*}
Then the capacities $\mathcal{R}_{\alpha,p;\rho}^{\omega}(\cdot)$ is defined to be
\begin{align*}
\mathcal{R}_{\alpha,p;\rho}^{\omega}(E)&=\inf\left\{\|f\|_{L_{\nu}^{p}(\mathbb{R}^{n+1})}^{p}:f\in L_{\nu}^{p}(\mathbb{R}^{n+1})^{+},~k(\cdot,f\nu)\geq 1~\text{on}~E\right\},
\end{align*}
where $E\subseteq\mathbb{R}^{n}$ is an arbitrary set. The corresponding nonlinear potential $\mathcal{V}_{\omega;\rho}^{\mu}$ of a positive measure $\mu$ on $\mathbb{R}^{n}$ is then
\begin{align*}
\mathcal{V}_{\omega;\rho}^{\mu}(x)&=V_{k,\nu,p}^{\mu}(x),\quad x\in\mathbb{R}^{n}\\
&=\int_{\mathbb{R}^{n+1}}k(x,(y,t))k(\mu,(y,t))^{p'-1}d\nu(y,t)\\
&=\int_{\mathbb{R}^{n}}\int_{0}^{\rho}t^{-(n-\alpha)}\chi\left(\frac{x-y}{t}\right)t^{-(n-\alpha)(p'-1)}\mu(Q_{t}(y))^{p'-1}\omega'(y)dy\frac{dt}{t}\\
&=\int_{0}^{\rho}\int_{|x-y|_{\infty}< t}\left(\frac{\mu(Q_{t}(y))}{t^{n-\alpha}}\right)^{\frac{1}{p-1}}\frac{\omega(y)^{-\frac{1}{p-1}}}{t^{n-\alpha}}dy\frac{dt}{t}.
\end{align*}
To instantiate Proposition \ref{trick} for $\mathcal{R}_{\alpha,p;\rho}^{\omega}(\cdot)$, we need the following technical lemma.
\begin{lemma}\label{technical density}
For any $\psi\in C_{0}(\mathbb{R}^{n+1})$, let $\varphi=t^{\frac{1}{p}}\psi(\omega')^{-1}$, $t>0$. Denote by $D$ the set of all such $\varphi$. If $\varphi\in D$, then the function defined by
\begin{align*}
k(x,\varphi\nu)=\int_{0}^{\rho}\left(\int_{|x-y|_{\infty}<t}\psi(y,t)dy\right)t^{\frac{1}{p}-(n-\alpha)}\frac{dt}{t},\quad x\in\mathbb{R}^{n}
\end{align*}
belongs to $C_{0}(\mathbb{R}^{n})$. Moreover, the set $D$ is dense in $L_{\nu}^{p}(\mathbb{R}^{n+1})$.
\end{lemma}

\begin{proof}
Suppose that ${\rm supp}(\psi)\subseteq[-M,M]^{n+1}$ for some $M>\rho$. Then 
\begin{align*}
{\rm supp}(k(\cdot,\varphi\nu))\subseteq[-2M,2M]^{n}.
\end{align*}
To see this, if $(y,t)\in{\rm supp}(\psi)$, $|x-y|_{\infty}<t$, and $0<t<\rho$, then 
\begin{align*}
|x|_{\infty}\leq|x-y|_{\infty}+|y|_{\infty}<t+|(y,t)|_{\infty}<2M,
\end{align*}
which yields $x\in[-2M,2M]^{n}$, as expected. To see that $k(\cdot,\varphi\nu)$ is continuous, let $x_{N}\rightarrow x_{0}$ in $\mathbb{R}^{n}$. For fixed $0<t<\rho$, let
\begin{align*}
F(x,t)=\int_{|x-y|_{\infty}\leq t}\psi(y,t)dy=\int_{\mathbb{R}^{n}}\chi_{|x-y|_{\infty}\leq t}\psi(y,t)dy,\quad x\in\mathbb{R}^{n}.
\end{align*}
Observe that for all $y$ with $|x_{0}-y|_{\infty}\ne t$, one has
\begin{align*}
\chi_{|x_{N}-\cdot|_{\infty}<t}(y)\rightarrow\chi_{|x_{0}-\cdot|_{\infty}<t}(y).
\end{align*}
Since $\{|x_{0}-\cdot|_{\infty}=t\}$ has measure zero, the above pointwise convergence holds almost everywhere. We also have
\begin{align*}
\left|\chi_{|x_{N}-\cdot|_{\infty}<t}\psi(\cdot,t)\right|\leq M_{0}\cdot\chi_{[-M,M]^{n}},\quad M_{0}=\sup_{(y,t)}|\psi(y,t)|,
\end{align*}
Using Lebesgue dominated convergence theorem, we have $F(x_{N})\rightarrow F(x_{0})$. On the other hand,
\begin{align*}
|F(x_{N},t)|\leq 2^{n}M_{0}t^{n},
\end{align*}
and $t^{n}$ is integrable with respect to $t^{\frac{1}{p}-(n-\alpha)}\frac{dt}{t}$, one can use Lebesgue dominated convergence theorem again to deduce that $k(x_{N},\varphi\nu)\rightarrow k(x_{0},\varphi\nu)$, which shows the continuity of $k(\cdot,\varphi)$.

Now we show for the density. Let $f\in L_{\nu}^{p}(\mathbb{R}^{n+1})$. Define $g=t^{-\frac{1}{p}}f\omega'$. Then
\begin{align*}
\int_{0}^{\rho}\left(\int_{\mathbb{R}^{n}}|g(y)|^{p}\omega(y)dy\right)\frac{dt}{t}=\int_{0}^{\rho}\left(\int_{\mathbb{R}^{n}}|f(y)|^{p}\omega'(y)dy\right)\frac{dt}{t}.
\end{align*}
Now we choose $\psi_{N}\in C_{0}(\mathbb{R}^{n+1})$ such that 
\begin{align*}
\int_{0}^{\rho}\left(\int_{\mathbb{R}^{n}}|\psi_{N}(y)-g(y)|^{p}\omega'(y)dy\right)\frac{dt}{t}\rightarrow 0.
\end{align*}
Let $\varphi_{N}=t^{\frac{1}{p}}\psi_{N}(\omega')^{-1}$. Then $\varphi_{N}\in D$ and 
\begin{align*}
\int_{0}^{\rho}\left(\int_{\mathbb{R}^{n}}|\varphi_{N}(y)-f(y)|^{p}\omega(y)dy\right)\frac{dt}{t}=\int_{0}^{\rho}\left(\int_{\mathbb{R}^{n}}|\psi_{N}(y)-g(y)|^{p}\omega'(y)dy\right)\frac{dt}{t},
\end{align*}
which completes the proof of the density.
\end{proof}

Combining Proposition \ref{trick} and Lemma \ref{technical density}, one obtains the following.
\begin{proposition}\label{use nonlinear}
If $E\subseteq\mathbb{R}^{n}$ is an arbitrary set with $0<\mathcal{R}_{\alpha,p;\rho}^{\omega}(E)<\infty$, then there is a measure $\mu^{E}\in\mathcal{M}^{+}\left(\overline{E}\right)$ such that
\begin{align*}
\mathcal{V}_{\omega;\rho}^{\mu^{E}}(x)&\geq 1\quad\mathcal{R}_{\alpha,p;\rho}^{\omega}(\cdot)\text{-quasi-everywhere}~x\in E,\\
\mathcal{V}_{\omega;\rho}^{\mu^{E}}(x)&\leq 1\quad\text{for every}~x\in{\rm supp}\left(\mu^{E}\right),
\end{align*} 
and 
\begin{align*}
\mu^{E}\left(\overline{E}\right)=\int_{\mathbb{R}^{n}}\mathcal{V}_{\omega;\rho}^{\mu^{E}}(x)d\mu(x)=\mathcal{R}_{\alpha,p;\rho}^{\omega}(E).
\end{align*}
\end{proposition}
We will prove that the capacities $R_{\alpha,p;\rho}^{\omega}(\cdot)$ and $\mathcal{R}_{\alpha,p;\rho}^{\omega}(\cdot)$ are equivalent under the assumption that $\omega\in A_{p}^{\rm loc}$. If the stronger premise that $\omega\in A_{p}$ is assumed, then the weighted Bessel capacities $B_{\alpha,p}^{\omega}(\cdot)$ are also equivalent with them. To prove these assertions, we need the Muckenhoupt-Wheeden type inequalities, which we will address in the next subsection.

\subsection{Muckenhoupt-Wheeden Type Inequalities}
\enskip

In the sequel, for $n\in\mathbb{N}$, $0<\alpha<n$, and $0<\rho<\infty$, denote by 
\begin{align*}
M_{\alpha,\rho}\mu(x)=\sup_{0<r\leq\rho}\frac{1}{r^{n-\alpha}}\mu(Q_{r}(x)),\quad x\in\mathbb{R}^{n},
\end{align*}
where $\mu$ is a positive measure on $\mathbb{R}^{n}$.

We begin with the following ``good-$\lambda$ inequality".
\begin{lemma}\label{good lambda}
Let $n\in\mathbb{N}$, $0<\alpha<n$, $0<\rho<\infty$, and $\omega\in A_{\infty}^{\rm loc}$. Suppose that $\mu$ is a positive measure on $\mathbb{R}^{n}$. Then there exists an $a=a_{n,\alpha}\geq 1$ for which, for every $\eta>0$, there is an $\varepsilon=\varepsilon(n,\alpha,[\omega]_{A_{\infty;\rho}^{\rm loc}},\eta)$ such that $0<\varepsilon\leq 1$ and the estimate
\begin{align*}
&\omega\left(\left\{x\in\mathbb{R}^{n}:(I_{\alpha,\rho}\ast\mu)(x)>a\lambda\right\}\right)\\
&\leq\eta\cdot\omega\left(\left\{x\in\mathbb{R}^{n}:(I_{\alpha,\rho}\ast\mu)(x)>\lambda\right\}\right)+\omega\left(\left\{x\in\mathbb{R}^{n}:M_{\alpha,\rho}\mu(x)>\varepsilon\lambda\right\}\right)
\end{align*}
holds for every $\lambda>0$.
\end{lemma}

\begin{proof}
Let $E_{\lambda}=\{x\in\mathbb{R}^{n}:(I_{\alpha,\rho}\ast\mu)(x)>\lambda\}$ for $\lambda>0$. Since $I_{\alpha,\rho}\ast\mu$ is lower semicontinuous, the set $E_{\lambda}$ is open. Using Whitney's theorem, we may decompose $E_{\lambda}$ into closed, dyadic cubes $\{Q_{j}\}$ with disjoint interiors such that for each cube $Q_{j}$, it holds that
\begin{align*}
{\rm diam}(Q_{j})\leq{\rm dist}(Q_{j},E_{\lambda}^{c})\leq 4\cdot{\rm diam}(Q_{j})
\end{align*}
(see \cite[Page 16]{SE}). We subdivide the cubes with diameter $>\frac{\rho}{8}$ into dyadic subcubes with diameters $\leq\frac{\rho}{8}$ but $>\frac{\rho}{16}$. Denote the modified decomposition by the same sequence $\{Q_{j}\}$.

Suppose that we can show that there is an $a\geq 1$ such that for every $\eta>0$, there is an $0<\varepsilon\leq 1$ satisfying  
\begin{align}\label{313}
\omega\left(\left\{x\in Q:(I_{\alpha,\rho}\ast\mu)(x)>a\lambda,~M_{\alpha,\rho}\mu(x)\leq\varepsilon\lambda\right\}\right)\leq\eta\cdot\omega(Q)
\end{align}
for every $\lambda>0$ and $Q\in\{Q_{j}\}$. This implies that 
\begin{align*}
\omega\left(\left\{x\in Q:(I_{\alpha,\rho}\ast\mu)(x)>a\lambda\right\}\right)\leq\eta\cdot\omega(Q)+\omega\left(\left\{x\in Q:M_{\alpha,\rho}\mu(x)>\varepsilon\lambda\right\}\right),
\end{align*}
and the lemma follows by summing over all such $Q$.

To show (\ref{313}), note that the diameter of cubes $\{Q_{j}\}$ in issue are $\leq\frac{\rho}{8}$. We measure the weight  $\omega\in A_{\infty}^{\rm loc}$ by $\left[\cdot\right]_{A_{\infty;\rho}^{\rm loc}}$. Given any $\eta>0$, we choose 
\begin{align}\label{choose the constant}
\delta=\left(\frac{\eta}{\eta+C_{2}}\right)^{\frac{1}{\varepsilon_{0}}}, 
\end{align}
where the constants $C_{2}$ and $\varepsilon_{0}$ appear as in Theorem \ref{further properties Ainfty} (4) (see also Remark \ref{crucial decreasing} for the behavior of $[\omega]_{A_{\infty;\rho}^{\rm loc}}$ in the parameters of $\delta$). As a consequence, using the aforementioned theorem, if $Q\in\{Q_{j}\}$ and $E$ is a measurable subset of $Q$ with $|E|\leq\delta|Q|$, then $\omega(E)\leq\eta\cdot\omega(Q)$. Thus, it suffices to prove that, for a given $\eta>0$, there is an $0<\varepsilon\leq 1$ such that 
\begin{align}\label{314}
\left|\left\{x\in Q:(I_{\alpha,\rho}\ast\mu)(x)>a\lambda,~M_{\alpha,\rho}\mu(x)\leq\varepsilon\lambda\right\}\right|\leq\delta|Q|
\end{align} 
for every $\lambda>0$ and $Q\in\{Q_{j}\}$. Let $\eta>0$ be arbitrary and $Q\in\{Q_{j}\}$ be one of the undivided Whitney cubes. Let $P=Q_{8\cdot{\rm diam}(Q)}(c_{Q})$ be the cube concentric to $Q$ with side length $16\cdot{\rm diam}(Q)$, $d\mu_{1}(x)=\chi_{P}(x)d\mu(x)$, and $\mu_{2}=\mu-\mu_{1}$. For arbitrary $a$ and $\varepsilon$, we define
\begin{align*}
M&=\left\{x\in Q:(I_{\alpha,\rho}\ast\mu)(x)>a\lambda,~M_{\alpha,\rho}\mu(x)\leq\varepsilon\lambda\right\},\\
M_{1}&=\left\{x\in Q:(I_{\alpha,\rho}\ast\mu_{1})(x)>\frac{a\lambda}{2},~M_{\alpha,\rho}\mu(x)\leq\varepsilon\lambda\right\},\\
\end{align*}
and 
\begin{align*}
M_{2}=\left\{x\in Q:(I_{\alpha,\rho}\ast\mu_{2})(x)>\frac{a\lambda}{2},~M_{\alpha,\rho}\mu(x)\leq\varepsilon\lambda\right\}.
\end{align*}
In the sequel, it will be shown that there is a constant $C=C(n,\alpha)>0$ such that 
\begin{align}\label{315}
|M_{1}|\leq C\left(\frac{\varepsilon}{a}\right)^{\frac{n}{n-\alpha}}|Q|
\end{align}
and, further, if $a\geq 4L^{n-\alpha}$, where $L$ is an absolute constant, and if $0<\varepsilon\leq 1$, then $M_{2}=\emptyset$. Fixing $a=\max\{4L^{n-\alpha},1\}$, we then find
\begin{align*}
|M|\leq|M_{1}|\leq C\left(\frac{\varepsilon}{a}\right)^{\frac{n}{n-\alpha}}|Q|\leq C\varepsilon^{\frac{n}{n-\alpha}}|Q|,
\end{align*}
and this proves (\ref{314}) if we choose 
\begin{align*}
0<\varepsilon\leq\varepsilon_{1}=\min\left\{\left(\frac{\delta}{C}\right)^{\frac{n-\alpha}{n}},1\right\}. 
\end{align*}
We now show that (\ref{315}) holds. We may assume that there is an $x_{0}\in Q$ such that $M_{\alpha,\rho}\mu(x_{0})\leq\varepsilon\lambda$. Then
\begin{align*}
|M_{1}|&\leq\left|\left\{x\in Q:(I_{\alpha,\rho}\ast\mu_{1})(x)>\frac{a\lambda}{2}\right\}\right|\\
&\leq\left|\left\{x\in Q:(I_{\alpha}\ast\mu_{1})(x)>\frac{a\lambda}{2}\right\}\right|\\
&\leq C(n,\alpha)\left(\frac{2}{a\lambda}\mu_{1}(\mathbb{R}^{n})\right)^{\frac{n}{n-\alpha}},
\end{align*}
where the last estimate is the weak Sobolev inequality for Riesz potentials (see \cite[Page 120]{SE}). If we let $\overline{P}=Q_{9\cdot{\rm diam}(Q)}(x_{0})$, then $P\subseteq\overline{P}$ and hence  
\begin{align*}
\mu_{1}(\mathbb{R}^{n})=\mu(P)\leq\mu\left(\overline{P}\right)\leq\left|\overline{P}\right|^{\frac{n-\alpha}{n}}M_{\alpha,\rho}\mu(x_{0})\leq C'(n,\alpha)|Q|^{\frac{n-\alpha}{n}}\varepsilon\lambda,
\end{align*}
which proves (\ref{315}).

Suppose that $x_{1}\in E_{\lambda}^{c}$ and ${\rm dist}(x_{1},Q)={\rm dist}(Q,E_{\lambda}^{c})$. Pick a $z\in Q$ such that
\begin{align*}
{\rm dist}(x_{1},Q)\leq|x_{1}-z|\leq 5\cdot{\rm diam}(Q).
\end{align*}

For any $x\in Q$ and $y\notin P$, we have
\begin{align*}
8\cdot{\rm diam}(Q)&<|y-c_{Q}|_{\infty}\\
&\leq|y-x|_{\infty}+|x-c_{Q}|_{\infty}\\
&\leq|y-x|_{\infty}+|x-c_{Q}|\\
&\leq|y-x|_{\infty}+{\rm diam}(Q),
\end{align*}
which yields $|y-x|_{\infty}>7\cdot{\rm diam}(Q)$. Moreover, it holds that 
\begin{align*}
|x-x_{1}|_{\infty}&\leq|x-z|_{\infty}+|z-x_{1}|_{\infty}\\
&\leq|x-z|+|z-x_{1}|\\
&\leq{\rm diam}(Q)+5\cdot{\rm diam}(Q)\\
&=6\cdot{\rm diam}(Q),
\end{align*}
which yields 
\begin{align*}
|x-y|_{\infty}&\geq|x_{1}-y|_{\infty}-|x-x_{1}|_{\infty}\\
&\geq|x_{1}-y|_{\infty}-6\cdot{\rm diam}(Q)\\
&>|x_{1}-y|_{\infty}-\frac{6}{7}|x-y|_{\infty}.
\end{align*}
Denote by $L=\frac{13}{7}$. Then
\begin{align*}
&(I_{\alpha,\rho}\ast\mu_{2})(x)\\
&\leq L^{n-\alpha}\int_{|x-y|_{\infty}<\rho}\frac{d\mu_{2}(y)}{|x_{1}-y|_{\infty}^{n-\alpha}}\\
&\leq L^{n-\alpha}\left(\int_{|x_{1}-y|_{\infty}<\rho}\frac{d\mu_{2}(y)}{|x_{1}-y|_{\infty}^{n-\alpha}}+\int_{|x-y|_{\infty}<\rho,|x_{1}-y|_{\infty}\geq\rho}\frac{d\mu_{2}(y)}{|x_{1}-y|_{\infty}^{n-\alpha}}\right)\\
&\leq L^{n-\alpha}\left((I_{\alpha,\rho}\ast\mu_{2})(x_{1})+\int_{|x-y|_{\infty}<\rho}\frac{d\mu_{2}(y)}{|x-y|_{\infty}^{n-\alpha}}\right)\\
&\leq L^{n-\alpha}\lambda+L^{n-\alpha}M_{\alpha,\rho}\mu(x).
\end{align*}
As a result, if $M_{\alpha,\rho}\mu(x)\leq\varepsilon\lambda$ and $0<\varepsilon\leq 1$, then $(I_{\alpha,\rho}\ast\mu_{2})(x)\leq 2L^{n-\alpha}\lambda$, and hence $M_{2}=\emptyset$ if $a\geq 4L^{n-\alpha}$.

We now consider one of the cubes $Q$ obtained by subdividing a Whitney cube. Then $\frac{\rho}{16}<{\rm diam}(Q)\leq\frac{\rho}{8}$ and $|M_{1}|\leq\delta|Q|$ for $0<\varepsilon\leq\varepsilon_{1}$. Furthermore, for any $x\in Q$ and $y\notin P$, we have 
\begin{align*}
|x-y|_{\infty}&\geq|y-c_{Q}|_{\infty}-|x-c_{Q}|_{\infty}\\
&>8\cdot{\rm diam}(Q)-|x-c_{Q}|\\
&\geq 8\cdot{\rm diam}(Q)-{\rm diam}(Q)\\
&=7\cdot{\rm diam}(Q), 
\end{align*}
which yields
\begin{align*}
(I_{\alpha,\rho}\ast\mu_{2})(x)&\leq\int_{\frac{7}{16}\rho<|x-y|_{\infty}<\rho}\frac{d\mu(y)}{|x-y|_{\infty}^{n-\alpha}}\\
&\leq\left(\frac{16}{7\rho}\right)^{n-\alpha}\int_{|x-y|_{\infty}<\rho}d\mu(y)\\
&\leq\left(\frac{16}{7}\right)^{n-\alpha}M_{\alpha,\rho}\mu(x).
\end{align*}
If $x\in M_{2}$, then we have 
\begin{align*}
(I_{\alpha,\rho}\ast\mu_{2})(x)\leq\left(\frac{16}{7}\right)^{n-\alpha}\varepsilon\lambda.
\end{align*}
This shows that $M_{2}=\emptyset$ for $0<\varepsilon\leq\varepsilon_{2}=\frac{1}{2}\left(\frac{7}{16}\right)^{n-\alpha}a$. Therefore, if we fix $\varepsilon=\min\{\varepsilon_{1},\varepsilon_{2}\}$, we obtain (\ref{314}) and the proof is then finished.
\end{proof}

The Muckenhoupt-Wheenden type inequality reads as the following (see \cite{MW}).
\begin{proposition}\label{312}
Let $n\in\mathbb{N}$, $0<\alpha<n$, $0<p<\infty$, $0<\rho<\infty$, and $\omega\in A_{\infty}^{\rm loc}$. For every positive measure $\mu$ on $\mathbb{R}^{n}$, the following estimate
\begin{align*}
\int_{\mathbb{R}^{n}}\left(I_{\alpha,\rho}\ast\mu\right)(x)^{p}\omega(x)dx\leq C(n,\alpha,p,[\omega]_{A_{\infty;\rho}^{\rm loc}})\int_{\mathbb{R}^{n}}M_{\alpha,\rho}\mu(x)^{p}\omega(x)dx
\end{align*}
holds.
\end{proposition}

\begin{proof}
Let $a$, $\eta$, and $\varepsilon$ be as in Lemma \ref{good lambda}. Assume first that $\mu$ has compact support. For any $R>0$, Lemma \ref{good lambda} yields
\begin{align*}
&\int_{0}^{R}\lambda^{p-1}\omega\left(\left\{x\in\mathbb{R}^{n}:(I_{\alpha,\rho}\ast\mu)(x)>a\lambda\right\}\right)d\lambda\\
&\leq\eta\int_{0}^{R}\lambda^{p-1}\omega\left(\left\{x\in\mathbb{R}^{n}:(I_{\alpha,\rho}\ast\mu)(x)>\lambda\right\}\right)d\lambda\\
&\qquad+\int_{0}^{R}\lambda^{p-1}\omega\left(\left\{x\in\mathbb{R}^{n}:M_{\alpha,\rho}\mu(x)>\varepsilon\lambda\right\}\right)d\lambda.
\end{align*}
Note that 
\begin{align*}
&a^{-p}\int_{0}^{aR}\lambda^{p-1}\omega\left(\left\{x\in\mathbb{R}^{n}:(I_{\alpha,\rho}\ast\mu)(x)>\lambda\right\}\right)d\lambda\\
&\leq\eta\int_{0}^{aR}\lambda^{p-1}\omega\left(\left\{x\in\mathbb{R}^{n}:(I_{\alpha,\rho}\ast\mu)(x)>\lambda\right\}\right)d\lambda\\
&\qquad+\varepsilon^{-p}\int_{0}^{\varepsilon R}\lambda^{p-1}\omega\left(\left\{x\in\mathbb{R}^{n}:M_{\alpha,\rho}\mu(x)>\lambda\right\}\right)d\lambda,
\end{align*}
where both integrals in the right-sided of the last inequality are finite since both $I_{\alpha,\rho}\ast\mu$ and $M_{\alpha,\rho}\mu$ have compact support. By choosing $\eta=\frac{a^{-p}}{2}$, one has
\begin{align*}
&\frac{1}{2}a^{-p}\int_{0}^{aR}\lambda^{p-1}\omega\left(\left\{x\in\mathbb{R}^{n}:(I_{\alpha,\rho}\ast\mu)(x)>\lambda\right\}\right)d\lambda\\
&\leq\varepsilon^{-p}\int_{0}^{\varepsilon R}\lambda^{p-1}\omega\left(\left\{x\in\mathbb{R}^{n}:M_{\alpha,\rho}\mu(x)>\lambda\right\}\right)d\lambda,
\end{align*}
and the proposition follows by taking $R\rightarrow\infty$. If ${\rm supp}(\mu)$ is not compact, we let $d\mu_{m}(x)=\chi_{Q_{m}(0)}(x)d\mu(x)$, $m=1,2,\cdots$. Then
\begin{align*}
\int_{\mathbb{R}^{n}}\left(I_{\alpha,\rho}\ast\mu\right)(x)^{p}\omega(x)dx&=\lim_{m\rightarrow\infty}\int_{\mathbb{R}^{n}}\left(I_{\alpha,\rho}\ast\mu_{m}\right)(x)^{p}\omega(x)dx\\
&\leq C(n,\alpha,p,[\omega]_{A_{\infty;\rho}^{\rm loc}})\lim_{m\rightarrow\infty}\int_{\mathbb{R}^{n}}M_{\alpha,\rho}\mu_{m}(x)^{p}\omega(x)dx\\
&=C(n,\alpha,p,[\omega]_{A_{\infty;\rho}^{\rm loc}})\int_{\mathbb{R}^{n}}M_{\alpha,\rho}\mu(x)^{p}\omega(x)dx.
\end{align*}
The proof is now complete.
\end{proof}

\begin{remark}\label{crucial decreasing}
\rm As noted in Remark \ref{increasing remark}, when $[\omega]_{A_{\infty;\rho}^{\rm loc}}$ increases, the term $\varepsilon_{0}$ decreases while $C_{2}=C_{1}$ increases in Theorem \ref{further properties Ainfty} (4). A consequence of this is that the choice of $\delta$ in (\ref{choose the constant}) is decreasing since $0<\frac{\eta}{\eta+C_{2}}<1$. Then the constant $C(n,\alpha,p,[\omega]_{A_{\infty;\rho}^{\rm loc}})$ in the proof of Proposition \ref{312}, which is essentially $\varepsilon^{-p}$, where $\varepsilon=\varepsilon(n,\alpha,[\omega]_{A_{\infty;\rho}^{\rm loc}},\eta)$ appears as in Lemma \ref{good lambda}, increases when $[\omega]_{A_{\infty;\rho}^{\rm loc}}$ increases.
\end{remark}

The following corollary is an immediate consequence of Proposition \ref{312} and Remark \ref{crucial decreasing} by the simple observation that $[\omega]_{A_{\infty;\rho}^{\rm loc}}\leq[\omega]_{A_{\infty;\rho_{0}}^{\rm loc}}$ for all $0<\rho\leq\rho_{0}$.
\begin{corollary}
Let $n\in\mathbb{N}$, $0<\alpha<n$, $0<p<\infty$, $0<\rho_{0}<\infty$, and $\omega\in A_{\infty}^{\rm loc}$. For every positive measure $\mu$ on $\mathbb{R}^{n}$, the following estimate
\begin{align*}
\int_{\mathbb{R}^{n}}\left(I_{\alpha,\rho}\ast\mu\right)(x)^{p}\omega(x)dx\leq C(n,\alpha,p,[\omega]_{A_{\infty;\rho_{0}}^{\rm loc}})\int_{\mathbb{R}^{n}}M_{\alpha,\rho}\mu(x)^{p}\omega(x)dx
\end{align*}
holds for every $0<\rho\leq\rho_{0}$.
\end{corollary}

For positive measure $\mu$ on $\mathbb{R}^{n}$, we can further show that $I_{\alpha,2\rho}\ast\mu$ can be controlled by $M_{\alpha,\rho}\mu$ in terms of integrals.
\begin{corollary}\label{crucial wolff}
Let $n\in\mathbb{N}$, $0<\alpha<n$, $0<p<\infty$, $0<\rho<\infty$, and $\omega\in A_{\infty}^{\rm loc}$. For every positive measure $\mu$ on $\mathbb{R}^{n}$, the following estimate
\begin{align*}
\int_{\mathbb{R}^{n}}\left(I_{\alpha,2\rho}\ast\mu\right)(x)^{p}\omega(x)dx\leq C(n,\alpha,p,[\omega]_{A_{\infty;20\rho}^{\rm loc}})\int_{\mathbb{R}^{n}}M_{\alpha,\rho}\mu(x)^{p}\omega(x)dx
\end{align*}
holds. As a result, it holds that 
\begin{align*}
\int_{\mathbb{R}^{n}}\left(I_{\alpha,2\rho}\ast\mu\right)(x)^{p}\omega(x)dx\leq C(n,\alpha,p,[\omega]_{A_{\infty;20\rho}^{\rm loc}})\int_{\mathbb{R}^{n}}(I_{\alpha,\rho}\ast\mu)(x)^{p}\omega(x)dx.
\end{align*}
Equivalently,
\begin{align*}
\int_{\mathbb{R}^{n}}V_{\omega;2\rho}^{\mu}(x)d\mu(x)\leq C(n,\alpha,p,[\omega]_{A_{\infty;20\rho}^{\rm loc}})\int_{\mathbb{R}^{n}}V_{\omega;\rho}^{\mu}(x)d\mu(x).
\end{align*}
\end{corollary}

\begin{proof}
For every $x\in\mathbb{R}^{n}$, we have 
\begin{align*}
(I_{\alpha,2\rho}\ast\mu)(x)&=\int_{|x-y|_{\infty}<\rho}\frac{d\mu(y)}{|x-y|_{\infty}^{n-\alpha}}+\int_{\rho\leq|x-y|_{\infty}<2\rho}\frac{d\mu(y)}{|x-y|_{\infty}^{n-\alpha}}\\
&\leq(I_{\alpha,\rho}\ast\mu)(x)+\frac{\mu(Q_{2\rho}(x))}{\rho^{n-\alpha}}.
\end{align*}
We are to show that 
\begin{align}\label{317}
\int_{\mathbb{R}^{n}}\left(\frac{\mu(Q_{2\rho}(x))}{\rho^{n-\alpha}}\right)^{p}\omega(x)dx\leq C(n,\alpha,p,[\omega]_{A_{\infty;20\rho}^{\rm loc}})\int_{\mathbb{R}^{n}}M_{\alpha,\rho}\mu(x)^{p}\omega(x)dx,
\end{align}
then the corollary will follow from Proposition \ref{312} and Remark \ref{crucial decreasing} for the larger quantity that $[\omega]_{A_{\infty;20\rho}^{\rm loc}}\geq[\omega]_{A_{\infty;\rho}^{\rm loc}}$. 

We partition $\mathbb{R}^{n}$ into closed cubes $\left\{\overline{Q}_{j}\right\}$ such that each cube $\overline{Q}_{j}=\overline{Q}_{\rho/4}(x_{j})$ has side length $\ell\left(\overline{Q}_{j}\right)=\frac{\rho}{2}$. Suppose that $x\in \overline{Q}_{j}$. Then
\begin{align*}
Q_{2\rho}(x)\subseteq\bigcup_{k=1}^{m}\overline{Q}_{k}^{(j)},
\end{align*}
where $\overline{Q}_{k}^{(j)}\in\left\{\overline{Q}_{l}\right\}$, $\overline{Q}_{k}^{(j)}\subseteq 20\overline{Q}_{k'}^{(j)}$ for $k,k'=1,...,m$, and $m$ depends only on the dimension $n$. Since $\overline{Q}_{j}\subseteq \overline{Q}_{\rho/2}(x)\subseteq \overline{Q}_{2\rho}(x)$, it follows that 
\begin{align*}
\int_{\mathbb{R}^{n}}\mu(Q_{2\rho}(x))^{p}\omega(x)dx&\leq\sum_{j}\int_{\overline{Q}_{j}}\mu(Q_{2\rho}(x))^{p}\omega(x)dx\\
&\leq C(n,p)\sum_{j}\sum_{k=1}^{m}\mu\left(\overline{Q}_{k}^{(j)}\right)^{p}\omega\left(\overline{Q}_{j}\right)\\
&\leq C(n,p)\sum_{j}\sum_{k=1}^{m}\mu\left(\overline{Q}_{k}^{(j)}\right)^{p}\omega\left(10\overline{Q}_{k}^{(j)}\right)\\
&\leq C(n,p,[\omega]_{A_{\infty;20\rho}^{\rm loc}})\sum_{j}\sum_{k=1}^{m}\mu\left(\overline{Q}_{k}^{(j)}\right)^{p}\omega\left(\overline{Q}_{k}^{(j)}\right)\\
&\leq C'(n,p,[\omega]_{A_{\infty;20\rho}^{\rm loc}})\sum_{j}\mu\left(\overline{Q}_{j}\right)^{p}\omega\left(\overline{Q}_{j}\right).
\end{align*}
The estimate (\ref{317}) follows since $\overline{Q}_{j}\subseteq \overline{Q}_{\rho/2}(x)\subseteq Q_{\rho}(x)$
\begin{align*}
\int_{\mathbb{R}^{n}}\left(\frac{\mu(Q_{2\rho}(x))}{\rho^{n-\alpha}}\right)^{p}\omega(x)dx
&\leq C(n,p,[\omega]_{A_{\infty;20\rho}^{\rm loc}})\sum_{j}\int_{\overline{Q}_{j}}\left(\frac{\mu\left(\overline{Q}_{j}\right)}{\rho^{n-\alpha}}\right)^{p}\omega(x)dx\\
&\leq C(n,p,[\omega]_{A_{\infty;20\rho}^{\rm loc}})\sum_{j}\int_{\overline{Q}_{j}}\left(\frac{\mu(Q_{\rho}(x))}{\rho^{n-\alpha}}\right)^{p}\omega(x)dx\\
&\leq C(n,p,[\omega]_{A_{\infty;20\rho}^{\rm loc}})\int_{\mathbb{R}^{n}}M_{\alpha,\rho}\mu(x)^{p}\omega(x)dx.
\end{align*}
The last assertion of this corollary follows by the fact that $M_{\alpha,\rho}\mu(x)\leq(I_{\alpha}\ast\mu)(x)$, $x\in\mathbb{R}^{n}$.
\end{proof}

\begin{remark}\label{very crucial remark}
\rm In the proof of (\ref{317}), it is clear that one can replace the constant $C(n,\alpha,p,[\omega]_{A_{\infty;20\rho}^{\rm loc}})$ by $C(n,\alpha,p,[\omega]_{A_{\infty;\rho'}^{\rm loc}})$ for any $\rho'\geq 20\rho$. Combining this fact with Remark \ref{crucial decreasing}, we obtain
\begin{align*}
\int_{\mathbb{R}^{n}}\left(I_{\alpha,2k\rho}\ast\mu\right)(x)^{p}\omega(x)dx\leq C(n,\alpha,p,[\omega]_{A_{\infty;20k\rho}^{\rm loc}})\int_{\mathbb{R}^{n}}M_{\alpha,\rho}\mu(x)^{p}\omega(x)dx
\end{align*}
for any $k\in\mathbb{N}$. This fact will be used from time to time without further reference in the sequel. 
\end{remark}

To deal with the Bessel kernel $G_{\alpha}(\cdot)$, we need to impose an extra condition on weights $\omega$, to wit:
\begin{align}\label{not too fast}
\omega(tQ)\leq\exp\left(\frac{t}{\delta}\right)\omega(Q),\quad t\geq 1,\quad\ell(Q)=1,
\end{align}
where $\delta>1$ is some constant independent of all $t\geq 1$ and cubes $Q$ with $\ell(Q)=1$.

Note that $A_{\infty}$ weights satisfy (\ref{not too fast}) since the doubling condition of $\omega$ entails $\omega(tQ)\leq Ct^{n}\omega(Q)$ for any $t\geq 1$ and cubes $Q$. On the other hand, $A_{\infty}^{\rm loc}$ weights do not necessarily satisfy (\ref{not too fast}), for instance, $\omega(\cdot)=e^{\left|\cdot\right|}$. 
\begin{proposition}\label{Bessel Wolff}
Let $n\in\mathbb{N}$, $0<\alpha<n$, $1<p<\infty$, and $\omega$ be a weight that satisfying $(\ref{not too fast})$ for some constant $\delta>1$. For every positive measure $\mu$ on $\mathbb{R}^{n}$, the following estimate
\begin{align*}
\int_{\mathbb{R}^{n}}\left(G_{\alpha}\ast\mu\right)(x)^{p}\omega(x)dx\leq C(n,\alpha,p,\delta)\int_{\mathbb{R}^{n}}M_{\alpha,1}\mu(x)^{p}\omega(x)dx
\end{align*}
holds.
\end{proposition}

\begin{proof}
Choose some $0<\varepsilon<1$ such that $\frac{1}{\delta}<\varepsilon$. Recall the asymptotic behavior (\ref{Bessel infinity}) of $G_{\alpha}(\cdot)$ at infinity that $G_{\alpha}(x)=O(e^{-\varepsilon|x|})$ as $|x|\rightarrow\infty$. Then we have 
\begin{align*}
G_{\alpha}\mu(x)\leq C(n,\alpha,\delta)(I_{\alpha,1}\ast\mu)(x)+C(n,\alpha,\delta)\int_{|x-y|_{\infty}\geq 1}e^{-\varepsilon|x-y|_{\infty}}d\mu(y),\quad x\in\mathbb{R}^{n}.
\end{align*}
Denote by 
\begin{align*}
I(x)=\int_{|x-y|_{\infty}\geq 1}e^{-\varepsilon|x-y|_{\infty}}d\mu(y),\quad x\in\mathbb{R}^{n}.
\end{align*}
It suffices to show that 
\begin{align*}
\int_{\mathbb{R}^{n}}I(x)^{p}\omega(x)dx\leq C(n,\alpha,p,\delta)\int_{\mathbb{R}^{n}}(M_{\alpha,1}\mu)(x)^{p}\omega(x)dx.
\end{align*}
Partition $\mathbb{R}^{n}$ into closed cubes $\left\{\overline{Q}_{j}\right\}$ with $\ell\left(\overline{Q}_{j}\right)=\frac{1}{2}$. We have
\begin{align}
I(x)^{p}&=\left(\sum_{j}\int_{\overline{Q}_{j}}e^{-\varepsilon|x-y|}d\mu(y)\right)^{p}\notag\\
&\leq\left(\sum_{j}\int_{Q_{j}}e^{-\varepsilon d_{\infty}\left(x,\overline{Q}_{j}\right)}d\mu(y)\right)^{p},\quad d_{\infty}\left(x,\overline{Q}_{j}\right)=\inf\{|x-z|_{\infty}:z\in Q_{j}\}\notag\\
&=\left(\sum_{j}e^{-\varepsilon d_{\infty}\left(x,\overline{Q}_{j}\right)}\mu\left(\overline{Q}_{j}\right)\right)^{p}\notag\\
&\leq C(n,p,\delta)\sum_{j}e^{-\varepsilon d_{\infty}\left(x,\overline{Q}_{j}\right)}\mu\left(\overline{Q}_{j}\right)^{p}\label{bessel proof},
\end{align}
where (\ref{bessel proof}) is due to H\"{o}lder's inequality provided that 
\begin{align}\label{bessel proof 2}
\sum_{j}e^{-\varepsilon d_{\infty}\left(x,\overline{Q}_{j}\right)}\leq C(n,\delta)
\end{align}
uniformly in $x\in\mathbb{R}^{n}$. To show (\ref{bessel proof 2}), note that $d_{\infty}\left(x,\overline{Q}_{j}\right)=d_{\infty}\left(0,\overline{R}_{j}\right)$, where $\left\{\overline{R}_{j}\right\}$ is a partition of $\mathbb{R}^{n}$ where each $\overline{R}_{j}$ is the translate of $\overline{Q}_{j}$ with $x$. On the other hand, for every partition $\left\{\overline{R}_{j}\right\}$ of $\mathbb{R}^{n}$ with $\overline{R}_{j}=\overline{Q}_{1/4}(c_{j})$, one has 
\begin{align*}
\sum_{j}e^{-\varepsilon d_{\infty}\left(0,\overline{R}_{j}\right)}&=\sum_{k=0}^{\infty}\sum_{k\leq|c_{j}|<k+1}e^{-\varepsilon d_{\infty}\left(0,\overline{R}_{j}\right)}\\
&\leq\sum_{k=0}^{\infty}\sum_{k\leq|c_{j}|<k+1}e^{-\varepsilon(|c_{j}|_{\infty}-1)}\\
&\leq C(n,\delta)\sum_{k=0}^{\infty}e^{-k\varepsilon}\\
&\leq C'(n,\delta),
\end{align*}
which yields (\ref{bessel proof 2}).

Subsequently, using assumption (\ref{not too fast}), we have 
\begin{align*}
\int_{\mathbb{R}^{n}}e^{-\varepsilon d_{\infty}\left(x,\overline{Q}_{j}\right)}\omega(x)dx&\leq\int_{\mathbb{R}^{n}}e^{-\varepsilon(|x-x_{j}|_{\infty}-1)}\omega(x)dx,\quad \overline{Q}_{j}=\overline{Q}_{1/4}(x_{j})\\
&=C(\delta)\sum_{k=0}^{\infty}\int_{k\leq|x-x_{j}|_{\infty}<k+1}e^{-\varepsilon|x-x_{j}|_{\infty}}\omega(x)dx\\
&\leq C(\delta)\sum_{k=0}^{\infty}e^{-\varepsilon k}\omega\left(\overline{Q}_{k+1}(x_{j})\right)\\
&\leq C(\delta)\sum_{k=0}^{\infty}e^{-\varepsilon k}e^{\frac{k+1}{\delta}}\omega\left(\overline{Q}_{j}\right)\\
&=C'(\delta)\sum_{k=0}^{\infty}e^{-(\varepsilon-\frac{1}{\delta})k}\omega\left(\overline{Q}_{j}\right)\\
&=C''(\delta)\omega\left(\overline{Q}_{j}\right)
\end{align*}
since $\varepsilon-\frac{1}{\delta}>0$. Summarizing the above estimates, we obtain
\begin{align*}
\int_{\mathbb{R}^{n}}I(x)^{p}\omega(x)dx&\leq C(n,p,\delta)\sum_{j}\int_{\mathbb{R}^{n}}e^{-\varepsilon d_{\infty}(x,Q_{j})}\mu(Q_{j})^{p}\\
&\leq C'(n,p,\delta)\sum_{j}\mu\left(\overline{Q}_{j}\right)^{p}\omega\left(\overline{Q}_{j}\right)\\
&=C'(n,p,\delta)\sum_{j}\int_{Q_{j}}\mu\left(\overline{Q}_{j}\right)^{p}\omega(x)dx\\
&\leq C'(n,p,\delta)\sum_{j}\int_{Q_{j}}\mu\left(\overline{Q}_{1/2}(x)\right)^{p}\omega(x)dx\\
&=C'(n,p,\delta)\int_{\mathbb{R}^{n}}\mu\left(\overline{Q}_{1/2}(x)\right)^{p}\omega(x)dx\\
&\leq C'(n,p,\delta)\int_{\mathbb{R}^{n}}M_{\alpha,1}\mu(x)^{p}\omega(x)dx,
\end{align*}
which completes the proof.
\end{proof}

\subsection{Wolff Type Inequalities}
\enskip

Let $n\in\mathbb{N}$, $0<\alpha<n$, $1<p<\infty$, $0<\rho<\infty$, and $\omega$ be a weight. Suppose that $\mu$ is a positive measure on $\mathbb{R}^{n}$. The Wolff potential $\mathcal{W}_{\omega;\rho}^{\mu}$ is defined by 
\begin{align*}
\mathcal{W}_{\omega;\rho}^{\mu}(x)=\int_{0}^{\rho}\left(\frac{t^{\alpha p}\mu(Q_{t}(x))}{\omega(Q_{t}(x))}\right)^{\frac{1}{p-1}}\frac{dt}{t},\quad x\in\mathbb{R}^{n}.
\end{align*}

Up to a constant $C(n,p)>0$ depending only on $n$ and $p$, we prefer to express the Wolff potential by 
\begin{align*}
\mathcal{W}_{\omega;\rho}^{\mu}(x)=C(n,p)\int_{0}^{\rho}\left(\frac{\mu(Q_{t}(x))}{t^{n-\alpha p}}\right)^{\frac{1}{p-1}}\left(\frac{1}{|Q_{t}(x)|}\int_{Q_{t}(x)}\omega(y)dy\right)^{-\frac{1}{p-1}}\frac{dt}{t}.
\end{align*}
Recall the nonlinear potential $\mathcal{V}_{\omega;\rho}^{\mu}$ associated with $\mathcal{R}_{\alpha,p;\rho}^{\omega}(\cdot)$. We have the pointwise estimate of $\mathcal{V}_{\omega;\rho}^{\mu}$ in terms of $\mathcal{W}_{\omega;\rho}^{\mu}$ and vice versa.
\begin{proposition}\label{367}
Let $n\in\mathbb{N}$, $0<\alpha<n$, $1<p<\infty$, $0<\rho<\infty$, and $\omega\in A_{p}^{\rm loc}$. Suppose that $\mu$ is a positive measure on $\mathbb{R}^{n}$. Then
\begin{align}\label{pointwise}
C'(n,\alpha,p)\mathcal{W}_{\omega;\rho/2}^{\mu}(x)\leq\mathcal{V}_{\omega;\rho}^{\mu}(x)\leq C(n,\alpha,p)[\omega]_{A_{p;2\rho}^{\rm loc}}\mathcal{W}_{\omega;2\rho}^{\mu}(x),\quad x\in\mathbb{R}^{n}.
\end{align}
\end{proposition}

\begin{proof}
Observe that $Q_{t}(y)\subseteq Q_{2t}(x)$ for $|x-y|_{\infty}<t$. Then $\omega\in A_{p;2\rho}^{\rm loc}$ yields
\begin{align*}
&\mathcal{V}_{\omega;\rho}^{\mu}(x)\\
&=\int_{0}^{\rho}\int_{|x-y|_{\infty}< t}\left(\frac{\mu(Q_{t}(y))}{t^{n-\alpha}}\right)^{\frac{1}{p-1}}\frac{\omega(y)^{-\frac{1}{p-1}}}{t^{n-\alpha}}dy\frac{dt}{t}\\
&\leq\int_{0}^{\rho}\int_{|x-y|_{\infty}< 2t}\left(\frac{\mu(Q_{2t}(x))}{t^{n-\alpha}}\right)^{\frac{1}{p-1}}\frac{\omega(y)^{-\frac{1}{p-1}}}{t^{n-\alpha}}dy\frac{dt}{t}\\
&\leq C(n)[\omega]_{A_{p;2\rho}^{\rm loc}}\int_{0}^{\rho}\left(\frac{\mu(Q_{2t}(x))}{t^{n-\alpha}}\right)^{\frac{1}{p-1}}\left(\frac{1}{|Q_{2t}(x)|}\int_{Q_{2t}(x)}\omega(y)dy\right)^{-\frac{1}{p-1}}t^{\alpha}\frac{dt}{t}\\
&=C(n,\alpha,p)[\omega]_{A_{p;2\rho}^{\rm loc}}\int_{0}^{2\rho}\left(\frac{\mu(Q_{t}(x))}{t^{n-\alpha p}}\right)^{\frac{1}{p-1}}\left(\frac{1}{|Q_{t}(x)|}\int_{Q_{t}(x)}\omega(y)dy\right)^{-\frac{1}{p-1}}\frac{dt}{t}\\
&=C'(n,\alpha,p)[\omega]_{A_{p;2\rho}^{\rm loc}}\mathcal{W}_{\omega;2\rho}^{\mu}(x).
\end{align*}
Besides that, by writing $1=\omega^{\frac{1}{p}}\omega^{-\frac{1}{p}}$, H\"{o}lder's inequality gives 
\begin{align*}
1\leq\left(\frac{1}{|Q_{t}(x)|}\int_{Q_{t}(x)}\omega(y)dy\right)\left(\frac{1}{|Q_{t}(x)|}\int_{Q_{t}(x)}\omega(y)^{-\frac{1}{p-1}}dy\right)^{p-1}.
\end{align*}
Since $Q_{t}(x)\subseteq Q_{2t}(y)$ for $|x-y|_{\infty}<t$, we have 
\begin{align*}
\mathcal{V}_{\omega;\rho}^{\mu}(x)&=C(n,\alpha,p)\int_{0}^{\frac{\rho}{2}}\int_{|x-y|_{\infty} 2t}\left(\frac{\mu(Q_{2t}(y))}{t^{n-\alpha}}\right)^{\frac{1}{p-1}}\frac{\omega(y)^{-\frac{1}{p-1}}}{t^{n-\alpha}}dy\frac{dt}{t}\\
&\geq C'(n,\alpha,p)\int_{0}^{\frac{\rho}{2}}\left(\frac{\mu(Q_{t}(x))}{t^{n-\alpha}}\right)^{\frac{1}{p-1}}\left(\frac{1}{|Q_{t}(x)|}\int_{Q_{t}(x)}\omega(y)dy\right)^{-\frac{1}{p-1}}t^{\alpha}\frac{dt}{t}\\
&=C''(n,\alpha,p)\int_{0}^{\frac{\rho}{2}}\left(\frac{t^{\alpha p}\mu(Q_{t}(x))}{\omega(Q_{t}(x))}\right)^{\frac{1}{p-1}}\frac{dt}{t}\\
&=C''(n,\alpha,p)\mathcal{W}_{\omega;\rho/2}^{\mu}(x),
\end{align*} 
and the estimate (\ref{pointwise}) follows.
\end{proof}

The nonlinear potential $V_{\omega;\rho}^{\mu}$ associated with $R_{\alpha,p;\rho}^{\omega}(\cdot)$ is comparable to $\mathcal{W}_{\omega;\rho}^{\mu}$ in terms of integrals with respect to $d\mu$. We start with the following technical lemma.
\begin{lemma}
Let $n\in\mathbb{N}$, $0<\alpha<n$, $0<\rho<\infty$, and $\mu$ be a positive measure on $\mathbb{R}^{n}$. For any $x\in\mathbb{R}^{n}$, it holds that 
\begin{align}\label{246}
\int_{|x-y|_{\infty}<\rho}\frac{d\mu(y)}{|x-y|_{\infty}^{n-\alpha}}=(n-\alpha)\int_{0}^{\rho}\frac{\mu(Q_{r}(x))}{r^{n-\alpha}}\frac{dr}{r}+\frac{\mu(Q_{\rho}(x))}{\rho^{n-\alpha}},
\end{align}
and
\begin{align}\label{247}
\int_{|x-y|_{\infty}\geq\rho}\frac{d\mu(y)}{|x-y|_{\infty}^{n-\alpha}}=(n-\alpha)\int_{\rho}^{\infty}\frac{\mu(Q_{r}(x))}{r^{n-\alpha}}\frac{dr}{r}-\frac{\mu(Q_{\rho}(x))}{\rho^{n-\alpha}}.
\end{align}
As a result, for any $0<\delta<\rho<\infty$ and $x\in\mathbb{R}^{n}$, it holds that 
\begin{align}\label{248}
\int_{\delta\leq|x-y|_{\infty}<\rho}\frac{d\mu(y)}{|x-y|_{\infty}^{n-\alpha}}=(n-\alpha)\int_{\delta}^{\rho}\frac{\mu(Q_{r}(x))}{r^{n-\alpha}}\frac{dr}{r}+\frac{\mu(Q_{\rho}(x))}{\rho^{n-\alpha}}-\frac{\mu(Q_{\delta}(x))}{\delta^{n-\alpha}}.
\end{align}
\end{lemma}

\begin{proof}
We first prove for (\ref{246}). Indeed, we have 
\begin{align*}
&(n-\alpha)\int_{0}^{\rho}\frac{\mu(Q_{r}(x))}{r^{n-\alpha}}\frac{dr}{r}+\frac{\mu(Q_{\rho}(x))}{\rho^{n-\alpha}}\\
&=(n-\alpha)\int_{0}^{\rho}\frac{1}{r^{n-\alpha}}\int_{|x-y|_{\infty}< r}d\mu(y)\frac{dr}{r}+\frac{\mu(Q_{\rho}(x))}{\rho^{n-\alpha}}\\
&=(n-\alpha)\int_{|x-y|_{\infty}<\rho}\int_{|x-y|_{\infty}}^{\rho}\frac{1}{r^{n-\alpha}}\frac{dr}{r}d\mu(y)+\frac{\mu(Q_{\rho}(x))}{\rho^{n-\alpha}}\\
&=\int_{|x-y|_{\infty}<\rho}\left(\frac{1}{|x-y|_{\infty}}-\frac{1}{\rho^{n-\alpha}}\right)d\mu(y)+\frac{\mu(Q_{\rho}(x))}{\rho^{n-\alpha}}\\
&=\int_{|x-y|_{\infty}<\rho}\frac{d\mu(y)}{|x-y|_{\infty}^{n-\alpha}}.
\end{align*}
Whereas for (\ref{247}), we have 
\begin{align*}
&(n-\alpha)\int_{\rho}^{\infty}\frac{\mu(Q_{r}(x))}{r^{n-\alpha}}\frac{dr}{r}-\frac{\mu(Q_{\rho}(x))}{\rho^{n-\alpha}}\\
&=(n-\alpha)\int_{\rho}^{\infty}\frac{1}{r^{n-\alpha}}\int_{|x-y|_{\infty}<r}d\mu(y)\frac{dr}{r}-\frac{\mu(Q_{\rho}(x))}{\rho^{n-\alpha}}\\
&=(n-\alpha)\int_{\mathbb{R}^{n}}\int_{\max(|x-y|_{\infty},\rho)}^{\infty}\frac{1}{r^{n-\alpha}}\frac{dr}{r}d\mu(y)-\frac{\mu(Q_{\rho}(x))}{\rho^{n-\alpha}}\\
&=\int_{\mathbb{R}^{n}}\frac{1}{\max(|x-y|_{\infty},\rho)^{n-\alpha}}d\mu(y)-\frac{\mu(Q_{\rho}(x))}{\rho^{n-\alpha}}\\
&=\int_{|x-y|_{\infty}\geq\rho}\frac{1}{|x-y|_{\infty}^{n-\alpha}}d\mu(y)+\int_{|x-y|_{\infty}<\rho}\frac{1}{\rho^{n-\alpha}}d\mu(y)-\frac{\mu(Q_{\rho}(x))}{\rho^{n-\alpha}}\\
&=\int_{|x-y|_{\infty}\geq\rho}\frac{1}{|x-y|_{\infty}^{n-\alpha}}d\mu(y).
\end{align*}
Finally, (\ref{248}) follows by (\ref{247}) by noting that 
\begin{align*}
\int_{\delta\leq|x-y|_{\infty}<\rho}\frac{d\mu(y)}{|x-y|_{\infty}^{n-\alpha}}=\int_{|x-y|_{\infty}\geq\delta}\frac{d\mu(y)}{|x-y|_{\infty}^{n-\alpha}}-\int_{|x-y|_{\infty}\geq\rho}\frac{d\mu(y)}{|x-y|_{\infty}^{n-\alpha}}.
\end{align*}
The proof is now complete.
\end{proof}

\begin{theorem}\label{366}
Let $n\in\mathbb{N}$, $0<\alpha<n$, $1<p<\infty$, $0<\rho<\infty$, and $\omega$ be a weight. For any positive measure $\mu$ on $\mathbb{R}^{n}$, it holds that
\begin{align}\label{pointwise wolff}
\mathcal{W}_{\omega;\rho}^{\mu}(x)\leq C(n,\alpha,p)V_{\omega;2\rho}^{\mu}(x),\quad x\in\mathbb{R}^{n}.
\end{align}
Assume further that $\omega\in A_{p}^{\rm loc}$. Then 
\begin{align}\label{integral wolff}
C^{-1}\int_{\mathbb{R}^{n}}V_{\omega;\rho}^{\mu}(x)d\mu(x)\leq\int_{\mathbb{R}^{n}}\mathcal{W}_{\omega;\rho}^{\mu}(x)d\mu(x)\leq C\int_{\mathbb{R}^{n}}V_{\omega;\rho}^{\mu}(x)d\mu(x),
\end{align}
where $C=C(n,\alpha,p,[\omega]_{A_{p;20\rho}^{\rm loc}})>0$ is a constant depending only on the indicated parameters.
\end{theorem}

\begin{proof}
We note that the second estimate in (\ref{integral wolff}) follows by (\ref{pointwise wolff}) and Corollary \ref{crucial wolff}. Indeed, since $\omega\in A_{p}^{\rm loc}$, it follows that $\omega'\in A_{\infty}^{\rm loc}$ and hence 
\begin{align*}
\int_{\mathbb{R}^{n}}\mathcal{W}_{\omega;\rho}^{\mu}(x)d\mu(x)&\leq C(n,\alpha,p)\int_{\mathbb{R}^{n}}V_{\omega;2\rho}^{\mu}(x)d\mu(x)\\
&=C(n,\alpha,p)\int_{\mathbb{R}^{n}}(I_{\alpha,2\rho}\ast\mu)(x)^{p'}\omega'(x)dx\\
&\leq C(n,\alpha,p,[\omega]_{A_{p;20\rho}^{\rm loc}})\int_{\mathbb{R}^{n}}M_{\alpha,\rho}\mu(x)^{p'}\omega'(x)dx\\
&\leq C(n,\alpha,p,[\omega]_{A_{p;20\rho}^{\rm loc}})\int_{\mathbb{R}^{n}}(I_{\alpha,\rho}\ast\mu)(x)^{p'}\omega'(x)dx\\
&=C(n,\alpha,p,[\omega]_{A_{p;20\rho}^{\rm loc}})\int_{\mathbb{R}^{n}}V_{\omega;\rho}^{\mu}(x)d\mu(x).
\end{align*}
Now we prove (\ref{pointwise wolff}). Let $f=(I_{\alpha,2\rho}\ast\mu)^{p'-1}$. Using (\ref{246}), one has 
\begin{align*}
V_{\omega;2\rho}^{\mu}(x)&=\int_{|x-y|_{\infty}<2\rho}\frac{f(y)\omega'(y)}{|x-y|_{\infty}^{n-\alpha}}dy\\
&=(n-\alpha)\int_{0}^{2\rho}\left(\int_{|x-y|_{\infty}<t}f(y)\omega'(y)dy\right)t^{-(n-\alpha)}\frac{dt}{t}\\
&\qquad+(2\rho)^{-(n-\alpha)}\int_{|x-y|<2\rho}f(y)\omega'(y)dy\\
&\geq(n-\alpha)\int_{0}^{\rho}\left(\int_{|x-y|_{\infty}<t}f(y)\omega'(y)dy\right)t^{-(n-\alpha)}\frac{dt}{t}.
\end{align*}
Further, if $|x-y|_{\infty}<t<\rho$, then $Q_{t}(x)\subseteq Q_{2t}(y)$, and hence 
\begin{align*}
f(y)&=\left(\int_{|y-z|_{\infty}<2\rho}\frac{d\mu(z)}{|y-z|_{\infty}^{n-\alpha}}\right)^{p'-1}\\
&\geq\left(\int_{|y-z|_{\infty}<2t}\frac{d\mu(z)}{|y-z|_{\infty}^{n-\alpha}}\right)^{p'-1}\\
&\geq\left(\frac{\mu(Q_{t}(x))}{(2t)^{n-\alpha}}\right)^{p'-1}.
\end{align*}
Hence H\"{o}lder's inequality gives
\begin{align*}
&V_{\omega;2\rho}^{\mu}(x)\\
&\geq C(n,\alpha,p)\int_{0}^{\rho}\left(\int_{|x-y|_{\infty}< t}\left(\frac{\mu(Q_{t}(x))}{t^{n-\alpha}}\right)^{p'-1}\omega'(y)dy\right)t^{-(n-\alpha)}\frac{dt}{t}\\
&=C'(n,\alpha,p)\int_{0}^{\rho}\frac{\mu(Q_{t}(x))^{p'-1}}{t^{(n-\alpha)p'-n}}\left(\frac{1}{|Q_{t}(x)|}\int_{|x-y|_{\infty}<t}\omega'(y)dy\right)t^{-(n-\alpha)}\frac{dt}{t}\\
&\geq C'(n,\alpha,p)\int_{0}^{\rho}\frac{\mu(Q_{t}(x))^{p'-1}}{t^{(n-\alpha)p'-n}}\left(\frac{1}{|Q_{t}(x)|}\int_{|x-y|_{\infty}<t}\omega(y)dy\right)^{-\frac{1}{p-1}}t^{-(n-\alpha)}\frac{dt}{t}\\
&=C''(n,\alpha,p)\mathcal{W}_{\omega;\rho}^{\mu}(x),
\end{align*}
which yields (\ref{pointwise wolff}).

It remains to prove the first estimate in (\ref{integral wolff}). As before, since $\omega'\in A_{\infty}^{\rm loc}$ for $\omega\in A_{p}^{\rm loc}$, Corollary \ref{crucial wolff} entails
\begin{align*}
\int_{\mathbb{R}^{n}}V_{\omega;\rho}^{\mu}(x)d\mu(x)&=\int_{\mathbb{R}^{n}}(I_{\alpha,\rho}\ast\mu)(x)^{p'}\omega'(x)dx\\
&\leq C(n,\alpha,p,[\omega]_{A_{p;20\rho}^{\rm loc}})\int_{\mathbb{R}^{n}}M_{\alpha;\rho/4}(x)^{p'}\omega'(x)dx.
\end{align*}
On the other hand, for any $0<r\leq\frac{\rho}{4}$ and $x\in\mathbb{R}^{n}$, we have 
\begin{align*}
\left(\int_{0}^{\frac{\rho}{2}}\left(\frac{\mu(Q_{t}(x))}{t^{n-\alpha}}\right)^{p'}\frac{dt}{t}\right)^{\frac{1}{p'}}\geq\left(\int_{r}^{2r}\left(\frac{\mu(Q_{t}(x))}{t^{n-\alpha}}\right)^{p'}\frac{dt}{t}\right)^{\frac{1}{p'}}\geq\frac{(\log 2)^{\frac{1}{p'}}}{2^{n-\alpha}}\frac{\mu(Q_{r}(x))}{r^{n-\alpha}},
\end{align*}
it follows that 
\begin{align*}
&\int_{\mathbb{R}^{n}}M_{\alpha;\rho/4}(x)^{p'}\omega'(x)dx\\
&\leq C(n,\alpha,p)\int_{\mathbb{R}^{n}}\int_{0}^{\frac{\rho}{2}}\left(\frac{\mu(Q_{t}(x))}{t^{n-\alpha}}\right)^{p'}\frac{dt}{t}\omega'(x)dx\\
&=C(n,\alpha,p)\int_{0}^{\frac{\rho}{2}}\left(\int_{\mathbb{R}^{n}}\mu(Q_{t}(x))^{p'}\omega'(x)dx\right)t^{-(n-\alpha)p'}\frac{dt}{t}\\
&=C(n,\alpha,p)\int_{0}^{\frac{\rho}{2}}\left(\int_{\mathbb{R}^{n}}\int_{|x-y|_{\infty}< t}\mu(Q_{t}(x))^{p'-1}d\mu(y)\omega'(x)dx\right)t^{-(n-\alpha)p'}\frac{dt}{t}\\
&\leq C(n,\alpha,p)\int_{0}^{\frac{\rho}{2}}\left(\int_{\mathbb{R}^{n}}\int_{|x-y|_{\infty}< t}\mu(Q_{2t}(y))^{p'-1}d\mu(y)\omega'(x)dx\right)t^{-(n-\alpha)p'}\frac{dt}{t}\\
&=C'(n,\alpha,p)\int_{0}^{\frac{\rho}{2}}\left(\int_{\mathbb{R}^{n}}\mu(Q_{2t}(y))^{p'-1}\frac{1}{|Q_{t}(y)|}\int_{Q_{t}(y)}\omega'(x)dxd\mu(y)\right)t^{n-(n-\alpha)p'}\frac{dt}{t}\\
&\leq C(n,\alpha,p,[\omega]_{A_{p;\rho}^{\rm loc}})\int_{0}^{\frac{\rho}{2}}\left(\int_{\mathbb{R}^{n}}\mu(Q_{2t}(y))^{p'-1}\left(\frac{\omega(Q_{t}(y))}{|Q_{t}(y)|}\right)^{1-p'}d\mu(y)\right)t^{n-(n-\alpha)p'}\frac{dt}{t}\\
&=C'(n,\alpha,p,[\omega]_{A_{p;\rho}^{\rm loc}})\int_{0}^{\frac{\rho}{2}}\left(\int_{\mathbb{R}^{n}}\left(\frac{t^{\alpha p}\mu(Q_{2t}(y))}{\omega(Q_{t}(y))}\right)^{\frac{1}{p-1}}d\mu(y)\right)\frac{dt}{t}\\
&\leq C''(n,\alpha,p,[\omega]_{A_{p;20\rho}^{\rm loc}})\int_{0}^{\frac{\rho}{2}}\left(\int_{\mathbb{R}^{n}}\left(\frac{t^{\alpha p}\mu(Q_{2t}(y))}{\omega(Q_{2t}(y))}\right)^{\frac{1}{p-1}}d\mu(y)\right)\frac{dt}{t}\\
&=C^{(3)}(n,\alpha,p,[\omega]_{A_{p;20\rho}^{\rm loc}})\mathcal{W}_{\omega;\rho}^{\mu}(x),
\end{align*}
where the last inequality follows by the local doubling property of $\omega$ with respect to $20\rho$, and the proof is now finished.
\end{proof}

A variant of Wolff potential is given as follows. Assume that $\omega$ is a weight such that $\omega'$ is locally integrable on $\mathbb{R}^{n}$. We define $W_{\omega;\rho}^{\mu}$ to be
\begin{align*}
W_{\omega;\rho}^{\mu}(x)=\int_{0}^{\rho}\left(\frac{\mu(Q_{t}(x))}{t^{n-\alpha p}}\right)^{\frac{1}{p-1}}\frac{1}{|Q_{t}(x)|}\int_{Q_{t}(x)}\omega'(y)dy\frac{dt}{t},\quad x\in\mathbb{R}^{n}.
\end{align*}
Note that the term $W_{\omega;\rho}^{\mu}$ can also be expressed by 
\begin{align*}
W_{\omega;\rho}^{\mu}(x)=C(n)\int_{0}^{\rho}\left(\frac{\mu(Q_{t}(x))}{t^{(n-\alpha)p }}\right)^{\frac{1}{p-1}}\int_{Q_{t}(x)}\omega'(y)dy\frac{dt}{t},\quad x\in\mathbb{R}^{n}.
\end{align*}
The ``homogeneous" version of $W_{\omega;\rho}^{\mu}$ has been introduced by \cite[Theorem 3.2]{AD}.
\begin{theorem}\label{homogeneous Wolff}
Let $n\in\mathbb{N}$, $0<\alpha<n$, $1<p<\infty$, $0<\rho<\infty$, and $\omega$ be a weight such that $\omega'$ is locally integrable on $\mathbb{R}^{n}$. For any positive measure $\mu$ on $\mathbb{R}^{n}$, it holds that 
\begin{align*}
W_{\omega;\rho}^{\mu}(x)\leq C(n,\alpha,p)V_{\omega;2\rho}^{\mu}(x),\quad x\in\mathbb{R}^{n}.
\end{align*}
Assume further that $\omega'\in A_{\infty}^{\rm loc}$. Then 
\begin{align*}
C^{-1}\int_{\mathbb{R}^{n}}W_{\omega;\rho}^{\mu}(x)d\mu(x)\leq\int_{\mathbb{R}^{n}}V_{\omega;\rho}^{\mu}(x)d\mu(x)\leq C\int_{\mathbb{R}^{n}}W_{\omega;\rho}^{\mu}(x)d\mu(x),
\end{align*}
where $C=C(n,\alpha,p,[\omega']_{A_{\infty;20\rho}^{\rm loc}})>0$ is a constant depending only on the indicated parameters.
\end{theorem}

\begin{proof}
We have 
\begin{align*}
&W_{\omega;\rho}^{\mu}(x)\\
&=\int_{0}^{\rho}\left(\frac{\mu(Q_{t}(x))}{t^{n-\alpha p}}\right)^{\frac{1}{p-1}}\left(\frac{1}{|Q_{t}(x)|}\int_{Q_{t}(x)}\omega'(y)dy\right)\frac{dt}{t}\\
&=C(n)\int_{0}^{\rho}\int_{|x-y|_{\infty}<t}\omega'(y)t^{-\frac{n-\alpha p}{p-1}-n}\left(\int_{|x-z|_{\infty}< t}d\mu(z)\right)^{\frac{1}{p-1}}dy\frac{dt}{t}\\
&\leq C(n)\int_{0}^{\rho}\int_{|x-y|_{\infty}<t}\omega'(y)t^{-(n-\alpha)}\left(\int_{|y-z|_{\infty}< 2t}t^{-(n-\alpha)}d\mu(z)\right)^{\frac{1}{p-1}}dy\frac{dt}{t}\\
&\leq C(n,\alpha,p)\int_{|x-y|_{\infty}< \rho}\omega'(y)\int_{|x-y|_{\infty}}^{\infty}t^{-(n-\alpha)}\frac{dt}{t}\left(\int_{|y-z|_{\infty}< 2\rho}\frac{d\mu(z)}{|y-z|_{\infty}^{n-\alpha}}\right)^{\frac{1}{p-1}}dy\\
&=C(n,\alpha,p)\int_{\mathbb{R}^{n}}\omega'(y) I_{\alpha,\rho}(x-y)\left(\int_{\mathbb{R}^{n}}I_{\alpha,2\rho}(y-z)d\mu(z)\right)^{\frac{1}{p-1}}dy\\
&=C(n,\alpha,p) I_{\alpha,\rho}\ast\left(\left(I_{\alpha,2\rho}\ast\mu\right)^{\frac{1}{p-1}}\omega'\right)(x)\\
&\leq C(n,\alpha,p)V_{\omega;2\rho}^{\mu}(x).
\end{align*}
On the other hand, if $\omega'\in A_{\infty}^{\rm loc}$, then Corollary \ref{crucial wolff} entails
\begin{align*}
&\int_{\mathbb{R}^{n}}V_{\omega;\rho}^{\mu}(x)d\mu(x)\\
&=\int_{\mathbb{R}^{n}}(I_{\alpha,\rho}\ast\mu)(x)^{p'}\omega'(x)dx\\
&\leq C(n,\alpha,p,\rho,[\omega']_{A_{\infty;20\rho}^{\rm loc}})\int_{\mathbb{R}^{n}}M_{\alpha,\frac{\rho}{2}}\mu(x)^{\frac{s}{s-1}}\omega'(x)dx\\
&\leq C'(n,\alpha,p,\rho,[\omega']_{A_{\infty;20\rho}^{\rm loc}})\int_{\mathbb{R}^{n}}\sup_{0<r\leq\frac{\rho}{2}}\int_{r}^{2r}\left(\frac{\mu(Q_{t}(x))}{t^{n-\alpha}}\right)^{p'}\frac{dt}{t}\omega'(x)dx\\
&\leq C'(n,\alpha,p,\rho,[\omega']_{A_{\infty;20\rho}^{\rm loc}})\int_{\mathbb{R}^{n}}\int_{0}^{\rho}\left(\frac{\mu(Q_{t}(x))}{t^{n-\alpha}}\right)^{p'}\frac{dt}{t}\omega'(x)dx\\
&= C'(n,\alpha,p,\rho,[\omega']_{A_{\infty;20\rho}^{\rm loc}})\int_{\mathbb{R}^{n}}\int_{0}^{\rho}\int_{|x-y|_{\infty}< t}\left(\frac{\mu(Q_{t}(x))}{t^{n-\alpha}}\right)^{\frac{1}{p-1}}d\mu(y)\frac{dt}{t^{n-\alpha+1}}\omega'(x)dx\\
&\leq C'(n,\alpha,p,\rho,[\omega']_{A_{\infty;20\rho}^{\rm loc}})\int_{\mathbb{R}^{n}}\int_{0}^{\rho}\int_{|x-y|_{\infty}<t}\left(\frac{\mu(Q_{2t}(y))}{t^{n-\alpha}}\right)^{\frac{1}{p-1}}d\mu(y)\frac{dt}{t}\omega'(x)dx\\
&=C'(n,\alpha,p,\rho,[\omega']_{A_{\infty;20\rho}^{\rm loc}})\int_{\mathbb{R}^{n}}\int_{0}^{\rho}\left(\frac{\mu(Q_{2t}(y))}{t^{n-\alpha p}}\right)^{\frac{1}{p-1}}\int_{|x-y|_{\infty}<t}\omega'(x)\frac{dx}{t^{n}}\frac{dt}{t}d\mu(y)\\
&\leq C''(n,\alpha,p,\rho,[\omega']_{A_{\infty;20\rho}^{\rm loc}})\int_{\mathbb{R}^{n}}\int_{0}^{2\rho}\left(\frac{\mu(Q_{t}(y))}{t^{n-\alpha p}}\right)^{\frac{1}{p-1}}\frac{1}{|Q_{t}(x)|}\int_{Q_{t}(x)}\omega'(y)dy\frac{dt}{t}d\mu(y)\\
&=C''(n,\alpha,p,\rho,[\omega']_{A_{\infty;20\rho}^{\rm loc}})\int_{\mathbb{R}^{n}}W_{\omega;2\rho}^{\mu}(x)d\mu(x).
\end{align*}
Repeating the first part argument in the proof of Theorem \ref{366}, one deduces that
\begin{align*}
\int_{\mathbb{R}^{n}}V_{\omega;\rho}^{\mu}(x)d\mu(x)&\leq C''(n,\alpha,p,\rho,[\omega']_{A_{\infty;20\rho}^{\rm loc}})\int_{\mathbb{R}^{n}}W_{\omega;2\rho}^{\mu}(x)d\mu(x)\\
&\leq C^{(3)}(n,\alpha,p,\rho,[\omega']_{A_{\infty;20\rho}^{\rm loc}})\int_{\mathbb{R}^{n}}V_{\omega;4\rho}^{\mu}(x)d\mu(x)\\
&\leq C(n,\alpha,p,\rho,[\omega']_{A_{\infty;20\rho}^{\rm loc}})\int_{\mathbb{R}^{n}}V_{\omega;\rho/2}^{\mu}(x)d\mu(x)\\
&\leq C'(n,\alpha,p,\rho,[\omega']_{A_{\infty;20\rho}^{\rm loc}})\int_{\mathbb{R}^{n}}W_{\omega;\rho}^{\mu}(x)d\mu(x),
\end{align*}
the same reasoning also applies to the following 
\begin{align*}
\int_{\mathbb{R}^{n}}W_{\omega;\rho}^{\mu}(x)d\mu(x)&\leq C(n,\alpha,p)\int_{\mathbb{R}^{n}}V_{\omega;2\rho}^{\mu}(x)d\mu(x)\\
&\leq C'(n,\alpha,p,[\omega']_{A_{\infty;20\rho}^{\rm loc}})\int_{\mathbb{R}^{n}}V_{\omega;\rho}^{\mu}(x)d\mu(x),
\end{align*}
which completes the proof.
\end{proof}

Combining Proposition \ref{367}, Theorems \ref{366} and \ref{homogeneous Wolff}, we obtain the following.
\begin{corollary}\label{unified}
Let $n\in\mathbb{N}$, $0<\alpha<n$, $1<p<\infty$, $0<\rho<\infty$, and $\omega\in A_{p}^{\rm loc}$ be a weight. For any positive measure $\mu$ on $\mathbb{R}^{n}$, it holds that 
\begin{align*}
\int_{\mathbb{R}^{n}}\mathcal{V}_{\omega;\rho}^{\mu}(x)d\mu(x)\approx\int_{\mathbb{R}^{n}}\mathcal{W}_{\omega;\rho}^{\mu}(x)d\mu(x)\approx\int_{\mathbb{R}^{n}}V_{\omega;\rho}^{\mu}(x)d\mu(x)\approx\int_{\mathbb{R}^{n}}W_{\omega;\rho}^{\mu}(x)d\mu(x),
\end{align*}
where the implicit constants depend only on $n$, $\alpha$, $p$, and $[\omega]_{A_{p;20\rho}^{\rm loc}}$.
\end{corollary}

\subsection{Weighted Capacities Associated with $A_{p;\rho}^{\rm loc}$}
\enskip

In this subsection we show that the capacities $R_{\alpha,p;\rho}^{\omega}(\cdot)$ and $\mathcal{R}_{\alpha,p;\rho}^{\omega}$ are equivalent provided that $\omega\in A_{p}^{\rm loc}$. If the stronger condition $\omega\in A_{p}$ is given, they are also equivalent to $B_{\alpha,p}^{\omega}(\cdot)$.

We start by showing that the capacities $R_{\alpha,p;\rho}^{\omega}(\cdot)$ give rise to the same capacity regardless of the scaling $\rho$.
\begin{proposition}\label{338}
Let $n\in\mathbb{N}$, $0<\alpha<n$, $1<p<\infty$, $0<\rho<\infty$, and $\omega'\in A_{\infty}^{\rm loc}$. Then 
\begin{align*}
R_{\alpha,p;\rho}^{\omega}(E)\leq C(n,\alpha,p,[\omega]_{A_{\infty;20\rho}^{\rm loc}})R_{\alpha,p;2\rho}^{\omega}(E),
\end{align*}
where $E\subseteq\mathbb{R}^{n}$ is an arbitrary set. As a result, it holds that 
\begin{align*}
R_{\alpha,p;\rho_{1}}^{\omega}(\cdot)\approx R_{\alpha,p;\rho_{2}}^{\omega}(\cdot)
\end{align*}
for all $0<\rho_{1}<\rho_{2}<\infty$, where the implicit constants depend only on $n$, $\alpha$, $p$,  $[\omega]_{A_{p;c\rho_{1}}^{\rm loc}}$, and $c$ is a suitable constant multiplied to $\frac{\rho_{2}}{\rho_{1}}$.
\end{proposition}

\begin{proof}
Let $\mu$ be a positive measure on $\mathbb{R}^{n}$. Then Corollary \ref{crucial wolff} gives
\begin{align*}
\int_{\mathbb{R}^{n}}(I_{\alpha,2\rho}\ast\mu)(x)^{p'}\omega'(x)dx\leq C(n,\alpha,p,[\omega']_{A_{\infty;20\rho}^{\rm loc}})\int_{\mathbb{R}^{n}}(I_{\alpha,\rho}\ast\mu)(x)^{p'}\omega'(x)dx.
\end{align*}
We deduce that
\begin{align*}
V_{\omega;2\rho}^{\mu}(x)\leq C(n,\alpha,p,[\omega']_{A_{\infty;20\rho}^{\rm loc}})V_{\omega;\rho}^{\mu}(x),
\end{align*}
and hence Proposition \ref{dual definition} entails
\begin{align*}
R_{\alpha,p;\rho}^{\omega}(K)\leq C(n,\alpha,p,[\omega']_{A_{\infty;20\rho}^{\rm loc}})R_{\alpha,p;2\rho}^{\omega}(K)
\end{align*}
for any compact set $K\subseteq\mathbb{R}^{n}$. For any open set $G\subseteq\mathbb{R}^{n}$, the inner regularity of $R_{\alpha,p;2\rho}^{\omega}(\cdot)$ holds for open sets by Proposition \ref{borel}, it follows that
\begin{align*}
R_{\alpha,p;\rho}^{\omega}(G)&=\sup\left\{R_{\alpha,p;\rho}^{\omega}(K):K\subseteq G,~K~\text{compact}\right\}\\
&\leq C(n,\alpha,p,[\omega']_{A_{\infty;20\rho}^{\rm loc}})\sup\left\{R_{\alpha,p;2\rho}^{\omega}(K):K\subseteq G,~K~\text{compact}\right\}\\
&\leq C(n,\alpha,p,[\omega']_{A_{\infty;20\rho}^{\rm loc}})R_{\alpha,p;2\rho}^{\omega}(G).
\end{align*}
Subsequently, we use Proposition \ref{outer}, the outer regularity of capacities, to obtain
\begin{align*}
R_{\alpha,p;\rho}^{\omega}(E)&=\sup\left\{R_{\alpha,p;\rho}^{\omega}(G):G\supseteq E,~G~\text{open}\right\}\\
&\leq C(n,\alpha,p,[\omega']_{A_{\infty;20\rho}^{\rm loc}})\sup\left\{R_{\alpha,p;2\rho}^{\omega}(G):G\supseteq E,~G~\text{open}\right\}\\
&=C(n,\alpha,p,[\omega']_{A_{\infty;20\rho}^{\rm loc}})R_{\alpha,p;2\rho}^{\omega}(E).
\end{align*}
The last assertion of this proposition simply follows by the simple observation that $I_{\alpha,\rho_{1}}(x)\leq I_{\alpha,\rho_{2}}(x)$, $x\in\mathbb{R}^{n}$, and hence 
\begin{align*}
R_{\alpha,p;\rho_{2}}^{\omega}(E)\leq R_{\alpha,p;\rho_{1}}^{\omega}(E),\quad E\subseteq\mathbb{R}^{n}
\end{align*}
holds for any $0<\rho_{1}<\rho_{2}<\infty$.
\end{proof}

\begin{theorem}\label{337}
Let $n\in\mathbb{N}$, $0<\alpha<n$, $1<p<\infty$, $0<\rho<\infty$, and $\omega'$ be a weight that satisfying $(\ref{not too fast})$. Then 
\begin{align}\label{333}
C(n,\alpha,p,\delta)R_{\alpha,p;1}^{\omega}(E)\leq B_{\alpha,p}^{\omega}(E)\leq C(n,\alpha)R_{\alpha,p;1}^{\omega}(E),
\end{align}
where $E\subseteq\mathbb{R}^{n}$ is an arbitrary set. Assume further that $\omega'\in A_{\infty}$. Then 
\begin{align}\label{333 prime}
C(n,\alpha,p,\rho,[\omega']_{A_{\infty}})R_{\alpha,p;\rho}^{\omega}(E)\leq B_{\alpha,p}^{\omega}(E)\leq C(n,\alpha,\rho) R_{\alpha,p;\rho}^{\omega}(E),
\end{align}
holds for all $0<\rho<\infty$ and arbitrary sets $E\subseteq\mathbb{R}^{n}$.
\end{theorem}

\begin{proof}
The second estimate of (\ref{333}) follows by the asymptotic behaviors (\ref{Bessel zero}) and (\ref{Bessel infinity}) of $G_{\alpha}(\cdot)$ that $I_{\alpha,1}(x)\leq C(n,\alpha)G_{\alpha}(x)$, $x\in\mathbb{R}^{n}$.

For the proof of the first estimate of (\ref{333}), we appeal to Proposition \ref{Bessel Wolff} that 
\begin{align*}
\int_{\mathbb{R}^{n}}(G_{\alpha}\ast\mu)(x)^{p'}\omega'(x)dx\leq C(n,\alpha,p,\delta)\int_{\mathbb{R}^{n}}M_{\alpha,1}\mu(x)^{p'}\omega'(x)dx.
\end{align*}
Using (\ref{dual definition}), one obtains
\begin{align*}
R_{\alpha,p;1}^{\omega}(K)\leq C(n,\alpha,p,\delta)B_{\alpha,p}^{\omega}(K)
\end{align*}
for all compact sets $K\subseteq\mathbb{R}^{n}$. To extend the above estimate to general sets, one argues as in the proof of Proposition \ref{338}, which yields the first estimate of (\ref{333}).

As noted in the first paragraph, the second estimate of (\ref{333 prime}) follows by the fact that $I_{\alpha,\rho}(x)\leq C(n,\alpha,\rho)G_{\alpha}(x)$, $x\in\mathbb{R}^{n}$. For the first estimate of (\ref{333 prime}), first note that
\begin{align}\label{333 double prime}
B_{\alpha,p}^{\omega}(E)\geq C(n,\alpha,p,\delta)R_{\alpha,p;1}^{\omega}(E),\quad E\subseteq\mathbb{R}^{n}
\end{align} 
holds with $\delta$ depending on $[\omega']_{A_{\infty}}$. To see this, the doubling property of $\omega'$ entails the existence of such a $\delta$ in (\ref{not too fast}), which depends on the constant $[\omega']_{A_{\infty}}$. By repeating the proof of Proposition \ref{Bessel Wolff}, one sees that 
\begin{align*}
\int_{\mathbb{R}^{n}}(G_{\alpha}\ast\mu)(x)^{p'}\omega'(x)dx\leq C(n,\alpha,p,[\omega']_{A_{\infty}})\int_{\mathbb{R}^{n}}M_{\alpha,1}\mu(x)^{p'}\omega'(x)dx,
\end{align*}
and hence (\ref{333 double prime}) follows. Denote $C(n,\alpha,p,[\omega']_{A_{\infty}})$ by $C$ at the moment. By repeating the proof of Proposition \ref{crucial wolff}, one obtains 
\begin{align*}
\int_{\mathbb{R}^{n}}(I_{\alpha,1}\ast\mu)(x)^{p'}\omega'(x)dx\leq C\int_{\mathbb{R}^{n}}(I_{\alpha,1/2}\ast\mu)(x)^{p'}\omega'(x)dx,
\end{align*}
and hence $R_{\alpha,p;1}^{\omega}(E)\geq C R_{\alpha,p;1/2}^{\omega}(E)$. Then for any $0<\rho\leq 1$, choose $k=k(\rho)$ such that $\frac{1}{2^{k}}\leq\rho$. Iteration gives $R_{\alpha,p;1}^{\omega}(E)\geq C^{k}R_{\alpha,p;1/2^{k}}^{\omega}(E)\geq C^{k}R_{\alpha,p;\rho}^{\omega}(E)$. We conclude that
\begin{align*}
R_{\alpha,p;1}^{\omega}(E)&\geq\begin{cases}
R_{\alpha,p;\rho}^{\omega}(E),\quad \rho>1,\\
C^{k(\rho)}R_{\alpha,p;\rho}^{\omega}(E),\quad \rho\leq 1.
\end{cases}
\end{align*}
We obtain the first estimate of (\ref{333 prime}) by combining the above with (\ref{333 double prime}).
\end{proof}

\begin{theorem}\label{R cal}
Let $n\in\mathbb{N}$, $0<\alpha<n$, $1<p<\infty$, $0<\rho<\infty$, and $\omega\in A_{p}^{\rm loc}$. Then 
\begin{align*}
C(n,\alpha,p,[\omega]_{A_{p;20\rho}^{\rm loc}})^{-1}R_{\alpha,p;\rho}^{\omega}(E)\leq\mathcal{R}_{\alpha,p;\rho}^{\omega}(E)\leq C(n,\alpha,p,[\omega]_{A_{p;20\rho}^{\rm loc}})R_{\alpha,p;\rho}^{\omega}(E)
\end{align*}
holds for arbitrary set $E\subseteq\mathbb{R}^{n}$. 
\end{theorem}

\begin{proof}
Let $\mu$ be a positive measure on $\mathbb{R}^{n}$. Using Corollary \ref{unified}, we have
\begin{align*}
C(n,\alpha,p,[\omega]_{A_{p;20\rho}^{\rm loc}})^{-1}\int_{\mathbb{R}^{n}}V_{\omega;\rho}^{\mu}(x)d\mu(x)
&\leq\int_{\mathbb{R}^{n}}\mathcal{V}_{\omega;\rho}^{\mu}(x)d\mu(x)\\
&\leq C(n,\alpha,p,[\omega]_{A_{p;20\rho}^{\rm loc}})\int_{\mathbb{R}^{n}}V_{\omega;\rho}^{\mu}(x)d\mu(x).
\end{align*}
The rest of the proof is similar to that of Proposition \ref{338}.
\end{proof}

By combining (\ref{333 prime}) of Theorem \ref{337} and the above theorem, we obtain the following corollary.
\begin{corollary}
Let $n\in\mathbb{N}$, $0<\alpha<n$, $1<p<\infty$, $0<\rho<\infty$, and $\omega\in A_{p}$. Then
\begin{align*}
R_{\alpha,p;\rho}^{\omega}(\cdot)\approx\mathcal{R}_{\alpha,p;\rho}^{\omega}(\cdot)\approx B_{\alpha,p}^{\omega}(\cdot),
\end{align*}
where the implicit constants depend only on $n$, $\alpha$, $p$, $\rho$, and $[\omega]_{A_{p}}$.
\end{corollary}
 
A result from Aikawa \cite[Lemma 20]{AH} states that the weighted Riesz capacity $R_{\alpha,p}^{\omega}(\cdot)$ is nontrivial, i.e., $R_{\alpha,p}^{\omega}(\mathbb{R}^{n})>0$, if and only if $\omega\in A_{p}$ and 
\begin{align*}
\int_{\mathbb{R}^{n}}\frac{\omega'(x)}{(1+|x|)^{(n-\alpha)p'}}dx<\infty.
\end{align*}
As suggested by (\ref{not too fast}), the weighted Bessel capacity $B_{\alpha,p}^{\omega}(\cdot)$ is nontrivial provided that the weight $\omega'$ does not grow too fast, equivalently, $\omega$ does not decay too fast. Therefore, $B_{\alpha,p}^{\omega}(\cdot)$ is not equivalent to $R_{\alpha,p;\rho}^{\omega}(\cdot)$ for general $\omega\in A_{\infty}^{\rm loc}$.
\begin{theorem}\label{becoming zero}
Let $n\in\mathbb{N}$, $0<\alpha<n$, and $1<p<\infty$. The $A_{\infty}^{\rm loc}$ weight $\omega$ defined by $\omega=e^{-3p\left|\cdot\right|}$ satisfies that 
\begin{align*}
B_{\alpha,p}^{\omega}(\mathbb{R}^{n})=0.
\end{align*}
\end{theorem}

\begin{proof}
Fix a cube $Q_{N}(0)$, $N\in\mathbb{N}$. Suppose that $\mu\in\mathcal{M}^{+}(Q_{N}(0))$. Note that the nonlinear potential of $B_{\alpha,p}^{\omega}(\cdot)$ is given by $G_{\alpha}\ast\left((G_{\alpha}\ast\mu)^{p'-1}\omega'\right)$. In view of Proposition \ref{dual definition}, we assume that
\begin{align}\label{nonlinear infty}
\int_{\mathbb{R}^{n}}(G_{\alpha,p}\ast\mu)(x)^{p'}\omega'(x)dx\leq 1<\infty.
\end{align}
Then by viewing $G_{\alpha}(\cdot)$ as a decreasing radial function $G_{\alpha}(r)$, $r>0$, one has 
\begin{align*}
&\int_{\mathbb{R}^{n}}(G_{\alpha,p}\ast\mu)(x)^{p'}\omega'(x)dx\\
&\geq\int_{|x|\geq N}\left(\int_{|y|\leq N}G_{\alpha}(x-y)d\mu(y)\right)^{p'}\omega'(x)dx\\
&\geq\int_{|x|\geq N}\left(\int_{|y|\leq N}G_{\alpha}(2|x|)d\mu(y)\right)^{p'}\omega'(x)dx\\
&=\mu(Q_{N}(0))^{p'}\int_{|x|\geq N}G_{\alpha}(2|x|)^{p'}\omega'(x)dx.
\end{align*}
Using the asymptotic behavior (\ref{Bessel infinity}) of $G_{\alpha}(\cdot)$, we obtain
\begin{align*}
\int_{|x|\geq N}G_{\alpha}(2|x|)^{p'}\omega'(x)dx&\geq C(n,\alpha,p)\int_{|x|\geq N}\left(\frac{1}{(2|x|)^{\frac{n+1-\alpha}{2}}}e^{-2|x|}\right)^{p'}\omega'(x)dx\\
&=C'(n,\alpha,p)\int_{|x|\geq N}\frac{1}{|x|^{\frac{n+1-\alpha}{2}p'}}e^{-2p'|x|}\omega'(x)dx\\
&=\infty
\end{align*}
by noting that $\omega'=e^{3p'\left|\cdot\right|}$. Since the term (\ref{nonlinear infty}) is finite, we have $\mu(Q_{N}(0))=0$, and hence Proposition \ref{dual definition} entails
\begin{align*}
B_{\alpha,p}^{\omega}(Q_{N}(0))=0.
\end{align*}
Then the monotonicity (\ref{monotone}) of $B_{\alpha,p}^{\omega}(\cdot)$ entails
\begin{align*}
B_{\alpha,p}^{\omega}(\mathbb{R}^{n})=\lim_{N\rightarrow\infty}B_{\alpha,p}^{\omega}(Q_{N}(0))=0,
\end{align*}
as expected.
\end{proof}

\subsection{Weighted Capacities of Cubes}
\enskip

In this subsection we compute the weighted capacities of cubes. Recall that  
\begin{align*}
&{\rm cap}_{\alpha,p}(Q_{r}(0))=r^{n-\alpha p}{\rm cap}_{\alpha,p}(Q_{r}(0)),\quad 0<r<\infty,\quad \alpha p<n,\\
&{\rm Cap}_{\alpha,p}(Q_{r}(0))\approx r^{n-\alpha p},\quad 0<r\leq 1,\quad \alpha p<n
\end{align*}
(see \cite[Propositions 5.1.2 and 5.1.4]{AH}). The estimate of general capacities of cubes is also given by \cite[Section 3.2]{KV}, but their kernels are not applicable to $I_{\alpha,\rho}(\cdot)$. 
\begin{lemma}\label{lemma 3312}
Let $n\in\mathbb{N}$, $0<\alpha<n$, $1<p<\infty$, $0<\rho<\infty$, and $\omega\in A_{p}^{\rm loc}$. Then 
\begin{align*}
C(n,\alpha,p,[\omega]_{A_{p;6\rho}^{\rm loc}})^{-1}\frac{\omega(Q_{\rho}(a))}{\rho^{\alpha p}}\leq R_{\alpha,p;\rho}^{\omega}(Q_{\rho}(a))\leq C(n,\alpha,p,[\omega]_{A_{p;6\rho}^{\rm loc}})\frac{\omega(Q_{\rho}(a))}{\rho^{\alpha p}}
\end{align*}
holds for every $a\in\mathbb{R}^{n}$.
\end{lemma}

\begin{proof}
We first prove the upper bound. Let $f=A\chi_{Q_{\rho}(a)}\omega'$, where the constant $A>0$ will be determined later. For any $x\in Q_{\rho}(a)$, it holds that 
\begin{align*}
(I_{\alpha,\rho}\ast f)(x)=A\int_{|x-y|_{\infty}<\rho,|y-a|_{\infty}<\rho}\frac{\omega'(y)}{|x-y|_{\infty}^{n-\alpha}}\geq A\frac{\omega'(E)}{\rho^{n-\alpha}},
\end{align*}
where $E=Q_{\rho}(x)\cap Q_{\rho}(a)$. We claim that $Q_{\rho/2}\left(\frac{x+a}{2}\right)\subseteq E$. Indeed, for any $z\in Q_{\rho/2}\left(\frac{x+a}{2}\right)$, we have 
\begin{align*}
|z-x|_{\infty}&\leq\left|z-\frac{x+a}{2}\right|_{\infty}+\left|\frac{x+a}{2}-x\right|_{\infty}\\
&<\frac{\rho}{2}+\frac{|x-a|_{\infty}}{2}\\
&<\rho,
\end{align*}
which yields $z\in Q_{\rho}(x)$. In a similar fashion, one shows that $z\in Q_{\rho}(a)$, the claim follows. Since $\omega'\in A_{p';2\rho}^{\rm loc}$, Proposition \ref{local strong} gives
\begin{align*}
\omega'(Q_{\rho}(a))&\leq[\omega']_{A_{p';2\rho}^{\rm loc}}\left(\frac{|Q_{\rho}(a)|}{|E|}\right)^{p}\omega'(E)\\
&\leq[\omega']_{A_{p';2\rho}^{\rm loc}}\left(\frac{|Q_{\rho}(a)|}{|Q_{\rho/2}\left(\frac{x+a}{2}\right)|}\right)^{p}\omega'(E)\\
&=C(n,p,[\omega]_{A_{p;2\rho}^{\rm loc}})\omega'(E).
\end{align*}
Now we let $A=\rho^{n-\alpha}C(n,p,[\omega]_{A_{p;2\rho}^{\rm loc}})\omega'(Q_{\rho}(a))^{-1}$, then $I_{\alpha,\rho}\ast f\geq 1$ on $Q_{\rho}(a)$. Further, H\"{o}lder's inequality implies
\begin{align*}
\int_{\mathbb{R}^{n}}f(x)^{p}\omega(x)dx&=A^{p}\omega'(Q_{\rho}(a))\\
&=\rho^{(n-\alpha)p}C(n,p,[\omega]_{A_{p;2\rho}^{\rm loc}})^{p}\omega'(Q_{\rho}(a))^{-p}\omega'(Q_{\rho}(a))\\
&=C'(n,\alpha,p,[\omega]_{A_{p;2\rho}^{\rm loc}})\rho^{n-\alpha p}\left(\frac{1}{|Q_{\rho}(a)|}\int_{Q_{\rho}(a)}\omega'(x)dx\right)^{1-p}\\
&\leq C'(n,\alpha,p,[\omega]_{A_{p;2\rho}^{\rm loc}})\rho^{n-\alpha p}\left(\frac{1}{|Q_{\rho}(a)|}\int_{Q_{\rho}(a)}\omega(x)dx\right)\\
&=C''(n,\alpha,p,[\omega]_{A_{p;2\rho}^{\rm loc}})\rho^{n-\alpha p}\omega(Q_{\rho}(a)),
\end{align*}
which yields the upper bound for $R_{\alpha,p;\rho}^{\omega}(Q_{\rho}(a))$.

Now we address on the lower bound. Let $\varepsilon>0$, $f\geq 0$, and $I_{\alpha;\rho}\ast f\geq 1$ on $Q_{\rho}(a)$ with 
\begin{align*}
\int_{\mathbb{R}^{n}}f(x)^{p}\omega(x)dx<R_{\alpha,p;\rho}^{\omega}(Q_{\rho}(a))+\varepsilon.
\end{align*}
On the other hand, we have
\begin{align*}
\omega(Q_{\rho}(a))=\int_{Q_{\rho}(a)}\omega(x)dx\leq\int_{Q_{\rho}(a)}(I_{\alpha,\rho}\ast f)(x)^{p}\omega(x)dx.
\end{align*}
Using the formula (\ref{246}), one easily deduces that 
\begin{align*}
(I_{\alpha,\rho}\ast f)(x)\leq C(n,\alpha)\rho^{\alpha}\mathcal{M}_{2\rho}^{\rm loc}f(x),\quad x\in\mathbb{R}^{n}.
\end{align*}
Hence Theorem \ref{center strong type} entails
\begin{align*}
\rho^{-\alpha p}\omega(Q_{\rho}(a))&\leq C(n,\alpha,p)\int_{\mathbb{R}^{n}}\mathcal{M}_{2\rho}^{\rm loc}f(x)^{p}\omega(x)dx\\
&\leq C(n,\alpha,p,[\omega]_{A_{p;6\rho}^{\rm loc}})\int_{\mathbb{R}^{n}}f(x)^{p}\omega(x)dx,
\end{align*}
which yields
\begin{align*}
\rho^{-\alpha p}\omega(Q_{\rho}(a))\leq C(n,\alpha,p,[\omega]_{A_{p;6\rho}^{\rm loc}})\left(R_{\alpha,p;\rho}^{\omega}(Q_{\rho}(a))+\varepsilon\right).
\end{align*}
The arbitrariness of $\varepsilon>0$ gives the lower bound for $R_{\alpha,p;\rho}^{\omega}(Q_{\rho}(a))$.
\end{proof}

\begin{theorem}\label{estimate of cube}
Let $n\in\mathbb{N}$, $0<\alpha<n$, $1<p<\infty$, and $0<\rho<\infty$. Suppose that $\omega$ is a weight such that $\omega'$ satisfies the local doubling property $(\ref{local doubling property})$ with respect to $8\rho$ that
\begin{align*}
\omega'(Q)\leq D\cdot\omega'\left(\frac{1}{3}Q\right),\quad\ell(Q)\leq 8\rho
\end{align*}
for some constant $D>0$. Then for any $a\in\mathbb{R}^{n}$ and $0<r\leq\rho$, one has 
\begin{align}\label{339}
R_{\alpha,p;\rho}^{\omega}(Q_{r}(a))\leq C(n,\alpha,p,D)\left(\int_{r}^{2\rho}\frac{\omega'(Q_{t}(a))}{t^{(n-\alpha)p'}}\frac{dt}{t}\right)^{1-p}.
\end{align}
If $\omega\in A_{p}^{\rm loc}$, then 
\begin{align}\label{339 prime}
\left(\int_{r}^{2\rho}\frac{\omega'(Q_{t}(a))}{t^{(n-\alpha)p'}}\frac{dt}{t}\right)^{1-p}\leq C(n,\alpha,p,[\omega]_{A_{p;8\rho}^{\rm loc}})R_{\alpha,p;\rho}^{\omega}(Q_{r}(a)).
\end{align}
\end{theorem} 

\begin{proof}
Let $C=\frac{2^{n-\alpha}D^{3}}{(n-\alpha)p'}$, $f=CA^{-1}g$, where 
\begin{align*}
A=\int_{r}^{2\rho}\frac{\omega'(Q_{t}(a))}{t^{(n-\alpha)p'}}\frac{dt}{t},
\end{align*}
and
\begin{align*}
g(x)&=\begin{cases}
r^{-(n-\alpha)(p'-1)}\omega'(x),\quad|x-a|_{\infty}<r,\\
|x-a|_{\infty}^{-(n-\alpha)(p'-1)}\omega'(x),\quad r\leq|x-a|_{\infty}<2\rho,\\
0,\quad\text{otherwise}.
\end{cases}
\end{align*}
Note that 
\begin{align}\label{3310}
\int_{r}^{2\rho}\frac{\omega'(Q_{t}(a))}{t^{(n-\alpha)p'}}\geq\int_{\rho}^{2\rho}\frac{\omega'(Q_{t}(a))}{t^{(n-\alpha)p'}}\geq\log 2\frac{\omega'(Q_{\rho}(a))}{(2\rho)^{(n-\alpha)p'}}\geq\frac{\log 2}{D}\frac{\omega'(Q_{2\rho}(a))}{(2\rho)^{(n-\alpha)p'}}.
\end{align}
Using formula (\ref{248}), we have 
\begin{align}
&\int_{\mathbb{R}^{n}}g(x)^{p}\omega(x)dx\notag\\
&=\frac{\omega'(Q_{r}(a))}{r^{(n-\alpha)p'}}+(n-\alpha)p'\int_{r}^{2\rho}\frac{\omega'(Q_{t}(a))}{t^{(n-\alpha)p'}}+\frac{\omega'(Q_{2\rho}(a))}{(2\rho)^{(n-\alpha)p'}}-\frac{\omega'(Q_{r}(a))}{r^{(n-\alpha)p'}}\notag\\
&\leq C(n,\alpha,p,D)A.\label{pre 3311}
\end{align}
Suppose that $x\in Q_{\rho}(a)$. Then
\begin{align*}
(I_{\alpha,\rho}\ast g)(x)
&=r^{-(n-\alpha)(p'-1)}\int_{|y-a|_{\infty}<r,|x-y|_{\infty}<\rho}\frac{\omega'(y)}{|x-y|_{\infty}^{n-\alpha}}dy\\
&\qquad+\int_{r\leq|y-a|_{\infty}< 2\rho,|x-y|_{\infty}<\rho}\frac{\omega'(y)}{|x-y|_{\infty}^{n-\alpha}|y-a|_{\infty}^{(n-\alpha)(p'-1)}}dy\\
&\geq r^{-(n-\alpha)(p'-1)}\frac{\omega'(Q_{r}(a)\cap Q_{\rho}(x))}{(2r)^{n-\alpha}}\\
&\qquad+2^{-(n-\alpha)}\int_{r\leq|y-a|_{\infty}< 2\rho}\frac{\chi_{Q_{\rho}(x)}\omega'(y)}{|y-a|_{\infty}^{(n-\alpha)(p'-1)}}dy\\
&=2^{-(n-\alpha)}\frac{\omega'(Q_{r}(a)\cap Q_{\rho}(a))}{r^{(n-\alpha)p'}}\\
&\qquad+(n-\alpha)p'2^{-(n-\alpha)}\int_{r}^{2\rho}\frac{\omega'(Q_{t}(a)\cap Q_{\rho}(t))}{t^{(n-\alpha)p'}}\frac{dt}{t}\\
&\qquad+2^{-(n-\alpha)}\left(\frac{\omega'(Q_{2\rho}(a)\cap Q_{\rho}(t))}{(2\rho)^{(n-\alpha)p'}}-\frac{\omega'(Q_{r}(a)\cap Q_{\rho}(t))}{r^{(n-\alpha)p'}}\right)\\
&\geq(n-\alpha)p'2^{-(n-\alpha)}\int_{r}^{2\rho}\frac{\omega'(Q_{t}(a)\cap Q_{\rho}(t))}{t^{(n-\alpha)p'}}\frac{dt}{t},
\end{align*}
where we have used a variant form of formula (\ref{248}) in the second equality. Now we claim that 
\begin{align}\label{3311}
\omega'(Q_{t}(a)\cap Q_{\rho}(x))\geq D^{-3}\omega'(Q_{t}(a)),\quad r\leq t\leq 2\rho.
\end{align}
When $t\leq\frac{1}{2}\rho$, we have $Q_{t}(a)\subseteq Q_{\rho}(x)$. Suppose that $\frac{1}{2}\rho\leq t\leq 2\rho$. Let $z=\frac{3a+x}{4}$. Then it is routine to check that $Q_{\rho/4}(z)\subseteq Q_{t}(a)\cap Q_{\rho}(x)$ and $Q_{t}(a)\subseteq Q_{4\rho}(z)$. The local doubling property of $\omega'$ with respect to $8\rho$ entails
\begin{align*}
\omega'(Q_{t}(a)\cap Q_{\rho}(a))&\geq\omega'(Q_{\rho/4}(z))\\
&\geq\omega'(Q_{4\rho/27}(z))\\
&\geq D^{-1}\omega'(Q_{4\rho/9}(z))\\
&\geq D^{-2}\omega'(Q_{4\rho/3}(z))\\
&\geq D^{-3}\omega'(Q_{4\rho}(z))\\
&\geq D^{-3}\omega'(Q_{t}(a)),
\end{align*}
which proves (\ref{3311}). Combining the above estimates with (\ref{pre 3311}), we obtain
\begin{align*}
(I_{\alpha,\rho}\ast g)(x)\geq 2^{-(n-\alpha)}D^{-3}(n-\alpha)p'\int_{r}^{2\rho}\frac{\omega'(Q_{t}(a))}{t^{(n-\alpha)p'}}\frac{dt}{t}=\frac{A}{C},
\end{align*}
which yields
\begin{align*}
(I_{\alpha,\rho}\ast f)(x)=\frac{C}{A}(I_{\alpha,\rho}\ast g)(x)\geq 1,\quad x\in Q_{r}(a).
\end{align*}
Further, we have
\begin{align*}
R_{\alpha,p;\rho}^{\omega}(Q_{r}(a))&\leq\int_{\mathbb{R}^{n}}f(x)^{p}\omega(x)dx\\
&=\frac{C^{p}}{A^{p}}\int_{\mathbb{R}^{n}}g(x)^{p}\omega(x)dx\\
&\leq C'(n,\alpha,p,D)\left(\int_{r}^{2\rho}\frac{\omega'(Q_{t}(a))}{t^{(n-\alpha)p'}}\frac{dt}{t}\right)^{1-p},
\end{align*}
which is (\ref{339}).

It remains to show (\ref{339 prime}). Let $\varepsilon>0$ be arbitrary and choose an $f\geq 0$ such that $I_{\alpha,\rho}\ast f\geq 1$ on $Q_{r}(a)$ and 
\begin{align*}
\int_{\mathbb{R}^{n}}f(x)^{p}\omega(x)dx<R_{\alpha,p;\rho}^{\omega}(Q_{r}(a))+\varepsilon.
\end{align*}
Let $E=\left\{x\in Q_{r}(a):(I_{\alpha,r}\ast f)(x)<\frac{1}{2}\right\}$. If $E\ne\emptyset$, then for some $x\in Q_{r}(a)$, it holds that
\begin{align}
\frac{1}{2}&\leq\int_{r\leq|x-y|_{\infty}<\rho}\frac{f(y)}{|x-y|_{\infty}^{n-\alpha}}dy\notag\\
&\leq\left(\int_{\mathbb{R}^{n}}f(x)^{p}\omega(x)dx\right)^{\frac{1}{p}}\left(\int_{r\leq|x-y|_{\infty}<\rho}\frac{\omega'(y)}{|x-y|_{\infty}^{(n-\alpha)p'}}dy\right)^{\frac{1}{p'}}.\label{3312}
\end{align}
Using formula (\ref{248}) once again, we have
\begin{align*}
&\int_{r\leq|x-y|_{\infty}<\rho}\frac{\omega'(y)}{|x-y|_{\infty}^{(n-\alpha)p'}}dy\\
&=(n-\alpha)p'\int_{r}^{\rho}\frac{\omega'(Q_{t}(x))}{t^{(n-\alpha)p'}}\frac{dt}{t}+\frac{\omega'(Q_{\rho}(x))}{\rho^{(n-\alpha)p'}}-\frac{\omega'(Q_{r}(x))}{r^{(n-\alpha)p'}}.
\end{align*}
If $t\geq r$, then $Q_{t}(x)\subseteq Q_{2t}(a)$, and hence $\omega'(Q_{t}(x))\leq\omega'(Q_{2t}(a))$. Combining this with (\ref{3310}), we obtain
\begin{align*}
&\int_{r\leq|x-y|_{\infty}<\rho}\frac{\omega'(y)}{|x-y|_{\infty}^{(n-\alpha)p'}}dy\\
&\leq C(n,\alpha,p,[\omega]_{A_{p;8\rho}^{\rm loc}})\left(\int_{r}^{\rho}\frac{\omega'(Q_{2t}(a))}{t^{(n-\alpha)p'}}\frac{dt}{t}+\int_{\rho}^{2\rho}\frac{\omega'(Q_{t}(a))}{t^{(n-\alpha)p'}}\frac{dt}{t}\right)\\
&=C(n,\alpha,p,[\omega]_{A_{p;8\rho}^{\rm loc}})\int_{r}^{2\rho}\frac{\omega'(Q_{t}(a))}{t^{(n-\alpha)p'}}\frac{dt}{t}.
\end{align*}
Combining the above with (\ref{3312}), it follows that
\begin{align*}
\frac{1}{2}\leq C'(n,\alpha,p,[\omega]_{A_{p;8\rho}^{\rm loc}})\left(R_{\alpha,p;\rho}^{\omega}(Q_{r}(a))+\varepsilon\right)^{\frac{1}{p}}\left(\int_{r}^{2\rho}\frac{\omega'(Q_{t}(a))}{t^{(n-\alpha)p'}}\frac{dt}{t}\right)^{\frac{p-1}{p}},
\end{align*}
and hence 
\begin{align*}
R_{\alpha,p;\rho}^{\omega}(Q_{r}(a))\geq C''(n,\alpha,p,[\omega]_{A_{p;8\rho}^{\rm loc}})^{-1}\left(\int_{r}^{2\rho}\frac{\omega'(Q_{t}(a))}{t^{(n-\alpha)p'}}\frac{dt}{t}\right)^{1-p}.
\end{align*}
If $E=\emptyset$, then $I_{\alpha,r}\ast(2f)\geq 1$ on $Q_{r}(a)$, and hence
\begin{align*}
R_{\alpha,p;r}^{\omega}(Q_{r}(a))\leq\int_{\mathbb{R}^{n}}(2f(x))^{p}\omega(x)dx\leq 2^{p}\left(R_{\alpha,p;\rho}^{\omega}(Q_{r}(a))+\varepsilon\right),
\end{align*}
which implies that
\begin{align}\label{combine}
R_{\alpha,p;\rho}^{\omega}(Q_{r}(a))\geq 2^{-p}R_{\alpha,p;r}^{\omega}(Q_{r}(a)).
\end{align}
On the other hand, H\"{o}lder's inequality gives
\begin{align*}
\frac{\omega(Q_{r}(a))}{|Q_{r}(a)|}\geq\left(\frac{\omega'(Q_{r}(a))}{|Q_{r}(a)|}\right)^{1-p},
\end{align*}
and hence
\begin{align*}
\frac{\omega(Q_{r}(a))}{r^{\alpha p}}\geq\left(\frac{\omega'(Q_{r}(a))}{r^{(n-\alpha)p'}}\right)^{1-p}.
\end{align*}
Since 
\begin{align*}
\int_{r}^{2\rho}\frac{\omega'(Q_{t}(a))}{t^{(n-\alpha)p'}}\frac{dt}{t}\geq\int_{r}^{2r}\frac{\omega'(Q_{t}(a))}{t^{(n-\alpha)p'}}\frac{dt}{t}\geq C(n,\alpha,p)\frac{\omega'(Q_{r}(a))}{r^{(n-\alpha)p'}},
\end{align*}
we have 
\begin{align*}
\frac{\omega(Q_{r}(a))}{r^{\alpha p}}\geq C'(n,\alpha,p)\left(\int_{r}^{2\rho}\frac{\omega'(Q_{t}(a))}{t^{(n-\alpha)p'}}\frac{dt}{t}\right)^{1-p}.
\end{align*}
Combining (\ref{combine}) with Lemma \ref{lemma 3312}, one obtains
\begin{align*}
R_{\alpha,p;\rho}^{\omega}(Q_{r}(a))\geq C(n,\alpha,p,[\omega]_{A_{p;6\rho}^{\rm loc}})^{-1}\left(\int_{r}^{2\rho}\frac{\omega'(Q_{t}(a))}{t^{(n-\alpha)p'}}\frac{dt}{t}\right)^{1-p},
\end{align*}
which yields (\ref{339 prime}).
\end{proof}

\subsection{Bounded Maximum Principle of Nonlinear Potential}
\enskip

The classical bounded maximum principle reads as 
\begin{align*}
(g\ast\mu)(x)\leq C(n)\sup_{y\in{\rm supp}(\mu)}(g\ast\mu)(y),\quad x\in\mathbb{R}^{n},
\end{align*}
where $g$ is a radially decreasing function and $\mu$ is a positive measure on $\mathbb{R}^{n}$ (see \cite[Theorem 3.6.2]{AH}). To put it another way, if $g\ast\mu$ is bounded on its support, then it is bounded everywhere. For the nonlinear potential $G_{\alpha}\ast(G_{\alpha}\ast\mu)^{p'-1}$ associated with the unweighted Bessel capacities ${\rm Cap}_{\alpha,p}(\cdot)$, the bounded maximum principle is also valid (see \cite[Theorem 2.6.3]{AH}). However, the proof of \cite[Theorem 2.6.3]{AH} seems to be working only for the convolution kernels. It is unknown if the same principle holds for the weighted nonlinear potential like $I_{\alpha,\rho}\ast\left((I_{\alpha,\rho}\ast\mu)^{p'-1}\omega'\right)$. Nevertheless, we can show that such a principle holds for the Wolff potentials.
\begin{theorem}\label{max}
Let $n\in\mathbb{N}$, $0<\alpha<n$, $1<p<\infty$, $0<\rho<\infty$, and $\mu$ be a positive measure on $\mathbb{R}^{n}$. If $\omega$ is a weight that satisfies 
\begin{align}\label{strong doubling}
\omega(Q)\leq C\left(\frac{|Q|}{|E|}\right)^{\varepsilon}\omega(E)
\end{align}
for some constants $0<C,\varepsilon<\infty$, where $Q$ is any cube with length $\ell(Q)\leq 4\rho$. Then
\begin{align*}
\mathcal{W}_{\omega;\rho}^{\mu}(x)\leq2^{\frac{n\varepsilon-\alpha p}{p-1}}C^{\frac{1}{p-1}}\sup\left\{\mathcal{W}_{\omega;2\rho}^{\mu}(z):z\in{\rm supp}(\mu)\right\},\quad x\in\mathbb{R}^{n}.
\end{align*}
In particular, if $\omega\in A_{p}^{\rm loc}$, then one may take $C=[\omega]_{A_{p;4\rho}^{\rm loc}}$ and $\varepsilon=p$.
\end{theorem}

\begin{proof}
Let $x\notin{\rm supp}(\mu)$ and $x_{0}\in K={\rm supp}(\mu)$ be the point that minimizes the distance from $x$ to ${\rm supp}(\mu)$ with respect to the norm $\left|\cdot\right|_{\infty}$. If $Q_{t}(x)\cap K\ne\emptyset$, then $t>|x-x_{0}|_{\infty}$, which in turn implies that $Q_{t}(x)\subseteq Q_{2t}(x_{0})$. Consequently, (\ref{strong doubling}) implies that 
\begin{align*}
W_{\omega;\rho}^{\mu}(x)&\leq\int_{0}^{\rho}\left(\frac{t^{\alpha p}\mu(Q_{2t}(x_{0}))}{\omega(Q_{t }(x))}\right)^{\frac{1}{p-1}}\frac{dt}{t}\\
&\leq 2^{\frac{n\varepsilon}{p-1}}C^{\frac{1}{p-1}}\int_{0}^{\rho}\left(\frac{t^{\alpha p}\mu(Q_{2t}(x_{0}))}{\omega(Q_{2t }(x_{0}))}\right)^{\frac{1}{p-1}}\frac{dt}{t}\\
&=2^{\frac{n\varepsilon-\alpha p}{p-1}}C^{\frac{1}{p-1}}W_{\omega;2\rho}^{\mu}(x_{0}),
\end{align*}
and the bounded maximum principle holds. The last assertion of this lemma follows by Proposition \ref{local strong}.
\end{proof}

If we consider the variant type of Wolff potential $W_{\omega;\rho}^{\mu}$, then the restriction on weights $\omega$ can be relaxed to the local integrability of $\omega'$.
\begin{theorem}\label{max 2}
Let $n\in\mathbb{N}$, $0<\alpha<n$, $1<p<\infty$, $0<\rho<\infty$, and $\mu$ be a positive measure on $\mathbb{R}^{n}$. Suppose that $\omega$ is a weight such that $\omega'$ is locally integrable on $\mathbb{R}^{n}$. Then
\begin{align*}
W_{\omega;\rho}^{\mu}(x)\leq 2^{\frac{(n-\alpha)p}{p-1}}\sup\left\{W_{\omega,2\rho}^{\mu}(y):y\in{\rm supp}(\mu)\right\},\quad x\in\mathbb{R}^{n}.
\end{align*}
\end{theorem}

\begin{proof}
As in the proof of Theorem \ref{max}, we let $x\notin{\rm supp}(\mu)$ and $x_{0}\in K={\rm supp}(\mu)$ be the point that minimizes the distance from $x$ to ${\rm supp}(\mu)$ with respect to the norm $\left|\cdot\right|_{\infty}$. If $Q_{t}(x)\cap K\ne\emptyset$, then $Q_{t}(x)\subseteq Q_{2t}(x_{0})$. Consequently,
\begin{align*}
W_{\omega;\rho}^{\mu}(x)&=\int_{0}^{\rho}\left(\frac{\mu(Q_{t}(x))}{t^{n-\alpha p}}\right)^{\frac{1}{p-1}}\left(\frac{1}{|Q_{t}(x)|}\int_{Q_{t}(x)}\omega'(y)dy\right)\frac{dt}{t}\\
&\leq 2^{n}\int_{0}^{\rho}\left(\frac{\mu(Q_{2t}(x_{0}))}{t^{n-\alpha p}}\right)^{\frac{1}{p-1}}\left(\frac{1}{|Q_{2t}(x_{0})|}\int_{Q_{2t}(x_{0})}\omega'(y)dy\right)\frac{dt}{t}\\
&=2^{\frac{(n-\alpha)p}{p-1}}W_{\omega;2\rho}^{\mu}(x_{0}),
\end{align*}
which completes the proof.
\end{proof}

The bounded maximum principle holds for the nonlinear potential $\mathcal{V}_{\omega;\rho}^{\mu}$ which associated with $\mathcal{R}_{\alpha,p;\rho}^{\omega}(\cdot)$. Note that if one assume further that $\omega\in A_{p}^{\rm loc}$, then the following is a consequence of Proposition \ref{367} and Theorem \ref{max} with the bound depending also on $[\omega]_{A_{p;4\rho}^{\rm loc}}$.
\begin{theorem}
Let $n\in\mathbb{N}$, $0<\alpha<n$, $1<p<\infty$, $0<\rho<\infty$, and $\omega$ be a weight. Suppose that $\mu$ is a positive measure on $\mathbb{R}^{n}$. Then
\begin{align*}
\mathcal{V}_{\omega;\rho}^{\mu}(x)\leq 3^{\frac{(n-\alpha)p}{p-1}}\sup\{\mathcal{V}_{\omega;3\rho}^{\mu}(y):y\in{\rm supp}(\mu)\},\quad x\in\mathbb{R}^{n}.
\end{align*}
\end{theorem}

\begin{proof}
As in the proof of Theorem \ref{max}, we let $x\notin{\rm supp}(\mu)$ and $x_{0}\in K={\rm supp}(\mu)$ be the point that minimizes the distance from $x$ to ${\rm supp}(\mu)$ with respect to the norm $\left|\cdot\right|_{\infty}$. If $Q_{t}(y)\cap K\ne\emptyset$ for $|x-y|_{\infty}<t$, say, $z\in Q_{t}(y)\cap K$, then $z\in Q_{2t}(x)\cap K$, which yields $|x-x_{0}|_{\infty}<2t$ as $x_{0}$ is the minimizer of the distance. As a consequence, we have $Q_{t}(x)\subseteq Q_{3t}(x_{0})$. Then
\begin{align*}
\mathcal{V}_{\omega;\rho}^{\mu}(x)&=\int_{0}^{\rho}\left(\int_{|x-y|_{\infty}< t}\left(\frac{\mu(Q_{t}(y))}{t^{n-\alpha}}\right)^{\frac{1}{p-1}}\frac{\omega(y)^{-\frac{1}{p-1}}}{t^{n-\alpha}}dy\right)\frac{dt}{t}\\
&\leq\int_{0}^{\rho}\left(\int_{|x_{0}-y|_{\infty}< 3t}\left(\frac{\mu(Q_{t}(y))}{t^{n-\alpha}}\right)^{\frac{1}{p-1}}\frac{\omega(y)^{-\frac{1}{p-1}}}{t^{n-\alpha}}dy\right)\frac{dt}{t}\\
&\leq \int_{0}^{\rho}\left(\int_{|x_{0}-y|_{\infty}< 3t}\left(\frac{\mu(Q_{3t}(y))}{t^{n-\alpha}}\right)^{\frac{1}{p-1}}\frac{\omega(y)^{-\frac{1}{p-1}}}{t^{n-\alpha}}dy\right)\frac{dt}{t}\\
&=3^{\frac{(n-\alpha)p}{p-1}}\mathcal{V}_{\omega;3\rho}(x_{0}),
\end{align*}
and the bounded maximum principle follows.
\end{proof}

\subsection{Weak Type Boundedness of Nonlinear Potential}
\enskip

In proving the boundedness of local maximal function on the spaces of Choquet integrals, we need the weak type estimate of nonlinear potential in the form that
\begin{align*}
\mathcal{R}_{\alpha,p;1}^{\omega}\left(\left\{x\in\mathbb{R}^{n}:\mathcal{V}_{1;1}^{\mu}(x)\right\}\right)\leq\frac{C}{t^{p-1}}\mu(\mathbb{R}^{n}),\quad 0<t<\infty
\end{align*}
(see \cite{OP2}). The above estimate also appears in the theory of thinness of sets (see \cite[Section 6.3]{AH}). The weighted analogue of such a weak type estimate reads as the following.
\begin{theorem}\label{weak type}
Let $n\in\mathbb{N}$, $0<\alpha<n$, $1<p<\infty$, $0<\rho<\infty$, and $\omega$ be a weight. Suppose that $\mu$ is a positive measure on $\mathbb{R}^{n}$. Then
\begin{align*}
\mathcal{R}_{\alpha,p;\rho}^{\omega}\left(\left\{x\in\mathbb{R}^{n}:\mathcal{V}_{\omega;\rho}^{\mu}(x)>t\right\}\right)\leq C(n)\frac{\mu(\mathbb{R}^{n})}{t^{p-1}},\quad 0<t<\infty.
\end{align*}
\end{theorem}

\begin{proof}
Let $\gamma$ be the measure associated with a compact subset $K$ of 
\begin{align}\label{set}
\left\{x\in\mathbb{R}^{n}:\mathcal{V}_{\omega;\rho}^{\mu}(x)>t\right\} 
\end{align}
such that $\gamma(K)=\mathcal{R}_{\alpha,p;\rho}^{\omega}(K)$ and $\mathcal{V}_{\omega;\rho}^{\gamma}(x)\leq 1$ for all $x\in{\rm supp}(\gamma)$ as in Proposition \ref{use nonlinear}. Let
\begin{align*}
{\bf{M}}_{\gamma}\mu(x)=\sup_{t>0}\dfrac{\mu(Q_{t}(x))}{\gamma(Q_{t}(x))},\quad x\in\mathbb{R}^{n}.
\end{align*}
For all $x\in{\rm supp}(\gamma)$, it follows that 
\begin{align*}
\mathcal{V}_{\omega;\rho}^{\mu}(x)&=\int_{0}^{\rho}\left(\int_{|x-y|_{\infty}<t}\left(\frac{\mu(Q_{t}(y))}{t^{n-\alpha}}\right)^{\frac{1}{p-1}}\frac{\omega(y)^{-\frac{1}{p-1}}}{t^{n-\alpha}}dy\right)\frac{dt}{t}\\
&=\int_{0}^{\rho}\left(\int_{|x-y|_{\infty}<t}\left(\frac{\mu(Q_{t}(y))}{\gamma(Q_{t}(y))}\right)^{\frac{1}{p-1}}\left(\frac{\gamma(Q_{t}(y))}{t^{n-\alpha}}\right)^{\frac{1}{p-1}}\frac{\omega(y)^{-\frac{1}{p-1}}}{t^{n-\alpha}}dy\right)\frac{dt}{t}\\
&\leq{\bf{M}}_{\gamma}\mu(x)^{\frac{1}{p-1}}\mathcal{V}_{\omega;\rho}^{\gamma}(x)\\
&\leq{\bf{M}}_{\gamma}\mu(x)^{\frac{1}{p-1}}.
\end{align*}
As a result, we have
\begin{align*}
\text{supp}(\gamma)\subseteq\left\{x\in\mathbb{R}^{n}:{\bf{M}}_{\gamma}\mu(x)>t^{p-1}\right\}.
\end{align*}
Note that 
\begin{align*}
{\bf{M}}_{\gamma}\mu(x)=\sup_{t>0}\dfrac{\mu(Q_{t}(x))}{\gamma(Q_{t}(x))}=\sup_{t>0}\lim_{N\rightarrow\infty}\dfrac{\mu\left(\overline{Q}_{t-1/N}(x)\right)}{\gamma\left(\overline{Q}_{t-1/N}(x)\right)}\leq\sup_{t>0}\dfrac{\mu\left(\overline{Q}_{t}(x)\right)}{\gamma\left(\overline{Q}_{t}(x)\right)},\quad x\in\mathbb{R}^{n}.
\end{align*}
By Besicovitch covering theorem (see \cite[Theorem 18.1]{DE}), there are collections of closed cubes $A_{i}=\left\{\overline{Q}_{n_{i}}\right\}$, $i=1,...,C(n)$ such that $A_{i}$ is disjoint and 
\begin{align*}
\text{supp}(\gamma)\subseteq\bigcup_{i=1}^{C(n)}\bigcup_{\overline{Q}\in A_{i}}\overline{Q},\quad\frac{\mu\left(\overline{Q}\right)}{\gamma\left(\overline{Q}\right)}>t^{p-1},\quad \overline{Q}\in A_{i}.
\end{align*}
As a consequence, we have
\begin{align*}
\mathcal{R}_{\alpha,p;\rho}^{\omega}(K)=\gamma(K)\leq\sum_{i=1}^{C(n)}\sum_{\overline{Q}\in A_{i}}\gamma\left(\overline{Q}\right)\leq\frac{1}{t^{p-1}}\sum_{i=1}^{C(n)}\sum_{\overline{Q}\in A_{i}}\mu\left(\overline{Q}\right)\leq\frac{C(n)}{t^{p-1}}\mu(\mathbb{R}^{n}).
\end{align*}
Since $\mathcal{V}_{\omega;\rho}^{\mu}$ is lower semicontinuous, the set in (\ref{set}) is open. The result then follows by Proposition (\ref{borel}), the inner regularity of $\mathcal{R}_{\alpha,p;\rho}^{\omega}(\cdot)$.
\end{proof}

Using Proposition \ref{367} that $\mathcal{W}_{\omega;\rho}^{\mu}(x)\leq C(n,\alpha,p)\mathcal{V}_{\omega;2\rho}^{\mu}(x)$, $x\in\mathbb{R}^{n}$, the first assertion of the following corollary is an immediate consequence of Theorem \ref{weak type}. While the second assertion follows by Theorem \ref{R cal}.
\begin{corollary}\label{use weak wolff}
Let $n\in\mathbb{N}$, $0<\alpha<n$, $1<p<\infty$, $0<\rho<\infty$, and $\omega$ be a weight. Suppose that $\mu$ is a positive measure on $\mathbb{R}^{n}$. Then 
\begin{align*}
\mathcal{R}_{\alpha,p;2\rho}^{\omega}(\{x\in\mathbb{R}^{n}:\mathcal{W}_{\omega;\rho}^{\mu}(x)>t\})\leq C(n,\alpha,p)\frac{\mu(\mathbb{R}^{n})}{t^{p-1}},\quad 0<t<\infty.
\end{align*}
Assume further that $\omega\in A_{p}^{\rm loc}$. Then
\begin{align*}
\mathcal{R}_{\alpha,p;\rho}^{\omega}(\{x\in\mathbb{R}^{n}:\mathcal{W}_{\omega;\rho}^{\mu}(x)>t\})\leq C(n,\alpha,p,[\omega]_{A_{p;40\rho}^{\rm loc}})\frac{\mu(\mathbb{R}^{n})}{t^{p-1}},\quad 0<t<\infty.
\end{align*}
\end{corollary}

The weak type estimate for the variant Wolff potential $W_{\omega;\rho}^{\mu}$ is given as follows.
\begin{theorem}\label{bounded}
Let $n\in\mathbb{N}$, $0<\alpha<n$, $1<p<\infty$, $0<\rho<\infty$, and $\omega$ be a weight. Suppose that $\mu$ is a positive measure on $\mathbb{R}^{n}$. Then
\begin{align*}
R_{\alpha,p;\rho}^{\omega}\left(\left\{x\in\mathbb{R}^{n}:W_{\omega;\rho}^{\mu}(x)>t\right\}\right)\leq C(n)\frac{\mu(\mathbb{R}^{n})}{t^{p-1}},\quad 0<t<\infty.
\end{align*}
\end{theorem}

\begin{proof}
The proof is almost resemble that of Theorem \ref{weak type}. Simply note that 
\begin{align*}
W_{\omega;\rho}^{\mu}(x)&=\int_{0}^{\rho}\left(\frac{\mu(Q_{t}(x))}{t^{n-\alpha p}}\right)^{\frac{1}{p-1}}\frac{1}{|Q_{t}(x)|}\int_{Q_{t}(x)}\omega'(y)dy\frac{dt}{t}\\
&=\int_{0}^{\rho}\left(\frac{\mu(Q_{t}(x))}{\gamma(Q_{t}(x))}\right)^{\frac{1}{p-1}}\left(\frac{\gamma(Q_{t}(x))}{t^{n-\alpha p}}\right)^{\frac{1}{p-1}}\frac{1}{|Q_{t}(x)|}\int_{Q_{t}(x)}\omega'(y)dy\frac{dt}{t}\\
&\leq{\bf{M}}_{\gamma}\mu(x)^{\frac{1}{p-1}}W_{\omega;\rho}^{\gamma}(x).
\end{align*}
The rest is identical to the argument therein.
\end{proof}

It is worth noting that the weak type estimate for the unweighted case is not valid for the nonlinear potential $G_{\alpha}\ast(G_{\alpha}\ast\mu)^{p'-1}$ for $1<p\leq 2-\frac{\alpha}{n}$. We claim that in this case when $\mu=\delta_{0}$ is the Dirac mass at the origin, then $G_{\alpha}\ast(G_{\alpha}\ast\mu)^{p'-1}$ is identical infinity. Indeed, using the asymptotic behavior (\ref{Bessel zero}) of $G_{\alpha}(\cdot)$ at zero, for any fixed $x_{0}\in\mathbb{R}^{n}$, one obtains
\begin{align*}
\left(G_{\alpha}\ast(G_{\alpha}\ast\mu)^{p'-1}\right)(x_{0})&=\int_{\mathbb{R}^{n}}G_{\alpha}(x_{0}-y)\left(\int_{\mathbb{R}^{n}}G_{\alpha}(y-z)d\mu(z)\right)^{p'-1}dy\\
&=\int_{\mathbb{R}^{n}}G_{\alpha}(x_{0}-y)G_{\alpha}(y)^{p'-1}dy\\
&\geq\int_{|y|<1}G_{\alpha}(x_{0}-y)G_{\alpha}(y)^{p'-1}dy\\
&\geq C(n,\alpha,p,x_{0})\int_{|y|<|x_{0}|+1}\frac{1}{|x_{0}-y|^{n-\alpha}}\frac{1}{|y|^{(n-\alpha)(p'-1)}}dy\\
&\geq C(n,\alpha,p,x_{0})\int_{|y|<|x_{0}|+1}\frac{1}{|y|^{(n-\alpha)(p'-1)}}dy\\
&=\infty
\end{align*}
since $1<p\leq 2-\frac{\alpha}{n}$ entails $(n-\alpha)(p'-1)\geq n$.

\subsection{Capacitary Strong Type Inequality}
\enskip

Given a positive measure $\mu$ on $\mathbb{R}^{n}$, it is known that the trace class inequalities of the form that 
\begin{align*}
\|G_{\alpha}\ast f\|_{L^{p}(\mu)}\leq C\|f\|_{L^{p}(\mathbb{R}^{n})}
\end{align*}
is equivalent to 
\begin{align*}
\mu(K)\leq C'{\rm Cap}_{\alpha,p}(K),\quad K\subseteq\mathbb{R}^{n}~\text{compact},
\end{align*}
where the constants $C,C'>0$ are comparable to each other (see \cite[Theorem 7.2.1]{AH} and \cite{MV2}). The proof therein uses the capacitary strong type inequality that 
\begin{align*}
\int_{0}^{\infty}{\rm Cap}_{\alpha,p}(\{x\in\mathbb{R}^{n}:(G_{\alpha}\ast f)(x)>t\})dt^{p}\leq C\int_{\mathbb{R}^{n}}f(x)^{p}dx,
\end{align*}
where $f\in L^{p}(\mathbb{R}^{n})$ with $f\geq 0$. The weighted analogue of the capacitary strong type inequality reads as the following.
\begin{theorem}\label{CSI}
Let $n\in\mathbb{N}$, $0<\alpha<n$, $1<p<\infty$, $0<\rho<\infty$, and $\omega'\in A_{\infty}^{\rm loc}$. For any $\varphi:\mathbb{R}^{n}\rightarrow[0,\infty]$, it holds that
\begin{align*}
&\int_{0}^{\infty}R_{\alpha,p;\rho}^{\omega}(\{x\in\mathbb{R}^{n}:(I_{\alpha,\rho}\ast\varphi)(x)>t\})dt^{p}\\
&\leq C(n,\alpha,p,[\omega']_{A_{\infty;2\rho}^{\rm loc}})\int_{\mathbb{R}^{n}}\varphi(x)^{p}\omega(x)dx.
\end{align*}
\end{theorem}

\begin{proof}
Note that the $>t$ in the left-sided of the above estimate can be replaced by $\geq t$ and vice versa. Assume first that $\varphi\in C_{0}(\mathbb{R}^{n})$ with $\varphi\geq 0$. Let 
\begin{align*}
E_{t}=\left\{x\in\mathbb{R}^{n}:(I_{\alpha,\rho}\ast\varphi)(x)\geq t\right\},\quad 0<t<\infty.
\end{align*}
Note that $E_{t}$ is compact. Hence $R_{\alpha,p;\rho}^{\omega}(E_{t})<\infty$ by Theorem \ref{estimate of cube}. Further, since $t\rightarrow R_{\alpha,p;\rho}^{\omega}(E_{t})$ is nonincreasing, it will be clear in the sequel that we may assume that $R_{\alpha,p;\rho}^{\omega}(E_{t})>0$, hence Proposition \ref{nonlinear use in CSI} is applicable. Denote by
\begin{align*}
J=\int_{0}^{\infty}R_{\alpha,p;\rho}^{\omega}(E_{t})dt^{p}<\infty.
\end{align*}
Let $\mu_{t}$ be the measure associated with $E_{t}$ as in Proposition \ref{nonlinear use in CSI} with respect to $R_{\alpha,p;32\rho}^{\omega}(\cdot)$, i.e.,
\begin{align*}
&V_{\omega;32\rho}^{\mu_{t}}(x)\leq 1\quad x\in E_{t},\\
&\mu_{t}(\mathbb{R}^{n})=R_{\alpha,p;32\rho}^{\omega}(E_{t}).
\end{align*}
Using Theorems \ref{homogeneous Wolff} and \ref{max 2}, we have 
\begin{align}\label{wolff bounded}
W_{\omega;8\rho}^{\mu_{t}}(x)\leq C(n,\alpha,p),\quad x\in\mathbb{R}^{n}.
\end{align}
We have 
\begin{align*}
J&\leq p\int_{0}^{\infty}\int_{\mathbb{R}^{n}}(I_{\alpha,\rho}\ast\varphi)(x)d\mu_{t}(x)t^{p-2}dt\\
&=p\int_{\mathbb{R}^{n}}\int_{0}^{\infty}(I_{\alpha,\rho}\ast\mu_{t})(y)t^{p-2}dt\varphi(y)dy\\
&\leq p\|\varphi\|_{L^{p}(\omega)}L^{\frac{1}{p'}},
\end{align*}
where
\begin{align*}
L=\int_{\mathbb{R}^{n}}\left(\int_{0}^{\infty}(I_{\alpha,\rho}\ast\mu_{t})(y)t^{p-2}dt\right)^{p'}\omega'(y)dy.
\end{align*}
To conclude the proof, we will show that $L\leq C(n,\alpha,p,[\omega]_{A_{\infty,\rho}^{\rm loc}})J$. Assume at the moment that $p\geq 2$. Let 
\begin{align*}
\lambda_{u}(E)=\int_{u}^{\infty}t^{p-2}\mu_{t}(E)dt,\quad 0\leq u<\infty.
\end{align*}
We have 
\begin{align*}
L&=\int_{\mathbb{R}^{n}}\left(\int_{0}^{\infty}\int_{|y-z|_{\infty}<\rho}\frac{1}{|y-z|_{\infty}^{n-\alpha}}d\mu_{t}(z)t^{p-2}dt\right)^{p'}\omega'(y)dy\\
&=C(n,\alpha,p)\int_{\mathbb{R}^{n}}\left(\int_{0}^{\infty}\int_{|y-z|_{\infty}<\rho}\int_{|y-z|_{\infty}}^{2|y-z|_{\infty}}\frac{1}{r^{n-\alpha}}\frac{dr}{r}d\mu_{t}(z)t^{p-2}dt\right)^{p'}\omega'(y)dy\\
&=C(n,\alpha,p)\int_{\mathbb{R}^{n}}\left(\int_{0}^{\infty}\int_{0}^{2\rho}\int_{|y-z|_{\infty}<r}d\mu_{t}(z)\frac{1}{r^{n-\alpha}}\frac{dr}{r}t^{p-2}dt\right)^{p'}\omega'(y)dy\\
&=C(n,\alpha,p)\int_{\mathbb{R}^{n}}\left(\int_{0}^{2\rho}\int_{0}^{\infty}\frac{\mu_{t}(Q_{r}(y))}{r^{n-\alpha}}t^{p-2}dt\frac{dr}{r}\right)^{p'}\omega'(y)dy\\
&=C(n,\alpha,p)\int_{\mathbb{R}^{n}}\left(\int_{0}^{2\rho}\frac{\lambda_{0}(Q_{r}(y))}{r^{n-\alpha}}\frac{dr}{r}\right)^{p'}\omega'(y)dy\\
&\leq C'(n,\alpha,p)\int_{\mathbb{R}^{n}}(I_{\alpha,2\rho}\ast\lambda_{0})(y)^{p'}\omega'(y)dy\\
&\leq C(n,\alpha,p,[\omega']_{A_{\infty;2\rho}^{\rm loc}})\int_{\mathbb{R}^{n}}M_{\alpha,2\rho}\lambda_{0}(y)^{s'}\omega'(y)dy\\
&\leq C'(n,\alpha,p,[\omega']_{A_{\infty;2\rho}^{\rm loc}})\int_{\mathbb{R}^{n}}\sup_{0<t\leq 2\rho}\int_{t}^{2t}\left(\frac{\lambda_{0}(Q_{r}(y))}{r^{n-\alpha}}\right)^{p'}\frac{dr}{r}\omega'(y)dy\\
&\leq C'(n,\alpha,p,[\omega']_{A_{\infty;2\rho}^{\rm loc}})\int_{\mathbb{R}^{n}}\int_{0}^{4\rho}\left(\frac{\lambda_{0}(Q_{r}(y))}{r^{n-\alpha}}\right)^{p'}\frac{dr}{r}\omega'(y)dy,
\end{align*}
where we have used the formula (\ref{246}) in the first inequality. On the other hand, integration by parts gives
\begin{align*}
\lambda_{0}(Q_{r}(y))^{p'}=p'\int_{0}^{\infty}\lambda_{u}(Q_{r}(y))^{p'-1}\mu_{u}(Q_{r}(y))u^{p-2}du.
\end{align*}
Express $\mu_{u}(B_{r}(y))=\mu_{u}(B_{r}(y))^{(2-p')p'}\cdot\mu_{u}(B_{r}(y))^{(p'-1)^{2}}$. Using H\"{o}lder's inequality with respect to the exponents that 
\begin{align*}
\frac{1}{(2-p')^{-1}}+\frac{1}{(p'-1)^{-1}}=1, 
\end{align*}
we obtain $L\leq C''(n,\alpha,p,[\omega]_{A_{\infty;2\rho}^{\rm loc}})L_{1}^{2-p'}L_{2}^{p'-1}$, where 
\begin{align*}
L_{1}=\int_{\mathbb{R}^{n}}\int_{0}^{4\rho}\int_{0}^{\infty}\left(\frac{\mu_{u}(Q_{r}(y))}{r^{n-\alpha}}\right)^{p'}u^{p-1}du\frac{dr}{r}\omega'(y)dy,
\end{align*}
and 
\begin{align*}
L_{2}=\int_{\mathbb{R}^{n}}\int_{0}^{4\rho}\int_{0}^{\infty}r^{(\alpha -n)p'}\mu_{u}(Q_{r}(y))^{p'-1}\lambda_{u}(Q_{r}(y))du\frac{dr}{r}\omega'(y)dy.
\end{align*}
We have 
\begin{align}
L_{1}&=\int_{0}^{\infty}\int_{\mathbb{R}^{n}}\int_{0}^{4\rho}\left(\frac{\mu_{u}(Q_{r}(y))}{r^{n-\alpha}}\right)^{\frac{1}{p-1}}\int_{Q_{r}(y)}d\mu_{u}(z)\frac{\omega'(y)}{r^{n-\alpha}}dy\frac{dr}{r}u^{p-1}du\notag\\
&=\int_{0}^{\infty}\int_{\mathbb{R}^{n}}\int_{0}^{4\rho}\int_{|y-z|_{\infty}<r}\left(\frac{\mu_{u}(Q_{r}(y))}{r^{n-\alpha}}\right)^{\frac{1}{p-1}}\frac{\omega'(y)}{r^{n-\alpha}}dy\frac{dr}{r}d\mu_{u}(z)u^{p-1}du\notag\\
&\leq\int_{0}^{\infty}\int_{\mathbb{R}^{n}}\int_{0}^{4\rho}\int_{|y-z|_{\infty}<r}\left(\frac{\mu_{u}(Q_{2r}(z))}{r^{n-\alpha}}\right)^{\frac{1}{p-1}}\frac{\omega'(y)}{r^{n-\alpha}}dy\frac{dr}{r}d\mu_{u}(z)u^{p-1}du\notag\\
&=\int_{0}^{\infty}\int_{\mathbb{R}^{n}}\int_{0}^{4\rho}\left(\frac{\mu_{u}(Q_{2r}(z))}{r^{n-\alpha p}}\right)^{\frac{1}{p-1}}\frac{1}{r^{n}}\int_{Q_{r}(z)}\omega'(y)dy\frac{dr}{r}d\mu_{u}(z)u^{p-1}du\notag\\
&\leq C(n,\alpha,p)\int_{0}^{\infty}\int_{\mathbb{R}^{n}}W_{\omega;8\rho}^{\mu_{u}}(z)d\mu_{u}(z)u^{p-1}du\notag\\
&\leq C'(n,\alpha,p)\int_{0}^{\infty}\mu_{u}(\mathbb{R}^{n})u^{p-1}du\label{use wolff}\\
&= C'(n,\alpha,p)\int_{0}^{\infty}R_{\alpha,p;32\rho}^{\omega}(E_{u})u^{p-1}du\notag\\
&\leq C'(n,\alpha,p)\int_{0}^{\infty}R_{\alpha,p;\rho}^{\omega}(E_{u})u^{p-1}du\notag\\
&= C''(n,\alpha,p)J\notag,
\end{align}
where we have used (\ref{wolff bounded}) in (\ref{use wolff}). Subsequently, we have
\begin{align*}
&L_{2}\\
&=\int_{\mathbb{R}^{n}}\int_{0}^{4\rho}\int_{0}^{\infty}r^{(\alpha-n)p'}\mu_{u}(Q_{r}(y))^{p'-1}\int_{u}^{\infty}t^{p-2}\mu_{t}(Q_{r}(y))dtdu\frac{dr}{r}\omega'(y)dy\\
&=\int_{0}^{\infty}\int_{u}^{\infty}t^{p-2}\int_{\mathbb{R}^{n}}\int_{0}^{4\rho}r^{(\alpha-n)p'}\mu_{u}(Q_{r}(y))^{p'-1}\int_{Q_{r}(y)}d\mu_{t}(z)\frac{dr}{r}\omega'(y)dydtdu\\
&=\int_{0}^{\infty}\int_{u}^{\infty}t^{p-2}\int_{\mathbb{R}^{n}}\int_{0}^{4\rho}\int_{|y-z|_{\infty}<r}\left(\frac{\mu_{u}(Q_{r}(y))}{r^{n-\alpha}}\right)^{\frac{1}{p-1}}\frac{\omega'(y)}{r^{n-\alpha}}dy\frac{dr}{r}d\mu_{t}(z)dtdu\\
&\leq\int_{0}^{\infty}\int_{u}^{\infty}t^{p-2}\int_{\mathbb{R}^{n}}\int_{0}^{4\rho}\int_{|y-z|_{\infty}<r}\left(\frac{\mu_{u}(Q_{2r}(z))}{r^{n-\alpha}}\right)^{\frac{1}{p-1}}\frac{\omega'(y)}{r^{n-\alpha}}dy\frac{dr}{r}d\mu_{t}(z)dtdu\\
&\leq C(n,\alpha,p)\int_{0}^{\infty}\int_{u}^{\infty}t^{p-2}\int_{\mathbb{R}^{n}}W_{\omega;8\rho}^{\mu_{u}}(z)d\mu_{t}(z)dtdu\\
&\leq C(n,\alpha,p)\int_{0}^{\infty}\int_{u}^{\infty}t^{p-2}\mu_{t}(\mathbb{R}^{n})dtdu\\
&\leq C(n,\alpha,p)J.
\end{align*}
Now we consider the case where $1<p<2$. We write 
\begin{align*}
\lambda^{u}(E)=\int_{0}^{u}t^{p-2}\mu_{t}(E)dt,\quad 0<u\leq\infty,
\end{align*}
and integration by parts gives 
\begin{align*}
\lambda^{\infty}(Q_{r}(y))^{p'}=p'\int_{0}^{\infty}\lambda^{u}(Q_{r}(y))^{p'-1}\mu_{u}(Q_{r}(y))u^{p-2}du.
\end{align*}
Repeating the estimates of $L$ in the first part, we obtain
\begin{align*}
L&=C(n,\alpha,p)\int_{\mathbb{R}^{n}}\left(\int_{0}^{2\rho}\frac{\lambda^{\infty}(Q_{r}(y))}{r^{n-\alpha}}\frac{dr}{r}\right)^{p'}\omega'(y)dy\\
&\leq C'(n,\alpha,p)\int_{\mathbb{R}^{n}}(I_{\alpha,2\rho}\ast\lambda^{\infty})(y)^{p'}\omega'(y)dy\\
&\leq C(n,\alpha,p,[\omega']_{A_{\infty;2\rho}^{\rm loc}})\int_{\mathbb{R}^{n}}\int_{0}^{2\rho}\left(\frac{\lambda^{\infty}(Q_{r}(y))}{r^{n-\alpha}}\right)^{p'}\frac{dr}{r}\omega'(y)dy\\
&=C(n,\alpha,p,[\omega']_{A_{\infty;2\rho}^{\rm loc}})\int_{0}^{\infty}\left\|\lambda^{u}\left(Q_{(\cdot)}(\cdot)\right)\right\|_{L^{p'-1}(\sigma_{u})}^{p'-1}u^{p-2}du,
\end{align*}
where 
\begin{align*}
d\sigma_{u}(y,r)=\chi_{0<r<4\rho}\cdot r^{(\alpha-n)p'}\mu_{u}(Q_{r}(y))\omega'(y)dy\frac{dr}{r}.
\end{align*}
Note that 
\begin{align*}
&\left\|\lambda^{u}\left(Q_{(\cdot)}(\cdot)\right)\right\|_{L^{p'-1}(\sigma_{u})}\\
&\leq\int_{0}^{u}t^{p-2}\left\|\mu_{t}\left(Q_{(\cdot)}(\cdot)\right)\right\|_{L^{p'-1}(\sigma_{u})}dt\\
&=\int_{0}^{u}t^{p-2}\left(\int_{0}^{4\rho}\int_{\mathbb{R}^{n}}\left(\frac{\mu_{t}(Q_{r}(y))}{r^{n-\alpha}}\right)^{\frac{1}{p-1}}\int_{Q_{r}(y)}d\mu_{u}(z)\frac{\omega'(y)}{r^{n-\alpha}}dy\frac{dr}{r}\right)^{p-1}dt\\
&\leq\int_{0}^{u}t^{p-2}\left(\int_{0}^{4\rho}\int_{\mathbb{R}^{n}}\left(\frac{\mu_{t}(Q_{2r}(z))}{r^{n-\alpha}}\right)^{\frac{1}{p-1}}\int_{Q_{r}(y)}d\mu_{u}(z)\frac{\omega'(y)}{r^{n-\alpha}}dy\frac{dr}{r}\right)^{p-1}dt\\
&=\int_{0}^{u}t^{p-2}\left(\int_{\mathbb{R}^{n}}\int_{0}^{4\rho}\left(\frac{\mu_{t}(Q_{2r}(z))}{r^{n-\alpha p}}\right)^{\frac{1}{p-1}}\frac{1}{r^{n}}\int_{Q_{r}(z)}\omega'(y)dy\frac{dr}{r}d\mu_{u}(z)\right)^{p-1}dt\\
&\leq C(n,\alpha,p)\int_{0}^{u}t^{p-2}\left(\int_{\mathbb{R}^{n}}W_{\omega;8\rho}^{\mu_{t}}(z)d\mu_{u}(z)\right)^{p-1}dt\\
&\leq C'(n,\alpha,p)\mu_{u}(\mathbb{R}^{n})^{p-1}u^{p-1},
\end{align*}
where (\ref{wolff bounded}) is used in the last inequality. We conclude that 
\begin{align*}
L&\leq C'(n,\alpha,p,[\omega']_{A_{\infty;2\rho}^{\rm loc}})\int_{0}^{\infty}\mu_{u}(\mathbb{R}^{n})u^{p-1}du\\
&=C'(n,\alpha,p,[\omega']_{A_{\infty;2\rho}^{\rm loc}})\int_{0}^{\infty}R_{\alpha,p;32\rho}^{\omega}(E_{u})u^{p-1}du\\
&\leq C'(n,\alpha,p,[\omega']_{A_{\infty;2\rho}^{\rm loc}})\int_{0}^{\infty}R_{\alpha,p;\rho}^{\omega}(E_{u})u^{p-1}du\\
&= C''(n,\alpha,p,[\omega']_{A_{\infty;2\rho}^{\rm loc}})J,
\end{align*}
and the result holds for the case where $\varphi\in C_{0}(\mathbb{R}^{n})$ with $\varphi\geq 0$. For general $\varphi\in L^{p}(\omega)$ with $\varphi\geq 0$, we approximate $\varphi$ by a sequence $\{\varphi_{j}\}$ of nonnegative functions in $C_{0}(\mathbb{R}^{n})$. Indeed, if $\varphi_{j}\rightarrow\varphi$ in $L^{p}(\omega)$ and $\varphi_{j}(x)\rightarrow\varphi(x)$ a.e., then
\begin{align*}
(I_{\alpha,\rho}\ast\varphi)(x)\leq\liminf_{j\rightarrow\infty}(I_{\alpha,\rho}\ast\varphi_{j})(x)
\end{align*}
everywhere $x\in\mathbb{R}^{n}$ and hence 
\begin{align*}
&R_{\alpha,p;\rho}^{\omega}\left(\left\{x\in\mathbb{R}^{n}:(I_{\alpha,\rho}\ast\varphi)(x)>t\right\}\right)\\
&\leq R_{\alpha,p;\rho}^{\omega}\left(\left\{x\in\mathbb{R}^{n}:\liminf_{j\rightarrow\infty}(I_{\alpha,\rho}\ast\varphi_{j})(x)>t\right\}\right)\\
&\leq R_{\alpha,p;\rho}^{\omega}\left(\bigcup_{j\in\mathbb{N}}\bigcap_{k\geq j}\left\{x\in\mathbb{R}^{n}:(I_{\alpha,\rho}\ast\varphi_{j})(x)>t\right\}\right)\\
&=\sup_{j\in\mathbb{N}}R_{\alpha,p;\rho}^{\omega}\left(\bigcap_{k\geq j}\left\{x\in\mathbb{R}^{n}:(I_{\alpha,\rho}\ast\varphi_{j})(x)>t\right\}\right)\\
&\leq\sup_{j\in\mathbb{N}}\inf_{k\geq j}R_{\alpha,p;\rho}^{\omega}\left(\left\{x\in\mathbb{R}^{n}:(I_{\alpha,\rho}\ast\varphi_{j})(x)>t\right\}\right)\\
&=\liminf_{j\rightarrow\infty}R_{\alpha,p;\rho}^{\omega}(\{x\in\mathbb{R}^{n}:(I_{\alpha,1}\ast\varphi_{j})(x)>t\}),
\end{align*}
where we have used Proposition \ref{monotone} in the first equality. The proof is concluded by an application of Fatou's lemma. 
\end{proof}

\newpage
\section{Some Applications}
Having established the theory of capacities associated with $A_{p}^{\rm loc}$ weights, we are ready to give some applications. We start with the rather easy application about the absolute continuity of $R_{\alpha,p;\rho}^{\omega}(\cdot)$ with respect to Lebesgue measure under certain constrains on the weights $\omega$.

\subsection{Absolute Continuity Associated with Weighted Capacities}
\enskip

We have the standard Sobolev embedding theorem that
\begin{align*}
|E|^{1-\frac{\alpha p}{n}}\leq C{\rm Cap}_{\alpha,p}(E),\quad E\subseteq\mathbb{R}^{n},\quad 1<p<\frac{n}{\alpha},\quad 0<\alpha<n,\quad n\in\mathbb{N}.
\end{align*}
An immediate consequence of the above embedding is that ${\rm Cap}_{\alpha,p}(E)=0$ entails $|E|=0$, which is known to be the absolute continuity of the Lebesgue measure with respect to ${\rm Cap}_{\alpha,p}(\cdot)$. The weighted analogue reads as the following.
\begin{theorem}\label{absolute continuity}
Let $n\in\mathbb{N}$, $0<\alpha<n$, $1<p<\infty$, $0<\rho<\infty$, and $\omega\in A_{p}^{\rm loc}$. Suppose that $E\subseteq\mathbb{R}^{n}$ satisfies $R_{\alpha,p;\rho}^{\omega}(E)=0$. Then $|E|=0$. The same conclusion holds if $R_{\alpha,p;\rho}^{\omega}(\cdot)$ is replaced by $\mathcal{R}_{\alpha,p;\rho}^{\omega}(\cdot)$.
\end{theorem}

\begin{proof}
Let $\varepsilon>0$ and $G$ be an open set with $G\supseteq E$ and $R_{\alpha,p;\rho}^{\omega}(G)<\varepsilon$. Assume that $f\geq 0$ satisfies $I_{\alpha,\rho}\ast f\geq 1$ on $G$ and 
\begin{align*}
\int_{\mathbb{R}^{n}}f(x)^{p}\omega(x)dx<\varepsilon.
\end{align*}
Note that the formula (\ref{246}) gives
\begin{align*}
(I_{\alpha,\rho}\ast f)(x)\leq C(n,\alpha,\rho)\mathcal{M}_{2\rho}^{\rm loc}f(x),\quad x\in\mathbb{R}^{n}.
\end{align*} 
Using Theorem \ref{center strong type}, one has 
\begin{align*}
\omega(G)&\leq\int_{G}(I_{\alpha,\rho}\ast f)(x)^{p}\omega(x)dx\\
&\leq C(n,\alpha,p,\rho,[\omega]_{p;6\rho}^{\rm loc})\int_{\mathbb{R}^{n}}f(x)^{p}\omega(x)dx\\
&\leq C(n,\alpha,p,\rho,[\omega]_{p;6\rho}^{\rm loc})\varepsilon.
\end{align*}
Since $\varepsilon>0$ is arbitrary, we deduce that $\omega(G)=0$ and hence $|G|=0$ as $\omega(x)>0$ almost everywhere. Then $|E|=0$ follows by the outer regularity of Lebesgue measure. The last assertion follows by Theorem \ref{R cal}.
\end{proof}

To get rid of the assumption that $\omega\in A_{p}^{\rm loc}$ in the previous theorem, we have the following variant form of absolute continuity.
\begin{theorem}
Let $n\in\mathbb{N}$, $0<\alpha<n$, $1<p<\infty$, $0<2\rho\leq\rho'<\infty$, and $\omega$ be a weight such that ${\bf M}_{\rho'}^{\rm loc}\omega$ is locally integrable on $\mathbb{R}^{n}$. Suppose that $E\subseteq\mathbb{R}^{n}$ satisfies $R_{\alpha,p;\rho}^{{\bf M}_{\rho'}^{\rm loc}\omega}(E)=0$. Then $|E|=0$.
\end{theorem}

\begin{proof}
The proof is similar to that of Theorem \ref{absolute continuity}. Note that 
\begin{align*}
(I_{\alpha,\rho}\ast f)(x)\leq C(n,\alpha,\rho){\bf M}_{2\rho}^{\rm loc}f(x)\leq C(n,\alpha,\rho){\bf M}_{\rho'}^{\rm loc}f(x),\quad x\in\mathbb{R}^{n}.
\end{align*} 
Using (\ref{FS2}) of Theorem \ref{FS}, one has
\begin{align*}
\omega(G)\leq\int_{G}(I_{\alpha,\rho}\ast f)(x)^{p}\omega(x)dx\leq C(n,\alpha,p,\rho)\int_{\mathbb{R}^{n}}f(x)^{p}{\bf M}_{\rho'}^{\rm loc}\omega(x)dx
\end{align*}
The rest follows just as in the proof of Theorem \ref{absolute continuity}.
\end{proof}

\subsection{Choquet Integrals Associated with Weighted Capacities}
\enskip

Recall the general construction of capacities as in subsection 3.1 together with the notations therein. We define the Choquet integrals associated with $C_{k,\nu,p}(\cdot)$ that
\begin{align*}
\int_{E}|f|dC_{k,\nu,p}=\int_{0}^{\infty}C_{k,\nu,p}(\{x\in\mathbb{R}^{n}:|f(x)|>t\})dt
\end{align*}
for any function $f:E\rightarrow[-\infty,\infty]$ and set $E\subseteq\mathbb{R}^{n}$. The Choquet integrals form the function spaces $L^{q}(E,C_{k,\nu,p})$ defined by 
\begin{align*}
L^{q}(E,C_{k,\nu,p})=\left\{f:E\rightarrow[-\infty,\infty]:\|f\|_{L^{q}(C_{k,\nu,p})}<\infty\right\},
\end{align*}
where $0<q<\infty$ and 
\begin{align*}
\|f\|_{L^{q}(E,C_{k,\nu,p})}=\left(\int_{E}|f|^{q}dC_{k,\nu,p}\right)^{\frac{1}{q}}.
\end{align*}
We also define $L^{\infty}(E,C_{k,\nu,p})$ and $L^{q}(E,C_{k,\nu,p})$ by
\begin{align*}
L^{q}(E,C_{k,\nu,p})&=\left\{f:E\rightarrow[-\infty,\infty]:\|f\|_{L^{\infty}(E,C_{k,\nu,p})}<\infty\right\},\\
L^{q,\infty}(E,C_{k,\nu,p})&=\left\{f:E\rightarrow[-\infty,\infty]:\|f\|_{L^{q,\infty}(E,C_{k,\nu,p})}<\infty\right\},\quad 0<q<\infty,
\end{align*}
where 
\begin{align*}
\|f\|_{L^{\infty}(E,C_{k,\nu,p})}&=\inf\{a>0:|f(x)|\leq a~C_{k,\nu,p}(\cdot)\text{-quasi-everywhere on}~E\},\\
\|f\|_{L^{q,\infty}(E,C_{k,\nu,p})}&=\sup_{t>0}t C_{k,\nu,p}(\{x\in E:|f(x)|>t\})^{\frac{1}{q}}.
\end{align*}
When $E=\mathbb{R}^{n}$, the symbol $E$ will be hidden, i.e., $L^{q}(E,C_{k,\nu,p})$ will be abbreviated as $L^{q}(C_{k,\nu,p})$, and similar convention applies to others. It is immediate that 
\begin{align*}
L^{q}(E,C_{k,\nu,p})\hookrightarrow L^{q,\infty}(E,C_{k,\nu,p}),\quad 0<q<\infty.
\end{align*}
To see this, simply note that
\begin{align*}
\|f\|_{L^{q,\infty}(E,C_{k,\nu,p})}&=\sup_{t>0}t C_{k,\nu,p}(\{x\in E:|f(x)|>t\})^{\frac{1}{q}}\\
&=\sup_{t>0}\left(\int_{0}^{t^{q}}C_{k,\nu,p}\left(\{x\in E:|f(x)|>t\}\right)d\lambda\right)^{\frac{1}{q}}\\
&\leq\sup_{t>0}\left(\int_{0}^{t^{q}}C_{k,\nu,p}\left(\{x\in E:|f(x)|^{q}>\lambda\}\right)d\lambda\right)^{\frac{1}{q}}\\
&\leq\left(\int_{0}^{\infty}C_{k,\nu,p}\left(\{x\in E:|f(x)|^{q}>\lambda\}\right)d\lambda\right)^{\frac{1}{q}}\\
&=\|f\|_{L^{q}(E,C_{k,\nu,p})}.
\end{align*}

On the other hand, we associate a functional $\mathcal{C}(\cdot)$ with $C_{k,\nu,p}(\cdot)$ that 
\begin{align*}
&\mathcal{C}(\varphi)\\
&=\inf\{\|f\|_{L_{\nu}^{p}(\mathbb{R}^{m})}^{p}:f\in L_{\nu}^{p}(\mathbb{R}^{m})^{+},~k(x,f\nu)\geq|\varphi(x)|^{\frac{1}{p}}~C_{k,\nu,p}(\cdot)\text{-quasi-everywhere}\},
\end{align*}
where $\varphi:\mathbb{R}^{n}\rightarrow[-\infty,\infty]$ is any function. Then it is trivial that 
\begin{align*}
\mathcal{C}(\chi_{E})=C_{k,\nu,p}(E)
\end{align*}
for any set $E\subseteq\mathbb{R}^{n}$.

\begin{proposition}\label{sub}
$\mathcal{C}(\cdot)$ is subadditive.
\end{proposition}

\begin{proof}
Let $\mathcal{H}$ be the set of all $\varphi$ on $\mathbb{R}^{n}$ such that $\varphi\geq 0$ and $k(x,f\nu)(x)\geq\varphi(x)^{\frac{1}{p}}$ $C_{k,\nu,p}(\cdot)$-quasi-everywhere for some function $f\in L_{\nu}^{p}(\mathbb{R}^{m})^{+}$ with $\|f\|_{L_{\nu}^{p}(\mathbb{R}^{m})}\leq 1$. We claim that $\mathcal{H}$ is convex. Suppose that $\varphi_{1},\varphi_{2}\in\mathcal{H}$ and $0<c<1$. Then there are $f_{1},f_{2}\in L_{+}^{s}(\nu)$ such that $k(x,f_{j}\nu)(x)\geq\varphi_{j}(x)^{\frac{1}{p}}$ $C_{k,\nu,p}(\cdot)$-quasi-everywhere and $\|f_{j}\|_{L_{\nu}^{p}(\mathbb{R}^{m})}\leq 1$ for $j=1,2$. It follows by the reverse Minkowski's inequality that
\begin{align*}
&\left(k\left(x,(cf_{1}^{p}+(1-c)f_{2}^{p})^{\frac{1}{p}}\nu\right)\right)^{p}\\
&=\left(\int_{\mathbb{R}^{m}}k(x,y)\left(cf_{1}(y)^{p}+(1-c)f_{2}(y)^{p}\right)^{\frac{1}{p}}d\nu(y)\right)^{p}\\
&\geq\left(\int_{\mathbb{R}^{m}}k(x,y)\left(cf_{1}(y)^{p}\right)^{\frac{1}{p}}d\nu(y)\right)^{p}+\left(\int_{\mathbb{R}^{m}}k(x,y)\left((1-c)f_{2}(y)^{p}\right)^{\frac{1}{p}}d\nu(y)\right)^{p}\\
&=c\left(k(x,f_{1}\nu)\right)^{p}+(1-c)\left(k(x,f_{2}\nu)\right)^{p}\\
&\geq(c\varphi_{1}+(1-c)\varphi_{2})(x).
\end{align*}
On the other hand, we have 
\begin{align*}
\left\|(cf_{1}^{p}+(1-c)f_{2}^{p})^{\frac{1}{p}}\right\|_{L_{\nu}^{p}(\mathbb{R}^{m})}&=\left(\int_{\mathbb{R}^{m}}\left(cf_{1}(y)^{p}+(1-c)f_{2}(y)^{p}\right)d\nu(y)\right)^{\frac{1}{p}}\\
&=\left(c\|f_{1}\|_{L_{\nu}^{p}(\mathbb{R}^{m})}^{p}+(1-c)\|f_{2}\|_{L_{\nu}^{p}(\mathbb{R}^{m})}^{p}\right)^{\frac{1}{p}}\\
&\leq(c+(1-c))^{\frac{1}{p}}\\
&=1.
\end{align*}
As a result, the convexity of $\mathcal{H}$ is justified. It is now routine to check that
\begin{align*}
\mathcal{C}(\varphi)=\inf\{c>0:|\varphi|\in c\mathcal{H}\},
\end{align*}
which yields the subadditivity of $\mathcal{C}(\cdot)$.
\end{proof}

\begin{theorem}\label{first}
Let $n\in\mathbb{N}$, $0<\alpha<n$, $1<p<\infty$, $0<\rho<\infty$, and $\omega'\in A_{\infty}^{\rm loc}$. Suppose that $C_{k,\nu,p}(\cdot)=R_{\alpha,p;\rho}^{\omega}(\cdot)$. Then
\begin{align}\label{equi}
\frac{1}{4}\mathcal{C}(\varphi)\leq\int_{\mathbb{R}^{n}}|\varphi|dC_{k,\nu,p}\leq C(n,\alpha,p,[\omega']_{A_{\infty;2\rho}^{\rm loc}})\mathcal{C}(\varphi)
\end{align}
holds for any function $\varphi:\mathbb{R}^{n}\rightarrow[-\infty,\infty]$. As a result, the space $L^{1}(R_{\alpha,p;\rho}^{\omega})$ is normable by Proposition $\ref{sub}$.
\end{theorem}

\begin{proof}
First of all, Theorem \ref{CSI} shows that 
\begin{align*}
\int_{0}^{\infty}C_{k,\nu,p}(\{x\in\mathbb{R}^{n}:|\varphi(x)|>t\})dt&\leq\int_{0}^{\infty}C_{k,\nu,p}(\{x\in\mathbb{R}^{n}:(I_{\alpha,\rho}\ast f)(x)^{p}>t\})dt\\
&=\int_{0}^{\infty}C_{k,\nu,p}(\{x\in\mathbb{R}^{n}:(I_{\alpha,\rho}\ast f)(x)>t\})dt^{p}\\
&\leq C(n,\alpha,p,[\omega]_{A_{p;2\rho}^{\rm loc}})\|f\|_{L^{p}(\omega)}^{p}
\end{align*}
for any $f\in L^{p}(\omega)^{+}$ with $(I_{\alpha,\rho}\ast f)(x)\geq|\varphi(x)|^{\frac{1}{p}}$ $C_{k,\nu,p}(\cdot)$-quasi-everywhere. Taking infimum with respect to all such $f$, one obtains the second inequality in (\ref{equi}).

For the first inequality in (\ref{equi}), it suffices to prove for $|\varphi|^{p}$ in place of $|\varphi|$.  Then we may assume that 
\begin{align*}
\int_{\mathbb{R}^{n}}|\varphi|^{p}dC_{k,\nu,p}<\infty.
\end{align*}
For every $N=1,2,\ldots$, we have 
\begin{align*}
C_{k,\nu,p}(\{x\in\mathbb{R}^{n}:|\varphi(x)|=\infty\})\leq\frac{1}{N^{p}}\int_{0}^{\infty}C_{k,\nu,p}(\{x\in\mathbb{R}^{n}:|\varphi(x)|^{p}>t\})dt.
\end{align*}
Then $C_{k,\nu,p}(\{x\in\mathbb{R}^{n}:|\varphi(x)|=\infty\})=0$. Denote by $\overline{\mathcal{C}}(\varphi)=\mathcal{C}(|\varphi|^{p})$ for functions $\varphi:\mathbb{R}^{n}\rightarrow[-\infty,\infty]$. Assume at the moment that  
\begin{align}\label{prove later}
\overline{\mathcal{C}}\left(\sum_{i=1}^{\infty}\varphi_{i}\right)\leq\sum_{i=1}^{\infty}\overline{C}(\varphi_{i}),
\end{align}
where $\varphi_{i}:\mathbb{R}^{n}\rightarrow[-\infty,\infty]$ with $\{\varphi_{i}\ne 0\}\cap\{\varphi_{j}\ne 0\}=\emptyset$ for $i\ne j$. Then 
\begin{align*}
\overline{\mathcal{C}}(\varphi)&\leq\overline{\mathcal{C}}(\varphi\chi_{\{|\varphi|<\infty\}})+\overline{\mathcal{C}}\left(\chi_{\{|\varphi|=\infty\}}\right)\\
&=\overline{\mathcal{C}}(\varphi\chi_{\{|\varphi|<\infty\}})\\
&=\overline{\mathcal{C}}\left(\sum_{i\in\mathbb{Z}}|\varphi|\chi_{\{2^{i}\leq|\varphi|^{p}<2^{i+1}\}}\right)\\
&\leq\sum_{i\in\mathbb{Z}}\overline{\mathcal{C}}\left(\varphi\chi_{\{2^{i}\leq|\varphi|^{p}<2^{i+1}\}}\right)\\
&\leq \sum_{i\in\mathbb{Z}}2^{i+1}C_{k,\nu,p}\left(\left\{x\in\mathbb{R}^{n}:2^{i}\leq|\varphi(x)|^{p}<2^{i+1}\right\}\right).
\end{align*}
On the other hand,
\begin{align*}
&\int_{0}^{\infty}C_{k,\nu,p}(\{x\in\mathbb{R}^{n}:|\varphi(x)|^{p}>t\})dt\\
&=\sum_{i\in\mathbb{Z}}\int_{2^{i-1}}^{2^{i}}C_{k,\nu,p}\left(\left\{x\in\mathbb{R}^{n}:|\varphi(x)|^{p}>t\right\}\right)dt\\
&\geq\sum_{i\in\mathbb{Z}}2^{i-1}C_{k,\nu,p}\left(\left\{x\in\mathbb{R}^{n}:2^{i}\leq|\varphi(x)|^{p}<2^{i+1}\right\}\right),
\end{align*}
which yields the first estimate of (\ref{equi}). It remains to show for (\ref{prove later}). Let $\varepsilon>0$ and $f_{i}\in L_{\nu}^{p}(\mathbb{R}^{m})^{+}$ be such that $k(x,f_{i}\nu)\geq|\varphi_{i}(x)|$ $C_{k,\nu,p}(\cdot)$-quasi-everywhere and 
\begin{align*}
\|f_{i}\|_{L_{\nu}^{p}(\mathbb{R}^{m})}^{p}<\overline{\mathcal{C}}(\varphi_{i})+\frac{\varepsilon}{2^{i}}.
\end{align*}
Let $f=\sup_{i}f_{i}$. Then the ``disjoint support" of $\varphi_{i}$ gives 
\begin{align*}
k(x,f\nu)\geq\sum_{i=1}^{\infty}|\varphi_{i}(x)|
\end{align*}
$C_{k,\nu,p}(\cdot)$-quasi-everywhere. Subsequently, we have 
\begin{align*}
\int_{\mathbb{R}^{m}}f(y)^{p}d\nu(y)&\leq\sum_{i=1}^{\infty}\int_{\mathbb{R}^{m}}f_{i}(y)^{p}d\nu(y)\\
&\leq\sum_{i=1}^{\infty}\left(\overline{\mathcal{C}}(\varphi_{i})+\frac{\varepsilon}{2^{i}}\right)\\
&=\sum_{i=1}^{\infty}\overline{\mathcal{C}}(\varphi_{i})+\varepsilon.
\end{align*}
The arbitrariness of $\varepsilon>0$ justifies the validity of (\ref{prove later}).
\end{proof}

We claim that the spaces $L^{q}(R_{\alpha,p;\rho}^{\omega})$ and $L^{q,\infty}(R_{\alpha,p;\rho}^{\omega})$ satisfy the $q$-convexity for $0<q<1$.
\begin{proposition}\label{convex}
Let $n\in\mathbb{N}$, $0<\alpha<n$, $1<p<\infty$, $0<\rho<\infty$, and $\omega'\in A_{\infty}^{\rm loc}$. For any $0<q\leq 1$, it holds that
\begin{align*}
\left\|\sum_{k=1}^{\infty}|f_{k}|\right\|_{L^{q}(R_{\alpha,p;\rho}^{\omega})}\leq C(n,\alpha,p,[\omega']_{A_{\infty;2\rho}^{\rm loc}})\left(\sum_{k=1}^{\infty}\left\|f_{k}\right\|_{L^{q}(R_{\alpha,p;\rho}^{\omega})}^{q}\right)^{\frac{1}{q}}.
\end{align*}
\end{proposition}

\begin{proof}
We estimate that 
\begin{align}
\left\|\sum_{k=1}^{\infty}|f_{k}|\right\|_{L^{q}(R_{\alpha,p;\rho}^{\omega})}&=\sup_{N\in\mathbb{N}}\left\|\sum_{k=1}^{N}|f_{k}|\right\|_{L^{q}(R_{\alpha,p;\rho}^{\omega})}\label{use monotone}\\
&=\sup_{N\in\mathbb{N}}\left\|\left(\sum_{k=1}^{N}|f_{k}|\right)^{q}\right\|_{L^{1}(R_{\alpha,p;\rho}^{\omega})}^{\frac{1}{q}}\notag\\
&\leq\sup_{N\in\mathbb{N}}\left\|\sum_{k=1}^{N}|f_{k}|^{q}\right\|_{L^{1}(R_{\alpha,p;\rho}^{\omega})}^{\frac{1}{q}}\notag\\
&\leq C(n,\alpha,p,[\omega']_{A_{\infty;2\rho}^{\rm loc}})\left(\sum_{k=1}^{\infty}\left\||f_{k}|^{q}\right\|_{L^{1}(R_{\alpha,p;\rho}^{\omega})}\right)^{\frac{1}{q}}\label{app normable}\\
&=C(n,\alpha,p,[\omega']_{A_{\infty;2\rho}^{\rm loc}})\left(\sum_{k=1}^{\infty}\|f_{k}\|_{L^{q}(R_{\alpha,p;\rho}^{\omega})}^{q}\right)^{\frac{1}{q}}\notag,
\end{align}
where we have used Proposition \ref{monotone} in (\ref{use monotone}) and also Theorem \ref{first} for the normability of $L^{1}(R_{\alpha,p;\rho}^{\omega})$ in (\ref{app normable}).
\end{proof}

\begin{lemma}\label{realization}
Let $n\in\mathbb{N}$, $0<q<\infty$, and $E\subseteq\mathbb{R}^{n}$ be an arbitrary set. For any sequence $\{f_{k}\}_{k=1}^{m}$ of functions $f_{k}:E\rightarrow[-\infty,\infty]$, $k=1,2,\ldots$, $m\in\mathbb{N}$, it holds that 
\begin{align*}
\left\|\max_{1\leq k\leq m}|f_{k}|\right\|_{L^{q,\infty}(E,C_{k,\nu,p})}^{q}\leq\sum_{k=1}^{m}\|f_{k}\|_{L^{q,\infty}(E,C_{k,\nu,p})}^{q}.
\end{align*}
\end{lemma}

\begin{proof}

Indeed, for any $t>0$, it holds that 
\begin{align*}
C_{k,\nu,p}\left(\left\{x\in E:\max_{1\leq k\leq m}|f_{k}(x)|>t\right\}\right)&\leq C_{k,\nu,p}\left(\bigcup_{k=1}^{m}\{x\in E:|f_{k}(x)|>t\}\right)\\
&\leq\sum_{k=1}^{m}C_{k,\nu,p}(\{x\in E:|f_{k}(x)|>t\})\\
&\leq\sum_{k=1}^{m}\frac{1}{t^{q}}\|f_{k}\|_{L^{q,\infty}(E,C_{k,\nu,p})}^{q},
\end{align*}
and hence
\begin{align*}
\left\|\max_{1\leq k\leq m}|f_{j}|\right\|_{L^{q,\infty}(E,C_{k,\nu,p})}^{q}&=\sup_{t>0}t^{q}C_{k,\nu,p}\left(\left\{x\in E:\max_{1\leq k\leq m}|f_{k}(x)|>\alpha\right\}\right)\\
&\leq\sum_{k=1}^{m}\|f_{k}\|_{L^{q,\infty}(E,C_{k,\nu,p})}^{q},
\end{align*}
as expected.
\end{proof}

\begin{proposition}\label{2.3}
Let $n\in\mathbb{N}$, $0<\alpha<n$, $1<p<\infty$, $0<\rho<\infty$, and $\omega'\in A_{\infty}^{\rm loc}$. For any $0<q\leq 1$, it holds that
\begin{align*}
\left\|\sum_{k=1}^{\infty}|f_{k}|\right\|_{L^{q,\infty}(R_{\alpha,p;\rho}^{\omega})}\leq C(n,\alpha,p,[\omega']_{A_{\infty;2\rho}^{\rm loc}},q)\left(\sum_{k=1}^{\infty}\left\|f_{k}\right\|_{L^{q,\infty}(R_{\alpha,p;\rho}^{\omega})}^{q}\right)^{\frac{1}{q}}.
\end{align*}
\end{proposition}

\begin{proof}
For any $t>0$, Proposition \ref{monotone} entails
\begin{align*}
&R_{\alpha,p;\rho}^{\omega}\left(\left\{\sum_{k}|f_{k}|>t\right\}\right)\\
&\leq R_{\alpha,p;\rho}^{\omega}\left(\left\{\sum_{k}|f_{k}|>t,\sup_{k}|f_{k}|\leq t\right\}\right)+R_{\alpha,p;\rho}^{\omega}\left(\left\{\sup_{k}|f_{k}|>t\right\}\right)\\
&=\sup_{N\in\mathbb{N}}\left(R_{\alpha,p;\rho}^{\omega}\left(\left\{\sum_{k=1}^{N}|f_{k}|>t,\sup_{k}|f_{k}|\leq t\right\}\right)+R_{\alpha,p;\rho}^{\omega}\left(\left\{\max_{1\leq k\leq N}|f_{k}|>t\right\}\right)\right)
\end{align*}
Let 
\begin{align*}
I=\sup_{t>0}t^{q}\sup_{N\in\mathbb{N}}R_{\alpha,p;\rho}^{\omega}\left(\left\{\sum_{k=1}^{N}|f_{k}|>t,\sup_{k}|f_{k}|\leq t\right\}\right).
\end{align*}
As an instantiation of Lemma \ref{realization}, it holds that
\begin{align}\label{2.9}
\left\|\max_{1\leq k\leq N}|f_{k}|\right\|_{L^{q,\infty}(R_{\alpha,p;\rho}^{\omega})}^{q}\leq\sum_{k=1}^{N}\|f_{k}\|_{L^{q,\infty}(R_{\alpha,p;\rho}^{\omega})}^{q},\quad N\in\mathbb{N}.
\end{align}
Then (\ref{2.9}) gives
\begin{align}\label{2.10}
\left\|\sum_{k=1}^{\infty}|f_{k}|\right\|_{L^{q,\infty}(R_{\alpha,p;\rho}^{\omega})}^{q}\leq I+\sum_{k=1}^{\infty}\|f_{k}\|_{L^{q,\infty}(R_{\alpha,p;\rho}^{\omega})}^{q}.
\end{align}
Now we estimate $I$. Note that for any $t>0$, we have 
\begin{align}
&t^{q}\sup_{N\in\mathbb{N}}R_{\alpha,p;\rho}^{\omega}\left(\left\{\sum_{k=1}^{N}|f_{k}|>t,\sup_{k}|f_{k}|\leq t\right\}\right)\notag\\
&\leq t^{q-1}\int_{0}^{t}\sup_{N\in\mathbb{N}}R_{\alpha,p;\rho}^{\omega}\left(\left\{\sum_{k=1}^{N}|f_{k}|>\lambda\right\}\cap\left\{\sup_{k}|f_{k}|\leq t\right\}\right)d\lambda\notag\\
&\leq t^{q-1}\sup_{N\in\mathbb{N}}\int_{0}^{\infty}R_{\alpha,p;\rho}^{\omega}\left(\left\{x\in \left\{\sup_{k}|f_{k}|\leq t\right\}:\sum_{k=1}^{N}|f_{k}(x)|>\lambda\right\}\right)d\lambda\notag\\
&\leq Ct^{q-1}\sum_{k=1}^{\infty}\int_{0}^{\infty}R_{\alpha,p;\rho}^{\omega}\left(\left\{x\in\left\{\sup_{k}|f_{k}|\leq t\right\}:|f_{k}(x)|>\lambda\right\}\right)d\lambda,\label{last estimate}
\end{align}
where $C=C(n,\alpha,p,[\omega']_{A_{\infty;2\rho}^{\rm loc}})$ and we have used Theorem \ref{sub} for the normability of $L^{1}(R_{\alpha,p;\rho}^{\omega})$ in (\ref{last estimate}). As a consequence,
\begin{align}
&t^{q}R_{\alpha,p;\rho}^{\omega}\left(\left\{\sum_{k}|f_{k}|>t,\sup_{k}|f_{k}|\leq t\right\}\right)\notag\\
&\leq C(n,\alpha,p,[\omega']_{A_{\infty;2\rho}^{\rm loc}})t^{q-1}\sum_{k=1}^{\infty}\int_{0}^{\infty}R_{\alpha,p;\rho}^{\omega}\left(\left\{x\in\{|f_{k}|\leq t\}:|f_{k}(x)|> \lambda\right\}\right)d\lambda\notag\\
&=C(n,\alpha,p,[\omega']_{A_{\infty;2\rho}^{\rm loc}}) t^{q-1}\sum_{k=1}^{\infty}\int_{0}^{t}R_{\alpha,p;\rho}^{\omega}(\{x\in\mathbb{R}^{n}:|f_{k}(x)|>\lambda\})d\lambda\notag\\
&\leq C(n,\alpha,p,[\omega']_{A_{\infty;2\rho}^{\rm loc}}) \sum_{k=1}^{\infty}\int_{0}^{t}R_{\alpha,p;\rho}^{\omega}(\{x\in\mathbb{R}^{n}:|f_{k}(x)|>\lambda\})\frac{1}{\lambda^{1-q}}d\lambda\notag\\
&\leq C(n,\alpha,p,[\omega']_{A_{\infty;2\rho}^{\rm loc}},q)\sum_{k=1}^{\infty}\|f_{k}\|_{L^{q,\infty}(R_{\alpha,p;\rho}^{\omega})}^{q}.\label{2.11}
\end{align}
The proof is complete by combining (\ref{2.10}) and (\ref{2.11}).
\end{proof}

\subsection{Boundedness of Local Maximal Function with Choquet Integrals}
\enskip

Since the Bessel capacities ${\rm Cap}_{\alpha,p}(\cdot)$ and $\mathcal{R}_{\alpha,p;1}(\cdot)=\mathcal{R}_{\alpha,p;1}^{1}(\cdot)$ are equivalent, the main results in \cite{OP2} can be rephrased as 
\begin{align}\label{math weak}
\left\|{\bf M}_{1}^{\rm loc}f\right\|_{L^{q,\infty}(\mathcal{R}_{\alpha,p;1})}\leq C(n,\alpha,p)\|f\|_{L^{q}(\mathcal{R}_{\alpha,p;1})},\quad q=\frac{n-\alpha p}{n},
\end{align}
and
\begin{align}\label{math strong}
\left\|{\bf M}_{1}^{\rm loc}f\right\|_{L^{q}(\mathcal{R}_{\alpha,p;1})}\leq C(n,\alpha,p,q)\|f\|_{L^{q}(\mathcal{R}_{\alpha,p;1})},\quad q>\frac{n-\alpha p}{n}.
\end{align}
This subsection is devoted to the weighted analogue of (\ref{math weak}) and (\ref{math strong}).

For any $0<\rho<\infty$ and positive measure $\mu$ on $\mathbb{R}^{n}$, we define
\begin{align*}
{\bf M}_{\rho}^{\rm loc}\mu(x)=\sup_{\substack{x\in Q\\\ell(Q)\leq\rho}}\frac{\mu(Q)}{|Q|},\quad x\in\mathbb{R}^{n}.
\end{align*}
Hence (\ref{uncenter}) is an instantiation of the above with $d\mu=|f|dx$.

\begin{lemma}\label{3.4}
Let $n\in\mathbb{N}$, $0<\alpha<n$, $1<p<\frac{n}{\alpha}$, $0<\rho<\infty$, and $\omega\in A_{1}^{\rm loc}$. Suppose that $\mu$ is a compactly supported positive measure on $\mathbb{R}^{n}$. For any $x\in\mathbb{R}^{n}$, it holds that
\begin{align*}
&\mathcal{W}_{\omega;\rho}^{\mu}(x)\\
&\leq C(n,\alpha,p,[\omega]_{A_{1;2\rho}^{\rm loc}})\mu(\mathbb{R}^{n})^{\frac{\alpha p}{n(p-1)}}\left(\frac{1}{|Q_{\rho}(x)|}\int_{Q_{\rho}(x)}\omega(y)dy\right)^{-\frac{1}{p-1}}{\bf{M}}_{2\rho}^{\rm loc}\mu (x)^{\frac{n-\alpha p}{n(p-1)}}.
\end{align*}
\end{lemma}

\begin{proof}
Let $0<\delta\leq\rho$ be determined later. We write
\begin{align*}
\mathcal{W}_{\omega;\rho}^{\mu}(x)&=\int_{0}^{\delta}\left(\frac{t^{\alpha p}\mu(Q_{t}(x))}{\omega(Q_{t}(x))}\right)^{\frac{1}{p-1}}\frac{dt}{t}+\int_{\delta}^{\rho}\left(\frac{t^{\alpha p}\mu(Q_{t}(x))}{\omega(Q_{t}(x))}\right)^{\frac{1}{p-1}}\frac{dt}{t}=I_{1}+I_{2}.
\end{align*}
Using Proposition \ref{local strong}, we have 
\begin{align*}
\frac{1}{|Q_{\rho}(x)|}\int_{Q_{\rho}(x)}\omega(y)dy\leq[\omega]_{A_{1;2\rho}^{\rm loc}}\frac{1}{|Q_{t}(x)|}\int_{Q_{t}(x)}\omega(y)dy,\quad 0<t<\delta.
\end{align*}
Then
\begin{align}
I_{1}&=\int_{0}^{\delta}t^{\frac{\alpha p}{p-1}}\left(\frac{\mu(Q_{t}(x))}{|Q_{t}(x)|}\right)^{\frac{1}{p-1}}\left(\frac{1}{|Q_{t}(x)|}\int_{Q_{t}(x)}\omega(y)dy\right)^{-\frac{1}{p-1}}\frac{dt}{t}\notag\\
&\leq[\omega]_{A_{1;2\rho}^{\rm loc}}^{\frac{1}{p-1}}{\bf{M}}_{2\rho}^{\rm loc}\mu(x)^{\frac{1}{p-1}}\left(\frac{1}{|Q_{\rho}(x)|}\int_{Q_{\rho}(x)}\omega(y)dy\right)^{-\frac{1}{p-1}}\int_{0}^{\delta}t^{\frac{\alpha p}{p-1}}\frac{dt}{t}\notag\\
&=C(\alpha,p)[\omega]_{A_{1;2\rho}^{\rm loc}}^{\frac{1}{p-1}}{\bf{M}}_{2\rho}^{\rm loc}\mu(x)^{\frac{1}{p-1}}\left(\frac{1}{|Q_{\rho}(x)|}\int_{Q_{\rho}(x)}\omega(y)dy\right)^{-\frac{1}{p-1}}\delta^{\frac{\alpha p}{p-1}}.\label{3.10}
\end{align}
On the other hand, 
\begin{align}
I_{2}&\leq C(n)\mu(\mathbb{R}^{n})^{\frac{1}{p-1}}\left(\frac{1}{|Q_{\rho}(x)|}\int_{Q_{\rho}(x)}\omega(y)dy\right)^{-\frac{1}{p-1}}\int_{\delta}^{\infty}\frac{1}{t^{\frac{n-\alpha p}{p-1}}}\frac{dt}{t}\notag\\
&=C(n,\alpha,p)\mu(\mathbb{R}^{n})^{\frac{1}{p-1}}\left(\frac{1}{|Q_{\rho}(x)|}\int_{Q_{\rho}(x)}\omega(y)dy\right)^{-\frac{1}{p-1}}\frac{1}{\delta^{\frac{n-\alpha p}{p-1}}}.\label{3.11}
\end{align}
Note that we may assume that $\mu$ is supported in $Q_{\rho}(x)$, this gives
\begin{align*}
{\bf{M}}_{2\rho}^{\rm loc}\mu(x)\geq\frac{\mu(Q_{\rho}(x))}{\rho^{n}|Q_{1}(0)|}=\frac{\mu(\mathbb{R}^{n})}{\rho^{n}|Q_{1}(0)|}.
\end{align*}
If we choose 
\begin{align*}
\delta=\left(\frac{\mu(\mathbb{R}^{n})}{|Q_{1}(0)|{\bf{M}}_{2\rho}^{\rm loc}\mu(x)}\right)^{\frac{1}{n}},
\end{align*}
then $0<\delta\leq\rho$ and the proof is complete by routine simplification of (\ref{3.10}) and (\ref{3.11}).
\end{proof}

\begin{lemma}\label{embed}
Let $0<p<\infty$ and $(X,\mu)$ be a positive measure space. Suppose that $E$ be a $\mu$-measurable subset of $X$ with $\mu(E)<\infty$. For $0<q<p$, it holds that 
\begin{align*}
\int_{E}|f|^{q}d\mu\leq\frac{p}{p-q}\mu(E)^{1-\frac{q}{p}}\|f\|_{L^{p,\infty}(X,\mu)}^{q}.
\end{align*}
\end{lemma}

\begin{proof}
For any $0<t<\infty$, we have 
\begin{align*}
\mu(E\cap\{x\in X:|f(x)|>t\})\leq\min\left(\mu(E),t^{-p}\|f\|_{L^{p,\infty}(X,\mu)}^{p}\right)
\end{align*}
Then
\begin{align*}
\int_{E}|f|^{q}d\mu&=q\int_{0}^{\infty}t^{q-1}\mu(E\cap\{x\in X:|f(x)|>t\})dt\\
&\leq q\int_{0}^{\mu(E)^{-\frac{1}{p}}\|f\|_{L^{p,\infty}(X,\mu)}}t^{q-1}\mu(E)dt\\
&\qquad+q\int_{\mu(E)^{-\frac{1}{p}}\|f\|_{L^{p,\infty}(X,\mu)}}^{\infty}t^{q-1}t^{-p}\|f\|_{L^{p,\infty}(X,\mu)}^{p}dt\\
&=\mu(E)\mu(E)^{-\frac{q}{p}}\|f\|_{L^{p,\infty}(X,\mu)}+\frac{q}{p-q}\mu(E)^{\frac{p-q}{p}}\|f\|_{L^{p,\infty}(X,\mu)}^{p-(p-q)}\\
&=\frac{p}{p-q}\mu(E)^{1-\frac{q}{p}}\|f\|_{L^{p,\infty}(X,\mu)}^{q},
\end{align*}
which yields the estimate.
\end{proof}

The next proposition shows that the Wolff potential $(\mathcal{W}_{\omega,\rho}^{\mu})^{\delta}$ with admissible exponents $\delta$ satisfies a condition very closed to $A_{1}^{\rm loc}$ but not exactly.
\begin{proposition}\label{almost a1}
Let $n\in\mathbb{N}$, $0<\alpha<n$, $1<p<\frac{n}{\alpha}$, $0<\rho<\infty$, and $\omega\in A_{1}^{\rm loc}$. Suppose that $\mu$ is a positive measure on $\mathbb{R}^{n}$. For any  $0<\delta<\frac{n(p-1)}{n-\alpha p}$ and $x\in\mathbb{R}^{n}$, it holds that 
\begin{align*}
{\bf M}_{2\rho}^{\rm loc}\left((\mathcal{W}_{\omega;\rho}^{\mu})^{\delta}\right)(x)\leq C(n,\alpha,p,[\omega]_{A_{1;6\rho}^{\rm loc}},\delta)\mathcal{W}_{\omega;3\rho}^{\mu}(x)^{\delta}.
\end{align*}
\end{proposition}

\begin{proof}
Let $Q$ be an arbitrary cube that containing $x$ with $\ell(Q)\leq 2\rho$. Denote by $0<r=\frac{\ell(Q)}{2}\leq\rho$. We need to estimate that 
\begin{align*}
I=\frac{1}{|Q|}\int_{Q}\left(\int_{0}^{\rho}\left(\frac{t^{\alpha p}\mu(Q_{t}(y))}{\omega(Q_{t}(y))}\right)^{\frac{1}{p-1}}\frac{dt}{t}\right)^{\delta}dy\leq C(I_{1}+I_{2}),
\end{align*}
where 
\begin{align*}
I_{1}&=\frac{1}{|Q|}\int_{Q}\left(\int_{0}^{r}\left(\frac{t^{\alpha p}\mu(Q_{t}(y))}{\omega(Q_{t}(y))}\right)^{\frac{1}{p-1}}\frac{dt}{t}\right)^{\delta}dy,\\
I_{2}&=\frac{1}{|Q|}\int_{Q}\left(\int_{r}^{\rho}\left(\frac{t^{\alpha p}\mu(Q_{t}(y))}{\omega(Q_{t}(y))}\right)^{\frac{1}{p-1}}\frac{dt}{t}\right)^{\delta}dy.
\end{align*}
For the integral $I_{1}$, we may assume that $\mu$ is supported in $Q_{2r}(x)$. By Lemma \ref{3.4}, it suffices to estimate 
\begin{align*}
\int_{Q}{\bf{M}}_{2r}^{\rm loc}\mu(y)^{\frac{\delta(n-\alpha p)}{n(p-1)}}dy.
\end{align*}
Let $0<q=\frac{\delta(n-\alpha p)}{n(p-1)}<1$. We appeal to the estimate in Lemma \ref{embed} that 
\begin{align*}
\int_{E}|F(y)|^{q}dy\leq C(n,q)|E|^{1-q}\|F\|_{L^{1,\infty}(E)}^{q}
\end{align*}
for any measurable set $E\subseteq\mathbb{R}^{n}$ with $|E|<\infty$ and $F\in L^{1,\infty}(E)$. As a consequence, 
\begin{align}
&I_{1}\notag\\
&=\frac{1}{|Q|}\int_{Q}\mathcal{W}_{\omega;r}^{\mu}(y)^{\delta}dy\notag\\
&\leq C\mu(\mathbb{R}^{n})^{\frac{\delta\alpha p}{n(p-1)}}\left(\frac{1}{|Q_{r}(x)|}\int_{Q_{r}(x)}\omega(y)dy\right)^{-\frac{\delta}{p-1}}\frac{1}{|Q|}\int_{Q}{\bf{M}}_{2r}^{\rm loc}\mu(y)^{\frac{\delta(n-\alpha p)}{n(p-1)}}dy\notag\\
&\leq C\mu(\mathbb{R}^{n})^{\frac{\delta\alpha p}{n(p-1)}}\left(\frac{1}{|Q_{r}(x)|}\int_{Q_{r}(x)}\omega(y)dy\right)^{-\frac{\delta}{p-1}}\frac{1}{|Q|}|Q|^{1-q}\left\|{\bf{M}}_{2r}^{\rm loc}\mu\right\|_{L^{1,\infty}(\mathbb{R}^{n})}^{q}\notag\\
&\leq C r^{-\frac{\delta(n-\alpha p)}{p-1}}\left(\frac{1}{|Q_{r}(x)|}\int_{Q_{r}(x)}\omega(y)dy\right)^{-\frac{\delta}{p-1}}\mu(\mathbb{R}^{n})^{\frac{\delta\alpha p}{n(p-1)}+q}\label{weak type new}\\
&=C\left(\frac{r^{\alpha p}\mu(Q_{2r}(x))}{\omega(Q_{r}(x))}\right)^{\frac{\delta}{p-1}}\notag\\
&\leq C\left(\int_{2r}^{3r}\left(\frac{t^{\alpha p}\mu(Q_{t}(x))}{\omega(Q_{t}(x))}\right)^{\frac{1}{p-1}}\frac{dt}{t}\right)^{\delta}\label{use A1}\\
&\leq CW_{\omega;3\rho}^{\mu}(x)^{\delta},\label{3.14}
\end{align}
where $C=C(n,\alpha,p,[\omega]_{A_{1;6\rho}^{\rm loc}})$ and the details of the above estimates are given as follows. We have used the weak type $(1,1)$ boundedness of ${\bf M}$ in (\ref{weak type new}) that 
\begin{align*}
{\bf M}\mu(z)=\sup_{r>0}\frac{\mu(Q_{r}(z))}{|Q_{r}(z)|},\quad z\in\mathbb{R}^{n}
\end{align*}
together with ${\bf M}_{2r}^{\rm loc}\mu(z)\leq{\bf M}\mu(z)$ for all $r>0$ and $z\in\mathbb{R}^{n}$. Further, we have applied Proposition \ref{local strong} to $\omega\in A_{1;6\rho}^{\rm loc}$ in (\ref{use A1}). 

For the integral $I_{2}$, observe that $Q_{t}(y)\subseteq Q_{2t}(x)$ for $y\in Q_{r}(x)$ and $t>r$. Using Proposition \ref{local strong} for $\omega\in A_{1;6\rho}^{\rm loc}$ again, one obtains
\begin{align}
I_{2}&\leq C(n,p,[\omega]_{A_{1;6\rho}^{\rm loc}},\delta)\left(\int_{r}^{\rho}\left(\frac{t^{\alpha p}\mu(Q_{2t}(x))}{\omega(Q_{2t}(x))}\right)^{\frac{1}{p-1}}\frac{dt}{t}\right)^{\delta}\notag\\
&\leq C(n,\alpha,p,[\omega]_{A_{1;6\rho}^{\rm loc}},\delta)\mathcal{W}_{\omega;2\rho}^{\mu}(x)^{\delta}.\label{3.15}
\end{align}
The proof is complete by combining (\ref{3.14}) and (\ref{3.15}).
\end{proof}

\begin{lemma}\label{technical lemma}
Let $n\in\mathbb{N}$, $0<\alpha<n$, $1<p<\frac{n}{\alpha}$, $0<\rho<\infty$, and $\omega\in A_{1}^{\rm loc}$. Suppose that $\mu$ is a compactly supported positive measure on $\mathbb{R}^{n}$. For any $0<r\leq\rho$ and $x_{0}\in\mathbb{R}^{n}$, it holds that
\begin{align*}
\left\|\mathcal{W}_{\omega;2\rho}^{\mu}\right\|_{L^{\frac{n(p-1)}{n-\alpha p},\infty}(Q_{r}(x_{0}))}\leq C(n,\alpha,p,[\omega]_{A_{1;8\rho}^{\rm loc}})|Q_{r}(x_{0})|^{\frac{n-\alpha p}{n(p-1)}}\mathcal{W}_{\omega;4\rho}^{\mu}(x_{0}).
\end{align*}
\end{lemma}

\begin{proof}
Let 
\begin{align*}
P_{1}=\left\|\int_{0}^{r}\left(\frac{\mu(Q_{t}(\cdot))}{t^{n-\alpha p}}\right)^{\frac{1}{p-1}}\left(\frac{1}{|Q_{t}(\cdot)|}\int_{Q_{t}(\cdot)}\omega(y)dy\right)^{-\frac{1}{p-1}}\frac{dt}{t}\right\|_{L^{\frac{n(p-1)}{n-\alpha p},\infty}(Q_{r}(x_{0}))}
\end{align*}
and 
\begin{align*}
P_{2}=\left\|\int_{r}^{4\rho}\left(\frac{\mu(Q_{t}(\cdot))}{t^{n-\alpha p}}\right)^{\frac{1}{p-1}}\left(\frac{1}{|Q_{t}(\cdot)|}\int_{Q_{t}(\cdot)}\omega(y)dy\right)^{-\frac{1}{p-1}}\frac{dt}{t}\right\|_{L^{\frac{n(p-1)}{n-\alpha p},\infty}(Q_{r}(x_{0}))}.
\end{align*}
It suffices to estimate the terms $P_{1}$ and $P_{2}$. To bound $P_{1}$, we may assume that $\mu$ is supported in $Q_{2r}(x_{0})$ since $Q_{t}(x)\subseteq Q_{2r}(x_{0})$ for $x\in Q_{r}(x_{0})$ and $0<t<r$. Using Lemma \ref{3.4}, the weak type $(1,1)$ boundedness of ${\bf{M}}$, and Proposition \ref{local strong} that 
\begin{align*}
\frac{1}{|Q_{2r}(x_{0})|}\int_{Q_{2r}(x_{0})}\omega(y)dy\leq[\omega]_{A_{1;4\rho}^{\rm loc}}\frac{1}{|Q_{t}(x)|}\int_{Q_{r}(x)}\omega(y)dy,~x\in Q_{r}(x_{0}),~0<t<r,
\end{align*}
one has 
\begin{align*}
P_{1}&\leq C\left(\frac{\omega(Q_{2r}(x_{0}))}{|Q_{2r}(x_{0})|}\right)^{-\frac{1}{p-1}}\mu(Q_{2r}(x_{0}))^{\frac{\alpha p}{n(p-1)}}\left\|{\bf M}(\mu)^{\frac{n-\alpha p}{n(p-1)}}\right\|_{L^{\frac{n(p-1)}{n-\alpha p},\infty}(B_{r}(x_{0}))}\\
&=C\left(\frac{\omega(Q_{2r}(x_{0}))}{|Q_{2r}(x_{0})|}\right)^{-\frac{1}{p-1}}\mu(Q_{2r}(x_{0}))^{\frac{\alpha p}{n(p-1)}}\left\|{\bf M}(\mu)\right\|_{L^{1,\infty}(Q_{r}(x_{0}))}^{\frac{n-\alpha p}{n(p-1)}}\\
&\leq C\mu(Q_{2r}(x_{0}))^{\frac{\alpha p}{n(p-1)}}\left(\frac{\omega(Q_{2r}(x_{0}))}{|Q_{2r}(x_{0})|}\right)^{-\frac{1}{p-1}}\mu(Q_{2r}(x_{0}))^{\frac{n-\alpha p}{n(p-1)}}\\
&=C|Q_{r}(x_{0})|^{\frac{n-\alpha p}{n(p-1)}}\left(\frac{\omega(Q_{2r}(x_{0}))}{|Q_{2r}(x_{0})|}\right)^{-\frac{1}{p-1}}\left(\frac{\mu(Q_{2r}(x_{0}))}{r^{n-\alpha p}}\right)^{\frac{1}{p-1}}\\
&\leq C|Q_{r}(x_{0})|^{\frac{n-\alpha p}{n(p-1)}}\int_{2r}^{4r}\left(\frac{\omega(Q_{t}(x_{0}))}{|Q_{t}(x_{0})|}\right)^{-\frac{1}{p-1}}\left(\frac{\mu(Q_{t}(x_{0}))}{t^{n-\alpha p}}\right)^{\frac{1}{p-1}}\frac{dt}{t}\\
&\leq C|Q_{r}(x_{0})|^{\frac{n-\alpha p}{n(p-1)}}\mathcal{W}_{\omega;4\rho}^{\mu}(x_{0}),
\end{align*}
where $C=C(n,\alpha,p,[\omega]_{A_{1;8\rho}^{\rm loc}})$. Now we bound for $P_{2}$. Observe that if $x\in Q_{r}(x_{0})$ and $r\leq t\leq 2\rho$, then $Q_{t}(x)\subseteq Q_{2t}(x_{0})$ and again Proposition \ref{local strong} yields
\begin{align*}
\frac{1}{|Q_{2t}(x_{0})|}\int_{Q_{2t}(x_{0})}\omega(y)dy\leq[\omega]_{A_{1;8\rho}^{\rm loc}}\frac{1}{|Q_{t}(x)|}\int_{Q_{t}(x)}\omega(y)dy,
\end{align*}
which implies that
\begin{align*}
&P_{2}\\
&\leq\left\|\int_{r}^{2\rho}\left(\frac{\mu(Q_{2t}(x_{0}))}{t^{n-\alpha p}}\right)^{\frac{1}{p-1}}\left(\frac{1}{|Q_{2t}(x_{0})|}\int_{Q_{2t}(x_{0})}\omega(y)dy\right)^{-\frac{1}{p-1}}\frac{dt}{t}\right\|_{L^{\frac{n(p-1)}{n-\alpha p},\infty}(Q_{r}(x_{0}))}\\
&\leq C(n,\alpha,p,[\omega]_{A_{1;8\rho}^{\rm loc}})|Q_{r}(x_{0})|^{\frac{n-\alpha p}{n(p-1)}}\mathcal{W}_{\omega;8\rho}^{\mu}(x_{0}),
\end{align*}
and the lemma now follows by combining the estimates $P_{1}$ and $P_{2}$.
\end{proof}

\begin{proposition}\label{2.2}
Let $n\in\mathbb{N}$, $0<\alpha<n$, $1<p<\frac{n}{\alpha}$, $0<\rho<\infty$, and $\omega\in A_{1}^{\rm loc}$. Suppose that $E$ is a measurable subset of $\mathbb{R}^{n}$. Then 
\begin{align*}
\left\|{\bf{M}}_{\rho}^{\rm loc}\chi_{E}\right\|_{L^{\frac{n-\alpha p}{n},\infty}(\mathcal{R}_{\alpha,p;\rho}^{\omega})}\leq C(n,\alpha,p,[\omega]_{A_{1;160\rho}^{\rm loc}})\mathcal{R}_{\alpha,p;1}^{\omega}(E)^{\frac{n}{n-\alpha p}}.
\end{align*}
\end{proposition}

\begin{proof}
By choosing the nonlinear potential $\mathcal{V}_{\omega;\rho}^{\mu}$ as in Proposition \ref{nonlinear use in CSI}, we have 
\begin{align*}
\mu(\mathbb{R}^{n})&=\mu\left(\overline{E}\right)=\mathcal{R}_{\alpha,p;\rho}^{\omega}(E),\\
\mathcal{V}_{\omega;\rho}^{\mu}(x)&\geq 1\quad\mathcal{R}_{\alpha,p;\rho}^{\omega}(\cdot)\text{-quasi-everywhere on}~E.
\end{align*}
Since $\omega\in A_{1}^{\rm loc}\subseteq A_{p}^{\rm loc}$, Theorem \ref{absolute continuity} entails
\begin{align*}
\chi_{E}(x)\leq\mathcal{V}_{\omega;\rho}^{\mu}(x)^{\frac{n(p-1)}{n-\alpha p}}\quad\text{a.e..}
\end{align*}
Let $x_{0}\in\mathbb{R}^{n}$ and $Q$ be a cube that containing $x_{0}$ with $\ell(Q)\leq\rho$. Denote by $r=\ell(Q)$. Then $Q\subseteq Q_{r}(x_{0})$ and $0<r\leq\rho$. We have by (\ref{pointwise}) and Lemma \ref{technical lemma} that
\begin{align*}
\frac{1}{|Q|}\int_{Q}\chi_{E}(y)dy&\leq 2^{n}\frac{1}{|Q_{r}(x_{0})|}\int_{Q_{r}(x_{0})}\chi_{E}(y)dy\\
&=2^{n}\frac{1}{|Q_{r}(x_{0})|}\|\chi_{E}\|_{L^{1,\infty}(Q_{r}(x_{0}))}\\
&\leq 2^{n}\frac{1}{|Q_{r}(x_{0})|}\left\|\mathcal{V}_{\omega;\rho}^{\mu}\right\|_{L^{\frac{n(p-1)}{n-\alpha p},\infty}(B_{r}(x_{0}))}^{\frac{n(p-1)}{n-\alpha p}}\\
&\leq C(n,\alpha,p,[\omega]_{A_{1;8\rho}^{\rm loc}})\mathcal{W}_{\omega;4\rho}^{\mu}(x_{0})^{\frac{n(p-1)}{n-\alpha p}},
\end{align*}
which yields
\begin{align*}
{\bf M}_{\rho}^{\rm loc}\chi_{E}(x_{0})\leq C(n,\alpha,p,[\omega]_{A_{1;8\rho}^{\rm loc}})\mathcal{W}_{\omega;4\rho}^{\mu}(x_{0})^{\frac{n(p-1)}{n-\alpha p}}.
\end{align*}
As a consequence, for any $t>0$, we have by Corollary \ref{use weak wolff} that 
\begin{align*}
&\mathcal{R}_{\alpha,p;8\rho}^{\omega}\left(\left\{x\in\mathbb{R}^{n}:{\bf M}^{\rm loc}\chi_{E}(x)>t\right\}\right)\\
&\leq \mathcal{R}_{\alpha,p;8\rho}^{\omega}\left(\left\{x\in\mathbb{R}^{n}:\mathcal{W}_{\omega;4\rho}^{\mu}(x)>\left(\frac{t}{C(n,\alpha,p,[\omega]_{A_{1;8\rho}^{\rm loc}})}\right)^{\frac{n-\alpha p}{n(p-1)}}\right\}\right)\\
&\leq C'(n,\alpha,p,[\omega]_{A_{1;8\rho}^{\rm loc}})\frac{\mu(\mathbb{R}^{n})}{t^{\frac{n-\alpha p}{n}}}\\
&=C'(n,\alpha,p,[\omega]_{A_{1;8\rho}^{\rm loc}})\frac{\mathcal{R}_{\alpha,p;\rho}^{\omega}(E)}{t^{\frac{n-\alpha p}{n}}}
\end{align*}
and the proof is complete by appealing to Proposition \ref{338} and Theorem \ref{R cal} that $\mathcal{R}_{\alpha,p;\rho}^{\omega}(\cdot)\leq C(n,\alpha,p,[\omega]_{A_{1;160\rho}^{\rm loc}})\mathcal{R}_{\alpha,p;8\rho}^{\omega}(\cdot)$.
\end{proof}

The weighted analogue of (\ref{math weak}) is given as follows.
\begin{theorem}\label{main1}
Let $n\in\mathbb{N}$, $0<\alpha<n$, $1<p<\frac{n}{\alpha}$, $0<\rho<\infty$, and $\omega\in A_{1}^{\rm loc}$. For any Lebesgue measurable function $f:\mathbb{R}^{n}\rightarrow[-\infty,\infty]$, it holds that
\begin{align*}
\left\|{\bf M}_{\rho}^{\rm loc}f\right\|_{L^{q,\infty}(\mathcal{R}_{\alpha,p;\rho}^{\omega})}\leq C(n,\alpha,p,[\omega]_{A_{1;160\rho}^{\rm loc}})\|f\|_{L^{q}(\mathcal{R}_{\alpha,p;\rho}^{\omega})}.
\end{align*}
\end{theorem}

\begin{proof}
First of all, Theorem \ref{R cal} and Proposition \ref{2.3} entail
\begin{align}\label{2.12}
\left\|\sum_{k\in\mathbb{Z}}|f_{k}|\right\|_{L^{q,\infty}(\mathcal{R}_{\alpha,p;\rho}^{\omega})}\leq C(n,\alpha,p,[\omega]_{A_{1;20\rho}^{\rm loc}})\left(\sum_{k\in\mathbb{Z}}\|f_{k}\|_{L^{q,\infty}(\mathcal{R}_{\alpha,p;\rho}^{\omega})}^{q}\right)^{\frac{1}{q}}
\end{align}
with $0<q=\frac{n-\alpha p}{n}<1$. Suppose that $f\in L^{q}(\mathcal{R}_{\alpha,p;\rho}^{\omega})$. Then
\begin{align*}
\mathcal{R}_{\alpha,p;\rho}^{\omega}(\{x\in\mathbb{R}^{n}:|f(x)|=\infty\})\leq\frac{1}{N^{p}}\int_{\mathbb{R}^{n}}|f|^{p}d\mathcal{R}_{\alpha,p;\rho}^{\omega},\quad N=1,2,\ldots
\end{align*}
gives $\mathcal{R}_{\alpha,p;\rho}^{\omega}(\{x\in\mathbb{R}^{n}:|f(x)|=\infty\})=0$. As a result, we can write
\begin{align*}
f=\sum_{k\in\mathbb{Z}}f\chi_{\{2^{k-1}<|f|\leq 2^{k}\}}\quad\mathcal{R}_{\alpha,p;\rho}^{\omega}(\cdot)\text{-quasi-everywhere},
\end{align*}
and hence 
\begin{align*}
|f|\leq\sum_{k\in\mathbb{Z}}2^{k}\chi_{\{2^{k-1}<|f|\leq 2^{k}\}}\quad\mathcal{R}_{\alpha,p;\rho}^{\omega}(\cdot)\text{-quasi-everywhere}.
\end{align*}
Then Theorem \ref{absolute continuity} entails
\begin{align*}
{\bf M}_{\rho}^{\rm loc}f(x)\leq\sum_{k\in\mathbb{Z}}2^{k}{\bf M}_{\rho}^{\rm loc}\chi_{\{2^{k-1}<|f|\leq 2^{k}\}}(x),\quad x\in\mathbb{R}^{n}.
\end{align*}
We obtain by (\ref{2.12}) and Proposition \ref{2.2} that
\begin{align*}
\left\|{\bf M}_{\rho}^{\rm loc}f\right\|_{L^{q,\infty}(\mathcal{R}_{\alpha,p;\rho}^{\omega})}&\leq C\left(\sum_{k\in\mathbb{Z}}2^{kq}\left\|{\bf{M}}^{\rm loc}\chi_{\{2^{k-1}<|f|\leq 2^{k}\}}\right\|_{L^{q}(R_{\alpha,p;\rho}^{\omega})}^{q}\right)^{\frac{1}{q}}\\
&\leq C\left(\sum_{k\in\mathbb{Z}}2^{kq}\mathcal{R}_{\alpha,p;\rho}^{\omega}\left(\left\{x\in\mathbb{R}^{n}:2^{k-1}<|f(x)|\leq 2^{k}\right\}\right)\right)^{\frac{1}{q}},
\end{align*}
which yields
\begin{align*}
&\left\|{\bf M}_{\rho}^{\rm loc}f\right\|_{L^{q,\infty}(\mathcal{R}_{\alpha,p;\rho}^{\omega})}\\
&\leq C\left(\sum_{k\in\mathbb{Z}}\int_{2^{k-2}}^{2^{k-1}}t^{q}\mathcal{R}_{\alpha,p;\rho}^{\omega}(\{x\in\mathbb{R}^{n}:2^{k-1}<|f(x)|\leq 2^{k}\})\frac{dt}{t}\right)^{\frac{1}{p}}\\
&=C\left(\int_{0}^{\infty}t^{q}\mathcal{R}_{\alpha,p;\rho}^{\omega}(\{x\in\mathbb{R}^{n}:2^{k-1}<|f(x)|\leq 2^{k}\})\frac{dt}{t}\right)^{\frac{1}{p}}\\
&=C\|f\|_{L^{q}(R_{\alpha,p;\rho}^{\omega})},
\end{align*}
where $C=C(n,\alpha,p,[\omega]_{A_{1;160\rho}^{\rm loc}})$, and the proof is now complete.
\end{proof}

\begin{proposition}\label{final strong}
Let $n\in\mathbb{N}$, $0<\alpha<n$, $1<p<\frac{n}{\alpha}$, $0<\rho<\infty$, and $\omega\in A_{1}^{\rm loc}$. Suppose that $E$ is a measurable subset of $\mathbb{R}^{n}$. For any $q>\frac{n-\alpha p}{n}$, it holds that 
\begin{align*}
\left\|{\bf M}_{\rho}^{\rm loc}\chi_{E}\right\|_{L^{q}(\mathcal{R}_{\alpha,p;\rho}^{\omega})}\leq C(n,\alpha,p,[\omega]_{A_{1;320\rho}^{\rm loc}},q)\mathcal{R}_{\alpha,p;\rho}^{\omega}(E)^{\frac{1}{q}}.
\end{align*}
\end{proposition}

\begin{proof}
For any $q>\frac{n-\alpha p}{n}$, choose an $\varepsilon=\varepsilon_{q}>0$ such that 
\begin{align*}
q>\frac{n-\alpha p+\varepsilon}{n}.
\end{align*}
Denote by $\delta=\frac{n(p-1)}{n-\alpha p+\varepsilon}$. By choosing the nonlinear potential $\mathcal{V}_{\omega;\rho}^{\mu}$ as in Proposition \ref{nonlinear use in CSI} and using (\ref{pointwise}), one has 
\begin{align*}
\mu(\mathbb{R}^{n})&=\mu\left(\overline{E}\right)=\mathcal{R}_{\alpha,p;\rho}^{\omega}(E),\\
\mathcal{V}_{\omega;\rho}^{\mu}(x)&\geq 1\quad\mathcal{R}_{\alpha,p;\rho}^{\omega}(\cdot)\text{-quasi-everywhere on}~E.
\end{align*}
Since $\omega\in A_{1}^{\rm loc}\subseteq A_{p}^{\rm loc}$, Theorem (\ref{absolute continuity}) and estimate (\ref{pointwise}) entail
\begin{align*}
\chi_{E}(x)\leq\mathcal{V}_{\omega;\rho}^{\mu}(x)^{\delta}\leq C(n,\alpha,p,[\omega]_{A_{1;2\rho}^{\rm loc}},q)\mathcal{W}_{\omega;2\rho}^{\mu}(x)^{\delta}\quad{\rm a.e.}.
\end{align*}
Then Proposition \ref{almost a1} yields
\begin{align*}
{\bf M}_{\rho}^{\rm loc}\chi_{E}(x)\leq C{\bf M}_{\rho}^{\rm loc}\left((\mathcal{W}_{\omega;2\rho}^{\mu})^{\delta}\right)(x)\leq C(n,\alpha,p,[\omega]_{A_{1;12\rho}^{\rm loc}},q)\mathcal{W}_{\omega;6\rho}^{\mu}(x)^{\delta},\quad x\in\mathbb{R}^{n}.
\end{align*}
Denote $C(n,\alpha,p,[\omega]_{A_{1;12\rho}^{\rm loc}},q)$ by $C$ for simplicity. Since ${\bf M}_{\rho}^{\rm loc}\chi_{E}\leq 1$ and $\frac{p-1}{q\delta}<1$, one has by Corollary \ref{use weak wolff} that
\begin{align*}
&\int_{0}^{\infty}\mathcal{R}_{\alpha,p;12\rho}^{\omega}\left(\left\{x\in\mathbb{R}^{n}:{\bf M}_{\rho}^{\rm loc}\chi_{E}(x)^{q}>t\right\}\right)dt\\
&=\int_{0}^{1}R_{\alpha,p;12\rho}^{\omega}\left(\left\{x\in\mathbb{R}^{n}:{\bf M}_{\rho}^{\rm loc}\chi_{E}(x)>t^{\frac{1}{q}}\right\}\right)dt\\
&\leq\int_{0}^{1}\mathcal{R}_{\alpha,p;12\rho}^{\omega}\left(\left\{x\in\mathbb{R}^{n}:\mathcal{W}_{\omega;6\rho}^{\mu}(x)>\left(\frac{t}{C}\right)^{\frac{1}{q\delta}}\right\}\right)dt\\
&= C\int_{0}^{C^{-1}}\mathcal{R}_{\alpha,p;12\rho}^{\omega}\left(\left\{x\in\mathbb{R}^{n}:\mathcal{W}_{\omega;6\rho}^{\mu}(x)>t^{\frac{1}{q\delta}}\right\}\right)dt\\
&\leq C'(n,\alpha,p,[\omega]_{A_{1;12\rho}^{\rm loc}},q)\int_{0}^{C^{-1}}\frac{\mu(\mathbb{R}^{n})}{t^{\frac{p-1}{q\delta}}}dt\\
&=C''(n,\alpha,p,[\omega]_{A_{1;12\rho}^{\rm loc}},q)\mathcal{R}_{\alpha,p;\rho}^{\omega}(E).
\end{align*}
The proof is complete by noting that $\mathcal{R}_{\alpha,p;\rho}^{\omega}(\cdot)\leq C(n,\alpha,p,[\omega]_{A_{1;320\rho}^{\rm loc}})\mathcal{R}_{\alpha,p;12\rho}^{\omega}(\cdot)$.
\end{proof}

The next theorem addresses the weighted analogue of (\ref{math strong}).
\begin{theorem}\label{main2}
Let $n\in\mathbb{N}$, $0<\alpha<n$, $1<p<\frac{n}{\alpha}$, $0<\rho<\infty$, and $\omega\in A_{1}^{\rm loc}$. For any Lebesgue measurable function $f:\mathbb{R}^{n}\rightarrow[-\infty,\infty]$ and $q>\frac{n-\alpha p}{n}$, it holds that
\begin{align*}
\left\|{\bf M}_{\rho}^{\rm loc}f\right\|_{L^{q}(\mathcal{R}_{\alpha,p;\rho}^{\omega})}\leq C(n,\alpha,p,[\omega]_{A_{1;160\rho}^{\rm loc}})\|f\|_{L^{q}(\mathcal{R}_{\alpha,p;\rho}^{\omega})}.
\end{align*}
\end{theorem}

\begin{proof}
Note that 
\begin{align*}
{\bf M}_{\rho}^{\rm loc}f(x)^{q}\leq{\bf M}_{\rho}^{\rm loc}(f^{q})(x),\quad q>1,\quad x\in\mathbb{R}^{n}.
\end{align*}
We only need to prove the estimate for the exponent that $\frac{n-\alpha p}{n}<q\leq 1$. By using Propositions \ref{convex} and \ref{final strong}, one obtains the result by repeating the argument given in the proof of Theorem \ref{main1}.
\end{proof}

\begin{remark}
\rm In view of Theorem \ref{absolute continuity}, the functions $f$ in Theorems \ref{main1} and \ref{main2} need no to be defined everywhere in $\mathbb{R}^{n}$. In other words, the assertions hold for $\mathcal{R}_{\alpha,p;\rho}^{\omega}(\cdot)$-quasi-everywhere defined functions on $\mathbb{R}^{n}$.
\end{remark}

\subsection{Thinness of Sets and the Kellogg Property}
\enskip

Let $n\in\mathbb{N}$, $0<\alpha<n$, $1<p<\infty$, $0<\rho<\infty$, and $\omega$ be a weight. We say that a set $E\subseteq\mathbb{R}^{n}$ is $R_{\alpha,p;\rho}^{\omega}(\cdot)$-thin at a point $a\in\mathbb{R}^{n}$ if 
\begin{align}\label{361}
\int_{0}^{\rho}\left(\frac{t^{\alpha p}}{\omega(Q_{\rho}(a))}\right)^{p'-1}\frac{dt}{t}=\infty
\end{align}
and
\begin{align}\label{362}
\int_{0}^{\rho}\left(\frac{t^{\alpha p}R_{\alpha,p;\rho}^{\omega}(E\cap Q_{t}(a))}{\omega(Q_{t}(a))}\right)^{p'-1}\frac{dt}{t}<\infty.
\end{align}
If $E$ is not $R_{\alpha,p;\rho}^{\omega}(\cdot)$-thin at $a$, then $E$ is said to be $R_{\alpha,p;\rho}^{\omega}(\cdot)$-thick at $a$. The set of points where $E$ is $R_{\alpha,p;\rho}^{\omega}(\cdot)$-thin is denoted by $e_{\alpha,p;\rho}^{\omega}(E)$ and its complement with respect to $\mathbb{R}^{n}$ is the set $b_{\alpha,p;\rho}^{\omega}(E)$. We remark that if $x\notin\overline{E}$, then $E\cap Q_{t}(a)=\emptyset$ for sufficiently small $t$, and hence (\ref{362}) holds trivially.
\begin{lemma}\label{368}
Let $n\in\mathbb{N}$, $0<\alpha<n$, $1<p<\infty$, $0<\rho<\infty$, and $\omega\in A_{p}^{\rm loc}$. Suppose that $E\subseteq\mathbb{R}^{n}$ is a Borel set such that $0<R_{\alpha,p;\rho}^{\omega}(E)<\infty$. Assume that $\mu^{E}$ is the measure associated with $E$ as in Proposition $\ref{nonlinear use in CSI}$. For any open set $G\subseteq\mathbb{R}^{n}$, it holds that 
\begin{align}\label{3611}
\mu^{E}(G)\leq R_{\alpha,p;\rho}^{\omega}(G\cap E).
\end{align}
\end{lemma}

\begin{proof}
By Proposition \ref{borel}, there is an increasing sequence $\{K_{N}\}_{N=1}^{\infty}$ of compact subsets of $E$ such that $\lim\limits_{N\rightarrow\infty}R_{\alpha,p;\rho}^{\omega}(K_{N})=R_{\alpha,p;\rho}^{\omega}(E)$. Further, by the proof of \cite[Proposition 3.2.15]{TB} (the proof therein mainly uses Banach-Alaoglu theorem and the uniform convexity of $L^{p}$ spaces), we may assume that $\mu_{N}=\mu^{K_{N}}$ converges weakly to $\mu^{E}$, where $\mu_{N}$ and $\mu^{E}$ are the measures associated with $K_{N}$ and $E$ respectively as in Proposition \ref{nonlinear use in CSI}. Let $K\subseteq G$ be compact and $\sigma_{N}=\mu_{N}|_{K}$. Then $V_{\omega;\rho}^{\sigma_{N}}\leq 1$ on ${\rm supp}(\sigma_{N})\subseteq K\cap K_{N}$. By Proposition \ref{3214}, we have
\begin{align*}
\sigma_{N}(K)=\sigma_{N}(K\cap K_{N})\leq R_{\alpha,p;\rho}^{\omega}(K\cap K_{N})\leq R_{\alpha,p;\rho}^{\omega}(G\cap E),
\end{align*}
which implies that
\begin{align*}
\mu_{N}(G)=\sup_{\substack{K\subseteq G\\K~\text{compact}}}\sigma_{N}(K)\leq R_{\alpha,p;\rho}^{\omega}(G\cap E).
\end{align*}
Let $\varphi\in C_{0}(G)$ be such that $0\leq\varphi\leq 1$ on $G$. Then
\begin{align*}
\int_{G}\varphi(x)d\mu^{E}(x)=\lim_{N\rightarrow\infty}\int_{G}\varphi_{N}(x)d\mu_{N}(x)\leq\liminf_{N\rightarrow\infty}\mu_{N}(G)\leq R_{\alpha,p;\rho}^{\omega}(G\cap E).
\end{align*}
Use the fact that $\mu^{E}(G)$ is the supremum of $\int_{G}\varphi(x)d\mu^{E}(x)$ over all such $\varphi$ (see \cite[Chapter II, Section 3]{DN}). This proves the estimate (\ref{3611}).
\end{proof}

\begin{lemma}\label{3610}
Let $n\in\mathbb{N}$, $0<\alpha<n$, $1<p<\infty$, $0<\rho<\infty$, and $\omega\in A_{p}^{\rm loc}$. If $(\ref{361})$ holds for some $a\in\mathbb{R}^{n}$, then 
\begin{align*}
\int_{0}^{\rho}\left(\frac{t^{\alpha p}R_{\alpha,p;\rho}^{\omega}(Q_{t}(a))}{\omega(Q_{t}(a))}\right)^{p'-1}\frac{dt}{t}=\infty.
\end{align*}
\end{lemma}

\begin{proof}
For any $0<t\leq 2\rho$, we define
\begin{align*}
\varphi(t)&=R_{\alpha,p;\rho}^{\omega}(E\cap Q_{\rho}(a))^{p'-1}\\
\psi(t)&=\left(\frac{t^{\alpha p}}{\omega(Q_{t}(a))}\right)^{p'-1}\frac{1}{t},
\end{align*}
and
\begin{align*}
\Psi(t)=\int_{t}^{2\rho}\psi(r)dr.
\end{align*}
In view of Theorem \ref{3310}, the term $R_{\alpha,p;\rho}^{\omega}(Q_{t}(a))^{p'-1}$ is comparable to $\Psi(t)^{-1}$ for $0<t\leq\rho$. It suffices to show that
\begin{align*}
\int_{0}^{\rho}\frac{\psi(t)}{\Psi(t)}dt=\infty.
\end{align*}
We note that
\begin{align*}
\int_{0}^{\rho}\frac{\psi(t)}{\Psi(t)}dt=\log\Psi(0)-\log\Psi(\rho)=\infty,
\end{align*}
since 
\begin{align*}
\Psi(0)=\int_{0}^{2\rho}\left(\frac{t^{\alpha p}}{\omega(Q_{t}(a))}\right)^{p'-1}\frac{dt}{t}=\infty,
\end{align*}
by assumption, and 
\begin{align*}
0<\Psi(\rho)=\int_{\rho}^{2\rho}\left(\frac{t^{\alpha p}}{\omega(Q_{t}(a))}\right)^{p'-1}\frac{1}{t}<\infty,
\end{align*}
which completes the proof.
\end{proof}

The next lemma says that the $R_{\alpha,p;\rho}^{\omega}(\cdot)$-capacitary density of a set is $0$ at every point where the set is $R_{\alpha,p;\rho}^{\omega}(\cdot)$-thin.
\begin{lemma}\label{369}
Let $n\in\mathbb{N}$, $0<\alpha<n$, $1<p<\infty$, $0<\rho<\infty$, $\omega\in A_{p}^{\rm loc}$, and $E\subseteq\mathbb{R}^{n}$. Suppose that $E$ is $R_{\alpha,p;\rho}^{\omega}(\cdot)$-thin at a point $a\in\mathbb{R}^{n}$. Then
\begin{align}\label{3612}
\lim_{r\rightarrow 0}\frac{R_{\alpha,p;\rho}^{\omega}(E\cap Q_{r}(a))}{R_{\alpha,p;\rho}^{\omega}(Q_{r}(a))}=0.
\end{align}
\end{lemma}

\begin{proof}
Let $\varphi$, $\psi$, and $\Psi$ be as in the proof of Lemma \ref{3610}. Note that $\Psi(2\rho)=0$ and $\Psi'(t)=-\psi(t)$. Integration by parts yields
\begin{align}\label{3613}
\int_{t}^{2\rho}\varphi(r)\psi(r)dr=\varphi(t)\Psi(t)+\int_{t}^{2\rho}\Psi(r)d\varphi(r).
\end{align}
Since $\varphi$ and $\Psi$ are nonnegative, it follows by (\ref{3612}) and (\ref{3613}) that 
\begin{align*}
\int_{t}^{2\rho}\Psi(r)d\varphi(r)\leq\int_{t}^{2\rho}\varphi(r)\psi(r)dr\leq\int_{0}^{2\rho}\varphi(r)\psi(r)dr<\infty,\quad 0<t\leq 2\rho.
\end{align*}
As $\varphi$ is nondecreasing, we have 
\begin{align*}
\int_{0}^{2\rho}\Psi(r)d\varphi(r)<\infty.
\end{align*}
Then we deduce from (\ref{3613}) that the limit $l=\lim\limits_{t\rightarrow 0}\varphi(t)\Psi(t)$ exists. If we can show that $l=0$, then (\ref{3612}) will follow by Theorem \ref{3310}. Suppose that $l>0$. Using Theorem \ref{3310}, we have
\begin{align*}
\frac{R_{\alpha,p;\rho}^{\omega}(E\cap Q_{r}(a))}{R_{\alpha,p;\rho}^{\omega}(Q_{r}(a))}\geq C\varphi(t)\Psi(t).
\end{align*} 
If $t$ is small enough, say, $0<t\leq t_{0}\leq\rho$, then 
\begin{align*}
R_{\alpha,p;\rho}^{\omega}(E\cap Q_{\rho}(a))^{p'-1}\geq\frac{Cl}{2}R_{\alpha,p;\rho}^{\omega}(Q_{\rho}(a))^{p'-1},
\end{align*}
which contradicts (\ref{362}), since
\begin{align*}
\int_{0}^{\rho}\left(\frac{t^{\alpha p}R_{\alpha,p;\rho}^{\omega}(E\cap Q_{t}(a))}{\omega(Q_{\rho}(a))}\right)^{p'-1}\frac{dt}{t}&\geq\int_{0}^{t_{0}}\left(\frac{t^{\alpha p}R_{\alpha,p;\rho}^{\omega}(E\cap Q_{t}(a))}{\omega(Q_{\rho}(a))}\right)^{p'-1}\frac{dt}{t}\\
&\geq\frac{Cl}{2}\int_{0}^{\rho}\left(\frac{t^{\alpha p}R_{\alpha,p;\rho}^{\omega}( Q_{t}(a))}{\omega(Q_{\rho}(a))}\right)^{p'-1}\frac{dt}{t}\\
&=\infty
\end{align*}
by Lemma \ref{3610}. Hence, we conclude that (\ref{3612}) holds.
\end{proof}

\begin{lemma}\label{lemma 3611}
Let $n\in\mathbb{N}$, $0<\alpha<n$, $1<p<\infty$, $0<\rho<\infty$, $\omega\in A_{p}^{\rm loc}$, and $E\subseteq\mathbb{R}^{n}$. Suppose that $a\in\overline{E}\cap e_{\alpha,p;\rho}^{\omega}(E)$. Assume that $\varepsilon>0$ and $a\in Q$ for some cube $Q$. Then
\begin{align*}
\mathcal{V}_{\omega;\rho}^{\mu^{E\cap Q}}(a)<\varepsilon
\end{align*}
if $Q$ is small enough, where $\mu^{E\cap Q}$ is the measure associated with $E\cap Q$ as in Proposition $\ref{nonlinear use in CSI}$.
\end{lemma}

\begin{proof}
Consider first that $E$ is a Borel set. Denote $\mu^{E\cap Q}$ by $\mu$ for simplicity. We have by (\ref{367}) that
\begin{align*}
\mathcal{V}_{\omega;\rho}^{\mu}(a)&\leq C\mathcal{W}_{\omega;2\rho}^{\mu}(a)\\
&=C\int_{0}^{2\rho}\left(\frac{t^{\alpha p}\mu(B_{t}(a))}{\omega(B_{t}(a))}\right)^{p'-1}\frac{dt}{t}\\
&=C\int_{0}^{\delta}\left(\frac{t^{\alpha p}\mu(Q_{t}(a))}{\omega(Q_{t}(a))}\right)^{p'-1}\frac{dt}{t}+C\int_{\delta}^{2\rho}\left(\frac{t^{\alpha p}\mu(Q_{t}(a))}{\omega(Q_{t}(a))}\right)^{p'-1}\frac{dt}{t},
\end{align*}
where $0<\delta\leq\rho$ is arbitrary. Now choose $\delta$ small enough such that 
\begin{align*}
\int_{0}^{\delta}\left(\frac{t^{\alpha p}R_{\alpha,p;\rho}^{\omega}(E\cap Q_{t}(a))}{\omega(Q_{\rho}(a))}\right)^{p'-1}\frac{dt}{t}<\varepsilon',
\end{align*}
where $\varepsilon'>0$ will be determined later. Then Lemma \ref{368} entails
\begin{align*}
\mu(Q_{t}(a))\leq R_{\alpha,p;\rho}^{\omega}(E\cap Q\cap Q_{t}(a))\leq R_{\alpha,p;\rho}^{\omega}(E\cap Q_{t}(a)),\quad 0<t<\infty,
\end{align*}
which yields
\begin{align*}
\int_{0}^{\delta}\left(\frac{t^{\alpha p}\mu(Q_{t}(a))}{\omega(Q_{t}(a))}\right)^{p'-1}\frac{dt}{t}\leq\int_{0}^{\delta}\left(\frac{t^{\alpha p}R_{\alpha,p;\rho}^{\omega}(E\cap Q_{t}(a))}{\omega(Q_{t}(a))}\right)^{p'-1}\frac{dt}{t}<C\varepsilon'.
\end{align*}
Now suppose that $Q\subseteq Q_{\delta}(a)$. Using Lemma \ref{368} once again, we have 
\begin{align*}
\mu(Q_{t}(a))\leq R_{\alpha,p;\rho}^{\omega}(E\cap Q\cap Q_{t}(a))\leq R_{\alpha,p;\rho}^{\omega}(E\cap Q_{t}(a)),\quad 0<t<\infty.
\end{align*}
Combining the above with Theorem \ref{3310}, we obtain
\begin{align*}
\int_{\delta}^{2\rho}\left(\frac{t^{\alpha p}\mu(Q_{t}(a))}{\omega(Q_{t}(a))}\right)^{p'-1}\frac{dt}{t}&\leq R_{\alpha,p;\rho}^{\omega}(E\cap Q_{\delta}(a))^{p'-1}\int_{\delta}^{2\rho}\left(\frac{t^{\alpha p}\mu(B_{t}(a))}{\omega(B_{t}(a))}\right)^{p'-1}\frac{dt}{t}\\
&\leq C'\left(\frac{R_{\alpha,p;\rho}^{\omega}(E\cap Q_{r}(a))}{R_{\alpha,p;\rho}^{\omega}(Q_{r}(a))}\right)^{p'-1}.
\end{align*}
In view of Lemma \ref{369}, we choose $\delta$ small enough such that 
\begin{align*}
C'\left(\frac{R_{\alpha,p;\rho}^{\omega}(E\cap Q_{r}(a))}{R_{\alpha,p;\rho}^{\omega}(Q_{r}(a))}\right)^{p'-1}<C\varepsilon'.
\end{align*}
Hence, we have
\begin{align*}
\mathcal{V}_{\omega;\rho}^{\mu}(a)<2C\varepsilon',
\end{align*}
and the assertion follows by taking $\varepsilon=\frac{\varepsilon}{2C}$.

Consider now that $E$ is a general set. We first claim that (\ref{362}) holds if and only if 
\begin{align}\label{need to claim this}
\sum_{k=0}^{\infty}\left(\frac{2^{-\alpha pk}R_{\alpha,p;\rho}^{\omega}\left(E\cap Q_{2^{-k\rho}}(a)\right)}{\omega(Q_{2^{-k\rho}}(a))}\right)^{p'-1}<\infty.
\end{align}
Assume that (\ref{362}) holds. Then
\begin{align*}
&\int_{0}^{\rho}\left(\frac{t^{\alpha p}R_{\alpha,p;\rho}^{\omega}(E\cap Q_{t}(a))}{\omega(Q_{t}(a))}\right)^{p'-1}\frac{dt}{t}\\
&=\sum_{k=0}^{\infty}\int_{\frac{\rho}{2^{k+1}}}^{\frac{\rho}{2^{k}}}\left(\frac{t^{\alpha p}R_{\alpha,p;\rho}^{\omega}(E\cap Q_{t}(a))}{\omega(Q_{t}(a))}\right)^{p'-1}\frac{dt}{t}\\
&\geq\rho^{\alpha p}\sum_{k=0}^{\infty}\int_{\frac{\rho}{2^{k+1}}}^{\frac{\rho}{2^{k}}}\left(\frac{2^{-\alpha p(k+1)}R_{\alpha,p;\rho}^{\omega}(E\cap Q_{2^{-(k+1)\rho}}(a))}{\omega(Q_{2^{-k\rho}}(a))}\right)^{p'-1}\frac{dt}{t}\\
&=(\log 2)\rho^{\alpha p}\sum_{k=0}^{\infty}\left(\frac{2^{-\alpha p(k+1)}R_{\alpha,p;\rho}^{\omega}(E\cap Q_{2^{-(k+1)\rho}}(a))}{\omega(Q_{2^{-k\rho}}(a))}\right)^{p'-1}\\
&\geq C(\log 2)\rho^{\alpha p}\sum_{k=0}^{\infty}\left(\frac{2^{-\alpha p(k+1)}R_{\alpha,p;\rho}^{\omega}(E\cap Q_{2^{-(k+1)\rho}}(a))}{\omega(Q_{2^{-(k+1)\rho}}(a))}\right)^{p'-1},
\end{align*}
where the last inequality follows by the local doubling property of $\omega$, which shows that (\ref{need to claim this}) holds. For the converse, using the local doubling property again, one has
\begin{align*}
&\int_{0}^{\rho}\left(\frac{t^{\alpha p}R_{\alpha,p;\rho}^{\omega}(E\cap Q_{t}(a))}{\omega(Q_{t}(a))}\right)^{p'-1}\frac{dt}{t}\\
&=\sum_{k=0}^{\infty}\int_{\frac{\rho}{2^{k+1}}}^{\frac{\rho}{2^{k}}}\left(\frac{t^{\alpha p}R_{\alpha,p;\rho}^{\omega}(E\cap Q_{t}(a))}{\omega(Q_{t}(a))}\right)^{p'-1}\frac{dt}{t}\\
&\leq\rho^{\alpha p}\sum_{k=0}^{\infty}\int_{\frac{\rho}{2^{k+1}}}^{\frac{\rho}{2^{k}}}\left(\frac{2^{-\alpha pk}R_{\alpha,p;\rho}^{\omega}(E\cap Q_{2^{-k\rho}}(a))}{\omega(Q_{2^{-(k+1)\rho}}(a))}\right)^{p'-1}\frac{dt}{t}\\
&\leq C(\log 2)\rho^{\alpha p}\sum_{k=0}^{\infty}\left(\frac{2^{-\alpha pk}R_{\alpha,p;\rho}^{\omega}(E\cap Q_{2^{-k\rho}}(a))}{\omega(Q_{2^{-k\rho}}(a))}\right)^{p'-1},
\end{align*}
which yields (\ref{362}), as expected.

We choose, for $k=0,1,\ldots$, $G_{\delta}$ sets $F_{k}\supseteq E\cap Q_{2^{-k\rho}}(a)$ such that 
\begin{align*}
R_{\alpha,p;\rho}^{\omega}(F_{k})=R_{\alpha,p;\rho}^{\omega}\left(E\cap Q_{2^{-k\rho}}(a)\right).
\end{align*}
Define 
\begin{align*}
F=\bigcap_{k=0}^{\infty}\left[F_{k}\cup\left(Q_{2^{-k\rho}}(a)\right)^{c}\right].
\end{align*}
The set $F$ is Borel and $F\supseteq E$. For each $k=0,1,\ldots$, it holds that 
\begin{align*}
E\cap Q_{2^{-k\rho}}(a)\subseteq F\cap Q_{2^{-k\rho}}(a)\subseteq F_{k}\cap Q_{2^{-k\rho}}(a),
\end{align*}
which yields 
\begin{align*}
R_{\alpha,p;\rho}^{\omega}\left(E\cap Q_{2^{-k\rho}}(a)\right)&\leq R_{\alpha,p;\rho}^{\omega}\left(F\cap Q_{2^{-k\rho}}(a)\right)\\
&\leq R_{\alpha,p;\rho}^{\omega}\left(F_{k}\cap Q_{2^{-k\rho}}(a)\right)\\
&\leq R_{\alpha,p;\rho}^{\omega}(F_{k})\\
&=R_{\alpha,p;\rho}^{\omega}\left(E\cap Q_{2^{-k\rho}}(a)\right),
\end{align*}
and hence $R_{\alpha,p;\rho}^{\omega}\left(F\cap Q_{2^{-k\rho}}(a)\right)=R_{\alpha,p;\rho}^{\omega}\left(E\cap Q_{2^{-k\rho}}(a)\right)$. The equivalence between (\ref{362}) and (\ref{need to claim this}) shows that $F$ is $R_{\alpha,p;\rho}^{\omega}(\cdot)$-thin at $a$. 

Now we choose a $G_{\delta}$ set $H$ such that $E\cap Q\subseteq H$ and $R_{\alpha,p;\rho}^{\omega}(E\cap Q)=R_{\alpha,p;\rho}^{\omega}(F)$. Since $E\subseteq F$, we have $R_{\alpha,p;\rho}^{\omega}(E\cap Q)=R_{\alpha,p;\rho}^{\omega}(F\cap Q)$. Let $\sigma$ be the measure associated with $F\cap Q$ as in Proposition \ref{nonlinear use in CSI}. Then \cite[Theorem 2.3.10]{AH} implies that $\mathcal{V}_{\omega;\rho}^{\mu}(a)=\mathcal{V}_{\omega;\rho}^{\nu}(a)$. Thus the lemma follows by the first part of the proof.
\end{proof}

The following theorem is known to be ``the Kellogg property".
\begin{theorem}\label{kellogg}
Let $n\in\mathbb{N}$, $0<\alpha<n$, $1<p<\infty$, $0<\rho<\infty$, and $\omega\in A_{p}^{\rm loc}$. Then
\begin{align*}
R_{\alpha,p;\rho}^{\omega}\left(E\cap e_{\alpha,p;\rho}^{\omega}(E)\right)=0.
\end{align*}
\end{theorem}

\begin{proof}
Let $\{Q_{N}\}_{N=1}^{\infty}$ be the enumeration of the cubes with rational vertices. Define $\mu_{N}=\mu^{E\cap Q_{N}}$ as in Lemma \ref{lemma 3611}. For each $N=1,2,\ldots$, let 
\begin{align*}
E_{N}=\left\{x\in E\cap Q_{N}:\mathcal{V}_{\omega;\rho}^{\mu_{N}}(x)<1\right\}.
\end{align*}
Then Proposition \ref{use nonlinear} and Theorem \ref{R cal} entail $R_{\alpha,p;\rho}^{\omega}(E_{N})=0$. The preceding lemma shows that 
\begin{align*}
E\cap e_{\alpha,p;\rho}^{\omega}(E)\subseteq\bigcup_{N=1}^{\infty}E_{N},
\end{align*} 
which proves the theorem.
\end{proof}

A different way in measuring the thinness of sets is given by 
\begin{align*}
\int_{0}^{\rho}\left(\frac{R_{\alpha,p;\rho}(E\cap Q_{t}(a))}{t^{n-\alpha p}}\right)^{p'-1}\frac{dt}{t}<\infty
\end{align*}
(see \cite[Definition 6.3.7]{AH}). The weighted analogue of the above estimate is given in the next theorem.
\begin{theorem}\label{proposition 3612}
Let $n\in\mathbb{N}$, $0<\alpha<n$, $1<p<\infty$, $0<\rho<\infty$, $\omega\in A_{p}^{\rm loc}$, and $E\subseteq\mathbb{R}^{n}$ be a Borel set. Suppose that $a\in e_{\alpha,p;\rho}^{\omega}(E)$. Then 
\begin{align*}
\int_{0}^{\rho}\left(\frac{t^{\alpha p}R_{\alpha,p;\rho}^{\omega}(E\cap Q_{t}(a))}{\omega(Q_{\rho}(a))}\right)^{p'-1}\frac{dt}{t}<\infty.
\end{align*}
\end{theorem}

\begin{proof}
It suffices to show that there exists a constant $C>0$ such that if $t$ is small enough, then 
\begin{align*}
R_{\alpha,p;t}^{\omega}(E\cap Q_{t}(a))\leq CR_{\alpha,p;\rho}^{\omega}(E\cap Q_{t}(a)).
\end{align*}
Assume without loss of generality that $R_{\alpha,p;t}^{\omega}(E\cap Q_{t}(a))>0$. Let $0<t\leq\frac{\rho}{2}$. By Proposition \ref{323}, there is a probability measure $\mu\in\mathcal{M}^{+}(E\cap Q_{t}(a))$ such that 
\begin{align*}
\|I_{\alpha,3t}\ast\mu\|_{L_{\omega'}^{p'}(\mathbb{R}^{n})}\leq 2R_{\alpha,p;3t}^{\omega}(E\cap Q_{t}(a))^{-\frac{1}{p}}.
\end{align*}
Using the same proposition again, one has 
\begin{align*}
R_{\alpha,p;\rho}^{\omega}(E\cap Q_{t}(a))^{-\frac{1}{p}}\leq\|I_{\alpha,\rho}\ast\mu\|_{L_{\omega'}^{p'}(\mathbb{R}^{n})}.
\end{align*}
We have
\begin{align*}
\|I_{\alpha,\rho}\ast\mu\|_{L_{\omega'}^{p'}(\mathbb{R}^{n})}&\leq\left(\int_{|x-a|_{\infty}<2t}(I_{\alpha,\rho}\ast\mu)(x)^{p'}\omega'(x)dx\right)^{\frac{1}{p'}}\\
&\qquad+\left(\int_{|x-a|_{\infty}\geq 2t}(I_{\alpha,\rho}\ast\mu)(x)^{p'}\omega'(x)dx\right)^{\frac{1}{p'}},
\end{align*}
and by Minkowski's inequality, it holds that
\begin{align*}
&\left(\int_{|x-a|_{\infty}\geq 2t}(I_{\alpha,\rho}\ast\mu)(x)^{p'}\omega'(x)dx\right)^{\frac{1}{p'}}\\
&\leq\int_{|y-a|_{\infty}<t}\left(\int_{|x-y|_{\infty}<\rho,|x-a|\geq 2t}\frac{\omega'(x)}{|x-y|_{\infty}^{(n-\alpha)p'}}dx\right)d\mu(y).
\end{align*}
Now we estimate the last integral. Suppose that $|y-a|_{\infty}<t$. Note that 
\begin{align*}
&|x-y|_{\infty}\geq\frac{1}{2}|x-a|_{\infty},\quad\text{if}~|x-a|_{\infty}\geq 2t,\\
&|x-a|_{\infty}<\frac{3\rho}{2},\quad\text{if}~|x-y|_{\infty}<\rho.
\end{align*}
Combining the above with (\ref{248}), we obtain
\begin{align*}
&\left(\int_{|x-a|_{\infty}\geq 2t}(I_{\alpha,\rho}\ast\mu)(x)^{p'}\omega'(x)dx\right)^{\frac{1}{p'}}\\
&\leq 2^{n-\alpha}\mu(E\cap Q_{t}(a))\left(\int_{2t\leq|x-a|_{\infty}<2\rho}\frac{\omega'(x)}{|x-a|_{\infty}^{(n-\alpha)p'}}dx\right)^{\frac{1}{p'}}\\
&=2^{n-\alpha}\left(\int_{2t}^{2\rho}\frac{\omega'(Q_{s}(a))}{s^{(n-\alpha)p'}}\frac{ds}{s}+\frac{\omega'(Q_{2\rho}(a))}{(2\rho)^{(n-\alpha)p'}}-\frac{\omega'(Q_{2t}(a))}{(2t)^{(n-\alpha)p'}}\right)^{\frac{1}{p'}}\\
&\leq C\left(\int_{2t}^{2\rho}\frac{\omega'(Q_{s}(a))}{s^{(n-\alpha)p'}}\frac{ds}{s}\right)^{\frac{1}{p'}},
\end{align*}
where the last inequality follows by the fact that 
\begin{align*}
\int_{2t}^{2\rho}\frac{\omega'(Q_{s}(a))}{s^{(n-\alpha)p'}}\frac{ds}{s}\geq\int_{\rho}^{2\rho}\frac{\omega'(Q_{s}(a))}{s^{(n-\alpha)p'}}\frac{ds}{s}\geq C\frac{\omega'(Q_{\rho}(a))}{(2\rho)^{(n-\alpha)p'}}\geq C'\frac{\omega'(Q_{2\rho}(a))}{(2\rho)^{(n-\alpha)p'}},
\end{align*}
which is a consequence of the local doubling property of $\omega$. Using Theorem \ref{3310} and Lemma \ref{369} together with the assumption, we obtain
\begin{align*}
\left(\int_{2t}^{2\rho}\frac{\omega'(Q_{s}(a))}{s^{(n-\alpha)p'}}\frac{ds}{s}\right)^{\frac{1}{p'}}&\leq \left(\int_{t}^{2\rho}\frac{\omega'(Q_{s}(a))}{s^{(n-\alpha)p'}}\frac{ds}{s}\right)^{\frac{1}{p'}}\\
&\leq R_{\alpha,p;\rho}^{\omega}(Q_{t}(a))^{-\frac{1}{p}}\\
&=\left(\frac{R_{\alpha,p;\rho}^{\omega}(E\cap Q_{t}(a))}{R_{\alpha,p;\rho}^{\omega}(Q_{t}(a))}\right)^{\frac{1}{p}}R_{\alpha,p;\rho}^{\omega}(E\cap Q_{t}(a))^{-\frac{1}{p}}\\
&\leq\frac{1}{2}R_{\alpha,p;\rho}^{\omega}(E\cap Q_{t}(a))^{-\frac{1}{p}}
\end{align*}
for $t$ sufficiently closed to $0$. We thus have shown that 
\begin{align}\label{3617}
R_{\alpha,p;\rho}^{\omega}(E\cap Q_{t}(a))^{-\frac{1}{p}}\leq 2\left(\int_{|x-a|_{\infty}<2t}(I_{\alpha,\rho}\ast\mu)(x)^{p'}\omega'(x)dx\right)^{\frac{1}{p'}}
\end{align}
if $t$ is small enough. If $|x-a|_{\infty}<2t$ and $|y-a|_{\infty}<t$, then $|x-y|_{\infty}<3t$. Hence
\begin{align*}
\int_{|x-a|_{\infty}<2t}(I_{\alpha,\rho}\ast\mu)(x)^{p'}\omega'(x)dx&\leq\int_{\mathbb{R}^{n}}(I_{\alpha,3t}\ast\mu)(x)^{p'}\omega'(x)dx\\
&\leq 2^{p'}R_{\alpha,p;3t}^{\omega}(E\cap Q_{t}(a))^{-\frac{p'}{p}}.
\end{align*}
Combining the above with (\ref{3617}), one shows that
\begin{align*}
R_{\alpha,p;3t}^{\omega}(E\cap Q_{t}(a))\leq CR_{\alpha,p;\rho}^{\omega}(E\cap Q_{t}(a)).
\end{align*}
Using Proposition \ref{338}, we have 
\begin{align*}
R_{\alpha,p;t}^{\omega}(E\cap Q_{t}(a))&\leq C(n,\alpha,p,[\omega]_{40t})R_{\alpha,p;3t}^{\omega}(E\cap Q_{t}(a))\\
&\leq C(n,\alpha,p,[\omega]_{20\rho})R_{\alpha,p;3t}^{\omega}(E\cap Q_{t}(a))
\end{align*}
for $0<t<\frac{\rho}{2}$, where the last inequality follows by Remark \ref{very crucial remark}.
\end{proof}

The converse of Theorem \ref{proposition 3612} also holds, as the following shown.
\begin{theorem}
Let $n\in\mathbb{N}$, $0<\alpha<n$, $1<p<\infty$, $0<\rho<\infty$, $\omega\in A_{p}^{\rm loc}$, and $E\subseteq\mathbb{R}^{n}$. Suppose that $a\in E$ satisfies
\begin{align*}
\int_{0}^{\rho}\left(\frac{t^{\alpha p}}{\omega(Q_{\rho}(a))}\right)^{p'-1}\frac{dt}{t}<\infty.
\end{align*}
Then
\begin{align}\label{3619}
\int_{0}^{\rho}\left(\frac{t^{\alpha p}R_{\alpha,p;\rho}^{\omega}(E\cap Q_{t}(a))}{\omega(Q_{\rho}(a))}\right)^{p'-1}\frac{dt}{t}=\infty.
\end{align}
\end{theorem}

\begin{proof}
Recall the function $\psi$
\begin{align*}
\psi(t)=\left(\frac{t^{\alpha p}}{\omega(Q_{t}(a))}\right)^{p'-1}\frac{1}{t},\quad 0<t\leq 2\rho
\end{align*}
defined as in Lemma \ref{3610}. Then $\psi$ is continuous on $(0,2\rho]$ and $\int_{0}^{2\rho}\psi(t)dt<\infty$. Moreover, if $0<\delta\leq\rho$, then 
\begin{align*}
\frac{1}{2}\int_{\delta}^{\rho}\frac{\psi(t)}{\displaystyle\int_{0}^{2t}\psi(r)dr}dt=\log\left(\int_{0}^{2\rho}\psi(t)dt\right)-\log\left(\int_{0}^{2\delta}\psi(t)dt\right).
\end{align*}
As $\delta\rightarrow 0$, the right-sided of the above identity tends to $\infty$.

On the other hand, since 
\begin{align*}
R_{\alpha,p;t}^{\omega}(Q_{\delta}(a))\leq R_{\alpha,p;t}^{\omega}\left(\overline{Q}_{\delta}(a)\right)\leq R_{\alpha,p;t}^{\omega}\left(Q_{\delta+1/N}(a)\right),\quad \delta>0,\quad N=1,2,\ldots,
\end{align*}
in view of Theorem \ref{3310}, we have 
\begin{align*}
R_{\alpha,p;t}^{\omega}\left(\overline{Q}_{\delta}(a)\right)\approx\left(\int_{\delta}^{2t}\psi(r)dr\right)^{-1},\quad 0<\delta\leq t.
\end{align*}
Since $\overline{Q}_{\delta}(a)$ is decreasing to $\{a\}$ as $\delta\downarrow 0$, using Proposition \ref{decreasing compact}, one has 
\begin{align*}
R_{\alpha,p;t}^{\omega}(\{a\})\approx\left(\int_{0}^{2t}\psi(r)dr\right)^{-1},\quad 0<t<\infty,
\end{align*}
which yields 
\begin{align*}
\int_{0}^{\rho}\left(\frac{t^{\alpha p}R_{\alpha,p;\rho}^{\omega}(E\cap Q_{t}(a))}{\omega(Q_{\rho}(a))}\right)^{p'-1}\frac{dt}{t}&\geq\int_{0}^{\rho}\left(\frac{t^{\alpha p}R_{\alpha,p;\rho}^{\omega}(\{a\})}{\omega(Q_{\rho}(a))}\right)^{p'-1}\frac{dt}{t}\\
&\approx\int_{0}^{\rho}\frac{\psi(t)}{\displaystyle\int_{0}^{2t}\psi(r)dr}dt\\
&=\infty,
\end{align*}
and the theorem follows.
\end{proof}

\begin{remark}
\rm Let $S=\{a\in E: (\ref{3619})~\text{does not hold}\}$. Since 
\begin{align*}
\int_{0}^{\rho}\left(\frac{t^{\alpha p}R_{\alpha,p;\rho}^{\omega}(E\cap Q_{1}(a))}{\omega(Q_{\rho}(a))}\right)^{p'-1}\frac{dt}{t}\leq\int_{0}^{\rho}\left(\frac{t^{\alpha p}R_{\alpha,p;\rho}^{\omega}(E\cap Q_{t}(a))}{\omega(Q_{\rho}(a))}\right)^{p'-1}\frac{dt}{t},
\end{align*}
the theorem just proved gives $S\subseteq E\cap e_{\alpha,p;\rho}^{\omega}(E)$. By ``the Kellogg property", Theorem \ref{kellogg}, we have $R_{\alpha,p;\rho}^{\omega}(S)=0$.
\end{remark}

\end{document}